\documentclass[lettersize, 11pt]{article}
\usepackage{etex}
\usepackage{graphicx, bbm,  latexsym, amsfonts, epsfig, verbatim, amsmath, algorithm, algpseudocode, color, epsfig, verbatim, enumerate, multirow, caption, pgfkeys}
\usepackage{url, multicol, color, latexsym, amsfonts, epsfig, verbatim, tikz, array, subeqnarray, amssymb}
\usepgflibrary{arrows}
\usetikzlibrary{arrows}
\usepgflibrary{decorations.pathmorphing}
\usetikzlibrary{decorations.pathmorphing}
\usepackage{booktabs}
\usepackage[title]{appendix}
\usepackage{hyperref}

\newtheorem{theorem}{Theorem}

\newtheorem{claim}{Claim}

\newtheorem{proposition}{Proposition}
\newtheorem{remark}{Remark}

\newcommand{\tabincell}[2]{\begin{tabular}{@{}#1@{}}#2\end{tabular}}

\definecolor{darkblue}{rgb}{0.0,0.0,1.0}
\hypersetup{colorlinks,breaklinks,
            linkcolor=darkblue,urlcolor=darkblue,
            anchorcolor=darkblue,citecolor=darkblue}

\def\B{{\mathbb{B}}}
\def\R{{\mathbb{R}}}
\def\Z{{\mathbb{Z}}}
\def\Gb{{\mathbb{G}}}
\def\V{{\mathbb{V}}}

\def\Cupper{{\overline C}}
\def\Clower{{\underline C}}
\def\Vupper{{\overline V}}

\definecolor{myblue}{RGB}{3,70,148}

\def\fblue{\textcolor{black}}
\def\fgreen{\textcolor{black}}
\def\sred{\textcolor{black}}
\def\sblue{\textcolor{blue}}

\allowdisplaybreaks

\makeatletter
\newcommand{\algmargin}{\the\ALG@thistlm}
\makeatother
\newlength{\whilewidth}
\algdef{SE}[parWHILE]{parWhile}{EndparWhile}[1]
  {\parbox[t]{\dimexpr\linewidth-\algmargin}{%
     \hangindent\whilewidth\strut\algorithmicwhile\ #1\ \algorithmicdo\strut}}{\algorithmicend\ \algorithmicwhile}%
\algnewcommand{\parState}[1]{\State%
  \parbox[t]{\dimexpr\linewidth-\algmargin}{\strut #1\strut}}

\begin{document}

\title{ {\bf A Polyhedral Study of the Integrated Minimum-Up/-Down Time and Ramping Polytope\footnote{Earlier versions are available online at \url{http://www.optimization-online.org/DB_HTML/2015/06/4942.html} and \url{http://www.optimization-online.org/DB_HTML/2015/08/5070.html}. }}\\[3mm]
\author{\normalsize {\bf Kai Pan} and {\bf Yongpei Guan}
\\ \\
{\small Department of Industrial and Systems Engineering}\\
{\small University of Florida, Gainesville, FL 32611}\\
{\small Emails: kpan@ufl.edu; guan@ise.ufl.edu}\\
} }
\date{\textcolor{red}{\small Submitted: July 10, 2015; Revised: April 1, 2016}}

\maketitle

\vspace{-0.3in}

\begin{abstract}
\setlength{\baselineskip}{18pt}
In this paper, we study the polyhedral structure of an integrated minimum-up/-down time and ramping polytope, which has broad applications in variant industries. The polytope we studied includes minimum-up/-down time, generation ramp-up/-down rate, logical, and generation upper/lower bound constraints. By exploring its specialized structures, we derive strong valid inequalities and explore a new proof technique to prove these inequalities are sufficient to provide convex hull descriptions for variant two-period and three-period polytopes, under different parameter settings. For multi-period cases, we derive generalized strong valid inequalities (including one, two, and three continuous variables, respectively) and further prove that these inequalities are facet-defining under mild conditions. Moreover, we discover efficient polynomial time separation algorithms for these inequalities to improve the computational efficiency. Finally, extensive computational experiments are conducted to verify the effectiveness of our proposed strong valid inequalities by testing the applications of these inequalities to solve both {self-scheduling and network-constrained} unit commitment problems, for which our derived approach outperforms the default CPLEX significantly. 

\vspace{0.10in}

\end{abstract}

\setlength{\baselineskip}{22pt}

\section{Introduction} \label{sec:introduction}
In this paper, we study the polyhedral structure of the integrated minimum-up/-down time and ramping polytope, which is the fundamental polytope describing the physical characteristics of operating large machines and production systems. This polytope describe the operations in which when the system/machine is turned on, it should stay {online} in certain time periods, which is referred as the minimum-up time. Similarly, when the system/machine is turned off, it should stay {offline} in certain time periods, which is referred as the minimum-down time. Meanwhile, when the system is {online}, besides the generation upper/lower bound restrictions, the difference between the generation amounts in two consecutive time periods is bounded from above, which is referred as the ramp-up/-down rate restrictions.

To mathematically describe the integrated polytope, we let $T$ be the number of time periods for the whole operational horizon, $L$ ($\ell$) be the minimum-up (-down) time limit, $\Cupper$ ($\Clower$) be the generation upper (lower) bound when the machine is {online}, $\Vupper$ be the start-up/shut-down ramp rate {(which is usually between $\Clower$ and $\Cupper$, i.e., $\Clower \leq \Vupper \leq \Cupper$)}, and $V$ be the ramp-up/-down rate in the stable generation region. In addition, we let binary decision variable $y$ represent the machine's {online} (i.e., $y_t = 1$) or {offline} (i.e., $y_t = 0$) status, and continuous decision variable $x$ represent the generation amount. We add an additional binary decision variable $u$ to represent whether the machine starts up (i.e., $u_t = 1$) or not (i.e., $u_t = 0$). The corresponding integrated minimum-up/-down time and ramping polytope can be described as follows:
\begin{subeqnarray}
& P := \Bigl\{ (x, y, u) \in \R_+^{T} \times \B^{T} \times \B^{T-1} :  & \sum_{i = t- L +1}^{t} u_i \leq y_t, \ \forall t \in [L + 1, T]_{\Z}, \slabel{eqn:p-minup} \\
&& \sum_{i = t- \ell +1}^{t} u_i \leq 1 - y_{t- \ell}, \ \forall t \in [\ell + 1, T]_{\Z}, \slabel{eqn:p-mindn} \\
&& y_t - y_{t-1} - u_t \leq 0, \ \forall t \in [2, T]_{\Z}, \slabel{eqn:p-udef} \\
&&  - x_t + \Clower y_t \leq 0, \ \forall t \in [1, T]_{\Z}, \slabel{eqn:p-lower-bound} \\
&& x_t - \overline{C} y_t \leq 0, \ \forall t \in [1, T]_{\Z}, \slabel{eqn:p-upper-bound} \\
&& x_t - x_{t-1} \leq V y_{t-1} + \Vupper (1 - y_{t-1}), \ \forall t \in [2, T]_{\Z}, \slabel{eqn:p-ramp-up} \\
&& x_{t-1} - x_t \leq V y_t + \Vupper (1 - y_t), \ \forall t \in [2, T]_{\Z} \slabel{eqn:p-ramp-down} \Bigr\},
\end{subeqnarray}
where constraints \eqref{eqn:p-minup} and \eqref{eqn:p-mindn} describe the minimum-up and minimum-down time limits \cite{lee2004min, rajan2005minimum}, respectively (i.e., if the machine starts up at {time} $t-L+1$, it should {stay} online in the following $L$ consecutive time periods until {time} $t$; if the machine shuts down at time $t-\ell+1$, it should {stay} offline in the following $\ell$ consecutive time periods until {time} $t$), constraints \eqref{eqn:p-udef} describe the logical relationship between $y$ and $u$, constraints \eqref{eqn:p-lower-bound} and \eqref{eqn:p-upper-bound} describe the generation lower and upper bounds, and constraints \eqref{eqn:p-ramp-up} and \eqref{eqn:p-ramp-down} describe the generation ramp-up and ramp-down rate limits. For notation convenience, we define $[a,b]_{\Z}$ with $a < b$ as the set of integer numbers between integers $a$ and $b$, i.e., $\{ a, a+1, \cdots, b \}$, and let conv($P$) represent the convex hull {description} of $P$. 

One well-known application of this fundamental polytope is the unit commitment (UC) problem (see, e.g., \cite{frangioni2006solving, dubost2005primal, muckstadt1977application}, and \cite{sagastizabal2012divide}, among others) in power generation scheduling. The UC problem decides the unit (referred to a thermal power generation unit) commitment status (online/offline) and power generation amount at each time period for each unit over a finite discrete time horizon so as to satisfy the load with a minimum total cost, with the associated physical restrictions, including generation upper/lower limits, ramp-rate limits, and minimum-up/-down time limits, to be satisfied. Due to its importance for power system operations, significant research progress has been made in terms of developing dynamic programming algorithms, Lagrangian relaxation/decomposition, and heuristic approaches. Readers are referred to \cite{padhy2004unit, saravanan2013solution} for the detailed reviews of these approaches. Among these approaches, the Lagrangian relaxation approach~\cite{dubost2005primal, muckstadt1977application} has been commonly adopted in industry, due to its advantages of targeting large-scale instances by decomposing the unit commitment problem into a group of subproblems with each subproblem solved by an efficient dynamic programming algorithm. However, the Lagrangian relaxation approach cannot guarantee to provide an optimal or even a feasible solution at the termination. On the other hand, mixed-integer linear programming (MILP) approaches can guarantee to obtain an optimal solution~\cite{nemhauser2013ip}. Considering this, MILP approaches have recently been adopted by all wholesale electricity markets in US \cite{carlson2012miso} and creates more than $500$ million annual savings \cite{bixby2010mixed}. 

{Besides power generation scheduling problems, this polytope has broader applications in other fields. Since the polytope captures the fundamental characteristics of minimum-up/-down time, ramp-up/-down rate, and capacity constraints, the polyhedral study results in this paper can be applied to any problems with one or two or all of these three characteristics. In fact, there are many engineering problems with these characteristics and the corresponding polytope $P$ embedded. For instance, we can observe the production ramping and capacity constraints in the production smoothing problems \cite{silver1967tutorial, pekelman1975production}, ramp-up production planning problems \cite{haller2003cycle}, and the refrigerating problems through an expansion valve \cite{outtagarts1997transient}. Moreover, all of these three characteristics appear in the operations of chemical pumps \cite{lin2002optimal, xia2007endoreversible} and the boilers \cite{lakshmanan2009study, havel2013optimal}. Meanwhile, besides the traditional thermal generators, polytope $P$ is also embedded in the models describing the operations of hydro/pump-storage hydro generators \cite{guan1999scheduling,lohndorf2013optimizing} and combined-cycle units~\cite{liu2009component}.}

Considering the fact that cutting planes are efficient approaches to strength MILP formulations and speed up the corresponding branch-and-cut algorithm~\cite{nw}, strong formulations for the polytope {$P$} have been crucial to help improve the computational efficiency to solve the application problem with {$P$} embedded. Along this direction, there has been research progress on developing cutting planes for the related polytopes with a part of constraints described in {$P$}. For instance, for the polytopes with only minimum-up/-down time constraints (i.e., only constraints~\eqref{eqn:p-minup} and~\eqref{eqn:p-mindn}), in \cite{lee2004min}, alternating up/down inequalities are proposed to strengthen the polytope without considering the start-up binary decision variables. In \cite{rajan2005minimum}, the convex hull of the minimum-up/-down time polytope with the {auxiliary} start-up binary decision variables considered. Recently, new families of strong valid inequalities are proposed in \cite{ostrowski2012tight, damci1777polyhedral} to tighten the ramping polytope (i.e., constraints~\eqref{eqn:p-ramp-up} and~\eqref{eqn:p-ramp-down}) of the unit commitment problem. 

{In the above studies, the minimum-up/-down time and ramping polytopes are studied separately. In this paper, we study the integrated polytope including both minimum-up/-down time and ramping constraints. Our main contributions can be summarized as follows:
\begin{itemize}
\setlength\itemsep{0em}
\item[(1)] Our studied integrated polytope extends and generalizes the minimum-up/-down time polytope only and ramping polytope only studies to consider both aspects in one polytope. 
\item[(2)] For the integrated polytope, we derive the convex hull descriptions for two- and three-period cases under different parameter settings. The derived strong valid inequalities can be applied to help solve multi-period cases. More importantly, we provide a new technique to prove the convex hull descriptions. 
\item[(3)] For the multi-period cases, we derive strong valid inequalities that are facet-defining and can be separated in polynomial time. These inequalities can help speed up the branch-and-cut algorithm significantly to solve general multiple period cases.
\item[(4)] Extensive computational studies for the network-constrained and self-scheduling unit commitment problems verify the effectiveness of the proposed strong valid inequalities for the polytope as cutting planes to help improve the computational efficiency to solve the related application problems. 
\end{itemize}}

In the remaining part of this paper, we provide the convex hull descriptions for two- and three-period cases in Section~\ref{sec:convexhullresults}. In Section \ref{sec:multi-period}, we {extend our study to derive} strong valid inequalities covering multiple time periods {so as} to further strengthen the general multi-period polytopes. In Section \ref{sec:comp-exper}, we perform computational {studies} on its applications in network-constrained and self-scheduling unit commitment problems to verify the effectiveness of our proposed strong valid inequalities. Finally, we conclude {our study} in Section \ref{sec:conclusion}.

\section{Convex Hulls} \label{sec:convexhullresults}
\subsection{Two-period Convex Hulls} \label{sec:two-period}
Before we provide the convex hull description of the two-period polytope, to the reader's attention, the convex hull descriptions for the separated studies on the {two-period} ramp-up only polytope (i.e., $\sred{P_2^{\mbox{\tiny up}}} := \{ (x, y, u) \in \sred{\R_+^{2} \times \B^{2} \times \B:}~\eqref{eqn:p-minup},~\eqref{eqn:p-mindn},~\eqref{eqn:p-udef},~\eqref{eqn:p-lower-bound}, \eqref{eqn:p-upper-bound}, \eqref{eqn:p-ramp-up}\}$) and {two-period} ramp-down only polytope (i.e., $\sred{P_2^{\mbox{\tiny down}}} := \{ (x, y, u) \in \sred{\R_+^{2} \times \B^{2} \times \B:}~\eqref{eqn:p-minup},~\eqref{eqn:p-mindn},~\eqref{eqn:p-udef},~\eqref{eqn:p-lower-bound}, \eqref{eqn:p-upper-bound}, \eqref{eqn:p-ramp-down} \}$) are provided in \cite{damci1777polyhedral}. Here we provide the convex hull description for two-period $P$ with {both ramp-up and ramp-down constraints} as follows:
\begin{theorem}\label{thm1}
For $T=2$ and $L=\ell=1$, {when $\Cupper - \Clower - V \geq 0$}, conv($P$) can be described as 
\begin{subeqnarray}
& {Q_2}:= \Bigl\{ (x, y, u) \in \R^5:  & u_2 \geq 0, \ u_2 \geq y_2 - y_1, \slabel{eqn-q2:u2-1} \\
&& u_2 \leq y_2, \ y_1 + u_2 \leq 1, \slabel{eqn-q2:u2-2} \\
&& x_1 \geq \Clower y_1, \ x_2 \geq \Clower y_2, \slabel{eqn-q2:lb} \\
&& x_1 \leq \Vupper y_1 + (\Cupper - \Vupper) (y_2 - u_2), \slabel{eqn-q2:x1-ub} \\
&& x_2 \leq \Cupper y_2 - (\Cupper - \Vupper) u_2, \slabel{eqn-q2:x2-ub} \\
&& x_2 - x_1 \leq (\Clower + V) y_2 - \Clower y_1 - (\Clower + V - \Vupper) u_2, \slabel{eqn-q2:x2-x1-ub} \\
&& x_1 - x_2 \leq \Vupper y_1 - (\Vupper - V) y_2 - (\Clower + V - \Vupper) u_2 \slabel{eqn-q2:x1-x2-ub} \Bigr\};
\end{subeqnarray}
{when $\Cupper - \Clower - V < 0$, conv($P$) can be described as ${\bar{Q}_2}=\{ (x, y, u) \in \R^5:~\eqref{eqn-q2:u2-1} -\eqref{eqn-q2:x2-ub}\}$}.
\end{theorem}
\begin{proof}
We omit the proof here since the same {proof} technique is applied to prove {the later on more complicated} Theorem \ref{thm:conv-t3l2} in Section \ref{sec:three-period}.
\end{proof}

\begin{remark}
Since the start-up decision is not considered in the {first time} period in $Q_2$, we do not need to consider the cases in which $L = 2$ or $\ell=2$ and it also follows that the strong valid inequalities in $Q_2$ (e.g., \eqref{eqn-q2:x1-ub} - \eqref{eqn-q2:x1-x2-ub}) can be applied to any two consecutive time periods {for the multi-period cases}.
\end{remark}

\subsection{Three-period Convex Hulls} \label{sec:three-period}
In this subsection, we {further} perform the polyhedral study for the three-period formulation, i.e., $T = 3$ in $P$, and {propose convex hull descriptions for variant cases under the following valid parameter settings (\textbf{five} in total).
\begin{itemize}
\item \textbf{Case 1}: $\Clower {\leq} \Vupper < \Clower + V$, $\Cupper - \Vupper - V \geq 0$ (which implies $\Cupper - \Clower - V \geq 0$);
\item \textbf{Case 2}: $\Clower {\leq} \Vupper < \Clower + V$, $\Cupper - \Clower - V \geq 0$,  and $\Cupper - \Vupper - V < 0$;
\item \textbf{Case 3}: $\Clower {\leq} \Vupper < \Clower + V$, $\Cupper - \Clower - V < 0$ (which implies $\Cupper - \Vupper - V < 0$);
\item \textbf{Case 4}: $\Vupper \geq \Clower + V$, $\Cupper - \Vupper - V \geq 0$ (which implies $\Cupper - \Clower - V \geq 0$);
\item \textbf{Case 5}: $\Vupper \geq \Clower + V$, $\Cupper - \Clower - V \geq 0$,  and $\Cupper - \Vupper - V < 0$.
\end{itemize}}

{Note here that the case in which $\Vupper \geq \Clower + V$ and $\Cupper - \Clower - V < 0$ (which implies $\Cupper - \Vupper - V < 0$) does not need to be considered since $\Cupper \geq \Vupper$, which ensures that $\Vupper \geq \Clower + V$ guarantees $\Cupper \geq \Vupper \geq \Clower + V$.}

{We first consider \textbf{Case 1}, which is the most common case for thermal generators.} {Furthermore, since} the derived strong valid inequalities are different for $\Cupper - \Clower - 2 V \geq 0$ and $\Cupper - \Clower - 2 V < 0$ cases, we {separate the study for each case and} start with the case in which $\Cupper - \Clower - 2 V \geq 0$. {In addition,} {we consider different minimum-up/-down time limits, including $L = \ell = 2$, $L=\ell=1$, $L=2$ and $\ell=1$, and $L=1$ and $\ell=2$.} {We first study the case in which $L=\ell=2$ before extending our study to other minimum-up/-down time settings.} {Under the setting $L = \ell = 2$ (with $\Cupper - \Clower - 2 V \geq 0$)},
the corresponding {original polytope} can be described as follows:
\begin{subeqnarray}
& P_3^2 := \Bigl\{& (x, y, u) \in \R_+^3 \times \B^3 \times \B^2 : \nonumber \\
&& u_2 + u_3 \leq y_3, \slabel{eqn:p-minup-t3l2} \\
&& y_1 + u_2 + u_3 \leq 1, \slabel{eqn:p-mindn-t3l2} \\
&& u_2 \geq y_2 - y_1, \ u_3 \geq y_3 - y_2, \slabel{eqn:p-udef-t3l2} \\
&& x_1 \geq \Clower y_1, \ x_2 \geq \Clower y_2, \ x_3 \geq \Clower y_3, \slabel{eqn:p-lower-bound-t3l2} \\
&& x_1 \leq \Cupper y_1, \ x_2 \leq \Cupper y_2, \ x_3 \leq \Cupper y_3, \slabel{eqn:p-upper-bound-t3l2} \\
&& x_2 - x_1 \leq V y_1 + \Vupper (1 - y_1), \ x_3 - x_2 \leq V y_2 + \Vupper (1 - y_2), \slabel{eqn:p-ramp-up-t3l2} \\
&& x_1 - x_2 \leq V y_2 + \Vupper (1 - y_2), \ x_2 - x_3 \leq V y_3 + \Vupper (1 - y_3) \slabel{eqn:p-ramp-down-t3l2} \Bigr\}.
\end{subeqnarray}

{For $P_3^2$, we first provide the strong valid inequalities in the following proposition. Then we provide a linear programming description $Q_3^2$ and further prove that $Q_3^2$ provides the convex hull description for $P_3^2$.}

\begin{proposition} \label{prop:valid-t3l2}
{For $P_3^2$}, the following inequalities
\begin{align}
 x_1 & \leq \Vupper y_1 + V (y_2 - u_2) + (\Cupper - \Vupper - V) (y_3 - u_3 - u_2), \label{eqn:p-x1-ub-t3l2} \\
 x_2 & \leq \Vupper y_2 + (\Cupper - \Vupper) (y_3 - u_3 - u_2), \label{eqn:p-x2-ub-t3l2} \\
 x_3 & \leq \Cupper y_3 - (\Cupper - \Vupper) u_3 - (\Cupper - \Vupper - V) u_2, \label{eqn:p-x3-ub-t3l2} \\
 x_2 - x_1 & \leq \Vupper y_2 - \Clower y_1 + (\Clower + V - \Vupper) (y_3 - u_3 - u_2), \label{eqn:p-x2-x1-ub-t3l2} \\
 x_3 - x_2 & \leq (\Clower + V) y_3 - \Clower y_2 - (\Clower + V - \Vupper) u_3, \label{eqn:p-x3-x2-ub-t3l2} \\
 x_1 - x_2 & \leq \Vupper y_1 - (\Vupper - V) y_2 - (\Clower + V - \Vupper) u_2, \label{eqn:p-x1-x2-ub-t3l2} \\
 x_2 - x_3 & \leq \Vupper y_2 - \Clower y_3 + (\Clower + V - \Vupper) (y_3 - u_3 - u_2), \label{eqn:p-x2-x3-ub-t3l2} \\
 x_3 - x_1 & \leq (\Clower + 2 V) y_3 - \Clower y_1 - (\Clower + 2 V - \Vupper) u_3 - (\Clower + V - \Vupper) u_2, \label{eqn:p-x3-x1-ub-t3l2} \\
 x_1 - x_3 & \leq \Vupper y_1 - \Clower y_3 + V (y_2 - u_2) + (\Clower + V - \Vupper) (y_3 - u_3 - u_2), \label{eqn:p-x1-x3-ub-t3l2} \\
 x_1 - x_2 + x_3 & \leq \Vupper y_1 - (\Vupper - V) y_2 + \Vupper y_3 + (\Cupper - \Vupper) (y_3 - u_3 - u_2), \label{eqn:p-x1-x2+x3-ub-t3l2} 
\end{align}
are valid for conv($P_3^2$).
\end{proposition}
\begin{proof}
{To prove the validity of \eqref{eqn:p-x1-ub-t3l2}, which essentially tightens $x_1 \leq \Cupper y_1$ in \eqref{eqn:p-upper-bound-t3l2}, we show how \eqref{eqn:p-x1-ub-t3l2} is obtained and {accordingly illustrate} the corresponding insights. Since $\Cupper y_1 = \Vupper y_1 + (\Cupper - \Vupper) y_1$, $x_1 \leq \Cupper y_1$ can be tightened to be
\begin{equation}
x_1 \leq \Vupper y_1 + (\Cupper - \Vupper) (y_2 - u_2) \label{inner-valid-t2-1}
\end{equation}
because $0 \leq y_2 - u_2 \leq y_1$ due to \eqref{eqn:p-udef-t3l2}. It is easy to observe that \eqref{inner-valid-t2-1} is valid when $y_2 - u_2 = y_1$, which reduces back to $x_1 \leq \Cupper y_1$. We only need to consider the case in which $y_2 - u_2 = 0$ and $y_1= 1$, i.e., $y_2 = 0$, from which \eqref{inner-valid-t2-1} reduces to $x_1 \leq \Vupper$, which is valid due to ramp-down constraints \eqref{eqn:p-ramp-down-t3l2} as the machine shuts down from the first time period {(online)} to the second time period {(offline)}.}

{Furthermore, {we can rewrite} $\Vupper y_1 + (\Cupper - \Vupper) (y_2 - u_2) = \Vupper y_1 + V (y_2 - u_2) + (\Cupper - \Vupper - V) (y_2 - u_2)$. {Thus}, \eqref{inner-valid-t2-1} can be further tightened to be 
\begin{equation}
x_1 \leq \Vupper y_1 + V (y_2 - u_2) + (\Cupper - \Vupper - V) (y_3 - u_3 - u_2) \label{inner-valid-t2-2}
\end{equation}
because $0 \leq y_3 - u_3 \leq y_2$ due to \eqref{eqn:p-udef-t3l2}. Similarly, it is easy to observe that \eqref{inner-valid-t2-2} is valid when $y_3 - u_3 = y_2$, which reduces back to \eqref{inner-valid-t2-1}. We only need to consider the case in which $y_3 - u_3 = 0$ and $y_2= 1$, which further leads to {$u_3=y_3=0$ and thus} $y_1=1$ due to minimum-up constraint \eqref{eqn:p-minup-t3l2} and {accordingly} $x_1 \leq \Vupper + V$, which is also valid due to ramp-down constraints \eqref{eqn:p-ramp-down-t3l2} as the machine shuts down from the second time period {(online)} to the third time period {(offline)}. Since we obtain inequality \eqref{inner-valid-t2-2} as exactly inequality \eqref{eqn:p-x1-ub-t3l2}, the validity proof of \eqref{eqn:p-x1-ub-t3l2} is done.}

{Similar argument as above for \eqref{eqn:p-x1-ub-t3l2} can be applied to prove that inequalities \eqref{eqn:p-x2-ub-t3l2} and \eqref{eqn:p-x1-x3-ub-t3l2} are valid and thus we omit the corresponding proofs here.}

{For \eqref{eqn:p-x3-ub-t3l2}, it is clearly valid when $y_3=0$ since it leads to $u_2=u_3=0$ due to \eqref{eqn:p-minup-t3l2}. When $y_3=1$,  we prove the validity of \eqref{eqn:p-x3-ub-t3l2} by considering when the machine starts up. Note here that there is at most one start-up due to \eqref{eqn:p-minup-t3l2}. If there is no start-up, then we have $y_1=y_2=y_3=1$ due to \eqref{eqn:p-udef-t3l2} and it follows that $x_3 \leq \Cupper$ from \eqref{eqn:p-x3-ub-t3l2}, which is clearly valid; otherwise, because of ramp-up constraints \eqref{eqn:p-ramp-up-t3l2}, then $x_3 \leq \Vupper + V$ if the machine starts up {in} the second time period (i.e., $u_2=1$) and $x_3 \leq \Vupper$ if the machine starts up {in} the third time period (i.e., $u_3=1$). It follows that we have
\begin{align*}
x_3 & \leq \Vupper u_3 + (\Vupper+V)u_2 \\
& = \Cupper (u_3+u_2) - (\Cupper - \Vupper) u_3 - (\Cupper - \Vupper - V) u_2 \\
& \leq \Cupper y_3 - (\Cupper - \Vupper) u_3 - (\Cupper - \Vupper - V) u_2,
\end{align*}
which indicates that \eqref{eqn:p-x3-ub-t3l2} is valid.}

{To prove the validity of \eqref{eqn:p-x2-x1-ub-t3l2}, which essentially tightens $x_2 - x_1 \leq V y_1 + \Vupper (1 - y_1)$ in \eqref{eqn:p-ramp-up-t3l2}, we show how \eqref{eqn:p-x2-x1-ub-t3l2} is obtained and {accordingly illustrate} the corresponding insights. Due to ramp-up process characteristics and $\Vupper < \Clower+V$, we have 
\begin{equation}
x_2 - x_1 \leq (\Clower+V)y_2 - \Clower y_1. \label{inner-valid-t2-3}
\end{equation}
Since $(\Clower+V)y_2 = \Vupper y_2 + (\Clower+V-\Vupper)y_2$ and $y_2 \geq y_3 - u_3 -u_2 \geq 0$ due to \eqref{eqn:p-udef-t3l2} and $u_2 \geq 0$, \eqref{inner-valid-t2-3} can be tightened to be 
\begin{equation}
x_2 - x_1 \leq \Vupper y_2 - \Clower y_1 + (\Clower+V-\Vupper) (y_3 - u_3 -u_2), \label{inner-valid-t2-4}
\end{equation}
which is clearly valid when $y_3 - u_3 -u_2 = y_2$. When $y_3 - u_3 -u_2 < y_2$, i.e., $y_2=1$ and $y_3=u_3+u_2$, it follows that $y_1=1, \ y_3=0$ or $y_1=0, \ y_3=1$. If $y_1=1$ and $y_3=0$, then the machine shuts down from the second time period to the third time period; else {if} $y_1=0$ and $y_3=1$, then the machine starts up at the second time period. In whichever case, we have $x_2 \leq \Vupper$ due to ramp-up/-down constraints and therefore \eqref{inner-valid-t2-4} is valid since it reduces to $x_2 - x_1 \leq \Vupper - \Clower y_1$. Since we obtain inequality \eqref{inner-valid-t2-4} as exactly inequality \eqref{eqn:p-x2-x1-ub-t3l2}, the validity proof of \eqref{eqn:p-x2-x1-ub-t3l2} is done.}

{Similar argument as above for \eqref{eqn:p-x2-x1-ub-t3l2} can be applied to prove that inequalities \eqref{eqn:p-x3-x2-ub-t3l2} - \eqref{eqn:p-x3-x1-ub-t3l2} are valid and thus we omit the corresponding proofs here.}

{To prove the validity of \eqref{eqn:p-x1-x2+x3-ub-t3l2}, we show how \eqref{eqn:p-x1-x2+x3-ub-t3l2} is obtained and {accordingly illustrate} the corresponding insights. Due to ramp-down process characteristics and $\Vupper < \Clower+V$, we have
\begin{equation}
x_1 - x_2 \leq \Vupper y_1 - (\Vupper - V) y_2. \nonumber
\end{equation}
It follows that 
\begin{equation}
x_1 - x_2 + x_3 \leq \Vupper y_1 - (\Vupper - V) y_2 + \Cupper y_3 \label{inner-valid-t2-5}
\end{equation}
because $x_3 \leq \Cupper y_3$ due to \eqref{eqn:p-upper-bound-t3l2}.
Since $\Cupper y_3 = \Vupper y_3 + (\Cupper - \Vupper) y_3$ and $u_2, \ u_3 \geq 0$, \eqref{inner-valid-t2-5} can be further tightened to be 
\begin{equation}
x_1 - x_2 + x_3 \leq \Vupper y_1 - (\Vupper - V) y_2 + \Vupper y_3 + (\Cupper - \Vupper) (y_3 - u_3 -u_2), \label{inner-valid-t2-6}
\end{equation}
which is clearly valid when $u_2 = u_3 = 0$. When $u_2 = 1$ or $u_3 =1$ (note {here} that there is at most one start-up due to minimum-up time constraint \eqref{eqn:p-minup-t3l2}), we have $y_3 - u_3 -u_2 = 0$ and \eqref{inner-valid-t2-6} reduces to be $x_3 - x_2 \leq \Vupper y_3 - (\Vupper - V) y_2$, which is clearly valid due to ramp-up process characteristics. Since we obtain inequality \eqref{inner-valid-t2-6} as exactly inequality \eqref{eqn:p-x1-x2+x3-ub-t3l2}, the validity proof of \eqref{eqn:p-x1-x2+x3-ub-t3l2} is done.}
\end{proof}

Now, through utilizing inequalities \eqref{eqn:p-x1-ub-t3l2} - \eqref{eqn:p-x1-x2+x3-ub-t3l2}, we introduce the linear {programming} description of conv($P_3^2$) by adding trivial inequalities as follows:
\begin{eqnarray}
& Q_3^2 := \Bigl\{ & (x, y, u) \in \R^8 :  \eqref{eqn:p-minup-t3l2} - \eqref{eqn:p-lower-bound-t3l2}, \eqref{eqn:p-x1-ub-t3l2} - \eqref{eqn:p-x1-x2+x3-ub-t3l2}, \nonumber \\
&& u_2 \geq 0, \ u_3 \geq 0 \label{eqn:u2-3period-t3l2} \Bigr\}.
\end{eqnarray}
{Note here that} the nonnegativity of $x$ {in $Q_3^2$} is guaranteed by \eqref{eqn:p-minup-t3l2}, \eqref{eqn:p-udef-t3l2} - \eqref{eqn:p-lower-bound-t3l2}, and \eqref{eqn:u2-3period-t3l2}. {In the following}, we show that $Q_3^2$ describes the convex hull of $P_3^2$, i.e., $Q_3^2 =$ conv($P_3^2$).

\begin{proposition} \label{prop:t3l2-full}
$Q_3^2$ is full-dimensional.
\end{proposition}
\begin{proof}
We prove that dim($Q_3^2$) = $8$, because there are eight decision variables in $Q_3^2$. The details are provided in Appendix \ref{apx:sec:t3l2-full}.
\end{proof}

\begin{proposition} \label{prop:t3l2-int-extpoint}
$Q_3^2 \subseteq conv(P_3^2)$.
\end{proposition}
\begin{proof}
It is sufficient to prove that every point $z \in Q_3^2$ can be written as $z = \sum_{s \in S} \lambda_s z^s$ for some $\lambda_s \geq 0$ and $\sum_{s \in S} \lambda_s = 1$, where $z^s \in {P_3^2}, \ s \in S$ and $S$ is the index set for the candidate points.

For a given point $z = (\bar{x}_1, \bar{x}_2, \bar{x}_3, \bar{y}_1, \bar{y}_2, \bar{y}_3, \bar{u}_2, \bar{u}_3) \in Q_3^2$, we let the candidate points $z^1, z^2, \cdots, z^6 \in {P_3^2}$ in the forms such that  $z^1 = (\hat{x}_1, 0, 0, 1, 0, 0, 0, 0)$, $z^2 = (\hat{x}_2, \hat{x}_3, 0, 1, 1, 0, 0, 0)$, $z^3 = (\hat{x}_4, \hat{x}_5, \allowbreak \hat{x}_6, \allowbreak 1, \allowbreak 1, \allowbreak 1, 0, 0)$, $z^4 = (0, \hat{x}_7, \hat{x}_8, 0, 1, 1, 1, 0)$, $z^5 = (0, 0, \hat{x}_9, 0, 0, 1, 0, 1)$, and $z^6 = (0, 0, 0, 0, 0, 0, 0, 0)$, {where $\hat{x}_i, i = 1, \cdots, 9$ are to be decided later}. {Meanwhile}, we let 
\begin{subeqnarray} \label{eqn:lamdba}
&& \lambda_1 = \bar{y}_1 - \bar{y}_2 + \bar{u}_2, \ \lambda_2 = \bar{y}_2 - \bar{y}_3 + \bar{u}_3, \ \lambda_3 = \bar{y}_3 - \bar{u}_2 - \bar{u}_3, \\
&& \lambda_4 = \bar{u}_2, \ \lambda_5 = \bar{u}_3, \mbox{ and } \lambda_6 = 1 - \bar{y}_1 - \bar{u}_2 - \bar{u}_3. 
\end{subeqnarray} 
First of all, {based on this construction, we can check} that $\sum_{s = 1}^{6} \lambda_s = 1$ and $\lambda_s \geq 0$ for $\forall s = 1, \cdots, 6$ due to \eqref{eqn:p-minup-t3l2} - \eqref{eqn:p-udef-t3l2} and \eqref{eqn:u2-3period-t3l2}.
{Meanwhile}, it {can be checked} that $\bar{y}_i = y_i(z) = \sum_{s=1}^{6} \lambda_s y_i(z^s)$ for $i = 1, 2, 3$ and $\bar{u}_i = u_i(z) = \sum_{s=1}^{6} \lambda_s u_i(z^s)$ for $i = 2, 3$, {where $y_i(z)$ represents the $\bar{y}_i$ component value in the given point $z$ and $u_i(z)$ represents the $\bar{u}_i$ component value in the given point $z$}. 

{Thus, in the remaining part of this proof, we only need to} decide the values of $\hat{x}_i$ for $i = 1, \cdots, 9$ {such that} $\bar{x}_i = x_i(z) = \sum_{s=1}^{6} \lambda_s x_i(z^s)$ for $i = 1, 2, 3$, i.e., 
\begin{align}
 \bar{x}_1 = \lambda_1 \hat{x}_1 + \lambda_2 \hat{x}_2 + \lambda_3 \hat{x}_4, \ \bar{x}_2 = \lambda_2 \hat{x}_3 + \lambda_3 \hat{x}_5 + \lambda_4 \hat{x}_7, \ \bar{x}_3 = \lambda_3 \hat{x}_6 + \lambda_4 \hat{x}_8 + \lambda_5 \hat{x}_9. \label{eqn-C: x1x2x3bar}
\end{align}

{To show~\eqref{eqn-C: x1x2x3bar}, in the following, we prove that for any $(\bar{x}_1, \bar{x}_2, \bar{x}_3)$ in its feasible region corresponding to a given $(\bar{y}_1, \bar{y}_2, \bar{y}_3, \bar{u}_2, \bar{u}_3)$, we can always find a $(\hat{x}_1, \hat{x}_2, \cdots, \hat{x}_9)$ in its feasible region, corresponding to the same given $(\bar{y}_1, \bar{y}_2, \bar{y}_3, \bar{u}_2, \bar{u}_3)$. Now we describe the feasible regions for $(\hat{x}_1, \hat{x}_2, \cdots, \hat{x}_9)$ and $(\bar{x}_1, \bar{x}_2, \bar{x}_3)$, respectively. First, since}
$y$ and $u$ in $z^1, \cdots, z^6$ are given, by substituting $z^1, \cdots, z^6$ into ${P_3^2}$, the corresponding feasible region for $(\hat{x}_1, \hat{x}_2, \cdots, \hat{x}_9)$ can be described as set $A = \{ (\hat{x}_1, \hat{x}_2, \allowbreak \cdots, \allowbreak \hat{x}_9) \in \R^9: \Clower \leq \hat{x}_1 \leq \Vupper, \ \Clower \leq \hat{x}_2 \leq \Vupper + V, \  \Clower \leq \hat{x}_3 \leq \Vupper, \ -V \leq \hat{x}_3 - \hat{x}_2 \leq \Vupper - \Clower, \  \Clower \leq \hat{x}_4 \leq \Cupper, \ \Clower \leq \hat{x}_5 \leq \Cupper, \ \Clower \leq \hat{x}_6 \leq \Cupper, \ -V \leq \hat{x}_5 - \hat{x}_4 \leq V, \ -V \leq \hat{x}_6 - \hat{x}_5 \leq V, \ \Clower \leq \hat{x}_7 \leq \Vupper, \ \Clower \leq \hat{x}_8 \leq \Vupper + V, \ \Clower - \Vupper \leq \hat{x}_8 - \hat{x}_7 \leq V, \ \Clower \leq \hat{x}_9 \leq \Vupper \}$.
{Second, corresponding to a given $(\bar{y}_1, \bar{y}_2, \bar{y}_3, \bar{u}_2, \bar{u}_3)$, following the description of $Q_3^2$, the feasible region for $(\bar{x}_1, \bar{x}_2, \bar{x}_3)$ can be described as follows:}
\begin{subeqnarray}
& \hspace{-0.5in} C =  \Bigl\{  (\bar{x}_1, \bar{x}_2, \bar{x}_3) \in \R^3: & \bar{x}_1 \geq \Clower \bar{y}_1, \ \bar{x}_2 \geq \Clower \bar{y}_2, \ \bar{x}_3 \geq \Clower \bar{y}_3, \slabel{eqn-C:p-lower-bound-t3l2} \\
&& \bar{x}_1  \leq \Vupper \bar{y}_1 + V (\bar{y}_2 - \bar{u}_2) + (\Cupper - \Vupper - V) (\bar{y}_3 - \bar{u}_3 - \bar{u}_2), \slabel{eqn-C:p-x1-ub-t3l2} \\
&& \bar{x}_2  \leq \Vupper \bar{y}_2 + (\Cupper - \Vupper) (\bar{y}_3 - \bar{u}_3 - \bar{u}_2), \slabel{eqn-C:p-x2-ub-t3l2} \\
&& \bar{x}_3  \leq \Cupper \bar{y}_3 - (\Cupper - \Vupper) \bar{u}_3 - (\Cupper - \Vupper - V) \bar{u}_2, \slabel{eqn-C:p-x3-ub-t3l2} \\
&& \bar{x}_2 - \bar{x}_1  \leq \Vupper \bar{y}_2 - \Clower \bar{y}_1 + (\Clower + V - \Vupper) (\bar{y}_3 - \bar{u}_3 - \bar{u}_2), \slabel{eqn-C:p-x2-x1-ub-t3l2} \\
&& \bar{x}_3 - \bar{x}_2  \leq (\Clower + V) \bar{y}_3 - \Clower \bar{y}_2 - (\Clower + V - \Vupper) \bar{u}_3, \slabel{eqn-C:p-x3-x2-ub-t3l2} \\
&& \bar{x}_1 - \bar{x}_2  \leq \Vupper \bar{y}_1 - (\Vupper - V) \bar{y}_2 - (\Clower + V - \Vupper) \bar{u}_2, \slabel{eqn-C:p-x1-x2-ub-t3l2} \\
&& \bar{x}_2 - \bar{x}_3  \leq \Vupper \bar{y}_2 - \Clower \bar{y}_3 + (\Clower + V - \Vupper) (\bar{y}_3 - \bar{u}_3 - \bar{u}_2), \slabel{eqn-C:p-x2-x3-ub-t3l2} \\
&& \bar{x}_3 - \bar{x}_1  \leq (\Clower + 2 V) \bar{y}_3 - \Clower \bar{y}_1 - (\Clower + 2 V - \Vupper) \bar{u}_3 - (\Clower + V - \Vupper) \bar{u}_2, \slabel{eqn-C:p-x3-x1-ub-t3l2} \\
&& \bar{x}_1 - \bar{x}_3  \leq \Vupper \bar{y}_1 - \Clower \bar{y}_3 + V (\bar{y}_2 - \bar{u}_2) + (\Clower + V - \Vupper) (\bar{y}_3 - \bar{u}_3 - \bar{u}_2), \slabel{eqn-C:p-x1-x3-ub-t3l2} \\
&& \bar{x}_1 - \bar{x}_2 + \bar{x}_3  \leq \Vupper \bar{y}_1 - (\Vupper - V) \bar{y}_2 + \Vupper \bar{y}_3 + (\Cupper - \Vupper) (\bar{y}_3 - \bar{u}_3 - \bar{u}_2) \slabel{eqn-C:p-x1-x2+x3-ub-t3l2} 
\Bigr\}.
\end{subeqnarray}

{Accordingly, we can set up the linear transformation $F$ from $(\hat{x}_1, \hat{x}_2, \cdots, \hat{x}_9) \in A$ to $(\bar{x}_1, \bar{x}_2, \bar{x}_3) \in C$ as follows:}
\begin{equation}       
F = \left(                 
  \begin{array}{ c   c  c   c c  c  c  c c}   
    \lambda_1 & \lambda_2 & 0 & \lambda_3 & 0 & 0 & 0 & 0 & 0  \\  
    0 & 0 & \lambda_2 & 0 & \lambda_3 & 0 & \lambda_4 & 0 & 0 \\  
    0 & 0 & 0 & 0 & 0 & \lambda_3 & 0 & \lambda_4 & \lambda_5 \\  
  \end{array}
\right), \nonumber
\end{equation}
{where $\lambda_1, \lambda_2, \cdots, \lambda_5$ follow the definitions described in \eqref{eqn:lamdba}. Thus, in the following, we only need to prove that $F: A \rightarrow C$ is surjective.}

Since $C$ is a closed and bounded polytope, any point can be expressed as a convex combination of the extreme points in $C$. 
Accordingly, we only need to show that for any extreme point $w^i \in C$ ($i = 1, \cdots, M$), there exists a point $p^i \in A$ such that $F p^i = w^i$, where $M$ represents the number of extreme points in $C$ (because for an arbitrary point $w \in C$, which can be {represented} as $w = \sum_{i=1}^M \mu_i w^i$ and $\sum_{i=1}^M \mu_i = 1$, there exists $p = \sum_{i=1}^M \mu_i p_i \in A$ such that $F p = w$ due to the linearity of $F$ and the convexity of $A$). 
Since it is difficult to enumerate all the extreme points in $C$, in the following proof we show the conclusion holds for any point in the faces of $C$, i.e., satisfying one of \eqref{eqn-C:p-lower-bound-t3l2} - \eqref{eqn-C:p-x1-x2+x3-ub-t3l2}  at equality, which implies the conclusion holds for extreme points.

\textbf{Satisfying $\bar{x}_1 \geq \Clower \bar{y}_1$ at equality}. For this case, substituting $\bar{x}_1 = \Clower \bar{y}_1$ into  \eqref{eqn-C:p-x1-ub-t3l2} - \eqref{eqn-C:p-x1-x2+x3-ub-t3l2}, we obtain the feasible region of $(\bar{x}_2, \bar{x}_3)$ as $C' = \{(\bar{x}_2, \bar{x}_3) \in \R^2: \Clower \bar{y}_2 \leq \bar{x}_2 \leq \Vupper \bar{y}_2 + (\Clower + V - \Vupper) (\bar{y}_3 - \bar{u}_3 - \bar{u}_2), \ \Clower \bar{y}_3 \leq  \bar{x}_3 \leq (\Clower + 2 V) \bar{y}_3 - (\Clower + 2 V - \Vupper) \bar{u}_3 - (\Clower + V - \Vupper) \bar{u}_2, \ \bar{x}_3 - \bar{x}_2  \leq (\Clower + V) \bar{y}_3 - \Clower \bar{y}_2 - (\Clower + V - \Vupper) \bar{u}_3 \}$.

First, by letting $\hat{x}_1 = \hat{x}_2 = \hat{x}_4 = \Clower$, it is easy to check that $\bar{x}_1 = \lambda_1 \hat{x}_1 + \lambda_2 \hat{x}_2 + \lambda_3 \hat{x}_4$. {Then \eqref{eqn-C: x1x2x3bar} holds for $\bar{x}_1$}. 
Note here that once $(\hat{x}_1, \hat{x}_2, \hat{x}_4)$ {is} fixed, the corresponding feasible region for $(\hat{x}_3, \hat{x}_5, \hat{x}_6, \hat{x}_7, \hat{x}_8, \hat{x}_9)$ can be described as set $A' = \{ (\hat{x}_3, \hat{x}_5, \hat{x}_6, \hat{x}_7, \hat{x}_8, \hat{x}_9) \in \R^6: \Clower \leq \hat{x}_3 \leq \Vupper, \  \Clower \leq \hat{x}_5 \leq \Clower + V, \ \Clower \leq \hat{x}_6 \leq \Cupper, \ -V \leq \hat{x}_6 - \hat{x}_5 \leq V, \ \Clower \leq \hat{x}_7 \leq \Vupper, \ \Clower \leq \hat{x}_8 \leq \Vupper + V, \ \Clower - \Vupper \leq \hat{x}_8 - \hat{x}_7 \leq V, \ \Clower \leq \hat{x}_9 \leq \Vupper \}$. In the following, we repeat the argument above to consider that one of inequalities in $C'$ is satisfied at equality to obtain the values of $(\hat{x}_3, \hat{x}_5, \hat{x}_6, \hat{x}_7, \hat{x}_8, \hat{x}_9)$ from $A'$.
\begin{enumerate}[1)]
\item Satisfying $\bar{x}_2 \geq \Clower \bar{y}_2$ at equality. We obtain $ \bar{x}_3 \in C'' = \{ \bar{x}_3 \in \R: \Clower \bar{y}_3 \leq  \bar{x}_3 \leq (\Clower + V) \bar{y}_3 - (\Clower + V - \Vupper) \bar{u}_3 \}$ through substituting $\bar{x}_2 = \Clower \bar{y}_2$ into $C'$. By letting $\hat{x}_3 = \hat{x}_5 = \hat{x}_7 = \Clower$, we have $\bar{x}_2 = \lambda_2 \hat{x}_3 + \lambda_3 \hat{x}_5 + \lambda_4 \hat{x}_7$. {Then \eqref{eqn-C: x1x2x3bar} holds for $\bar{x}_2$}. 
Thus, the corresponding feasible region for $(\hat{x}_6, \hat{x}_8, \hat{x}_9)$ can be described as set $A'' = \{ (\hat{x}_6, \hat{x}_8, \hat{x}_9) \in \R^3: \Clower \leq \hat{x}_6 \leq \Clower + V, \ \Clower \leq \hat{x}_8 \leq \Clower + V, \ \Clower \leq \hat{x}_9 \leq \Vupper \}$. If $\bar{x}_3 \geq \Clower \bar{y}_3$ is satisfied at equality, we let $\hat{x}_6 = \hat{x}_8 = \hat{x}_9 = \Clower$; if $\bar{x}_3 \leq (\Clower + V) \bar{y}_3 - (\Clower + V - \Vupper) \bar{u}_3$ is satisfied at equality, we let $\hat{x}_6 = \hat{x}_8 = \Clower + V$ and $\hat{x}_9 = \Vupper$. It is easy to check that $\bar{x}_3 = \lambda_3 \hat{x}_6 + \lambda_4 \hat{x}_8 + \lambda_5 \hat{x}_9$. {Then \eqref{eqn-C: x1x2x3bar} holds for $\bar{x}_3$}.

\item Satisfying $\bar{x}_2 \leq \Vupper \bar{y}_2 + (\Clower + V - \Vupper) (\bar{y}_3 - \bar{u}_3 - \bar{u}_2)$ at equality. We obtain $ \bar{x}_3 \in C'' = \{ \bar{x}_3 \in \R: \Clower \bar{y}_3 \leq  \bar{x}_3 \leq (\Clower + 2 V) \bar{y}_3 - (\Clower + 2 V - \Vupper) \bar{u}_3 - (\Clower + V - \Vupper) \bar{u}_2 \}$. By letting $\hat{x}_3 = \hat{x}_7 = \Vupper$ and $\hat{x}_5 = \Clower +V$, we have $\bar{x}_2 = \lambda_2 \hat{x}_3 + \lambda_3 \hat{x}_5 + \lambda_4 \hat{x}_7$. {Then \eqref{eqn-C: x1x2x3bar} holds for $\bar{x}_2$}. 
Thus, the corresponding feasible region for $(\hat{x}_6, \hat{x}_8, \hat{x}_9)$ can be described as set $A'' = \{ (\hat{x}_6, \hat{x}_8, \hat{x}_9) \in \R^3: \Clower \leq \hat{x}_6 \leq \Clower + 2V, \ \Clower \leq \hat{x}_8 \leq \Vupper + V, \ \Clower \leq \hat{x}_9 \leq \Vupper \}$. If $\bar{x}_3 \geq \Clower \bar{y}_3$ is satisfied at equality, we let $\hat{x}_6 = \hat{x}_8 = \hat{x}_9 = \Clower$; if $\bar{x}_3 \leq (\Clower + 2 V) \bar{y}_3 - (\Clower + 2 V - \Vupper) \bar{u}_3 - (\Clower + V - \Vupper) \bar{u}_2$ is satisfied at equality, we let $\hat{x}_6 = \Clower + 2V$, $\hat{x}_8 = \Vupper + V$ and $\hat{x}_9 = \Vupper$. 
{For both cases}, we have $\bar{x}_3 = \lambda_3 \hat{x}_6 + \lambda_4 \hat{x}_8 + \lambda_5 \hat{x}_9$. {Then \eqref{eqn-C: x1x2x3bar} holds for $\bar{x}_3$}.

\item Satisfying $\bar{x}_3 \geq \Clower \bar{y}_3$ at equality. We obtain $ \bar{x}_2 \in C'' = \{ \bar{x}_2 \in \R: \Clower \bar{y}_2 \leq  \bar{x}_2 \leq \Vupper \bar{y}_2 + (\Clower + V - \Vupper) (\bar{y}_3 - \bar{u}_3 - \bar{u}_2) \}$. By letting $\hat{x}_6 = \hat{x}_8 = \hat{x}_9 = \Clower$, we have $\bar{x}_3 = \lambda_3 \hat{x}_6 + \lambda_4 \hat{x}_8 + \lambda_5 \hat{x}_9$. {Then \eqref{eqn-C: x1x2x3bar} holds for $\bar{x}_3$}. 
Thus, the corresponding feasible region for $(\hat{x}_3, \hat{x}_5, \hat{x}_7)$ can be described as set $A'' = \{ (\hat{x}_3, \hat{x}_5, \hat{x}_7) \in \R^3: \Clower \leq \hat{x}_3 \leq \Vupper, \ \Clower \leq \hat{x}_5 \leq \Clower + V, \ \Clower \leq \hat{x}_7 \leq \Vupper \}$. 
If $\bar{x}_2 \geq \Clower \bar{y}_2$ is satisfied at equality, we let $\hat{x}_3 = \hat{x}_5 = \hat{x}_7 = \Clower$; if $\bar{x}_2 \leq \Vupper \bar{y}_2 + (\Clower + V - \Vupper) (\bar{y}_3 - \bar{u}_3 - \bar{u}_2)$ \fblue{is satisfied at equality}, we let $\hat{x}_3 = \hat{x}_7 = \Vupper$ and $\hat{x}_5 = \Clower +V$. {For both cases}, we have $\bar{x}_2 = \lambda_2 \hat{x}_3 + \lambda_3 \hat{x}_5 + \lambda_4 \hat{x}_7$. {Then \eqref{eqn-C: x1x2x3bar} holds for $\bar{x}_2$}.

\item Satisfying $\bar{x}_3 \leq (\Clower + 2 V) \bar{y}_3 - (\Clower + 2 V - \Vupper) \bar{u}_3 - (\Clower + V - \Vupper) \bar{u}_2$ at equality. We obtain $ \bar{x}_2 \in C'' = \{ \bar{x}_2 \in \R: \Clower \bar{y}_2 + V(\bar{y}_3 - \bar{u}_3 - \bar{u}_2) + (\Vupper - \Clower) \bar{u}_2 \leq  \bar{x}_2 \leq \Vupper \bar{y}_2 + (\Clower + V - \Vupper) (\bar{y}_3 - \bar{u}_3 - \bar{u}_2) \}$. By letting $\hat{x}_6 = \Clower + 2V$, $\hat{x}_8 = \Vupper + V$, and $\hat{x}_9 = \Vupper$, we have $\bar{x}_3 = \lambda_3 \hat{x}_6 + \lambda_4 \hat{x}_8 + \lambda_5 \hat{x}_9$. {Then \eqref{eqn-C: x1x2x3bar} holds for $\bar{x}_3$}. 
Thus, the corresponding feasible region for $(\hat{x}_3, \hat{x}_5, \hat{x}_7)$ can be described as set $A'' = \{ (\hat{x}_3, \hat{x}_5, \hat{x}_7) \in \R^3: \Clower \leq \hat{x}_3 \leq \Vupper, \hat{x}_5 = \Clower + V, \hat{x}_7 = \Vupper \}$. If $\bar{x}_2 \geq \Clower \bar{y}_2 + V(\bar{y}_3 - \bar{u}_3 - \bar{u}_2) + (\Vupper - \Clower) \bar{u}_2$ is satisfied at equality, we let $\bar{x}_3 = \Clower$; if $\bar{x}_2 \leq \Vupper \bar{y}_2 + (\Clower + V - \Vupper) (\bar{y}_3 - \bar{u}_3 - \bar{u}_2)$ is satisfied at equality, we let $\bar{x}_3 = \Vupper$.
{For both cases}, we have $\bar{x}_2 = \lambda_2 \hat{x}_3 + \lambda_3 \hat{x}_5 + \lambda_4 \hat{x}_7$. {Then \eqref{eqn-C: x1x2x3bar} holds for $\bar{x}_2$}.

\item Satisfying $\bar{x}_3 - \bar{x}_2  \leq (\Clower + V) \bar{y}_3 - \Clower \bar{y}_2 - (\Clower + V - \Vupper) \bar{u}_3$ at equality. We obtain $ \bar{x}_2 \in C'' = \{ \bar{x}_2 \in \R: \Clower \bar{y}_2 \leq \bar{x}_2 \leq \Clower \bar{y}_2 + V(\bar{y}_3 - \bar{u}_3) - (\Clower + V - \Vupper) \bar{u}_2 \}$ through substituting $\bar{x}_3 = \bar{x}_2  + (\Clower + V) \bar{y}_3 - \Clower \bar{y}_2 - (\Clower + V - \Vupper) \bar{u}_3$ into set $C'$. 
By letting $\hat{x}_3 = \Clower$, $\hat{x}_9 = \Vupper$, and $\hat{x}_6 - \hat{x}_5 = \hat{x}_8 - \hat{x}_7 = V$, we have $\bar{x}_3 - \bar{x}_2 = (\lambda_3 \hat{x}_6 + \lambda_4 \hat{x}_8 + \lambda_5 \hat{x}_9) - (\lambda_2 \hat{x}_3 + \lambda_3 \hat{x}_5 + \lambda_4 \hat{x}_7)$. 
If \fblue{$\bar{x}_2 \geq \Clower \bar{y}_2$ is satisfied at equality}, we let $\hat{x}_5 = \hat{x}_7 = \Clower$; if \fblue{$\bar{x}_2 \leq \Clower \bar{y}_2 + V(\bar{y}_3 - \bar{u}_3) - (\Clower + V - \Vupper) \bar{u}_2$ is satisfied at equality}, we let $\hat{x}_5 = \Clower + V$, and $\hat{x}_7 = \Vupper$. For both cases, we have $\bar{x}_2 = \lambda_2 \hat{x}_3 + \lambda_3 \hat{x}_5 + \lambda_4 \hat{x}_7$ and thus $\bar{x}_3 = \lambda_3 \hat{x}_6 + \lambda_4 \hat{x}_8 + \lambda_5 \hat{x}_9$. {Then \eqref{eqn-C: x1x2x3bar} holds for both $\bar{x}_2$ and $\bar{x}_3$}.
\end{enumerate}

Similar analyses hold for $\bar{x}_2 \geq \Clower \bar{y}_2$ and $\bar{x}_3 \geq \Clower \bar{y}_3$ due to the similar structure {among} $\bar{x}_1 \geq \Clower \bar{y}_1$, $\bar{x}_2 \geq \Clower \bar{y}_2$, and $\bar{x}_3 \geq \Clower \bar{y}_3$ and thus are omitted here.
Furthermore, similar analyses also hold for inequalities \eqref{eqn-C:p-x1-ub-t3l2}-\eqref{eqn-C:p-x1-x2+x3-ub-t3l2}, with the details included in Appendix \ref{apx:sec:int-extpoint}.
\end{proof}

\begin{theorem} \label{thm:conv-t3l2}
{For $T=3$, $L=\ell=2$, and $\Cupper - \Clower - 2V \geq 0$, the convex hull representation can be described as} $Q_3^2 = $ conv($P_3^2$).
\end{theorem}
\begin{proof}
Since all the inequalities in $Q_3^2$ are valid for conv($P_3^2$) from Propositions \ref{prop:valid-t3l2}, we have $Q_3^2 \supseteq$ conv($P_3^2$). Meanwhile, we have $Q_3^2 \subseteq$ conv($P_3^2$) from Proposition \ref{prop:t3l2-int-extpoint}. Thus $Q_3^2 =$ conv($P_3^2$).
\end{proof}

\begin{remark}
{Note here that we provide a new proof technique to show $Q_3^2 \subseteq$ conv($P_3^2$), which is the key part in this convex hull proof. Our proof utilizes the mapping $F: A \rightarrow C$ and considers the equality of each inequality in $C$, to develop a system of integral points {in} $P_3^2$ to represent any point in $Q_3^2$ in the form of a convex combination of these integral points. Therefore, this proof is different from the conventional proofs (e.g., total-unimodularity \cite{nw}, integral extreme points \cite{sherali2001convex, lee2004min, rajan2005minimum, damci1777polyhedral}, primal-dual proof \cite{barany1984uncapacitated, kuccukyavuz2009uncapacitated}, and optimality condition construction \cite{atamturk2006strong, shebalov2015lifting}) and the alternative proofs \cite{luedtke2009strategic, sullivan2014convex}, among others.}
\end{remark}

{Under the setting} $L = \ell = 1$ ({with} $\Cupper - \Clower - 2V \geq 0$), we can obtain the convex hull representation of the original polytope (e.g., defined as $P_3^{1}$) described as follows:
\begin{theorem}
For {$T=3$}, $L = \ell = 1$, and $\Cupper - \Clower - 2V \geq 0$, the convex hull representation {can be described as} 
\begin{subeqnarray}
& Q_3^1 = conv(P_3^1) = \Bigl\{ & (x, y, u) \in \R^8: \eqref{eqn:p-udef-t3l2} - \eqref{eqn:p-lower-bound-t3l2}, \eqref{eqn:u2-3period-t3l2}, \nonumber \\
&& u_2 \leq y_2, \ u_3 \leq y_3, \slabel{eqn:Q-3-1-mu}  \\
&& y_1 + u_2 \leq 1, \ y_2+ u_3 \leq 1, \slabel{eqn:Q-3-1-md} \\
&& x_1 \leq \Vupper y_1 + (\Cupper - \Vupper) (y_2 - u_2), \slabel{eqn:Q-3-1-x1} \\
&& x_1 \leq \Vupper y_1 + V (y_2 - u_2) + (\Cupper - \Vupper - V) (y_3 - u_3), \slabel{eqn:Q-3-1-x1-2} \\
&& x_2 \leq \Cupper y_2 - (\Cupper - \Vupper) u_2, \slabel{eqn:Q-3-1-x2-1} \\
&& x_2 \leq \Vupper y_2 + (\Cupper - \Vupper) (y_3 - u_3), \slabel{eqn:Q-3-1-x2} \\
&& x_3 \leq \Cupper y_3 - (\Cupper - \Vupper) u_3, \slabel{eqn:Q-3-1-x3-1} \\
&& x_3 \leq (\Vupper + V) y_3 - V u_3 + (\Cupper - \Vupper - V) (y_2 - u_2), \slabel{eqn:Q-3-1-x3-2} \\
&& x_2 - x_1 \leq \Vupper y_2 - \Clower y_1 + (\Clower + V - \Vupper) (y_3 - u_3), \slabel{eqn:Q-3-1-x2-x1-1} \\
&& x_2 - x_1 \leq (\Clower + V) y_2 - \Clower y_1 - (\Clower + V - \Vupper) u_2, \slabel{eqn:Q-3-1-x2-x1} \\
&& x_3 - x_2 \leq (\Clower + V) y_3 - \Clower y_2 - (\Clower + V - \Vupper) u_3, \slabel{eqn:Q-3-1-x3-x2} \\
&& x_1 - x_2 \leq \Vupper y_1 - (\Vupper - V) y_2 - (\Clower + V - \Vupper) u_2, \slabel{eqn:Q-3-1-x1-x2} \\
&& x_2 - x_3 \leq \Vupper y_2 - (\Vupper - V) y_3 - (\Clower + V - \Vupper) u_3, \slabel{eqn:Q-3-1-x2-x3} \\
&& x_2 - x_3 \leq (\Clower + V) y_2 - \Clower y_3 - (\Clower + V - \Vupper) u_2, \slabel{eqn:Q-3-1-x2-x3-2} \\ 
&& x_3 - x_1 \leq (\Clower + 2 V) y_3 - \Clower y_1 - (\Clower + 2 V - \Vupper) u_3, \slabel{eqn:Q-3-1-x3-x1} \\
&& x_3 - x_1 \leq (\Vupper + V) y_3 - V u_3 - \Clower y_1 + (\Clower + V - \Vupper) (y_2 - u_2), \slabel{eqn:Q-3-1-x3-x1-2} \\
&& x_1 - x_3 \leq \Vupper y_1 - \Clower y_3 + (\Clower + 2V - \Vupper) (y_2 - u_2), \slabel{eqn:Q-3-1-x1-x3} \\
&& x_1 - x_3 \leq \Vupper y_1 - \Clower y_3 + V (y_2 - u_2) + (\Clower + V - \Vupper) (y_3 - u_3) \slabel{eqn:Q-3-1-x1-x3-2} \Bigr\}.
\end{subeqnarray}
\end{theorem}
\begin{proof}
The proofs are similar with those for Theorem \ref{thm:conv-t3l2} and thus are omitted here.
\end{proof}

Following the similar approach as described above, we can obtain the convex hull representations for other cases in terms of different minimum-up/-down times.

\begin{theorem} \label{thm:t3l1l2-2v}
For $T=3$, $L = 1$ and $\ell = 2$, and $\Cupper - \Clower - 2V \geq 0$, the convex hull representation is the same as $Q_3^1$ except that \eqref{eqn:Q-3-1-md} is replaced by~\eqref{eqn:p-mindn-t3l2}. For $T=3$, $L = 2$ and $\ell = 1$, and $\Cupper - \Clower - 2V \geq 0$, the convex hull representation is the same as $Q_3^2$ except that~\eqref{eqn:p-mindn-t3l2} is replaced by \eqref{eqn:Q-3-1-md}.
\end{theorem}

\begin{remark}
Note here that \sred{we can observe} convex hulls corresponding to different values of $\ell$ \sred{for the same $L$ value}
can be described by replacing~\eqref{eqn:Q-3-1-md} with~\eqref{eqn:p-mindn-t3l2} \sred{($\ell=2$ from $\ell=1$)} or replacing~\eqref{eqn:p-mindn-t3l2} with~\eqref{eqn:Q-3-1-md} \sred{($\ell=1$ from $\ell=2$)}. This conclusion holds for all the remaining theorems. Thus, we only need to provide the convex hulls for the cases $L=\ell=2$ and $L=\ell=1$ in the presentation.
\end{remark}

Now we consider the case in which $\Cupper - \Clower - 2V < 0$. We realize that the condition $\Cupper - \Clower - 2V < 0$ is more restrictive than the condition $\Cupper - \Clower - 2V \geq 0$. Thus, a smaller number of constraints are required to describe the convex hull.
 The proofs are similar with those for Theorem \ref{thm:conv-t3l2}, and thus we present the convex hull descriptions below without proofs. 

\begin{theorem}
For {$T=3$}, $L = \ell = 1$, and $\Cupper - \Clower - 2V < 0$, the convex hull representation of the original polytope (e.g., denoted as $\hat{P}_3^{1}$) can be described as $\hat{Q}_3^{1} =$ conv($\hat{P}_3^{1}$) $= \{ (x, y, u) \in \R^8 : \eqref{eqn:p-udef-t3l2} - \eqref{eqn:p-lower-bound-t3l2}, \eqref{eqn:u2-3period-t3l2}, \eqref{eqn:Q-3-1-mu} - \eqref{eqn:Q-3-1-x2-x3-2}  \}$.
For {$T=3$}, $L = \ell = 2$, and $\Cupper - \Clower - 2V < 0$, the convex hull representation of the original polytope (e.g., denoted as $\hat{P}_3^{2}$) can be described as $\hat{Q}_3^{2} =$ conv($\hat{P}_3^{2}$) $= \{ (x, y, u) \in \R^8 : \eqref{eqn:p-minup-t3l2} - \eqref{eqn:p-lower-bound-t3l2}, \eqref{eqn:p-x1-ub-t3l2} - \eqref{eqn:p-x2-x3-ub-t3l2}, \eqref{eqn:p-x1-x2+x3-ub-t3l2} - \eqref{eqn:u2-3period-t3l2}\}$.
\end{theorem}

{Next, we extend the study to other cases (i.e., \textbf{Cases 2 - 5}) in the rest of this section.
The derived convex hull results are very similar and we list them below for the readers' reference. Proofs for these results are omitted for description brevity, except that for Theorem~\ref{thm:general1} because it is the most complicated case among the rest convex hull results.}

{For \textbf{Cases 2} and \textbf{3}, we have $\Cupper - \Clower - 2V < 0$ for both of them due to their parameter settings. Therefore, we do not need to consider the case in which $\Cupper - \Clower - 2V \geq 0$ as described above {and the conclusion holds as follows}.}

\begin{theorem} \label{thm:general6}
{For \textbf{Case 2}, the convex hull can be described as
\begin{itemize}
\item[1.] $\{(x, y, u) \in \R^8 :  \eqref{eqn:p-minup-t3l2} - \eqref{eqn:p-lower-bound-t3l2},~\eqref{eqn-q2:x1-ub},~\eqref{eqn:p-x2-ub-t3l2},~\eqref{eqn:Q-3-1-x3-1},~\eqref{eqn:p-x2-x1-ub-t3l2}-\eqref{eqn:p-x2-x3-ub-t3l2}, \mbox{and}~\eqref{eqn:p-x1-x2+x3-ub-t3l2} \}$ {for $L = \ell = 2$ and}
\item[2.] $\{(x, y, u) \in \R^8 :  \eqref{eqn:p-udef-t3l2} - \eqref{eqn:p-lower-bound-t3l2},~\eqref{eqn:Q-3-1-mu}- \eqref{eqn:Q-3-1-x1},~\eqref{eqn:Q-3-1-x2-1}-\eqref{eqn:Q-3-1-x3-1}, \ \mbox{and} \ \eqref{eqn:Q-3-1-x2-x1-1}-\eqref{eqn:Q-3-1-x2-x3-2} \}$ {for $L = \ell = 1$.}
\end{itemize} 
For \textbf{Case 3}, the convex hull can be described as
\begin{itemize}
\item[1.] $\{(x, y, u) \in \R^8 :  \eqref{eqn:p-minup-t3l2} - \eqref{eqn:p-lower-bound-t3l2},~\eqref{eqn-q2:x1-ub},~\eqref{eqn:p-x2-ub-t3l2}, \mbox{and}~\eqref{eqn:Q-3-1-x3-1} \}$ {for $L = \ell = 2$ and}
\item[2.] $\{(x, y, u) \in \R^8 : \eqref{eqn:p-udef-t3l2} - \eqref{eqn:p-lower-bound-t3l2},~\eqref{eqn:Q-3-1-mu}- \eqref{eqn:Q-3-1-x1},~\eqref{eqn:Q-3-1-x2-1}-\eqref{eqn:Q-3-1-x3-1} \}$ {for $L = \ell = 1$.}
\end{itemize}}
\end{theorem}

{For \textbf{Case 4}, we have $\Cupper - \Clower - 2V =\Cupper - V - (\Clower + V) \geq \Vupper - (\Clower + V) \geq 0$ due to its parameter setting. Therefore, we do not need to consider the case in which $\Cupper - \Clower - 2V < 0$ as described above. In the following, we report the corresponding convex hull results in Theorems \ref{thm:general1} and \ref{thm:general4}.}

\begin{theorem} \label{thm:general1}
{For {\textbf{Case 4}, when} $L = \ell = 1$, the convex hull representation of the original polytope (e.g., denoted as $\bar{P}_3^{1}$) can be described as }
\begin{subeqnarray}
&\hspace{-0.3in} {\bar{Q}_3^{1} = conv(\bar{P}_3^{1})  =} \Bigl\{ & {(x, y, u) \in \R^8: \eqref{eqn:p-udef-t3l2} - \eqref{eqn:p-lower-bound-t3l2}, \eqref{eqn:u2-3period-t3l2}, \eqref{eqn:Q-3-1-mu}-\eqref{eqn:Q-3-1-x3-2}, \eqref{eqn:Q-3-1-x2-x1}-\eqref{eqn:Q-3-1-x2-x3}, \eqref{eqn:p-x1-x3-ub-t3l2},} \nonumber \\
&& {x_3 - x_1 \leq (\Clower + 2 V) y_3 - \Clower y_1 - (\Clower + 2 V - \Vupper) u_3 - (\Clower + V - \Vupper) u_2,}  \slabel{ineq:general1-3-1} \\
&& {(\Cupper - \Clower - 2V)(x_1 - \Vupper y_1 - V(y_2 - u_2)) \leq  (\Cupper - \Vupper - V) (x_3 - \Clower y_3),}  \slabel{ineq:general1-1-3-Maverick}\\
&& {(\Cupper - \Clower - 2V) (x_3 - (\Vupper + V) y_3 + V u_3) \leq (\Cupper - \Vupper - V) (x_1 - \Clower y_1)} \footnotemark \slabel{ineq:general1-3-1-Maverick} \Bigr\}. 
\end{subeqnarray}
\end{theorem} \footnotetext{\footnotesize Note that \eqref{ineq:general1-3-1-Maverick} was discovered independently and early in \cite{liustudy}.}
\begin{proof}
{Validity proof is similar to that described in Propositions~\ref{prop:valid-t3l2} and thus {is} omitted here. In this part, we only prove $\bar{Q}_3^{1} \subseteq$ conv$(\bar{P}_3^1)$. The details are shown in Appendix \ref{apx:thm:general1}.}
\end{proof}

\begin{theorem} \label{thm:general4}
{For {\textbf{Case 4}, when} $L = \ell = 2$, the convex hull representation of the original polytope (e.g., denoted as $\bar{P}_3^{2}$) can be described as $\bar{Q}_3^{2} =$conv($\bar{P}_3^{2}$), which is the same as $Q_3^2$ except that 
\eqref{eqn:p-x2-x1-ub-t3l2} and \eqref{eqn:p-x2-x3-ub-t3l2} are replaced by \eqref{eqn:Q-3-1-x2-x1} and \eqref{eqn:Q-3-1-x2-x3}, plus the following one:
\begin{equation}
x_1 - x_2 + x_3 \geq \Clower y_1 - (\Clower + V) y_2 + \Clower y_3. \label{newoneforuse}
\end{equation}}
\end{theorem}

{Finally, we consider \textbf{Case 5} and report the corresponding convex hull results in Theorem \ref{thm:general7} by further considering $\Cupper - \Clower - 2V \geq$ or $< 0$ and different minimum-up/-down time limits.}

\begin{theorem} \label{thm:general7}
{For {\textbf{Case 5}, when} $\Cupper - \Clower - 2V \geq 0$, the convex hull {can be} described as
\begin{itemize}
\item[1.] $\{(x, y, u) \in \R^8 :  \eqref{eqn:p-minup-t3l2} - \eqref{eqn:p-lower-bound-t3l2}, \eqref{eqn:Q-3-1-x1}, \eqref{eqn:p-x2-ub-t3l2}, \eqref{eqn:Q-3-1-x3-1}, \eqref{eqn:Q-3-1-x2-x1}-\eqref{eqn:Q-3-1-x2-x3}, \eqref{eqn:p-x1-x2+x3-ub-t3l2}, \eqref{newoneforuse}, \eqref{ffinal1}, \mbox{and} \  \eqref{ffinal2}\}$ {for $L = \ell = 2$ and}
\item[2.] $\{(x, y, u) \in \R^8 :  \eqref{eqn:p-udef-t3l2} - \eqref{eqn:p-lower-bound-t3l2},~\eqref{eqn:Q-3-1-mu}- \eqref{eqn:Q-3-1-x1},~\eqref{eqn:Q-3-1-x2-1}-\eqref{eqn:Q-3-1-x3-1}, \eqref{eqn:Q-3-1-x2-x1}-\eqref{eqn:Q-3-1-x2-x3}, \eqref{ffinal1}, \mbox{and} \  \eqref{ffinal2}\}$ {for $L = \ell = 1$}, \\
{where \eqref{ffinal1} and \eqref{ffinal2} are defined as follows:}
\begin{subeqnarray}
&&\hspace{-0.4in}  x_3 - x_1 \leq (\Clower + 2V) y_3 - \Clower y_1 - (\Clower + 2V - \Vupper) u_3 + (\Cupper - \Clower - 2V) u_2, \slabel{ffinal1} \\
&&\hspace{-0.4in} x_1 - x_3 \leq \Vupper y_1 + (\Cupper - \Vupper) y_2 - (\Cupper - 2V) y_3 - (\Clower + 2V - \Vupper) u_2 + (\Cupper - \Clower - 2V) u_3. \slabel{ffinal2}
\end{subeqnarray}
\end{itemize} 
For {\textbf{Case 5}, when} $\Cupper - \Clower - 2V < 0$, the convex hull descriptions for the $L = \ell = 2$ and $L = \ell = 1$ cases {can be obtained} by removing \eqref{ffinal1} and \eqref{ffinal2} from the above expressions.}
\end{theorem}

\begin{remark}
Since the start-up decision is not considered in the {first time} period, we do not need to consider the cases in which $L = 3$ or $\ell=3$ and it also follows that the strong valid inequalities in this section can be applied to any three consecutive time periods for the multi-period {cases}.
\end{remark}

\section{Strong Valid Inequalities for Multi-period Cases} \label{sec:multi-period}

In this section, we further strengthen the formulation for the general polytope $P$ by exploring the inequalities covering multiple {time} periods. 
Strong valid inequalities containing one, two, and three continuous variables are derived in Sections \ref{subsec:multi-period-single}, \ref{subsec:multi-period-two}, and \ref{subsec:multi-period-three} respectively, through considering the  effects of minimum-up/-down time, ramp rate, start-up decision, and capacity constraints.
{Without loss of generality, we} assume $T \geq L+1$.

\subsection{Strong Valid Inequalities with Single Continuous Variable}  \label{subsec:multi-period-single}
In this subsection, we target to bound the generation amount at each time period (e.g., $x_t$) from above (i.e., tightening constraints \eqref{eqn:p-upper-bound}) and provide better upper bound representations for $x_t$ in inequalities \eqref{eqn:x_t-ub-2-multi-period}, \eqref{eqn:x_t-up-exp}, and \eqref{eqn:x_t-down-exp}.

{Besides the capacity upper bound (i.e., $\Cupper$) to bound $x_t$ from above, {the start-up decisions before $t$ also have effects on the generation upper bound at $t$ when a machine} is online at time $t$.
For instance, if the machine starts up at $t-s$ ($s \geq 0$), then $x_t$ should be bounded from above by $\Vupper+sV$ due to ramp-up constraints \eqref{eqn:p-ramp-up}, which is a tighter upper bound for $x_t$ than $\Cupper$ if $\Vupper+sV < \Cupper$.
Based on this observation, we have the following strong valid inequality \eqref{eqn:x_t-ub-2-multi-period} to tighten \eqref{eqn:p-upper-bound}.}

\begin{proposition} \label{prop:x_t-ub-2-multi-period}
For $1 \leq k \leq \min \{L, {\lfloor (\Cupper - \Vupper)/V \rfloor}  + 2 \}$ {and} $t \in [k, T-1]_{\Z}$, the {following} inequality
\begin{equation}
x_t \leq \Vupper y_t + (\Cupper - \Vupper) (y_{t+1} - u_{t+1}) - \sum_{s=1}^{k-1} \bigg(\Cupper - \Vupper - (s-1) V\bigg)  u_{t-s+1} \label{eqn:x_t-ub-2-multi-period}
\end{equation}
is valid for conv($P$). Furthermore, it is facet-defining for conv($P$) when one of the following conditions is satisfied:
(1) $L \leq 3$ {and} $k = \min \{L, {\lfloor (\Cupper - \Vupper)/V \rfloor} + 2 \}$ for all $t \in [k,T-1]_{\Z}$;
(2) $L \geq 4$ {and} $k = \min \{L, {\lfloor (\Cupper - \Vupper)/V \rfloor} + 2 \}$ for $t = T-1$.
\end{proposition}
\begin{proof}
See Appendix \ref{apx:subsec:x_t-ub-2-multi-period} for the proof.
\end{proof}

Note here that the inequalities derived in Proposition \ref{prop:x_t-ub-2-multi-period} are polynomial in the order of $\mathcal{O}(T)$.
In the following, we explore a further study and discover \sred{several} families of inequalities in exponential sizes
with single continuous variable (i.e., $x_t$) to strengthen the multi-period formulation.

{The key idea of deriving these inequalities is based on the fact that
the online/offline status of a machine at time $t$ is affected by not only the online/offline status before $t$ (e.g., inequality \eqref{eqn:x_t-ub-2-multi-period}) but also the status after $t$. 
For instance, in Figure \ref{fig:insight-1}, as shown in the upper figure with the horizontal axis indicating the time range and the vertical axis indicating the generation amount $x$, if a machine starts up at time $t-m$ (i.e., $u_{t-m}=y_{t-m}=1$), then $x_t$ should be bounded from above by $\Vupper+mV$ due to ramp-up constraints \eqref{eqn:p-ramp-up}, which is a tighter upper bound than $\Cupper$ if $\Vupper+mV<\Cupper$. Meanwhile, if this machine shuts down at time $t+n+1$ (i.e., $y_{t+n}=1$ and $y_{t+n+1}=0$) with $n \leq m$, then $x_t$ should be bounded from above by $\Vupper+nV$ due to ramp-down constraints \eqref{eqn:p-ramp-down}, which leads to a further tighter upper bound than $\Vupper+mV$ since $n \leq m$. The lower figure in Figure \ref{fig:insight-1} provides the online/offline status of this machine.
Therefore, through considering {the effects generated by both the start-up decision before $t$ and the shut-down decision after $t$, we derive} the following strong valid inequality \eqref{eqn:x_t-up-exp} to provide a {tighter} upper bound representation for $x_t$.}

\begin{figure}
\centering
\begin{tikzpicture}[
	scale=1, 
	wdot/.style = {
      draw,
      fill = white,
      circle,
      inner sep = 0pt,
      minimum size = 4pt
    },
   	bdot/.style = {
      draw,
      fill = black,
      circle,
      inner sep = 0pt,
      minimum size = 4pt
    } ]
\coordinate (O) at (0,0);
\draw[->] (0,0) -- (10.5,0) coordinate[label = {below:$T$}] (tmax);
\draw[->] (0,0) -- (0,5.5) coordinate[label = {left:$x$}] (ymax);

\draw (1,0) node[wdot, label = {below: \footnotesize{$t-m$}}] {};
\draw (2,0) node[wdot] {};
\draw (6,0) node[wdot, label = {below: \footnotesize{$t$}}] {};
\draw (9,0) node[wdot, label = {below: \footnotesize{$t+n$}}] {};
\draw (10,0) node[wdot] {};
\draw (4,-0.18) node{   \footnotesize{$\cdots$}  };
\draw (7.5,-0.18) node{   \footnotesize{$\cdots$}  };

\draw (0,0.6) node[wdot, label = {left: \footnotesize{$\Clower$}}] {};
\draw (0,1) node[wdot, label = {left:  \footnotesize{$\Vupper$}}] {};
\draw (0,4.5) node[wdot, label = {left: \footnotesize{$\Vupper+mV$} } ] {};
\draw (0,5) node[wdot, label = {left:  \footnotesize{$\Cupper$}}] {};

\draw[dotted] (2,1.7) -- (2,0);
\draw[dotted] (2,1.7) -- (0,1.7);
\draw[thick, dashdotted] (1,1) -- (6,4.5);
\draw (6,4.5) node[bdot] {};
\draw (1,1) node[bdot] {};
\draw[->, dotted] (1,0) -- (1,0.9);
\draw[dotted] (1,1) -- (0,1);
\draw (0,1.7) node[wdot, label = {left: \footnotesize{$\Vupper+V$} } ] {};

\draw[dotted] (6,4.5) -- (6,0);
\draw[dotted] (6,4.5) -- (0,4.5);
\draw (7.5,4.5) node{   \footnotesize{$x_t \leq \Vupper+mV$}  };

\draw (0,3.1) node[wdot, label = {left: \footnotesize{$\Vupper+nV$} } ] {};
\draw (9,1) node[bdot] {};
\draw (6,3.1) node[bdot] {};
\draw[thick, draw=myblue] (6,3.1) -- (9,1);
\draw[->, dotted] (10,0.5) -- (10,0.1);
\draw (10,0.7) node{   \footnotesize{$t+n+1$}  };
\draw[dotted] (9,0) -- (9,1) -- (0,1);
\draw[dotted] (6,3.1) -- (0,3.1);
\draw (7.5,3.1) node{   \footnotesize{$x_t \leq \Vupper+nV$}  };

\draw (9,3.8) node{   \scriptsize{$\Longrightarrow x_t \leq \min \{\Vupper+nV, \Vupper + mV \} = \Vupper+nV$}  };
\draw (1,1) node[bdot] {};
\draw[->, dotted] (1,0) -- (1,0.9);
\draw[thick, draw=myblue] (6,3.1) -- (1,1);
\end{tikzpicture}
\begin{tikzpicture}[
	scale=1, 
	wdot/.style = {
      draw,
      fill = white,
      circle,
      inner sep = 0pt,
      minimum size = 4pt
    },
   	bdot/.style = {
      draw,
      fill = black,
      circle,
      inner sep = 0pt,
      minimum size = 4pt
    } ]
\coordinate (O) at (0,0);
\draw[->] (0,0) -- (10.75,0) coordinate[label = {below:$T$}] (tmax);
\draw[->] (0,0) -- (0,1.5) coordinate[label = {left:$y$}] (ymax);

\draw (2,0) node[wdot] {};
\draw (6,0) node[wdot, label = {below: \footnotesize{$t$}}] {};
\draw (9,0) node[wdot, label = {below: \footnotesize{$t+n$}}] {};
\draw (10,0) node[wdot] {};
\draw (4,-0.18) node{   \footnotesize{$\cdots$}  };
\draw (7.5,-0.18) node{   \footnotesize{$\cdots$}  };
\draw (9,0.5) node{   \footnotesize{$n \leq m$}  };

\draw (0,1) node[wdot, label = {left:$1$}] {};
\draw[thick, draw=red] (0,0) -- (1,0);
\draw (1,0) node[wdot, label = {below: \footnotesize{$t-m$}}] {};
\draw (0,0) node[wdot, label = {left:$0$}] {};
\draw (1,-0.8) node{$u_{t-m}=1$};
\draw[->, dotted] (1,0) -- (1,1);

\draw[thick, draw=myblue] (1,1) -- (10,1);
\draw[->, dotted] (10,1) -- (10,0);
\draw[thick, draw=red] (10,0) -- (10.5,0);
\end{tikzpicture}
\caption{The basic insight}\label{fig:insight-1}
\end{figure}
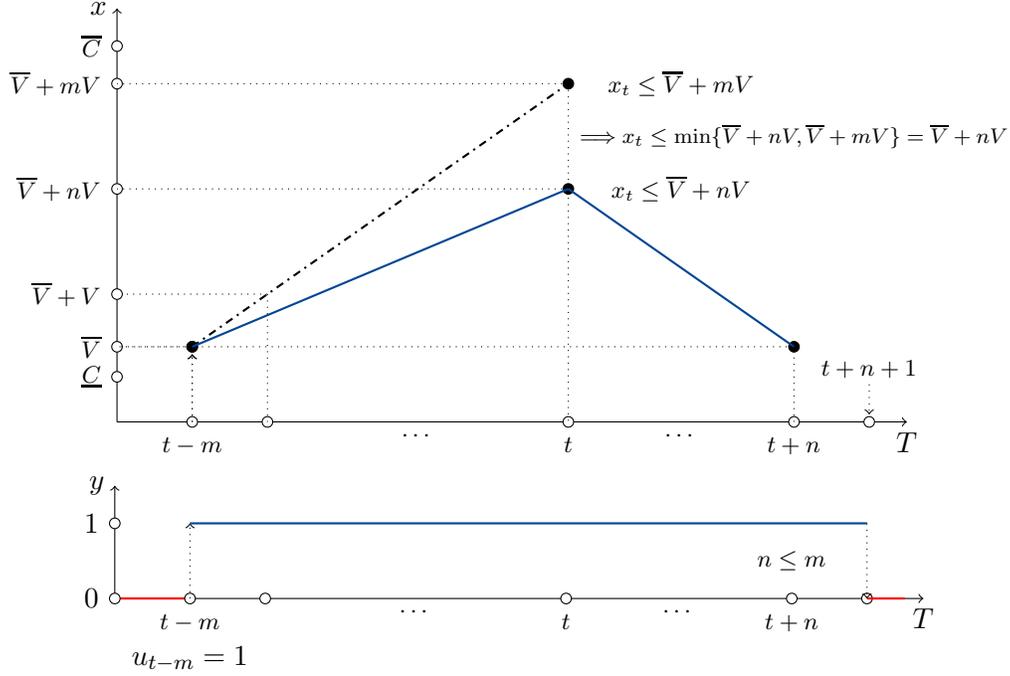

\begin{proposition} \label{prop:x_t-up-exp}
{For each $t \in [L+1, T]_{\Z}$, $m \in [0, \min \{t-L-1, {[(\Cupper-\Vupper)/V-L+1]^+}\}]_{\Z}$, $n \in [\min \{1, T-t\}, \min \{L-1, T-t\}]_{\Z}$ if $L \geq 2$ and $n = 0$ if $L=1$, {and} $S \subseteq [t-m+1, t]_{\Z}$, the {following} inequality
\begin{align}
x_t & \leq \Vupper y_t + V \sum_{i \in S} \bigg(i - d_i\bigg) \bigg(y_i - \sum_{j=0}^{L-1} u_{i-j}\bigg) + V \sum_{k=1}^{{[n-1]^+}} \bigg(y_{t+k} - \sum_{j=0}^{L-1} u_{t+k-j}\bigg) \nonumber \\
 & + \alpha(S,t) V\bigg(y_{t+n} - \sum_{j=0}^{L-1} u_{t+n-j}\bigg) + \beta(t) \bigg(y_{t-m} - \sum_{j=0}^{L-1} u_{t-m-j}\bigg) + \phi(u,t), \label{eqn:x_t-up-exp}
\end{align}
where $d_i = \max \{a \in S \cup \{t-m\}: a < i \}$ for each $i \in S$, $\alpha(S,t) = m+L-1 - \sum_{i \in S} (i - d_i) - {[n-1]^+}$, $\beta(t) = \Cupper - \Vupper - (m+L-1) V$, and $\phi(u,t) = V \sum_{k=1}^{ {t+L-T-1} } k u_{t-k} + V \sum_{k={[t+L-T]^+}}^{L-1} \min \{L-1-k,k\}u_{t-k}$,
is valid and facet-defining for conv($P$),
under the condition that $n \geq {(L-1)/2}$ {if $n \leq T-t-1$}.}
\end{proposition}
\begin{proof}
See Appendix \ref{apx:subsec:x_t-up-exp} for the proof.
\end{proof}

{\textbf{Separation:} Since the size of {inequalities} \eqref{eqn:x_t-up-exp} is exponential, we explore a separation scheme to find the most violated inequality (correspondingly the set $S$ in \eqref{eqn:x_t-up-exp}) in polynomial time.
Without loss of generality, {we} consider the case in which $L \geq 2$. {For a given point $(\hat{x}, \hat{y}, \hat{u}) \in \R_+^{3n-1}$}, to find {the most} violated {inequality} \eqref{eqn:x_t-up-exp} {corresponding to each combination of $(t,m,n)$}, we construct a shortest path problem on a directed acyclic graph $\Gb = (\V, \mathbb{A})$ as shown in Figure \ref{fig:dag-1}. 
The graph is represented as follows:
\begin{enumerate}[(i)]
\item Node set $\V = \{o,d,r\} \cup \V'$ with $o$ representing the origin, $d$ representing the destination, $r$ {aggregating} time indices from $t+1$ to $t+n$, {and} $\V'=\{t, t-1, \cdots, t-m\}$ representing a set of time indices from $t-m$ to $t$ in \eqref{eqn:x_t-up-exp};
\item Arc set $\mathbb{A} = \{a_{or}, a_{(t-m)d}\} \cup \mathbb{A}_1 \cup \mathbb{A}_2$ with $\mathbb{A}_1 = \cup_{t-m \leq s \leq t} a_{rs}$ (i.e., dashed arcs) and $\mathbb{A}_2 = \cup_{t-m \leq t_2 < t_1 \leq t} a_{t_1 t_2}$ (i.e., dashdotted arcs). Accordingly, we let $w_{ij}$ represent the cost of arc $(i,j)$ and provide the details as follows:
	\begin{enumerate}[(1)]
	\item $w_{or} = \Vupper \hat{y}_t - \hat{x}_t$;
	\item $w_{(t-m)d}=(\Cupper - \Vupper - (m+L-1) V) (\hat{y}_{t-m} - \sum_{j=0}^{L-1} \hat{u}_{t-m-j}) + {\phi}(\hat{u},t)$, where ${\phi}(\hat{u},t) = V \sum_{k=1}^{{t+L-T-1}} k \hat{u}_{t-k} + V \sum_{k={[t+L-T]^+}}^{L-1} \min \{L-1-k,k\} \hat{u}_{t-k}$;
	\item $w_{rs} = V \sum_{k=1}^{{[n-1]^+}} (\hat{y}_{t+k} - \sum_{j=0}^{L-1} \hat{u}_{t+k-j}) + (L-1+t-s - {[n-1]^+})V(\hat{y}_{t+n} - \sum_{j=0}^{L-1} \hat{u}_{t+n-j})$ for $s \in [t-m,t]_{\Z}$  (Note {here} that $\sum_{i \in S} (i - d_i) = 0$ if $S = \emptyset$ and $\sum_{i \in S} (i - d_i) = q-(t-m)$ otherwise {with} $q=\max\{a \in S\}$);
	\item $w_{t_1 t_2} = V(t_1-t_2)(\hat{y}_{t_1} - \sum_{j=0}^{L-1} \hat{u}_{t_1-j})$ for $t_1,t_2 \in [t-m,t]_{\Z}$ and $t_1 > t_2$.
	\end{enumerate}
\end{enumerate}
It is obvious that the shortest path from nodes $o$ to $d$ represents the maximum violation of inequality \eqref{eqn:x_t-up-exp} if the value is negative.
Meanwhile, the visited nodes in $\V'$ on the shortest path determine the set $S$. Since it is an acyclic graph and there are $\mathcal{O}(T^2)$ arcs and $\mathcal{O}(T)$ nodes, the shortest path can be found in $\mathcal{O}(T^2)$ time following the Topological Sorting Algorithm \cite{cormen2009introduction} for each combination of $(t,m,n)$.
Therefore, considering the fact that $m$ and $n$ are bounded from above by two constant numbers, i.e., ${[(\Cupper-\Vupper)/V-L+1]^+}$ and $L-1$, respectively, there is an $\mathcal{O}(T^3)$ algorithm to solve the separation problem for all $t$.}

\begin{figure}[htb]
\centering
\begin{tikzpicture}[thick,scale=0.7, every node/.style={scale=0.7}] 

		\tikzstyle{state} = [draw, thick, fill=white, circle, text width = 2em, 
		 text badly centered, node distance=8em, inner sep=1pt,font=\sffamily]
		\tikzstyle{stateEdgePortion} = [black];
		\tikzstyle{stateEdge} = [stateEdgePortion,->];
		\tikzstyle{edgeLabel} = [pos=0.5, text centered, font={\sffamily}];
  
\node[state, name=d] at (0,0) {\Large $d$};  
\node[state, name=nj] at (3,0) {\small $t-m$};
\draw (5,0) node{ \Large  \textcolor{brown}{$\cdots$}  };
\node[state, name=nq] at (7,0) { \Large $t_2$};
\draw (9,0) node{ \Large  \textcolor{brown}{$\cdots$}  };
\node[state, name=np] at (11,0) {\Large $t_1$};
\draw (13,0) node{  \Large \textcolor{brown}{$\cdots$}  };
\node[state, name=nt] at (15,0) {\Large $t$};
\node[state, name=t] at (18,0) {\Large $r$};
\node[state, name=o] at (21,0) {\Large $o$};

\draw (18,-2) node{  \Large $t-m \leq t_2 < t_1 \leq t$  };
  
\draw (o) edge[stateEdge] node[edgeLabel, xshift=0em, yshift=0.8em]{\Large $\Vupper \hat{y}_t - \hat{x}_t$} (t);
\draw (t) edge[stateEdge, draw=myblue, dashed] node[edgeLabel, xshift=0em, yshift=0.8em]{\Large $w_{rt}$} (nt);
\draw (t) edge[stateEdge,  bend right=35, draw=myblue, dashed] node[edgeLabel, xshift=0em, yshift=-0.8em]{\Large $w_{rt_1}$} (np);
\draw (t) edge[stateEdge,  bend right=40, draw=myblue, dashed] node[edgeLabel, xshift=0em, yshift=0.8em]{\Large $w_{rt_2}$} (nq);
\draw (t) edge[stateEdge,  bend right=45, draw=myblue, dashed] node[edgeLabel, xshift=0em, yshift=0.8em]{\Large $w_{r(t-m)}$} (nj);
\draw (nj) edge[stateEdge] node[edgeLabel, xshift=0em, yshift=0.8em]{\Large $w_{(t-m)d}$} (d);

\draw (nt) edge[stateEdge,  bend left=35, draw=brown, dashdotted] node[edgeLabel, xshift=0em, yshift=0.8em]{ \Large $w_{tt_1}$} (np);
\draw (nt) edge[stateEdge,  bend left=40, draw=brown, dashdotted] node[edgeLabel, xshift=0em, yshift=0.8em]{ \Large $w_{tt_2}$} (nq);
\draw (nt) edge[stateEdge,  bend left=45, draw=brown, dashdotted] node[edgeLabel, xshift=0em, yshift=0.8em]{ \Large $w_{t(t-m)}$} (nj);
  
\draw (np) edge[stateEdge,  bend left=35, draw=brown, dashdotted] node[edgeLabel, xshift=0em, yshift=0.8em]{\Large $w_{t_1 t_2}$} (nq);
\draw (nq) edge[stateEdge,  bend left=35, draw=brown, dashdotted] node[edgeLabel, xshift=0em, yshift=0.8em]{ \Large $w_{t_2(t-m)}$} (nj);
  
\draw (np) edge[stateEdge,  bend right=40, draw=brown, dashdotted] node[edgeLabel, xshift=0em, yshift=0.8em]{ \Large $w_{t_1(t-m)}$} (nj);
\end{tikzpicture}
\caption{Directed acyclic graph}\label{fig:dag-1}
\end{figure}

\begin{proposition}
{Given a point $(\hat{x}, \hat{y}, \hat{u}) \in \R_+^{3n-1}$, there {exists} an $\mathcal{O}(T^3)$ {time separation} algorithm to find the most violated inequality \eqref{eqn:x_t-up-exp}, if any.}
\end{proposition}

{Now we consider the effect on $x_t$ from the machine status {only} after time $t$ and develop {a family} of inequalities {in exponential size} in \eqref{eqn:x_t-down-exp} as follows. 
Note that since we only consider the machine status after time $t$, the corresponding separation problem for \eqref{eqn:x_t-down-exp} can be solved faster than that for \eqref{eqn:x_t-up-exp}, which considers the machine statuses from both before and after time $t$.}

\begin{proposition} \label{prop:x_t-down-exp}
{For each $t \in [1, T-1]_{\Z}$, $m \in [{[\hat{t}-t-1]^+}, \min \{T-t-1, {(\Cupper-\Vupper)/V}\}]_{\Z}$ {with $\hat{t} = t+\min \{t-2, L-2\}$ if $\min \{t-2, L-2\} \geq L/2$ and $\hat{t} = \max \{t+1, L+1\}$ otherwise}, $S_0 = [t+1, \hat{t}-1]_{\Z}$, {and} $S \subseteq [\hat{t}+1, t+m]_{\Z}$, the {following} inequality
\begin{align}
x_t & \leq \Vupper y_t + V \sum_{i \in S_0} \bigg(y_i - \sum_{j=0}^{\min \{L-1,i-2\}} u_{i-j}\bigg) + V \sum_{i \in S \cup \{\hat{t}\}} \bigg(d_i - i\bigg) \bigg(y_i - \sum_{j=0}^{L-1} u_{i-j}\bigg)  \nonumber \\
 & + \bigg(\Cupper - \Vupper - m V\bigg) \bigg(y_{t+m+1} - \sum_{j=0}^{L-1} u_{t+m+1-j}\bigg) + \phi (u,t), \label{eqn:x_t-down-exp}
\end{align}
where $d_i = \min \{a \in S \cup \{t+m+1\}: a > i \}$ for each $i \in S \cup \{\hat{t}\}$ and $\phi (u,t) = V \sum_{k=1}^{{t+L-T-1}} k u_{t-k} + V \sum_{k={[t+L-T]^+}}^{\min\{L-1, t-2\}} \min \{L-1-k,k\}u_{t-k}$,
is valid and facet-defining for conv($P$).}
\end{proposition}
\begin{proof}
See Appendix \ref{apx:subsec:x_t-down-exp} for the proof.
\end{proof}

{\textbf{Separation:} 
Since the size of {inequalities} \eqref{eqn:x_t-down-exp} is exponential, we explore a separation scheme to find the most violated inequality (correspondingly the set $S$ in \eqref{eqn:x_t-down-exp}) in polynomial time. 
{Similar to the separation procedure for inequalities \eqref{eqn:x_t-up-exp}, for} a given point $(\hat{x}, \hat{y}, \hat{u}) \in \R_+^{3n-1}$, to find the most violated {inequality} \eqref{eqn:x_t-down-exp} corresponding to each combination of $(t,m)$,
we construct a shortest path problem on a directed acyclic graph $\Gb = (\V, \mathbb{A})$ {with different node and arc sets shown} as follows:
\begin{enumerate}[(i)]
\item Node set $\V = \{o,d\} \cup \V'$ with $o$ representing the origin, $d$ representing the destination, {and} $\V'=\{\hat{t}, \hat{t}+1, \cdots, t+m, t+m+1\}$ representing a set of time indices from $\hat{t}$ to $t+m+1$ in \eqref{eqn:x_t-down-exp};
\item Arc set $\mathbb{A} = \{a_{or}, a_{(t-m)d}\} \cup \mathbb{A}'$ with $\mathbb{A}' = \cup_{\hat{t} \leq t_1 < t_2 \leq t+m+1} a_{t_1 t_2}$. Accordingly, we let $w_{ij}$ represent the cost of arc $(i,j)$ and provide the details as follows:
	\begin{enumerate}[(1)]
	\item $w_{o\hat{t}} = \Vupper \hat{y}_t - \hat{x}_t + V \sum_{i \in S_0} (\hat{y}_i - \sum_{j=0}^{\min \{L-1,i-2\}} \hat{u}_{i-j})$;
	\item $w_{(t+m+1)d} = (\Cupper - \Vupper - m V) (\hat{y}_{t+m+1} - \sum_{j=0}^{L-1} \hat{u}_{t+m+1-j}) + {\phi} (\hat{u},t)$, where ${\phi} (\hat{u},t) \allowbreak = V \sum_{k=1}^{{t+L-T-1}} \allowbreak k \hat{u}_{t-k} \allowbreak + V \sum_{k={[t+L-T]^+}}^{\min\{L-1, t-2\}} \min \{L-1-k,k\} \hat{u}_{t-k}$;
	\item $w_{t_1 t_2} = V(t_2-t_1)(\hat{y}_{t_1} - \sum_{j=0}^{L-1} \hat{u}_{t_1-j})$ for $t_1,t_2 \in [\hat{t}, t+m+1]_{\Z}$ and $t_1 < t_2$.
	\end{enumerate}
\end{enumerate}
{Similarly,} the shortest path from nodes $o$ to $d$ represents the maximum violation of inequality \eqref{eqn:x_t-down-exp} if the value is negative, {and} the visited nodes in $\V'$ on the shortest path determine the set $S$. Since it is an acyclic graph and there are $\mathcal{O}(T^2)$ arcs and $\mathcal{O}(T)$ nodes, the shortest path can be found in $\mathcal{O}(T^2)$ time following the Topological Sorting Algorithm \cite{cormen2009introduction} for each combination of $(t,m)$. 
Therefore, there is an $\mathcal{O}(T^3)$ algorithm to solve the separation problem for all $t$, considering the fact that $m$ is bounded from above by a constant number, i.e., ${(\Cupper-\Vupper)/V}$.}


\begin{proposition}
{Given a point $(\hat{x}, \hat{y}, \hat{u}) \in \R_+^{3n-1}$, there {exists} an $\mathcal{O}(T^3)$ {time separation} algorithm to find the most violated inequality \eqref{eqn:x_t-down-exp}, if any.}
\end{proposition}

Finally, it can be easily observed that {inequalities} \eqref{eqn:x_t-ub-2-multi-period} {are} neither dominated by {inequalities} \eqref{eqn:x_t-up-exp} nor by {inequalities} \eqref{eqn:x_t-down-exp}.

\subsection{Strong Valid Inequalities with Two Continuous Variables}  \label{subsec:multi-period-two}
In this subsection, we extend the study to derive strong valid inequalities to bound the difference of generation amounts at two different time periods, e.g., $x_t - x_{t-m}$ {or} $x_{t} - x_{t+m}$. These values are bounded from above and below by the combination of generation bound constraints \eqref{eqn:p-lower-bound} - \eqref{eqn:p-upper-bound} and ramp-rate constraints \eqref{eqn:p-ramp-up} - \eqref{eqn:p-ramp-down}.

{First, we consider the ramp-up process from times $t-m$ to $t$, from which we have $x_t-x_{t-m}\leq mV$ due to ramp-up constraints \eqref{eqn:p-ramp-up} if a machine stays online through $t-m$ to $t$. Through additionally {considering} the generation amount at $x_t$, which is affected by the machine status (i.e., online/offline and start-up statuses) both before $t$ and after $t$, we derive an exponential number of inequalities in \eqref{eqn:ru-2-exp-1} and \eqref{eqn:ru-2-exp-2}. Meanwhile, the minimum-up time limits are embedded {in the consideration for} the ramp-up process.}

\begin{proposition} \label{prop:ru-2-exp}
{If $\Vupper < \Clower + V$, then for each $t \in [2, T]_{\Z}$, $m \in [1, \min \{t-1, {(\Cupper-\Clower)/V}\}]_{\Z}$, $n \in [\min \{1, T-t\}, \min \{m, L, T-t\}]_{\Z}$, {and} $S \subseteq [t-m+L, t]_{\Z}$, the {following} inequalities
\begin{align}
x_t - x_{t-m} & \leq \Vupper y_t - \Clower y_{t-m} +  V \sum_{k=1}^{{[n-1]^+}} \bigg(y_{t+k} - \sum_{j=0}^{\min \{L-1,k+m-1\}} u_{t+k-j}\bigg) \nonumber \\
 & + \alpha(t) \bigg(y_{t+n} - \sum_{j=0}^{\min \{L-1, n+m-1\}} u_{t+n-j}\bigg) + \phi (u,t) \ (if \ S=\emptyset), \label{eqn:ru-2-exp-1} \\
 x_t - x_{t-m} & \leq \Vupper y_t - \Clower y_{t-m} + V \sum_{i \in S \setminus \{t-m+L\} } \bigg(i - d_i\bigg) \bigg(y_i - \sum_{j=0}^{L-1} u_{i-j}\bigg) \nonumber \\
 & + V \sum_{k=1}^{{[n-1]^+}} \bigg(y_{t+k} - \sum_{j=0}^{\min \{L-1,k+m-1\}} u_{t+k-j}\bigg)  + \beta(S,t) V \bigg(y_{t+n} - \sum_{j=0}^{ \min \{L-1, n+m-1\}} u_{t+n-j}\bigg) \nonumber \\
 & + \bigg(\Clower + V-\Vupper\bigg) \bigg(y_q - \sum_{j=0}^{\min \{L-1, q-t+m-1\}} u_{q-j}\bigg) + \phi (u,t) \ (if \ S \neq \emptyset), \label{eqn:ru-2-exp-2}
\end{align}
where 
$\alpha(t) = \Clower + (m-{[n-1]^+})V - \Vupper$, 
$\phi (u,t) = V \sum_{k=1}^{\min \{{t+L-T-1}, m-1\}} k u_{t-k} + V \sum_{k= {[t+L-T]^+} }^{\min \{L-1, m-1\}} \min \{L-1-k,k\}u_{t-k}$, and
$\beta(S,t) = m-1-\sum_{i \in S \setminus \{t-m+L\} } (i - d_i)-{[n-1]^+}$,
are valid and facet-defining for conv($P$) under the {conditions}
(i) if $\min \{m-1, L-2\} \geq L/2$, then $n \geq \min \{m-1,L-2,T-t\}$, 
and (ii) when $m \leq L-1$ (i.e., $S=\emptyset$), if $n \leq L-1-m$, then $n \geq \min \{m, L, T-t\}$.}
\end{proposition}
\begin{proof}
See Appendix \ref{apx:subsec:ru-2-exp} for the proof.
\end{proof}

{Following the similar separation procedure as described for \eqref{eqn:x_t-up-exp}, the above inequality \eqref{eqn:ru-2-exp-2} can be separated in polynomial time.}

\begin{proposition}
{Given a point $(\hat{x}, \hat{y}, \hat{u}) \in \R_+^{3n-1}$, there exists an $\mathcal{O}(T^3)$ time separation algorithm to find the most violated inequality \eqref{eqn:ru-2-exp-2}, if any.}
\end{proposition}

{Next, we consider the ramp-down process from times $t$ to $t+m$, from which we have $x_t-x_{t+m}\leq mV$ due to ramp-down constraints \eqref{eqn:p-ramp-down} if a machine stays online through $t$ to $t+m$. Through additionally {considering} the generation amount $x_t$, which is affected by the machine status after time $t$ until $t+m$, we {derive} the following strong valid inequalities \eqref{eqn:rd-2-exp}, which is {exponentially} sized.}

\begin{proposition} \label{prop:rd-2-exp}
{If $\Vupper < \Clower + V$, then for each $t \in [1, T-1]_{\Z}$, $\tilde{t} = \min \{\hat{t}, t+m\}$ {with $\hat{t} = t+\min \{t-2, L-2\}$ if $\min \{t-2, L-2\} \geq L/2$ and $\hat{t} = \max \{t+1, L+1\}$ otherwise}, $m \in [1, \min \{T-t, {(\Cupper-\Clower)/V}\}]_{\Z}$, $S_0 = [t+1, \tilde{t}-1]_{\Z}$, $S_1 \subseteq [\tilde{t}+1, t+m]_{\Z}$, {and} $S = S_1 \cup \{\tilde{t}\}$, the {following} inequality
\begin{align}
x_t - x_{t+m} & \leq \Vupper y_t - \Clower y_{t+m} + V \sum_{i \in S_0} \bigg(y_i - \sum_{j=0}^{\min \{L-1,i-2\}} u_{i-j}\bigg) + V \sum_{i \in S \setminus \{t+m\}} \bigg(d_i - i\bigg) \bigg(y_i - \sum_{j=0}^{L-1} u_{i-j}\bigg)  \nonumber \\
 & + \bigg(\Clower + V - \Vupper\bigg) \bigg(y_{q} - \sum_{j=0}^{ \min \{L-1, q-2\} } u_{q-j}\bigg) + \phi(u,t), \label{eqn:rd-2-exp}
\end{align}
 where $q = \max \{a \in S\}$, $d_i = \min \{a \in S \cup \{t+m\}: a > i \}$ for each $i \in S \setminus \{t+m\}$, and $\phi(u,t) = V \sum_{k=1}^{{t+L-T-1}} k u_{t-k} + V \sum_{k={[t+L-T]^+}}^{\min\{L-1, t-2\}} \min \{L-1-k,k\}u_{t-k}$,
is valid and facet-defining for conv($P$) under the condition that if $q = t+m$ and $m \leq L-1$, then $m \geq \lfloor {(L+1)/2} \rfloor$.}
\end{proposition}
\begin{proof}
See Appendix \ref{apx:subsec:rd-2-exp} for the proof.
\end{proof}

{Following the similar separation procedure as described for \eqref{eqn:x_t-down-exp}, the above inequality \eqref{eqn:rd-2-exp} can be separated in polynomial time.}

\begin{proposition}
{Given a point $(\hat{x}, \hat{y}, \hat{u}) \in \R_+^{3n-1}$, there {exists} an $\mathcal{O}(T^3)$ {time separation} algorithm to find the most violated inequality \eqref{eqn:rd-2-exp}, if any.}
\end{proposition}

\subsection{Strong Valid Inequalities with Three Continuous Variables}  \label{subsec:multi-period-three}
{In this subsection, we extend the study to strengthen the general polytope $P$ through deriving strong valid inequalities considering three continuous variables (e.g., $x_{t-2}$, $x_{t-1}$, and $x_{t}$). We develop strong valid inequalities to bound $x_{t-2} - x_{t-1} + x_{t}$, which is {easily observed to be} bounded from above by $\Cupper+V$ since $x_{t-3} \leq \Cupper$ and $x_{t-1} - x_{t-2} \leq V$ due to constraints \eqref{eqn:p-upper-bound} and \eqref{eqn:p-ramp-up}. 
{From this observation, we notice that there is} a mixture of capacity upper bound and ramp rate embedded in the three continuous variables, whereas capacity upper bound is embedded in {the} single continuous variable {as described} in Section \ref{subsec:multi-period-single} and ramp rate is embedded in {the} two continuous variables {as described} in Section \ref{subsec:multi-period-two}.}

{First, to bound {$x_{t-2} - x_{t-1} + x_{t}$}, we consider the start-up decisions before $t-1$ and the online/offline status at $t$ to develop strong valid inequalities in the following \eqref{eqn:3degree-1-multi-period}, which is polynomial in the order of $\mathcal{O}(T)$.}


\begin{proposition} \label{prop:3degree-1-multi-period}
If $\Vupper < \Clower + V$, then for each $t \in [\max \{L+1, 3\}, T-1]_{\Z}$, the {following} inequality
\begin{align}
x_{t-2} - x_{t-1} + x_{t} & \leq \Vupper y_{t-2} - (\Vupper - V) y_{t-1} + \Vupper y_{t} + (\Clower + V - \Vupper) (y_{t+1} - u_{t+1} - y_{t}) \nonumber \\
& + (\Cupper - \Vupper) (y_{t} - u_{t} - u_{t-1}) - \sum_{s=0}^{L-3} (\Cupper - \Vupper - s V ) u_{t-s-2} \label{eqn:3degree-1-multi-period}
\end{align}
is valid for conv($P$) when $L \geq 2$. Furthermore, it is facet-defining for conv($P$) for each $t \in [\max \{L+1, 3\}, T-1]_{\Z}$ when $L \leq 3$.
\end{proposition}
\begin{proof}
See Appendix \ref{apx:subsec:3degree-1-multi-period} for the proof.
\end{proof}

{Next, we derive {another family of} strong valid inequalities with three continuous variables (i.e., $x_{t-2} - x_{t-1} + x_t$) to strengthen the multi-period formulation. {In particular},
we develop the strong valid inequalities in {the following} \eqref{eqn:ru-3-exp} by considering {how} the machine status before time $t$ {to affect the bound} {of $x_{t-2} - x_{t-1} + x_{t}$}. Different from inequality \eqref{eqn:3degree-1-multi-period}, the machine status after time $t$ is not considered in \eqref{eqn:ru-3-exp}.}

\begin{proposition} \label{prop:ru-3-exp}
{For each $t \in [L+1, T]_{\Z}$, $m \in [{[3-L]^+}, \min \{{[t-L-1]^+, [(\Cupper-\Vupper)/V-L+3]^+} \}]_{\Z}$, {and $S \subseteq [t-m+1, t-1]_{\Z}$ if $L \geq 3$ and $S \subseteq [t-m+1, t-2]_{\Z}$ if $L=2$}, the {following} inequality
\begin{align}
\hspace{-0.1in} x_{t-2} - x_{t-1} + x_t & \leq \Vupper y_{t-2} - (\Vupper - V) y_{t-1} + \Vupper y_t + V \sum_{i \in S} \bigg(i - d_i\bigg) \bigg(y_i - \sum_{j=0}^{L-1} u_{i-j}\bigg) - \phi(u,t) \nonumber \\
 & + \alpha(S,t) V \bigg(y_t - \sum_{k=0}^{L-1} u_{t-k}\bigg) + \sum_{k=3}^{L-1} (k-2) V u_{t-k} + \beta(t) \bigg(y_{t-m} - \sum_{j=0}^{L-1} u_{t-m-j}\bigg),  \label{eqn:ru-3-exp}
\end{align}
where $d_i = \max \{a \in S \cup \{t-m\}: a < i \}$ for each $i \in S$, $\phi(u,t) = {[\Clower+V-\Vupper]^+} u_{t-2}$ if $L=2$ and $\phi(u,t) =0$ otherwise, $\alpha(S,t) = {[m+L-3]^+} - \sum_{i \in S} (i - d_i)$, and $\beta(t) = \Cupper - \Vupper - {[m+L-3]^+} V$, is valid for conv($P$) when $L \geq 2$.
Furthermore, it is facet-defining for conv($P$) when $t=T$.}
\end{proposition}
\begin{proof}
See Appendix \ref{apx:subsec:ru-3-exp} for the proof.
\end{proof}

{Following the similar separation procedure as described for \eqref{eqn:x_t-up-exp}, the above inequality \eqref{eqn:ru-3-exp} can be separated in polynomial time.}

\begin{proposition}
{Given a point $(\hat{x}, \hat{y}, \hat{u}) \in \R_+^{3n-1}$, there {exists} an $\mathcal{O}(T^3)$ {time separation} algorithm to find the most violated inequality \eqref{eqn:ru-3-exp}, if any.}
\end{proposition}

{Similar to inequalities \eqref{eqn:x_t-down-exp} and \eqref{eqn:rd-2-exp}, we continue to consider the machine status after time $t$ to provide a {tighter} upper bound for $x_{t-2} - x_{t-1} + x_t$ in {the following} \eqref{eqn:rd-3-exp}, which is {exponentially} sized. Note here that the machine status after $t$ will also have an effect on $x_{t-2}$ and $x_{t-1}$, besides $x_t$, which is mainly studied in \eqref{eqn:x_t-down-exp} and \eqref{eqn:rd-2-exp}.}

\begin{proposition} \label{prop:rd-3-exp}
{For each $t \in [3, T-1]_{\Z}$, $\hat{t} = \max \{t+1, L+1\}$, $m \in [{[\hat{t}-t-1]^+}, \min \{T-t-1, {(\Cupper-\Vupper)/V}\}]_{\Z}$, $S_0 = [t+1, \hat{t}-1]_{\Z}$, {and} $S \subseteq [\hat{t}+1, t+m]_{\Z}$, the {following} inequality
\begin{align}
x_{t-2} - x_{t-1} + x_t & \leq \Vupper y_{t-2} - (\Vupper - V) y_{t-1} + \Vupper y_t + V \sum_{i \in S_0} \bigg(y_i - \sum_{j=0}^{\min \{L-1,i-2\}} u_{i-j}\bigg) + \phi(u,t) \nonumber \\
 & + V \sum_{i \in S \cup \{\hat{t}\}} \bigg(d_i - i\bigg) \bigg(y_i - \sum_{j=0}^{L-1} u_{i-j}\bigg)  + \alpha(t) \bigg(y_{t+m+1} - \sum_{j=0}^{L-1} u_{t+m+1-j}\bigg),  \label{eqn:rd-3-exp}
\end{align}
where $\phi(u,t) = V \sum_{k=3}^{\min \{t-2, L-1\}} (k-2)u_{t-k}$, $d_i = \min \{a \in S \cup \{t+m+1\}: a > i \}$ for each $i \in S \cup \{\hat{t}\}$, and $\alpha(t) = \Cupper - \Vupper - m V$, is valid for conv($P$).
Furthermore, it is facet-defining for conv($P$) when $t=3$.}
\end{proposition}
\begin{proof}
See Appendix \ref{apx:subsec:rd-3-exp} for the proof.
\end{proof}

{Following the similar separation procedure as described for \eqref{eqn:x_t-down-exp}, the above inequality \eqref{eqn:rd-3-exp} can be separated in polynomial time.}

\begin{proposition}
{Given a point $(\hat{x}, \hat{y}, \hat{u}) \in \R_+^{3n-1}$, there {exists} an $\mathcal{O}(T^3)$ {time separation} algorithm to find the most violated inequality \eqref{eqn:rd-3-exp}, if any.}
\end{proposition}

Finally, it can be easily observed that {inequalities} \eqref{eqn:3degree-1-multi-period} {are} neither dominated by {inequalities} \eqref{eqn:ru-3-exp} nor by {inequalities} \eqref{eqn:rd-3-exp}.

\begin{remark}
{Note that inequalities \eqref{eqn:ru-2-exp-1}, \eqref{eqn:ru-2-exp-2}, \eqref{eqn:rd-2-exp}, and \eqref{eqn:3degree-1-multi-period} are derived for the most common cases in practice (i.e., $\Vupper < \Clower + V$). Strong valid inequalities for other parameter settings can be derived similarly. Here we list three families of inequalities for illustration purpose. For instance, under the condition $\Vupper > \Clower + V,\ \Cupper - \Clower - V > 0, \ \mbox{and} \ \Cupper - \Vupper - V > 0$, the following conclusion holds.}
\end{remark}
\begin{proposition} 
{The following inequalities
\begin{align}
x_t-x_{t-1} & \leq (\Clower + V)y_t - \Clower y_{t-1} - (\Clower + V - \Vupper)u_t, \ \forall t \in [2, T]_{\Z}, \nonumber \\
x_{t-1}-x_t & \leq \Vupper y_{t-1} - (\Vupper-V) y_t -(\Clower+V-\Vupper)u_t, \ \forall t \in [2, T]_{\Z}, \nonumber \\
 x_t-x_{t+1} + x_{t+2} & \geq \Clower y_t - (\Clower + V)y_{t+1} + \Clower y_{t+2}, \ \forall t \in [1, T-2]_{\Z} \nonumber
\end{align}
are valid and facet-defining for conv($P$).}
\end{proposition}
\begin{proof}
The proofs are similar and thus are omitted here.
\end{proof}

\section{Computational Experiments} \label{sec:comp-exper}
In this section, we show the effectiveness of our proposed strong valid inequalities on solving both the self-scheduling unit commitment (used by market participants)
and network-constrained unit commitment (used by system operators) problems. {Both are fundamental optimization problems in the power industry. In particular, for the U.S. electricity market only, there are more than 1,000 independent power producers (market participants), three regional synchronized power grids, eight electric reliability councils {(independent system operators)}, and about 150 control-area operators \cite{ElectricityRegulation, uselectricstat}. All of these system operators and market participants run the network-constrained and self-scheduling unit commitment problems respectively everyday.}

In our experiments, these two problems were solved on a computer node with two AMD Opteron 2378 Quad Core Processors at 2.4GHz and 4GB memory. 
IBM ILOG CPLEX 12.3 with a single thread was utilized as the MIP solver and the time limit {was set at} one hour per run.

\subsection{Self-Scheduling Unit Commitment Problem} \label{subsec:ssuc-experiment}
In this subsection, we report the computational results for the self-scheduling unit commitment {problem} used by market participants.
{For the self-scheduling problem, a market participant provides the generation scheduling for a unit and sells the electricity to the market so as to maximize {the} total profit, based on which the market participant obtain the optimal bidding strategies submitted to the system operators. The self-scheduling unit commitment problem is an important problem in power generation scheduling and extensive studies exist on this problem (see, e.g., \cite{ni2004optimal, li2005price, plazas2005multimarket, frangioni2006solving, cerisola2009stochastic}, among others).}

Besides the notation defined in Section \ref{sec:introduction}, {we let $\mbox{SU}$ ($\mbox{SD}$) represent the start-up (shut-down) cost of this unit and}
$p_t$ represent the electricity price at time period $t$. The mathematical formulation can be described as follows:
\begin{subeqnarray} \label{eqn-B:whole}
\max\limits_{x,y,u} \ && \sum_{t=1}^{T} \bigg( p_t x_t - f(x_t) \bigg)  - \sum_{t=2}^{T} \bigg( \mbox{SU} u_t + \mbox{SD} (y_{t-1} - y_t + u_t) \bigg)   \slabel{eqn-B:obj} \\
\mbox{s.t.} \ &&  \eqref{eqn:p-minup} - \eqref{eqn:p-ramp-down}, \nonumber \\
&& y_t \in \{0, 1\}, \ \forall t \in [1,T]_{\Z}; \ u_t  \in \{0,1\}, \ \forall t \in [2,T]_{\Z},
\slabel{eqn-B:nonnegativity}
\end{subeqnarray}
where the objective is to maximize the total profit, i.e., the total revenue (from selling electricity) minus the total cost (for producing electricity). Note here that the total cost includes the start-up cost, the shut-down cost, and the generation cost that is represented by $f(x_t)$, which is typically a nondecreasing quadratic function, i.e., $f(x_t) = a (x_t)^2 + b x_t + c y_t$. This function can usually be approximated by a piecewise linear function \cite{carrion2006computationally}.

{We run the experiments for the eight generators {from the data in \cite{carrion2006computationally} and \cite{ostrowski2012tight}, as shown} in Table \ref{tab:8-gen-data} and set the operational time horizon $T=5000$. For each generator in Table \ref{tab:8-gen-data}, we test three instances with the price $p_t$, $\forall t \in [1,T]_{\Z}$, randomly generated {in the interval described in Table \ref{tab:self-sch} (e.g., $p_t \in [0,35]$ for generators 1 and 2)} and report the average result.
We compare two formulations as follows.
\begin{itemize}
\item ``\textbf{MILP}'': The original MILP formulation {as described} in \eqref{eqn-B:whole}.
\item ``\textbf{Strong}'': The original MILP formulation plus our proposed strong valid inequalities.
\end{itemize}
In particular, here in ``Strong'' we add the strong valid inequalities in Section  \ref{sec:convexhullresults}, inequalities \eqref{eqn:x_t-ub-2-multi-period} and \eqref{eqn:ru-2-exp-1} in Section \ref{sec:multi-period}, and a subset of inequalities \eqref{eqn:x_t-up-exp}, \eqref{eqn:x_t-down-exp}, \eqref{eqn:ru-2-exp-2}, \eqref{eqn:rd-2-exp}, \eqref{eqn:ru-3-exp}, and \eqref{eqn:rd-3-exp} as constraints in the {original} model and all the remaining inequalities in the root node through separation process using the callback function.
Meanwhile, the optimality tolerance is set at $0.01\%$.}

\begin{table}[!htb]
  \centering
  \caption{Generator Data}
    \begin{tabular}{| @{} >{\footnotesize} c @{} | >{\footnotesize}c | >{\footnotesize}c |  >{\footnotesize}c | >{\footnotesize}c |  >{\footnotesize}c |  >{\footnotesize}c | >{\footnotesize}c |  >{\footnotesize}c   | >{\footnotesize}c | }
    \hline
    Generators & \tabincell{c}{$\Clower$ \\(MW)}  & \tabincell{c}{$\Cupper$ \\(MW)}  & \tabincell{c}{$L$/$\ell$ \\ (h)} & \tabincell{c}{ $V$ \\ (MW/h)} & \tabincell{c}{ $\Vupper$ \\ (MW/h)} & \tabincell{c}{\mbox{SU}\\(\$/h)} & \tabincell{c}{a \\(\$/$\mbox{MW}^2$h)}   & \tabincell{c}{b\\(\$/MWh)}   & \tabincell{c}{c \\ (\$/h)} \\
    \hline
    1     & 150   & 455   & 8     & 91    & 180   & 2000  & 0.00048 & 16.19 & 1000 \\
    2     & 150   & 455   & 8     & 91    & 180   & 2000  & 0.00031 & 17.26 & 970 \\
    3     & 20    & 130   & 5     & 26    & 35    & 500   & 0.002 & 16.6  & 700 \\
    4     & 20    & 130   & 5     & 26    & 35    & 500   & 0.00211 & 16.5  & 680 \\
    5     & 25    & 162   & 6     & 32.4  & 40    & 700   & 0.00398 & 19.7  & 450 \\
    6     & 20    & 80    & 3     & 16    & 28    & 150   & 0.00712 & 22.26 & 370 \\
    7     & 25    & 85    & 3     & 17    & 33    & 200   & 0.00079 & 27.74 & 480 \\
    8     & 10    & 55    & 1     & 11    & 15    & 60    & 0.00413 & 25.92 & 660 \\
    \hline
    \end{tabular}
  \label{tab:8-gen-data}
\end{table}

\begin{table}[htbp]
  \centering
  \caption{Computational Performance for Eight Single Generators}
	\begin{tabular}{ | @{} >{\footnotesize } c @{} | >{\footnotesize } c | @{} >{\footnotesize } c @{} |  >{\footnotesize } c |  >{\footnotesize } c |   >{\footnotesize } c  |  >{\footnotesize } c  |   >{\footnotesize } c  |   >{\footnotesize } c  |  >{\footnotesize } c   |  @{} >{\footnotesize } c  @{} |}    
    \hline
    \multirow{2}[0]{*}{Gen} & \multirow{2}[0]{*}{ \tabincell{c}{Price \\ (\$)}   } & \multirow{2}[0]{*}{ \tabincell{c}{Integer \\ OBJ. (\$)} } & \multicolumn{2}{>{\footnotesize } c|}{IGap (\%)} & \multirow{2}[0]{*}{\tabincell{c}{Percent \\ -age (\%)}} & \multicolumn{2}{>{\footnotesize } c|}{CPU Time(s) (TGap (\%))} & \multicolumn{2}{ >{\footnotesize } c|}{\# of Nodes} & \# of User \\
    \cline{4-5} \cline{7-11}
          &   &    & MILP  & Strong &       & MILP  & Strong & MILP  & Strong & Strong \\
          \hline
    1     & [0,35] & 2378708 & 30.94 & 0.07  & 99.76 & *** (3.62) [3] & 73    & 31157 & 0     & 0 \\
    2     & [0,35] & 1640999 & 38.88 & 0.12  & 99.69 & *** (4.86) [3] & 178.7 & 46947 & 0     & 0 \\
    3     & [0,44] & 991532 & 36.44 & 0.08  & 99.78 & *** (4.84) [3] & 104.1 & 26640 & 0     & 0 \\
    4     & [0,44] & 1047174 & 35.08 & 0.1   & 99.7  & *** (3.92) [3] & 105.7 & 22841 & 0     & 0 \\
    5     & [0,44] & 1126133 & 31.29 & 0.32  & 98.99 & *** (1.71) [3] & 126.7 & 58398 & 0     & 0 \\
    6     & [0,48] & 295958 & 58.33 & 0.34  & 99.42 & *** (7.68) [3] & 124.1 & 40190 & 0     & 0 \\
    7     & [0,60] & 562857 & 49.86 & 0.19  & 99.61 & *** (5.87) [3] & 48.7  & 31366 & 0     & 0 \\
    8     & [0,67] & 147698 & 79.6  & 0.17  & 99.78 & *** (24.08) [3] & 193.2 & 32211 & 0     & 144 \\
    \hline
    \end{tabular}
  \label{tab:self-sch}
\end{table}

{The computational results are reported in Table \ref{tab:self-sch}.
The column labelled ``Price (\$)'' provides the price setting and the column labelled ``Integer OBJ. (\$)'' {(denoted as $Z_{\tiny \mbox{MILP}}$)} provides the best objective value obtained from both of ``MILP'' and ``Strong'' within the time limit.
The column labelled ``IGap (\%)'' provides the {root-node} integrality gaps of ``MILP'' and ``Strong'', respectively, in which the integrality gap is defined as $(Z_{\tiny \mbox{LP}} - Z_{\tiny \mbox{MILP}})/Z_{\tiny \mbox{LP}}$, where $Z_{\tiny \mbox{LP}}$ is the objective value of the LP relaxation for the corresponding formulation.
We can observe that, our proposed strong valid inequalities tighten the LP relaxation dramatically  because the integrality gap reduction (from ``MILP'' to ``Strong'') is more than $99\%$ for most instances as reported in the column labelled ``Percentage (\%)''.
In the column labelled ``CPU Time(s) (TGap ($\%$))'', we report the computational time to solve the problem for each approach. 
The number in the square bracket indicates the number of instances (out of three for each generator) not solved into optimality within the time limit. {The label} ``***'' indicates that none of three cases are solved into optimality within the time limit, and accordingly we report the terminating gap labelled ``TGap ($\%$)'', which indicates the relative gap between the objective value corresponding to the best integer solution and the best upper bound when the time limit is reached. We can observe that our ``Strong'' approach can solve all instances within $200$ seconds, while ``MILP'' cannot solve any instance within the time limit (i.e., one hour).
The key advantage here is that our approach solves all instances at the root node without getting into the branch-and-bound procedure, due to the strength of our derived inequalities, which is indicated as $0$ in the ``\# of Nodes'' column that provides the number of explored branch-and-bound nodes.
The final column labelled ``\# of User'' reports the number of user cuts added to solve the problem for our proposed approach.}

\subsection{Network-Constrained Unit Commitment Problem} \label{subsec:uc-experiment}
{For the network-constrained unit commitment problem, 
the system operators schedule generators to be online/offline in order to satisfy the system load over a finite discrete time horizon with the objective of minimizing the total cost for a wholesale electricity market. Due to its importance, various algorithmic approaches have been proposed to find the (near-) optimal solutions \cite{padhy2004unit, saravanan2013solution}.}
In our experiments, we report the computational results for the power system data based on \cite{carrion2006computationally} and \cite{ostrowski2012tight}, and a modified IEEE 118-bus system based on the one given online at \url{http://motor.ece.iit.edu/data/SCUC_118/}. The operational time interval is set to be $24$ hours (i.e., $T=24$). 

{We first introduce} the mathematical formulation. {Besides the notation described in Sections \ref{sec:introduction} and \ref{subsec:ssuc-experiment} (adding superscript $k$ to represent generator $k$), we let $\mathbb{G}$ ($\mathbb{G}_b$), $\mathbb{B}$,  and $\mathbb{E}$} represent the set of generators {(the set of generators at bus $b$), the set of buses, and the set of} transmission lines linking two buses.
In addition, we let $d_t^b$ represent the load (demand) at bus $b$ at time period $t$ and $r_t$ represent the system reserve factor at $t$.
For each transmission line {$(m, n) \in \mathbb{E}$}, we let {$C_{mn}$} represent {its} capacity 
and {$K_{mn}^b$} represent the line flow distribution factor for the flow on the transmission line {$(m, n)$} contributed by the net injection at bus $b$.
Accordingly, the network-constrained unit commitment problem can be described as follows:
\begin{subeqnarray} \label{eqn-A:whole}
\min\limits_{x,y,u} \ && \sum_{{k \in \mathbb{G}}} \bigg(\sum_{t=2}^{T} \Bigl( \mbox{SU}^{{k}} u_t^{{k}} + \mbox{SD}^{{k}} (y_{t-1}^{{k}} - y_t^{{k}} + u_t^{{k}}) \Bigr) + \sum_{t=1}^{T} f^{{k}}(x_t^{{k}}) \bigg) \slabel{eqn-A:obj} \\
\mbox{s.t.} \ && \sum_{i=t- L^{{k}} +1}^t u_i^{{k}} \leq y_{t}^{{k}}, \ \ \forall t \in [L^{{k}} +1, T]_{\Z}, \forall {k \in \mathbb{G}}, \slabel{eqn-A:minup} \\
&& \sum_{i=t- \ell^{{k}} +1}^t u_i^{{k}} \leq 1 - y_{t - \ell^{{k}} }^{{k}}, \ \ \forall t \in [\ell^{{k}} +1, T]_{\Z}, \forall {k \in \mathbb{G}},  \slabel{eqn-A:mindown} \\
&& - y_{t-1}^{{k}} + y_t^{{k}} - u_t^{{k}} \leq 0, \ \ \forall t \in [2, T]_{\Z}, \forall {k \in \mathbb{G}}, \slabel{eqn-A:Startup} \\
&& \underline{C}^{{k}} y_t^{{k}} \leq  x_t^{{k}} \leq \overline{C}^{{k}} y_t^{{k}}, \ \ \forall t \in [1,T]_{\Z}, \forall {k \in \mathbb{G}}, \slabel{eqn-A:Generationbound}\\
&& x_t^{{k}} - x_{t-1}^{{k}} \leq V^{{k}} y_{t-1}^{{k}} + \Vupper^{{k}} (1 - y_{t-1}^{{k}}), \ \ \forall t \in [2, T]_{\Z}, \forall {k \in \mathbb{G}}, \slabel{eqn-A:rampup}\\
&& x_{t-1}^{{k}} - x_t^{{k}} \leq V^{{k}} y_t^{{k}} + \Vupper^{{k}} (1 - y_t^{{k}}), \ \ \forall t \in [2, T]_{\Z}, \forall {k \in \mathbb{G}}, \slabel{eqn-A:rampdown}\\
&& \sum_{{k \in \mathbb{G}}} x_t^{{k}} = \sum_{{b \in \mathbb{B}}} d_t^b, \ \ \forall t \in [1, T]_{\Z}, \slabel{eqn-A:Demand}\\
&& \sum_{{k \in \mathbb{G}}} \Cupper^{{k}} y_t^{{k}} \geq (1+r_t) \sum_{{b \in \mathbb{B}}} d_t^b, \ \ \forall t \in [1, T]_{\Z}, \slabel{eqn-A:spinning-reserve}\\
&& -C_{{mn}} \leq \sum_{{b \in \mathbb{B}}} K^b_{{mn}} \Bigl( \sum_{{k \in \mathbb{G}_b}} x_t^{{k}} - d_t^b \Bigr) \leq C_{{mn}}, \ \ \forall t \in [1,T]_{\Z}, \forall {(m,n) \in \mathbb{E}}, \slabel{eqn-A:transmission-constraint} \\
&& y_t^{{k}} \in \{0, 1\}, \ \forall t \in [1,T]_{\Z}, {\forall k \in \mathbb{G}}; \ u_t^{{k}}  \in \{0,1\}, \ \forall t \in [2,T]_{\Z}, \forall {k \in \mathbb{G}},
\slabel{eqn-A:nonnegativity}
\end{subeqnarray}
where the objective is to minimize the total cost.
{Besides the physical constraints for each generator (i.e., constraints \eqref{eqn-A:minup} - \eqref{eqn-A:rampdown}) as described in Section \ref{sec:introduction}, we have} 
constraints \eqref{eqn-A:Demand} {enforcing} the load balance at each time period $t$,
constraints \eqref{eqn-A:spinning-reserve} {describing} the system reserve requirements,
and constraints \eqref{eqn-A:transmission-constraint} {representing} the capacity limit of each transmission line {$(m, n)$} (see, e.g.,~\cite{wang1995short}).

\subsubsection{Power System Data Based on \cite{carrion2006computationally} and \cite{ostrowski2012tight}} \label{subsubsec:ostdata}
{We test twenty instances (see Table \ref{tab:ost-inst-comb}) with each containing different combinations of generators as shown in Table \ref{tab:8-gen-data}.} The system load setting is provided in Table \ref{tab:ost-load} and the system reserve factor $r_t = 3\%$ for all $t \in [1,T]_{\Z}$.
{In this experiment}, constraints \eqref{eqn-A:transmission-constraint} are not included since the transmission data are not provided in \cite{carrion2006computationally} and \cite{ostrowski2012tight}. Meanwhile, we set the optimality tolerance to {be} $0.05\%$.

\begin{table}[!htb]
  \centering
  \caption{Problem Instances \cite{ostrowski2012tight}}
    \begin{tabular}{| >{\footnotesize} c | >{\footnotesize} c >{\footnotesize} c >{\footnotesize} c >{\footnotesize} c >{\footnotesize} c >{\footnotesize} c >{\footnotesize} c >{\footnotesize} c | >{\footnotesize} c | }
    \hline
    \multirow{2}[0]{*}{Instances} & \multicolumn{8}{>{\footnotesize} c |}{Generators}   & \multirow{2}[0]{*}{\tabincell{c}{\# of \\ Generators}} \\
    \cline{2-9}
          & 1     & 2     & 3     & 4     & 5     & 6     & 7     & 8     &  \\
	\hline
    1     & 12    & 11    & 0     & 0     & 1     & 4     & 0     & 0     & 28 \\
    2     & 13    & 15    & 2     & 0     & 4     & 0     & 0     & 1     & 35 \\
    3     & 15    & 13    & 2     & 6     & 3     & 1     & 1     & 3     & 44 \\
    4     & 15    & 11    & 0     & 1     & 4     & 5     & 6     & 3     & 45 \\
    5     & 15    & 13    & 3     & 7     & 5     & 3     & 2     & 1     & 49 \\
    6     & 10    & 10    & 2     & 5     & 7     & 5     & 6     & 5     & 50 \\
    7     & 17    & 16    & 1     & 3     & 1     & 7     & 2     & 4     & 51 \\
    8     & 17    & 10    & 6     & 5     & 2     & 1     & 3     & 7     & 51 \\
    9     & 12    & 17    & 4     & 7     & 5     & 2     & 0     & 5     & 52 \\
    10    & 13    & 12    & 5     & 7     & 2     & 5     & 4     & 6     & 54 \\
    11    & 46    & 45    & 8     & 0     & 5     & 0     & 12    & 16    & 132 \\
    12    & 40    & 54    & 14    & 8     & 3     & 15    & 9     & 13    & 156 \\
    13    & 50    & 41    & 19    & 11    & 4     & 4     & 12    & 15    & 156 \\
    14    & 51    & 58    & 17    & 19    & 16    & 1     & 2     & 1     & 165 \\
    15    & 43    & 46    & 17    & 15    & 13    & 15    & 6     & 12    & 167 \\
    16    & 50    & 59    & 8     & 15    & 1     & 18    & 4     & 17    & 172 \\
    17    & 53    & 50    & 17    & 15    & 16    & 5     & 14    & 12    & 182 \\
    18    & 45    & 57    & 19    & 7     & 19    & 19    & 5     & 11    & 182 \\
    19    & 58    & 50    & 15    & 7     & 16    & 18    & 7     & 12    & 183 \\
    20    & 55    & 48    & 18    & 5     & 18    & 17    & 15    & 11    & 187 \\
    \hline
    \end{tabular}
  \label{tab:ost-inst-comb}
\end{table}

\begin{table}[!htb]
  \centering
  \caption{System Load (\% of Total Capacity) \cite{ostrowski2012tight}}
    \begin{tabular}{ | >{\footnotesize} c | >{\footnotesize} c >{\footnotesize} c >{\footnotesize} c >{\footnotesize} c >{\footnotesize} c >{\footnotesize} c >{\footnotesize} c >{\footnotesize} c >{\footnotesize} c >{\footnotesize} c >{\footnotesize} c >{\footnotesize} c | }
    \hline
    Time  & 1     & 2     & 3     & 4     & 5     & 6     & 7     & 8     & 9     & 10    & 11    & 12 \\
    Load  & 71\%    & 65\%    & 62\%    & 60\%    & 58\%    & 58\%    & 60\%    & 64\%    & 73\%    & 80\%    & 82\%    & 83\% \\
    \hline
    Time  & 13    & 14    & 15    & 16    & 17    & 18    & 19    & 20    & 21    & 22    & 23    & 24 \\
    Load  & 82\%    & 80\%    & 79\%    & 79\%    & 83\%    & 91\%    & 90\%    & 88\%    & 85\%    & 84\%    & 79\%    & 74\% \\
    \hline
    \end{tabular}
  \label{tab:ost-load}
\end{table}

{The computational results are reported in Table \ref{tab:ucresult-1} {with similar format as described in Table \ref{tab:self-sch}}.
For each instance, we compare two formulations as follows.
\begin{itemize}
\item ``\textbf{MILP}'': The original MILP formulation {as described} in \eqref{eqn-A:whole}.
\item ``\textbf{Strong}'': The original MILP formulation plus our proposed strong valid inequalities. {In our experiments}, strong valid inequalities {describing} the two-period convex hull (i.e., \eqref{eqn-q2:x1-ub} - \eqref{eqn-q2:x1-x2-ub}) in Section \ref{sec:two-period} are added as constraints and all the strong valid inequalities in Sections \ref{sec:three-period} and \ref{sec:multi-period} are added as user cuts.
{In particular, we perform separation for the majority of the {exponentially sized} inequalities (i.e., \eqref{eqn:x_t-up-exp}, \eqref{eqn:x_t-down-exp}, \eqref{eqn:ru-2-exp-2}, \eqref{eqn:rd-2-exp}, \eqref{eqn:ru-3-exp}, and \eqref{eqn:rd-3-exp}) and add them in the first fifty branch-and-bound nodes for instances 1-10 and in the root node for instances 11-20 to improve the computational performance. All the remaining inequalities are added into the user cut pool.}
\end{itemize}}

\begin{table}[!htb]
  \centering
  \caption{Computational Performance for the Data Based on \cite{carrion2006computationally} and \cite{ostrowski2012tight}}
	\begin{tabular}{ | @{} >{\footnotesize } c @{} |  >{\footnotesize } c |  >{\footnotesize } c |  >{\footnotesize } c |   >{\footnotesize } c  |  >{\footnotesize } c  |   >{\footnotesize } c  |   >{\footnotesize } c  |  >{\footnotesize } c   |  @{} >{\footnotesize } c  @{} |}    
    \hline
    \multirow{2}[0]{*}{Instance} & \multirow{2}[0]{*}{ \tabincell{c}{Integer \\ OBJ. (\$)} } & \multicolumn{2}{>{\footnotesize } c|}{IGap (\%)} & \multirow{2}[0]{*}{\tabincell{c}{Percent \\ -age (\%)}} & \multicolumn{2}{>{\footnotesize } c|}{CPU Time(s) (TGap (\%))} & \multicolumn{2}{ >{\footnotesize } c|}{\# of Nodes} & \# of User \\
    \cline{3-4} \cline{6-10}
          &       & MILP  & Strong &       & MILP  & Strong & MILP  & Strong & Strong \\
          \hline
    1     & 3794100 & 0.76  & 0.12  & 84.94 & *** (0.176) & *** (0.064) & 69160 & 74943 & 514 \\
    2     & 4770702 & 0.78  & 0.14  & 82.56 & *** (0.316) & *** (0.069) & 37322 & 53580 & 757 \\
    3     & 5080219 & 0.82  & 0.06  & 92.54 & *** (0.209) & 495.2 & 76580 & 7527  & 811 \\
    4     & 4755459 & 0.78  & 0.05  & 93.28 & *** (0.128) & 118   & 82003 & 173   & 538 \\
    5     & 5354093 & 0.91  & 0.04  & 95.63 & *** (0.186) & 192.6 & 64961 & 520   & 1169 \\
    6     & 4383515 & 1.1   & 0.05  & 95.73 & *** (0.092) & 72.8  & 76723 & 20    & 456 \\
    7     & 5784806 & 0.75  & 0.08  & 88.9  & *** (0.224) & 1136.9 & 81341 & 15490 & 455 \\
    8     & 5136989 & 0.96  & 0.04  & 95.63 & *** (0.179) & 270.3 & 68900 & 2016  & 605 \\
    9     & 5584255 & 0.91  & 0.05  & 94.74 & *** (0.202) & 317.5 & 49942 & 3139  & 1242 \\
    10    & 5046533 & 1.16  & 0.07  & 94    & *** (0.255) & *** (0.053) & 83485 & 83336 & 978 \\
    11    & 15681261 & 0.72  & 0.07  & 89.82 & *** (0.371) & 2458.8 & 19929 & 8900  & 2026 \\
    12    & 17080641 & 0.79  & 0.05  & 94.28 & *** (0.26) & 537.3 & 14070 & 454   & 1496 \\
    13    & 16758967 & 0.86  & 0.04  & 95.62 & *** (0.264) & 349.5 & 16302 & 210   & 1200 \\
    14    & 19981542 & 0.82  & 0.06  & 92.72 & *** (0.275) & 974.8 & 8407  & 469   & 2666 \\
    15    & 17245739 & 0.95  & 0.05  & 95.14 & *** (0.264) & 535   & 20800 & 270   & 3262 \\
    16    & 19345288 & 0.75  & 0.06  & 92.25 & *** (0.346) & 920.3 & 11924 & 240   & 2259 \\
    17    & 19537985 & 0.89  & 0.04  & 95.3  & *** (0.251) & 532   & 12471 & 354   & 2339 \\
    18    & 19459519 & 0.87  & 0.05  & 94.55 & *** (0.251) & 3137.3 & 12640 & 1220  & 2765 \\
    19    & 19966729 & 0.83  & 0.05  & 94.44 & *** (0.257) & 389.4 & 11667 & 80    & 1870 \\
    20    & 19574287 & 0.87  & 0.04  & 95.27 & *** (0.237) & 658.3 & 15531 & 80    & 2028 \\
    \hline
    \end{tabular}%
  \label{tab:ucresult-1}
\end{table}%

{In Table \ref{tab:ucresult-1}, the  integrality gap is defined as $(Z_{\tiny \mbox{MILP}} - Z_{\tiny \mbox{LP}})/Z_{\tiny \mbox{MILP}}$.
Note here that since the network-constrained unit commitment problem  is a minimization problem, the integrality gap definition is different from that defined for the self-scheduling unit commitment problem in Section \ref{subsec:ssuc-experiment}. 
From the table, we can observe significant advantages of applying our derived strong valid inequalities as cutting planes.
In particular, the integrality gap reduction ({from} ``MILP'' {to} ``Strong'') is more than $90\%$ for most instances.
All instances cannot be solved into optimality by ``MILP'' within the time limit, while most instances can be solved by the  ``Strong'' approach with our proposed strong valid inequalities added. 
For the instances which cannot be solved into optimality by either one of them,
i.e., Instances 1, 2, and 10, ``Strong'' leads to much better terminating gaps than ``MILP'' does. Meanwhile, our ``Strong'' approach explores much less branch-and-bound nodes than ``MILP'' does.}

\subsubsection{Modified IEEE 118-Bus System} \label{subsubsec:118busdata}
For this experiment {of testing the modified IEEE 118-Bus System}, there are $54$ generators, $118$ buses, $186$ transmission lines, and $91$ load buses. In the experiment, both constraints \eqref{eqn-A:spinning-reserve} and \eqref{eqn-A:transmission-constraint} are included with the system reserve factor $r_t$ set at $3 \%$ for each time period $t \in [1, T]_{\Z}$. Meanwhile, the optimality tolerance {is set at} $0.01\%$.

{We first test seven instances, with different instances corresponding to different load settings, as described in Table \ref{tab:ucresult-2-multi-load}. 
For example, for Instance 1, corresponding to each nominal load {$d_t^b$} given in the IEEE 118-bus system, we randomly generate a load {$\bar{d}_t^b \in [0.5d_t^b, 0.7 d_t^b]$}.
In addition, for each instance, we randomly generate three cases and report the average result. We compare two formulations ``MILP'' and ``Strong'' similarly as {described} in Section \ref{subsubsec:ostdata}.}

\begin{table}[htbp]
  \centering
  \caption{Computational Performance for the IEEE 118-Bus System - Multiple Load Settings}
\begin{tabular}{ | @{} >{\footnotesize } c @{} | >{\footnotesize } c  | @{} >{\footnotesize } c @{} | @{}  >{\footnotesize } c  | @{} >{\footnotesize } c @{} | @{} >{\footnotesize } c @{} | @{} >{\footnotesize } c | @{} >{\footnotesize } c @{} | @{} >{\footnotesize } c @{} | @{} >{\footnotesize } c @{} | @{} >{\footnotesize } c | @{} >{\footnotesize } c @{} | @{} >{\footnotesize } c @{} |}    
    \hline
    \multirow{2}[0]{*}{Inst} & \multirow{2}[0]{*}{Load} & \multirow{2}[0]{*}{ \tabincell{c}{Integer \\ OBJ. (\$)} } & \multicolumn{2}{ >{\footnotesize } c|}{IGap (\%)} & \multirow{2}[0]{*}{ \tabincell{c}{Percent \\ -age (\%)}    } & \multicolumn{2}{ >{\footnotesize }  c|}{ CPU Time (s)  } & \multicolumn{2}{>{\footnotesize }  c|}{TGap (\%)} & \multicolumn{2}{>{\footnotesize }  c|}{\# of Nodes} & \# of User \\
    \cline{4-5} \cline{7-13}
          &       &       & MILP  & Strong &       & MILP  & Strong & MILP  & Strong & MILP  & Strong & Strong \\
          \hline
    1     & [0.5d-0.7d] & 974357 & 1.59  & 0.77  & 51.38 & *** [3] & 1572.1 [2] & 0.365 & 0.084 & 40564 & 17518 & 1511 \\
    2     & [0.7d-0.9d] & 1300697 & 1.14  & 0.33  & 71.18 & *** [3] & 1249.8 & 0.126 & 0.000 & 32544 & 8927  & 1138 \\
    3     & [0.9d-1.1d] & 1627653 & 1.17  & 0.37  & 68.35 & *** [3] & 2283.6 [1] & 0.229 & 0.052 & 23592 & 10576 & 1897 \\
    4     & [1.1d-1.3d] & 1958258 & 1.14  & 0.34  & 70.27 & *** [3] & 2199.6 [1] & 0.263 & 0.077 & 26810 & 16891 & 2213 \\
    5     & [1.3d-1.5d] & 2289810 & 1.09  & 0.18  & 83.73 & *** [3] & 3155.6 [2] & 0.101 & 0.061 & 94305 & 27337 & 3171 \\
    6     & [1.5d-1.7d] & 2635350 & 1.16  & 0.14  & 87.98 & *** [3] & 2470.4 & 0.114 & 0.000 & 56800 & 15305 & 2839 \\
    7     & [1.7d-1.9d] & 2992284 & 1.31  & 0.08  & 93.62 & *** [3] & *** [3] & 0.152 & 0.038 & 63380 & 49099 & 1945 \\
    \hline
    \end{tabular}%
  \label{tab:ucresult-2-multi-load}
\end{table}%

{In Table \ref{tab:ucresult-2-multi-load}, for each instance, we provide the load setting in the column labelled ``Load'' and report the average objective value over three cases in the column labelled ``Integer OBJ. (\$)''.
Besides these, similarly as we did in Table \ref{tab:self-sch}, we report the average root-node integrality gap (IGap (\%)), the percentage of gap reduction, the corresponding time for each approach (CPU Time (s)), the terminating gap (TGap (\%)), the number of nodes, and the number of user cuts added.
We can observe that the ``Strong'' approach performs much better than ``MILP'' does as ``Strong'' tightens the original formulation dramatically (with over $50\%-90\%$ integrality gap reduction), solves the problem faster (e.g., no instance was solved into optimality by the default MILP formulation, while our strong formulation can solve quite a few instances into optimality within the time limit), obtains better (smaller) terminating gaps, and explores less branch-and-bound nodes.}

Next, we explore more instances with the same load setting.
{To create more instances, corresponding to each nominal load $d_t^b$ given in the IEEE 118-bus system, we construct $15$ instances {with} the load for each instance uniformly distributed in $[1.8d_t^b, 2.2 d_t^b]$.}
We compare two formulations ``MILP'' and ``Strong'' as defined in Section \ref{subsubsec:ostdata}. {The} only difference is that here we add all inequalities in Section \ref{sec:three-period}, {inequalities \eqref{eqn:x_t-ub-2-multi-period} and \eqref{eqn:ru-2-exp-1} in Section \ref{sec:multi-period}, and a subset of inequalities \eqref{eqn:x_t-up-exp}, \eqref{eqn:x_t-down-exp}, \eqref{eqn:ru-2-exp-2}, \eqref{eqn:rd-2-exp}, \eqref{eqn:ru-3-exp}, and \eqref{eqn:rd-3-exp}} in the user cut pool instead of using {the callback function}.

\begin{table}[!htb]
  \centering
  \caption{Computational Performance for the IEEE 118-Bus System - Single Load Setting}
    \begin{tabular}{|  >{\footnotesize} c  |  >{\footnotesize} c |  >{\footnotesize} c | >{\footnotesize} c  |  >{\footnotesize} c |  >{\footnotesize} c |  >{\footnotesize} c  | >{\footnotesize} c |  >{\footnotesize} c | @{} >{\footnotesize} c @{} |}
    \hline
    \multirow{2}[0]{*}{Inst} & \multirow{2}[0]{*}{\tabincell{c}{Integer \\ OBJ. (\$)}} & \multicolumn{2}{>{\footnotesize} c|}{IGap (\%)} & \multirow{2}[0]{*}{ \tabincell{c}{Percent \\ -age (\%)} } & \multicolumn{2}{>{\footnotesize} c |}{CPU Time(s) (TGap ($10^{-4}$))} & \multicolumn{2}{>{\footnotesize} c | }{\# of Nodes} & \# of User \\
    \cline{3-4} \cline{6-10}
          &       & MILP  & Strong &       & MILP  & Strong & MILP  & Strong &  Strong \\
          \hline
    1     & 3358217 & 1.54  & 0.09  & 94.42 & *** (1.39) & 1432.02 & 180121 & 85936 & 100 \\
    2     & 3356847 & 1.37  & 0.05  & 96.65 & *** (1.43) & 2371.36 & 229259 & 342774 & 222 \\
    3     & 3367104 & 1.61  & 0.06  & 96.29 & *** (3) & *** (1.8) & 159795 & 136426 & 340 \\
    4     & 3362632 & 1.64  & 0.06  & 96.26 & *** (1.96) & *** (1.37) & 272480 & 238904 & 225 \\
    5     & 3349280 & 1.47  & 0.09  & 93.97 & *** (2.23) & *** (1.47) & 150695 & 373875 & 299 \\
    6     & 3364177 & 1.45  & 0.07  & 95.28 & *** (1.28) & 848.11 & 152427 & 69191 & 257 \\
    7     & 3353272 & 1.58  & 0.08  & 95.19 & *** (2.29) & *** (1.51) & 180557 & 594986 & 182 \\
    8     & 3348885 & 1.27  & 0.04  & 97.12 & 758.44 & 289.94 & 54354 & 28080 & 215 \\
    9     & 3354399 & 1.5   & 0.06  & 96.02 & *** (3.27) & *** (1.9) & 127050 & 102107 & 199 \\
    10    & 3352652 & 1.53  & 0.06  & 96.21 & *** (1.91) & *** (1.38) & 191125 & 187788 & 280 \\
    11    & 3357921 & 1.54  & 0.06  & 95.85 & *** (1.31) & 665.88 & 166568 & 58687 & 249 \\
    12    & 3359379 & 1.55  & 0.05  & 96.57 & 1074.87 & 405.07 & 94365 & 29781 & 262 \\
    13    & 3359624 & 1.57  & 0.07  & 95.78 & *** (1.23) & 1162.33 & 166052 & 66590 & 236 \\
    14    & 3362072 & 1.57  & 0.06  & 96.07 & 671.6 & 480.58 & 36746 & 19262 & 271 \\
    15    & 3351562 & 1.51  & 0.1   & 93.61 & *** (2.12) & 2615.75 & 142626 & 98899 & 294 \\
    \hline
    \end{tabular}
  \label{tab:ucresult-2}
\end{table}

The results are reported in Table \ref{tab:ucresult-2} with the similar format as described in Table \ref{tab:ucresult-1}. Again, we
observe that the strong valid inequalities tighten the LP relaxation significantly, with about $95\%$ reduction {for the integrality gaps between ``MILP'' and ``Strong''.} ``Strong'' also performs much better {in terms of the computational time and terminating gap {than ``MILP'' does}. 
{For instance, ``MILP'' can only solve three (i.e., instances $8$, $12$, and $14$) out of fifteen instances into optimality within the time limit, while ``Strong'' can solve up to nine. Meanwhile, we can observe that significant terminating gap reduction within the time limit is reached.}
The number of explored branch-and-bound nodes is also reduced for most instances.

\section{Conclusion} \label{sec:conclusion}
In this paper, we performed the polyhedral study of the integrated minimum-up/-down time and ramping polytope, {which has broad applications in variant industries}.
We derived strong valid inequalities to strengthen the original MILP formulation. 
{Through developing a new proof technique, we can show that} our derived valid inequalities are strong enough to provide the convex hull descriptions for the polytope up to three time periods with variant minimum-up/-down time limits. To the best of our knowledge, this is {one of the first studies} that provide the convex hull descriptions for the three-period cases. 
In addition, our derived strong valid inequalities for the general multi-period case cover one, two, and three continuous variables, respectively. All these inequalities are facet-defining under certain conditions {and can be separated efficiently through polynomial {time} algorithms}.
Finally, the computational results \sred{verified} the {effectiveness} of our proposed strong valid inequalities by solving both {the self-scheduling unit commitment problem for a market participant and the network-constrained unit commitment problem for a system operator} under various data settings.

\baselineskip=12pt
\bibliographystyle{plain}
\bibliography{duc}

\begin{thebibliography}{10}

\bibitem{ElectricityRegulation}
Electricity \mbox{R}egulation \mbox{I}n the \mbox{US}: \mbox{A Guide}.
\newblock
  \url{http://www.raponline.org/docs/RAP_Lazar_ElectricityRegulationInTheUS_Guide_2011_03.pdf},
  Access on March 28, 2016.

\bibitem{uselectricstat}
\mbox{U.S.} \mbox{E}lectric \mbox{U}tility \mbox{I}ndustry \mbox{S}tatistics.
\newblock
  \url{http://www.publicpower.org/files/PDFs/USElectricUtilityIndustryStatistics.pdf},
  Access on March 28, 2016.

\bibitem{atamturk2006strong}
A.~Atamt{\"u}rk.
\newblock Strong formulations of robust mixed 0--1 programming.
\newblock {\em Mathematical Programming}, 108(2-3):235--250, 2006.

\bibitem{barany1984uncapacitated}
I.~Barany, T.~Van~Roy, and L.~A. Wolsey.
\newblock Uncapacitated lot-sizing: The convex hull of solutions.
\newblock {\em Mathematical Programming Study}, 22:32--43, 1984.

\bibitem{bixby2010mixed}
R.~E. Bixby.
\newblock Mixed-integer programming: It works better than you may think.
\newblock In {\em FERC Conference}, 2010.

\bibitem{carlson2012miso}
B.~Carlson, Y.~Chen, M.~Hong, R.~Jones, K.~Larson, X.~Ma, P.~Nieuwesteeg,
  H.~Song, K.~Sperry, M.~Tackett, D.~Taylor, J.~Wan, and E.~Zak.
\newblock \mbox{MISO} unlocks billions in savings through the application of
  operations research for energy and ancillary services markets.
\newblock {\em Interfaces}, 42(1):58--73, 2012.

\bibitem{carrion2006computationally}
M.~Carri{\'o}n and J.~M. Arroyo.
\newblock A computationally efficient mixed-integer linear formulation for the
  thermal unit commitment problem.
\newblock {\em IEEE Transactions on Power Systems}, 21(3):1371--1378, 2006.

\bibitem{cerisola2009stochastic}
S.~Cerisola, {\'A}.~Ba{\'\i}llo, J.~M. Fern{\'a}ndez-L{\'o}pez, A.~Ramos, and
  R.~Gollmer.
\newblock Stochastic power generation unit commitment in electricity markets: A
  novel formulation and a comparison of solution methods.
\newblock {\em Operations Research}, 57(1):32--46, 2009.

\bibitem{cormen2009introduction}
T.~H. Cormen, C.~E. Leiserson, R.~L. Rivest, and C.~Stein.
\newblock {\em Introduction to Algorithms}.
\newblock MIT press, 2009.

\bibitem{damci1777polyhedral}
P.~Damc{\i}-Kurt, S.~K{\"u}{\c{c}}{\"u}kyavuz, D.~Rajan, and A.~Atamt{\"u}rk.
\newblock A polyhedral study of production ramping.
\newblock {\em Mathematical Programming}, Online First, June 2015.

\bibitem{dubost2005primal}
L.~Dubost, R.~Gonzalez, and C.~Lemar{\'e}chal.
\newblock A primal-proximal heuristic applied to the \mbox{F}rench
  unit-commitment problem.
\newblock {\em Mathematical Programming}, 104(1):129--151, 2005.

\bibitem{frangioni2006solving}
A.~Frangioni and C.~Gentile.
\newblock Solving nonlinear single-unit commitment problems with ramping
  constraints.
\newblock {\em Operations Research}, 54(4):767--775, 2006.

\bibitem{guan1999scheduling}
X.~Guan, A.~Svoboda, and C.~Li.
\newblock Scheduling hydro power systems with restricted operating zones and
  discharge ramping constraints.
\newblock {\em IEEE Transactions on Power Systems}, 14(1):126--131, 1999.

\bibitem{haller2003cycle}
M.~Haller, A.~Peikert, and J.~Thoma.
\newblock Cycle time management during production ramp-up.
\newblock {\em Robotics and Computer-Integrated Manufacturing}, 19(1):183--188,
  2003.

\bibitem{havel2013optimal}
P.~Havel and T.~{\v{S}}imovi{\v{c}}.
\newblock Optimal planning of cogeneration production with provision of
  ancillary services.
\newblock {\em Electric Power Systems Research}, 95:47--55, 2013.

\bibitem{kuccukyavuz2009uncapacitated}
S.~K{\"u}{\c{c}}{\"u}kyavuz and Y.~Pochet.
\newblock Uncapacitated lot sizing with backlogging: the convex hull.
\newblock {\em Mathematical Programming}, 118(1):151--175, 2009.

\bibitem{lakshmanan2009study}
S.~P. Lakshmanan, M.~Pandey, P.~P. Kumar, and K.~N. Iyer.
\newblock Study of startup transients and power ramping of natural circulation
  boiling systems.
\newblock {\em Nuclear Engineering and Design}, 239(6):1076--1083, 2009.

\bibitem{lee2004min}
J.~Lee, J.~Leung, and F.~Margot.
\newblock Min-up/min-down polytopes.
\newblock {\em Discrete Optimization}, 1(1):77--85, 2004.

\bibitem{li2005price}
T.~Li and M.~Shahidehpour.
\newblock Price-based unit commitment: a case of \mbox{L}agrangian relaxation
  versus mixed integer programming.
\newblock {\em IEEE Transactions on Power Systems}, 20(4):2015--2025, 2005.

\bibitem{lin2002optimal}
G.~Lin, J.~Chen, and B.~Hua.
\newblock Optimal analysis on the performance of a chemical engine-driven
  chemical pump.
\newblock {\em Applied Energy}, 72(1):359--370, 2002.

\bibitem{liu2009component}
C.~Liu, M.~Shahidehpour, Z.~Li, and M.~Fotuhi-Firuzabad.
\newblock Component and mode models for the short-term scheduling of
  combined-cycle units.
\newblock {\em IEEE Transactions on Power Systems}, 24(2):976--990, 2009.

\bibitem{liustudy}
S.~Liu and D.~Rajan.
\newblock A study of three-period ramp-up polytope.
\newblock {\em Technical Report}, [Online]
  \url{http://www.optimization-online.org/DB_HTML/2015/09/5124.html}, 2015.

\bibitem{lohndorf2013optimizing}
N.~L{\"o}hndorf, D.~Wozabal, and S.~Minner.
\newblock Optimizing trading decisions for hydro storage systems using
  approximate dual dynamic programming.
\newblock {\em Operations Research}, 61(4):810--823, 2013.

\bibitem{luedtke2009strategic}
J.~Luedtke and G.~L. Nemhauser.
\newblock Strategic planning with start-time dependent variable costs.
\newblock {\em Operations Research}, 57(5):1250--1261, 2009.

\bibitem{muckstadt1977application}
J.~A. Muckstadt and S.~A. Koenig.
\newblock An application of \mbox{L}agrangian relaxation to scheduling in
  power-generation systems.
\newblock {\em Operations Research}, 25(3):387--403, 1977.

\bibitem{nemhauser2013ip}
G.~L. Nemhauser.
\newblock Integer \mbox{P}rogramming: Global \mbox{I}mpact.
\newblock
  \url{https://smartech.gatech.edu/bitstream/handle/1853/49829/presentation.pdf?sequence=1},
  2013.

\bibitem{nw}
G.~L. Nemhauser and L.~A. Wolsey.
\newblock {\em Integer and Combinatorial Optimization}.
\newblock Wiley, 1988.

\bibitem{ni2004optimal}
E.~Ni, P.~B. Luh, and S.~Rourke.
\newblock Optimal integrated generation bidding and scheduling with risk
  management under a deregulated power market.
\newblock {\em IEEE Transactions on Power Systems}, 19(1):600--609, 2004.

\bibitem{ostrowski2012tight}
J.~Ostrowski, M.~F. Anjos, and A.~Vannelli.
\newblock Tight mixed integer linear programming formulations for the unit
  commitment problem.
\newblock {\em IEEE Transactions on Power Systems}, 27(1):39--46, 2012.

\bibitem{outtagarts1997transient}
A.~Outtagarts, P.~Haberschill, and M.~Lallemand.
\newblock The transient response of an evaporator fed through an electronic
  expansion valve.
\newblock {\em International Journal of Energy Research}, 21(9):793--807, 1997.

\bibitem{padhy2004unit}
N.~P. Padhy.
\newblock Unit commitment-a bibliographical survey.
\newblock {\em IEEE Transactions on Power Systems}, 19(2):1196--1205, 2004.

\bibitem{pekelman1975production}
D.~Pekelman.
\newblock Production smoothing with fluctuating price.
\newblock {\em Management Science}, 21(5):576--590, 1975.

\bibitem{plazas2005multimarket}
M.~A. Plazas, A.~J. Conejo, and F.~J. Prieto.
\newblock Multimarket optimal bidding for a power producer.
\newblock {\em IEEE Transactions on Power Systems}, 20(4):2041--2050, 2005.

\bibitem{rajan2005minimum}
D.~Rajan and S.~Takriti.
\newblock Minimum up/down polytopes of the unit commitment problem with
  start-up costs.
\newblock {\em IBM, Research Report RC23628}, Jun. 2005.

\bibitem{sagastizabal2012divide}
C.~Sagastiz{\'a}bal.
\newblock Divide to conquer: decomposition methods for energy optimization.
\newblock {\em Mathematical Programming}, 134(1):187--222, 2012.

\bibitem{saravanan2013solution}
B.~Saravanan, S.~Das, S.~Sikri, and D.~P. Kothari.
\newblock A solution to the unit commitment problem-a review.
\newblock {\em Frontiers in Energy}, 7(2):223--236, 2013.

\bibitem{shebalov2015lifting}
S.~Shebalov, Y.~W. Park, and D.~Klabjan.
\newblock Lifting for mixed integer programs with variable upper bounds.
\newblock {\em Discrete Applied Mathematics}, 186:226--250, 2015.

\bibitem{sherali2001convex}
H.~D. Sherali, J.~C. Smith, and S.~Z. Selim.
\newblock Convex hull representations of models for computing collisions
  between multiple bodies.
\newblock {\em European Journal of Operational Research}, 135(3):514--526,
  2001.

\bibitem{silver1967tutorial}
E.~A. Silver.
\newblock A tutorial on production smoothing and work force balancing.
\newblock {\em Operations Research}, 15(6):985--1010, 1967.

\bibitem{sullivan2014convex}
K.~M. Sullivan, J.~C. Smith, and D.~P. Morton.
\newblock Convex hull representation of the deterministic bipartite network
  interdiction problem.
\newblock {\em Mathematical Programming}, 145(1-2):349--376, 2014.

\bibitem{wang1995short}
S.~Wang, S.~Shahidehpour, D.~Kirschen, S.~Mokhtari, and G.~Irisarri.
\newblock Short-term generation scheduling with transmission and environmental
  constraints using an augmented {L}agrangian relaxation.
\newblock {\em IEEE Transactions on Power Systems}, 10(3):1294--1301, 1995.

\bibitem{xia2007endoreversible}
D.~Xia, L.~Chen, F.~Sun, and C.~Wu.
\newblock Endoreversible four-reservoir chemical pump.
\newblock {\em Applied Energy}, 84(1):56--65, 2007.

\end{thebibliography}

\newpage
\begin{appendices}
\baselineskip=22pt
\setcounter{proposition}{0}
\setcounter{lemma}{0}
\setcounter{theorem}{0}

\section{Proofs for Three-period Formulations} \label{apx:sec:3period}

\subsection{Proof for Proposition \ref{prop:t3l2-full}} \label{apx:sec:t3l2-full}
\begin{proof}
We prove that dim($Q_3^2$) = $8$, because there are eight decision variables in $Q_3^2$. {Thus, we need to} generate nine affinely independent points in $Q_3^2$. Since $0 \in Q_3^2$, {it is sufficient to} generate other eight linearly independent points in $Q_3^2$ as shown in Table \ref{tab:t3l2-full}.
\begin{table}[htbp]
  \centering
  \caption{Eight linearly independent points in $Q_3^2$}
    \begin{tabular}{ >{\footnotesize} c >{\footnotesize} c >{\footnotesize} c | >{\footnotesize} c >{\footnotesize} c >{\footnotesize} c | >{\footnotesize} c >{\footnotesize} c}
    \hline
    $x_1$    & $x_2$    & $x_3$    & $y_1$    & $y_2$    & $y_3$    & $u_2$    & $u_3$ \\
    \hline
    $\Clower$     & 0     & 0     & 1     & 0     & 0     & 0     & 0 \\
    $\Clower+\epsilon$     & 0     & 0     & 1     & 0     & 0     & 0     & 0 \\
    $\Clower$     & $\Clower$     & 0     & 1     & 1     & 0     & 0     & 0 \\
    $\Clower+\epsilon$     & $\Clower+\epsilon$     & 0     & 1     & 1     & 0     & 0     & 0 \\
    $\Clower$     & $\Clower$     & $\Clower$     & 1     & 1     & 1     & 0     & 0 \\
    $\Clower+\epsilon$     & $\Clower+\epsilon$     & $\Clower+\epsilon$     & 1     & 1     & 1     & 0     & 0 \\
    0     & $\Clower$     & $\Clower$     & 0     & 1     & 1     & 1     & 0 \\
    0     & 0     & $\Clower$     & 0     & 0     & 1     & 0     & 1 \\
    \hline
    \end{tabular}
  \label{tab:t3l2-full}
\end{table}
\end{proof}

\subsection{Proof for Proposition \ref{prop:t3l2-int-extpoint}} \label{apx:sec:int-extpoint}
\begin{proof}
\textbf{Satisfying \eqref{eqn-C:p-x1-ub-t3l2} at equality}. For this case, substituting $\bar{x}_1 = \Vupper \bar{y}_1 + V (\bar{y}_2 - \bar{u}_2) + (\Cupper - \Vupper - V) (\bar{y}_3 - \bar{u}_3 - \bar{u}_2)$ into \eqref{eqn-C:p-x2-x1-ub-t3l2} - \eqref{eqn-C:p-x1-x2+x3-ub-t3l2}, we obtain the feasible region of $(\bar{x}_2, \bar{x}_3)$ as $C' = \{(\bar{x}_2, \bar{x}_3) \in \R^2: \Vupper \bar{y}_2 - (\Cupper - \Clower - V) \bar{u}_2 + (\Cupper - \Vupper - V) (\bar{y}_3 - \bar{u}_3) \leq \bar{x}_2 \leq \Vupper \bar{y}_2 + (\Cupper - \Vupper) (\bar{y}_3 - \bar{u}_3 - \bar{u}_2), \ \Clower \bar{y}_3 + (\Cupper - \Clower - 2V) (\bar{y}_3 - \bar{u}_3 - \bar{u}_2) \leq \bar{x}_3 \leq \Cupper \bar{y}_3 - (\Cupper - \Vupper) \bar{u}_3 - (\Cupper - \Vupper - V) \bar{u}_2, \ \bar{x}_3 - \bar{x}_2 \leq (\Vupper + V) \bar{y}_3 - \Vupper \bar{y}_2 - V \bar{u}_3, \ \bar{x}_2 - \bar{x}_3 \leq \Vupper \bar{y}_2 - \Clower \bar{y}_3 + (\Clower + V - \Vupper) (\bar{y}_3 - \bar{u}_3 - \bar{u}_2) \}$.

First, by letting $\hat{x}_1 = \Vupper$, $\hat{x}_2 = \Vupper + V$, and $\hat{x}_4 = \Cupper$, we have $\bar{x}_1 = \lambda_1 \hat{x}_1 + \lambda_2 \hat{x}_2 + \lambda_3 \hat{x}_4$. {Then \eqref{eqn-C: x1x2x3bar} holds for $\bar{x}_1$}.
 Then the corresponding feasible region for $(\hat{x}_3, \hat{x}_5, \hat{x}_6, \hat{x}_7, \hat{x}_8, \hat{x}_9)$ can be described as set $A' = \{ (\hat{x}_3, \hat{x}_5, \hat{x}_6, \hat{x}_7, \hat{x}_8, \hat{x}_9) \in \R^6: \hat{x}_3 = \Vupper, \  \Cupper - V \leq \hat{x}_5 \leq \Cupper, \ \Clower \leq \hat{x}_6 \leq \Cupper, \ -V \leq \hat{x}_6 - \hat{x}_5 \leq V, \ \Clower \leq \hat{x}_7 \leq \Vupper, \ \Clower \leq \hat{x}_8 \leq \Vupper + V, \ \Clower - \Vupper \leq \hat{x}_8 - \hat{x}_7 \leq V, \ \Clower \leq \hat{x}_9 \leq \Vupper \}$.
We consider that one of inequalities in $C'$ is satisfied at equality to obtain the values of $(\hat{x}_3, \hat{x}_5, \hat{x}_6, \hat{x}_7, \hat{x}_8, \hat{x}_9)$ from $A'$ as follows.
\begin{enumerate}[1)]
\item Satisfying $\bar{x}_2 \geq \Vupper \bar{y}_2 - (\Cupper - \Clower - V) \bar{u}_2 + (\Cupper - \Vupper - V) (\bar{y}_3 - \bar{u}_3)$ at equality. We obtain $\Clower \bar{y}_3 + (\Cupper - \Clower - 2V) (\bar{y}_3 - \bar{u}_3 - \bar{u}_2) \leq \bar{x}_3 \leq \Cupper \bar{y}_3 - (\Cupper - \Vupper) \bar{u}_3 - (\Cupper - \Clower - V) \bar{u}_2$. By letting $\hat{x}_3 = \Vupper$, $\hat{x}_5 = \Cupper - V$, and $\hat{x}_7 = \Clower$, we have $\bar{x}_2 = \lambda_2 \hat{x}_3 + \lambda_3 \hat{x}_5 + \lambda_4 \hat{x}_7$. {Then \eqref{eqn-C: x1x2x3bar} holds for $\bar{x}_2$}.
As a result, we have $(\hat{x}_6, \hat{x}_8, \hat{x}_9) \in A'' = \{ (\hat{x}_6, \hat{x}_8, \hat{x}_9) \in \R^3: \Cupper - 2V \leq \hat{x}_6 \leq \Cupper, \ \Clower \leq \hat{x}_8 \leq \Clower + V, \ \Clower \leq \hat{x}_9 \leq \Vupper \}$. If $\bar{x}_3 = \Clower \bar{y}_3 + (\Cupper - \Clower - 2V) (\bar{y}_3 - \bar{u}_3 - \bar{u}_2)$, we let $\hat{x}_6 = \Cupper - 2V$ and $\hat{x}_8 = \hat{x}_9 = \Clower$; if $\bar{x}_3 = \Cupper \bar{y}_3 - (\Cupper - \Vupper) \bar{u}_3 - (\Cupper - \Clower - V) \bar{u}_2$, we let $\hat{x}_6 = \Cupper$, $\hat{x}_8 = \Clower + V$, and $\hat{x}_9 = \Vupper$.
{For both cases}, we have $\bar{x}_3 = \lambda_3 \hat{x}_6 + \lambda_4 \hat{x}_8 + \lambda_5 \hat{x}_9$. {Then \eqref{eqn-C: x1x2x3bar} holds for $\bar{x}_3$}.

\item Satisfying $\bar{x}_2 \leq \Vupper \bar{y}_2 + (\Cupper - \Vupper) (\bar{y}_3 - \bar{u}_3 - \bar{u}_2)$ at equality. We obtain $\Clower \bar{y}_3 + (\Cupper - \Clower - V) (\bar{y}_3 - \bar{u}_3 - \bar{u}_2) \leq \bar{x}_3 \leq \Cupper \bar{y}_3 - (\Cupper - \Vupper) \bar{u}_3 - (\Cupper - \Vupper - V) \bar{u}_2$.
By letting $\hat{x}_3 = \hat{x}_7 = \Vupper$ and $\hat{x}_5 = \Cupper$, we have $\bar{x}_2 = \lambda_2 \hat{x}_3 + \lambda_3 \hat{x}_5 + \lambda_4 \hat{x}_7$. {Then \eqref{eqn-C: x1x2x3bar} holds for $\bar{x}_2$}.
 As a result, we have $(\hat{x}_6, \hat{x}_8, \hat{x}_9) \in A'' = \{ (\hat{x}_6, \hat{x}_8, \hat{x}_9) \in \R^3: \Cupper - V \leq \hat{x}_6 \leq \Cupper, \ \Clower \leq \hat{x}_8 \leq \Vupper + V, \ \Clower \leq \hat{x}_9 \leq \Vupper \}$. If $\bar{x}_3 = \Clower \bar{y}_3 + (\Cupper - \Clower - V) (\bar{y}_3 - \bar{u}_3 - \bar{u}_2)$, we let $\hat{x}_6 = \Cupper - V$ and $\hat{x}_8 = \hat{x}_9 = \Clower$; if $\bar{x}_3 = \Cupper \bar{y}_3 - (\Cupper - \Vupper) \bar{u}_3 - (\Cupper - \Vupper - V) \bar{u}_2$, we let $\hat{x}_6 = \Cupper$, $\hat{x}_8 = \Vupper + V$, and $\hat{x}_9 = \Vupper$.
\fblue{{For both cases}, we have $\bar{x}_3 = \lambda_3 \hat{x}_6 + \lambda_4 \hat{x}_8 + \lambda_5 \hat{x}_9$}. {Then \eqref{eqn-C: x1x2x3bar} holds for $\bar{x}_3$}.

\item Satisfying $\bar{x}_3 \geq \Clower \bar{y}_3 + (\Cupper - \Clower - 2V) (\bar{y}_3 - \bar{u}_3 - \bar{u}_2)$ at equality. We obtain $\Vupper \bar{y}_2 - (\Cupper - \Clower - V) \bar{u}_2 + (\Cupper - \Vupper - V) (\bar{y}_3 - \bar{u}_3) \leq \bar{x}_2 \leq \Vupper \bar{y}_2 + (\Cupper - \Vupper - V) (\bar{y}_3 - \bar{u}_3 - \bar{u}_2)$. By letting $\hat{x}_6 = \Cupper - 2V$ and $\hat{x}_8 = \hat{x}_9 = \Clower$, we have $\bar{x}_3 = \lambda_3 \hat{x}_6 + \lambda_4 \hat{x}_8 + \lambda_5 \hat{x}_9$. {Then \eqref{eqn-C: x1x2x3bar} holds for $\bar{x}_3$}.
As a result, we have $(\hat{x}_3, \hat{x}_5, \hat{x}_7) \in A'' = \{ (\hat{x}_3, \hat{x}_5, \hat{x}_7) \in \R^3: \hat{x}_3 = \Vupper, \ \hat{x}_5 = \Cupper - V, \ \Clower \leq \hat{x}_7 \leq \Vupper \}$. If $\bar{x}_2 = \Vupper \bar{y}_2 - (\Cupper - \Clower - V) \bar{u}_2 + (\Cupper - \Vupper - V) (\bar{y}_3 - \bar{u}_3)$, we let $\hat{x}_7 = \Clower$; if $\bar{x}_2 \leq \Vupper \bar{y}_2 + (\Cupper - \Vupper - V) (\bar{y}_3 - \bar{u}_3 - \bar{u}_2)$, we let $\hat{x}_7 = \Vupper$.
\fblue{{For both cases}, we have $\bar{x}_2 = \lambda_2 \hat{x}_3 + \lambda_3 \hat{x}_5 + \lambda_4 \hat{x}_7$}. {Then \eqref{eqn-C: x1x2x3bar} holds for $\bar{x}_2$}.

\item Satisfying $\bar{x}_3 \leq \Cupper \bar{y}_3 - (\Cupper - \Vupper) \bar{u}_3 - (\Cupper - \Vupper - V) \bar{u}_2$ at equality. We obtain $\Vupper \bar{y}_2 + (\Cupper - \Vupper - V) (\bar{y}_3 - \bar{u}_3 - \bar{u}_2) \leq \bar{x}_2 \leq \Vupper \bar{y}_2 + (\Cupper - \Vupper) (\bar{y}_3 - \bar{u}_3 - \bar{u}_2)$.
By letting $\hat{x}_6 = \Cupper$, $\hat{x}_8 = \Vupper + V$, and $\hat{x}_9 = \Vupper$, we have $\bar{x}_3 = \lambda_3 \hat{x}_6 + \lambda_4 \hat{x}_8 + \lambda_5 \hat{x}_9$. {Then \eqref{eqn-C: x1x2x3bar} holds for $\bar{x}_3$}.
 As a result, we have $(\hat{x}_3, \hat{x}_5, \hat{x}_7) \in A'' = \{ (\hat{x}_3, \hat{x}_5, \hat{x}_7) \in \R^3: \hat{x}_3 = \Vupper, \ \Cupper - V \leq \hat{x}_5 \leq \Cupper, \ \hat{x}_7 = \Vupper\}$. If $\bar{x}_2 = \Vupper \bar{y}_2 + (\Cupper - \Vupper - V) (\bar{y}_3 - \bar{u}_3 - \bar{u}_2)$, we let $\hat{x}_5 = \Cupper - V$; if $\bar{x}_2 = \Vupper \bar{y}_2 + (\Cupper - \Vupper) (\bar{y}_3 - \bar{u}_3 - \bar{u}_2)$, we let $\hat{x}_5 = \Cupper$.
\fblue{{For both cases}, we have $\bar{x}_2 = \lambda_2 \hat{x}_3 + \lambda_3 \hat{x}_5 + \lambda_4 \hat{x}_7$}. {Then \eqref{eqn-C: x1x2x3bar} holds for $\bar{x}_2$}.

\item Satisfying $\bar{x}_3 - \bar{x}_2 \leq (\Vupper + V) \bar{y}_3 - \Vupper \bar{y}_2 - V \bar{u}_3$ at equality. We obtain $\Vupper \bar{y}_2 - (\Cupper - \Clower - V) \bar{u}_2 + (\Cupper - \Vupper - V) (\bar{y}_3 - \bar{u}_3) \leq \bar{x}_2 \leq \Vupper \bar{y}_2 + (\Cupper - \Vupper - V) (\bar{y}_3 - \bar{u}_3 - \bar{u}_2)$ through substituting $\bar{x}_3 = \bar{x}_2 + (\Vupper + V) \bar{y}_3 - \Vupper \bar{y}_2 - V \bar{u}_3$ into set $C'$. By letting $\hat{x}_3 = \hat{x}_9 = \Vupper$ and $\hat{x}_6 - \hat{x}_5 = \hat{x}_8 - \hat{x}_7 = V$, we have $\bar{x}_3 - \bar{x}_2 = (\lambda_3 \hat{x}_6 + \lambda_4 \hat{x}_8 + \lambda_5 \hat{x}_9) - (\lambda_2 \hat{x}_3 + \lambda_3 \hat{x}_5 + \lambda_4 \hat{x}_7)$. If $\bar{x}_2 = \Vupper \bar{y}_2 - (\Cupper - \Clower - V) \bar{u}_2 + (\Cupper - \Vupper - V) (\bar{y}_3 - \bar{u}_3)$, we let $\hat{x}_5 = \Cupper - V$ and $\hat{x}_7 = \Clower$; if $\bar{x}_2 \leq \Vupper \bar{y}_2 + (\Cupper - \Vupper - V) (\bar{y}_3 - \bar{u}_3 - \bar{u}_2)$, we let $\hat{x}_5 = \Cupper - V$ and $\hat{x}_7 = \Vupper$. For both cases, we have $\bar{x}_2 = \lambda_2 \hat{x}_3 + \lambda_3 \hat{x}_5 + \lambda_4 \hat{x}_7$ and thus $\bar{x}_3 = \lambda_3 \hat{x}_6 + \lambda_4 \hat{x}_8 + \lambda_5 \hat{x}_9$. {Then \eqref{eqn-C: x1x2x3bar} holds for both $\bar{x}_2$ and $\bar{x}_3$}.

\item Satisfying $\bar{x}_2 - \bar{x}_3 \leq \Vupper \bar{y}_2 - \Clower \bar{y}_3 + (\Clower + V - \Vupper) (\bar{y}_3 - \bar{u}_3 - \bar{u}_2)$ at equality. We obtain $\Clower \bar{y}_3 + (\Cupper - \Clower - 2V) (\bar{y}_3 - \bar{u}_3 - \bar{u}_2) \leq \bar{x}_3 \leq \Clower \bar{y}_3 + (\Cupper - \Clower - V) (\bar{y}_3 - \bar{u}_3 - \bar{u}_2)$. By letting $\hat{x}_3 = \Vupper$, $\hat{x}_5 - \hat{x}_6 = V$, $\hat{x}_8 - \hat{x}_7 = \Clower - \Vupper$, and $\hat{x}_9 = \Clower$, we have $\bar{x}_2 - \bar{x}_3 = (\lambda_2 \hat{x}_3 + \lambda_3 \hat{x}_5 + \lambda_4 \hat{x}_7) - (\lambda_3 \hat{x}_6 + \lambda_4 \hat{x}_8 + \lambda_5 \hat{x}_9)$. 
If \fblue{$\bar{x}_3 \geq \Clower \bar{y}_3 + (\Cupper - \Clower - 2V) (\bar{y}_3 - \bar{u}_3 - \bar{u}_2)$ is satisfied at equality}, we let $\hat{x}_6 = \Cupper - 2V$ and $\hat{x}_8  = \Clower$; if \fblue{$\bar{x}_3 \leq \Clower \bar{y}_3 + (\Cupper - \Clower - V) (\bar{y}_3 - \bar{u}_3 - \bar{u}_2)$ is satisfied at equality}, we let $\hat{x}_6 = \Cupper - V$ and $\hat{x}_8  = \Clower$.
\fblue{For both cases, we have $\bar{x}_3 = \lambda_3 \hat{x}_6 + \lambda_4 \hat{x}_8 + \lambda_5 \hat{x}_9$ and thus $\bar{x}_2 = \lambda_2 \hat{x}_3 + \lambda_3 \hat{x}_5 + \lambda_4 \hat{x}_7$}. {Then \eqref{eqn-C: x1x2x3bar} holds for both $\bar{x}_2$ and $\bar{x}_3$}.
\end{enumerate}

Similar analyses hold for \eqref{eqn-C:p-x2-ub-t3l2} and \eqref{eqn-C:p-x3-ub-t3l2} due to the similar structure {among} \eqref{eqn-C:p-x1-ub-t3l2}, \eqref{eqn-C:p-x2-ub-t3l2}, and \eqref{eqn-C:p-x3-ub-t3l2} and thus are omitted here.

\textbf{Satisfying \eqref{eqn-C:p-x2-x1-ub-t3l2} at equality}. For this case, substituting $\bar{x}_2 = \bar{x}_1 + \Vupper \bar{y}_2 - \Clower \bar{y}_1 + (\Clower + V - \Vupper) (\bar{y}_3 - \bar{u}_3 - \bar{u}_2)$ into \eqref{eqn-C:p-lower-bound-t3l2} - \eqref{eqn-C:p-x1-x2+x3-ub-t3l2}, we obtain the feasible region of $(\bar{x}_1, \bar{x}_3)$ as $C' = \{(\bar{x}_1, \bar{x}_3) \in \R^2: \Clower \bar{y}_1 \leq \bar{x}_1 \leq \Clower \bar{y}_1 + (\Cupper - \Clower - V) (\bar{y}_3 - \bar{u}_3 - \bar{u}_2), \ \Clower \bar{y}_3 \leq \bar{x}_3 \leq \Cupper \bar{y}_3 - (\Cupper - \Vupper) \bar{u}_3 - (\Cupper - \Vupper - V) \bar{u}_2, \ \bar{x}_3 - \bar{x}_1 \leq (\Clower + 2 V) \bar{y}_3 - \Clower \bar{y}_1 - (\Clower + 2 V - \Vupper) \bar{u}_3 - (\Clower + V - \Vupper) \bar{u}_2, \ \bar{x}_1 - \bar{x}_3 \leq \Clower \bar{y}_1 - \Clower \bar{y}_3 \}$.

First, by letting $\hat{x}_1 = \Clower$, $\hat{x}_3 - \hat{x}_2 = \Vupper - \Clower$, $\hat{x}_5 - \hat{x}_4 = V$, and $\hat{x}_7 = \Vupper$, we have $\bar{x}_2 - \bar{x}_1 = (\lambda_2 \hat{x}_3 + \lambda_3 \hat{x}_5 + \lambda_4 \hat{x}_7) - (\lambda_1 \hat{x}_1 + \lambda_2 \hat{x}_2 + \lambda_3 \hat{x}_4)$. Since $\Clower \leq \hat{x}_3 \leq \Vupper$, it follows that $\hat{x}_2 = \Clower$ and $\hat{x}_3 = \Vupper$.
Then the corresponding feasible region for $(\hat{x}_4, \hat{x}_6, \hat{x}_8, \hat{x}_9)$ can be described as set $A' = \{ (\hat{x}_4, \hat{x}_6, \hat{x}_8, \hat{x}_9) \in \R^6: \Clower \leq \hat{x}_4 \leq \Cupper - V, \ \Clower \leq \hat{x}_6 \leq \Cupper, \ 0 \leq \hat{x}_6 - \hat{x}_4 \leq 2 V, \ \Clower \leq \hat{x}_8 \leq \Vupper + V, \ \Clower \leq \hat{x}_9 \leq \Vupper \}$.
Next, we only need to show $\bar{x}_1 = \lambda_1 \hat{x}_1 + \lambda_2 \hat{x}_2 + \lambda_3 \hat{x}_4$ and $\bar{x}_3 = \lambda_3 \hat{x}_6 + \lambda_4 \hat{x}_8 + \lambda_5 \hat{x}_9$. {Then \eqref{eqn-C: x1x2x3bar} will hold}.
We consider that one of inequalities in $C'$ is satisfied at equality to obtain the values of $(\hat{x}_4, \hat{x}_6, \hat{x}_8, \hat{x}_9)$ from $A'$ as follows.
\begin{enumerate}[1)]
\item Satisfying $\bar{x}_1 \geq \Clower \bar{y}_1$ at equality. We obtain $\Clower \bar{y}_3 \leq \bar{x}_3 \leq (\Clower + 2 V) \bar{y}_3 - (\Clower + 2 V - \Vupper) \bar{u}_3 - (\Clower + V - \Vupper) \bar{u}_2$. 
By letting $\hat{x}_4 = \Clower$, we have $\bar{x}_1 =\lambda_1 \hat{x}_1 + \lambda_2 \hat{x}_2 + \lambda_3 \hat{x}_4$.  
As a result, we have $(\hat{x}_6, \hat{x}_8, \hat{x}_9) \in A'' = \{ (\hat{x}_6, \hat{x}_8, \hat{x}_9) \in \R^3: \Clower \leq \hat{x}_6 \leq \Clower + 2V, \ \Clower \leq \hat{x}_8 \leq \Vupper + V, \ \Clower \leq \hat{x}_9 \leq \Vupper \}$. If $\bar{x}_3 = \Clower \bar{y}_3$, we let $\hat{x}_6 = \hat{x}_8 = \hat{x}_9 = \Clower$; if $\bar{x}_3 = (\Clower + 2 V) \bar{y}_3 - (\Clower + 2 V - \Vupper) \bar{u}_3 - (\Clower + V - \Vupper) \bar{u}_2$, we let $\hat{x}_6 =  \Clower + 2V$, $\hat{x}_8 = \Vupper + V$, and $\hat{x}_9 = \Vupper$.
\fblue{For both cases, we have $\bar{x}_3 = \lambda_3 \hat{x}_6 + \lambda_4 \hat{x}_8 + \lambda_5 \hat{x}_9$.}

\item Satisfying $\bar{x}_1 \leq \Clower \bar{y}_1 + (\Cupper - \Clower - V) (\bar{y}_3 - \bar{u}_3 - \bar{u}_2)$ at equality. We obtain $\Clower \bar{y}_3 + (\Cupper - \Clower - V) (\bar{y}_3 - \bar{u}_3 - \bar{u}_2) \leq \bar{x}_3 \leq \Cupper \bar{y}_3 - (\Cupper - \Vupper) \bar{u}_3 - (\Cupper - \Vupper - V) \bar{u}_2$. By letting $\hat{x}_4 = \Cupper - V$, we have $\bar{x}_1 =\lambda_1 \hat{x}_1 + \lambda_2 \hat{x}_2 + \lambda_3 \hat{x}_4$. As a result, we have $(\hat{x}_6, \hat{x}_8, \hat{x}_9) \in A'' = \{ (\hat{x}_6, \hat{x}_8, \hat{x}_9) \in \R^3: \Cupper - V \leq \hat{x}_6 \leq \Cupper, \ \Clower \leq \hat{x}_8 \leq \Vupper + V, \ \Clower \leq \hat{x}_9 \leq \Vupper \}$. If $\bar{x}_3 = \Clower \bar{y}_3 + (\Cupper - \Clower - V) (\bar{y}_3 - \bar{u}_3 - \bar{u}_2)$, we let $\hat{x}_6 =  \Cupper - V$ and $\hat{x}_8 = \hat{x}_9 = \Clower$; if $ \bar{x}_3 = \Cupper \bar{y}_3 - (\Cupper - \Vupper) \bar{u}_3 - (\Cupper - \Vupper - V) \bar{u}_2$, we let $\hat{x}_6 = \Cupper$, $\hat{x}_8 = \Vupper + V$, and $\hat{x}_9 = \Vupper$.
\fblue{For both cases, we have $\bar{x}_3 = \lambda_3 \hat{x}_6 + \lambda_4 \hat{x}_8 + \lambda_5 \hat{x}_9$.}

\item Satisfying $\bar{x}_3 \geq \Clower \bar{y}_3$ at equality. We obtain $\bar{x}_1 = \Clower \bar{y}_1$ since $\bar{x}_1 - \bar{x}_3 \leq \Clower \bar{y}_1 - \Clower \bar{y}_3$. By letting $\hat{x}_4 = \hat{x}_6 = \hat{x}_8 = \hat{x}_9 = \Clower$, we have $\bar{x}_1 =\lambda_1 \hat{x}_1 + \lambda_2 \hat{x}_2 + \lambda_3 \hat{x}_4$ and $\bar{x}_3 = \lambda_3 \hat{x}_6 + \lambda_4 \hat{x}_8 + \lambda_5 \hat{x}_9$.

\item Satisfying $\bar{x}_3 \leq \Cupper \bar{y}_3 - (\Cupper - \Vupper) \bar{u}_3 - (\Cupper - \Vupper - V) \bar{u}_2$ at equality. 
We obtain $\Clower \bar{y}_1 + (\Cupper - \Clower - 2V) (\bar{y}_3 - \bar{u}_3 - \bar{u}_2) \leq \bar{x}_1 \leq \Clower \bar{y}_1 + (\Cupper - \Clower - V) (\bar{y}_3 - \bar{u}_3 - \bar{u}_2)$.
By letting $\hat{x}_6 = \Cupper$, $\hat{x}_8 = \Vupper + V$, and $\hat{x}_9 = \Vupper$, we have $\bar{x}_3 = \lambda_3 \hat{x}_6 + \lambda_4 \hat{x}_8 + \lambda_5 \hat{x}_9$. Thus it follows that $\Cupper - 2V \leq \hat{x}_4 \leq \Cupper - V$. If $\bar{x}_1 = \Clower \bar{y}_1 + (\Cupper - \Clower - 2V) (\bar{y}_3 - \bar{u}_3 - \bar{u}_2) $, we let $\hat{x}_4 = \Cupper - 2V$; if $\bar{x}_1 = \Clower \bar{y}_1 + (\Cupper - \Clower - V) (\bar{y}_3 - \bar{u}_3 - \bar{u}_2)$, we let $\hat{x}_4 = \Cupper - V$.
\fblue{For both cases, we have $\bar{x}_1 =\lambda_1 \hat{x}_1 + \lambda_2 \hat{x}_2 + \lambda_3 \hat{x}_4$.}

\item Satisfying $\bar{x}_3 - \bar{x}_1 \leq (\Clower + 2 V) \bar{y}_3 - \Clower \bar{y}_1 - (\Clower + 2 V - \Vupper) \bar{u}_3 - (\Clower + V - \Vupper) \bar{u}_2$ at equality. We obtain $\Clower \bar{y}_1 \leq \bar{x}_1 \leq \Clower \bar{y}_1 + (\Cupper - \Clower - 2 V) (\bar{y}_3 - \bar{u}_3 - \bar{u}_2)$. By letting $\hat{x}_6 - \hat{x}_4 = 2 V$, $\hat{x}_8 = \Vupper + V$, and $\hat{x}_9 = \Vupper$, we have $\bar{x}_3 - \bar{x}_1 = (\lambda_3 \hat{x}_6 + \lambda_4 \hat{x}_8 + \lambda_5 \hat{x}_9) - (\lambda_1 \hat{x}_1 + \lambda_2 \hat{x}_2 + \lambda_3 \hat{x}_4)$. If $\bar{x}_1 = \Clower \bar{y}_1$, we let $\hat{x}_4 = \Clower$ and thus $\hat{x}_6 = \Clower + 2V$; if $\bar{x}_1 \leq \Clower \bar{y}_1 + (\Cupper - \Clower - 2 V) (\bar{y}_3 - \bar{u}_3 - \bar{u}_2)$, we let $\hat{x}_4 = \Cupper - 2V$.
\fblue{For both cases, we have $\bar{x}_1 =\lambda_1 \hat{x}_1 + \lambda_2 \hat{x}_2 + \lambda_3 \hat{x}_4$ and thus $\bar{x}_3 = \lambda_3 \hat{x}_6 + \lambda_4 \hat{x}_8 + \lambda_5 \hat{x}_9$.}

\item Satisfying $\bar{x}_1 - \bar{x}_3 \leq \Clower \bar{y}_1 - \Clower \bar{y}_3$ at equality. We obtain $\Clower \bar{y}_3 \leq \bar{x}_3 \leq \Clower \bar{y}_3 + (\Cupper - \Clower - V) (\bar{y}_3 - \bar{u}_3 - \bar{u}_2)$. By letting $\hat{x}_4 = \hat{x}_6$, $\hat{x}_8 = \hat{x}_9 = \Clower$, we have $\bar{x}_1 - \bar{x}_3 =  (\lambda_1 \hat{x}_1 + \lambda_2 \hat{x}_2 + \lambda_3 \hat{x}_4) - (\lambda_3 \hat{x}_6 + \lambda_4 \hat{x}_8 + \lambda_5 \hat{x}_9)$. If $\bar{x}_3 = \Clower \bar{y}_3$, we let $\hat{x}_6 = \Clower$; if $\bar{x}_3 \leq \Clower \bar{y}_3 + (\Cupper - \Clower - V) (\bar{y}_3 - \bar{u}_3 - \bar{u}_2)$, we let $\hat{x}_6 = \Cupper- V$.
\fblue{For both cases, we have $\bar{x}_3 = \lambda_3 \hat{x}_6 + \lambda_4 \hat{x}_8 + \lambda_5 \hat{x}_9$ and thus $\bar{x}_1 =\lambda_1 \hat{x}_1 + \lambda_2 \hat{x}_2 + \lambda_3 \hat{x}_4$.}
\end{enumerate}

Similar analyses hold for \eqref{eqn-C:p-x3-x2-ub-t3l2} - \eqref{eqn-C:p-x2-x3-ub-t3l2} due to the similar structure {among} \eqref{eqn-C:p-x2-x1-ub-t3l2} and \eqref{eqn-C:p-x3-x2-ub-t3l2} - \eqref{eqn-C:p-x2-x3-ub-t3l2} and thus are omitted here.

\textbf{Satisfying \eqref{eqn-C:p-x3-x1-ub-t3l2} at equality}. For this case, substituting $\bar{x}_3 = \bar{x}_1  + (\Clower + 2 V) \bar{y}_3 - \Clower \bar{y}_1 - (\Clower + 2 V - \Vupper) \bar{u}_3 - (\Clower + V - \Vupper) \bar{u}_2$ into \eqref{eqn-C:p-lower-bound-t3l2} - \eqref{eqn-C:p-x1-x2+x3-ub-t3l2}, we obtain the feasible region of $(\bar{x}_1, \bar{x}_2)$ as $C' = \{(\bar{x}_1, \bar{x}_2) \in \R^2: \Clower \bar{y}_1 \leq \bar{x}_1 \leq \Clower \bar{y}_1 + (\Cupper - \Clower - 2 V) (\bar{y}_3 - \bar{u}_3 - \bar{u}_2), \ \bar{x}_2 - \bar{x}_1 \leq \Vupper \bar{y}_2 - \Clower \bar{y}_1 + (\Clower + V - \Vupper) (\bar{y}_3 - \bar{u}_3 - \bar{u}_2), \ \bar{x}_1 - \bar{x}_2 \leq \Clower \bar{y}_1 - \Clower \bar{y}_2 - V \bar{y}_3 + V \bar{u}_3 + (\Clower + V - \Vupper) \bar{u}_2 \}$.

First, by letting $\hat{x}_1 = \hat{x}_2 = \Clower$, $\hat{x}_6 - \hat{x}_4 = 2V$, $\hat{x}_8 = \Vupper + V$, and $\hat{x}_9 = \Vupper$, we have $\bar{x}_3 - \bar{x}_1 = (\lambda_3 \hat{x}_6 + \lambda_4 \hat{x}_8 + \lambda_5 \hat{x}_9) - (\lambda_1 \hat{x}_1 + \lambda_2 \hat{x}_2 + \lambda_3 \hat{x}_4)$.
Since $\Clower \leq \hat{x}_7 \leq \Vupper $ and $\Clower - \Vupper \leq \hat{x}_8 - \hat{x}_7 \leq V$, we have $\hat{x}_7 = \Vupper$.
Then the corresponding feasible region for $(\hat{x}_3, \hat{x}_4, \hat{x}_5)$ can be described as set $A' = \{ (\hat{x}_3, \hat{x}_4, \hat{x}_5) \in \R^3: \Clower \leq \hat{x}_3 \leq \Vupper, \ \Clower \leq \hat{x}_4 \leq \Cupper - 2V, \ \Clower \leq \hat{x}_5 \leq \Cupper - V, \hat{x}_5 - \hat{x}_4 = V \}$.
Next, we only need to show $\bar{x}_1 = \lambda_1 \hat{x}_1 + \lambda_2 \hat{x}_2 + \lambda_3 \hat{x}_4$ and $\bar{x}_2 = \lambda_2 \hat{x}_3 + \lambda_3 \hat{x}_5 + \lambda_4 \hat{x}_7$. {Then \eqref{eqn-C: x1x2x3bar} will hold}.
We consider that one of inequalities in $C'$ is satisfied at equality to obtain the values of $(\hat{x}_3, \hat{x}_4, \hat{x}_5)$ from $A'$ as follows.
\begin{enumerate}[1)]
\item Satisfying $\bar{x}_1 \geq \Clower \bar{y}_1$ at equality. We obtain $\Clower \bar{y}_2 + V (\bar{y}_3 - \bar{u}_3) - (\Clower + V - \Vupper) \bar{u}_2 \leq \bar{x}_2 \leq \Vupper \bar{y}_2 + (\Clower + V - \Vupper) (\bar{y}_3 - \bar{u}_3 - \bar{u}_2) $. By letting $\hat{x}_4 = \Clower$, we have $\bar{x}_1 = \lambda_1 \hat{x}_1 + \lambda_2 \hat{x}_2 + \lambda_3 \hat{x}_4$. As a result, we have $\Clower \leq \hat{x}_3 \leq \Vupper$ and $\hat{x}_5 = \Clower +V$. If $\bar{x}_2 = \Clower \bar{y}_2 + V (\bar{y}_3 - \bar{u}_3) - (\Clower + V - \Vupper) \bar{u}_2$, we let $\hat{x}_3 = \Clower$; if  $\bar{x}_2 = \Vupper \bar{y}_2 + (\Clower + V - \Vupper) (\bar{y}_3 - \bar{u}_3 - \bar{u}_2)$, we let $\hat{x}_3 = \Vupper$.
\fblue{For both cases, we have $\bar{x}_2 = \lambda_2 \hat{x}_3 + \lambda_3 \hat{x}_5 + \lambda_4 \hat{x}_7$.}

\item Satisfying $\bar{x}_1 \leq \Clower \bar{y}_1 + (\Cupper - \Clower - 2 V) (\bar{y}_3 - \bar{u}_3 - \bar{u}_2)$ at equality. We obtain  $ \Clower \bar{y}_2 + (\Cupper - \Clower - V) (\bar{y}_3 - \bar{u}_3) - (\Cupper - \Vupper - V) \bar{u}_2 \leq \fblue{\bar{x}_2} \leq \Vupper \bar{y}_2 + (\Cupper - \Vupper - V) (\bar{y}_3 - \bar{u}_3 - \bar{u}_2)$. By letting $\hat{x}_4 = \Cupper - 2V$, we have $\bar{x}_1 = \lambda_1 \hat{x}_1 + \lambda_2 \hat{x}_2 + \lambda_3 \hat{x}_4$. As a result, we have $\Clower \leq \hat{x}_3 \leq \Vupper$ and $\hat{x}_5 = \Cupper - V$. If $\bar{x}_2 = \Clower \bar{y}_2 + (\Cupper - \Clower - V) (\bar{y}_3 - \bar{u}_3) - (\Cupper - \Vupper - V) \bar{u}_2$, we let $\hat{x}_3 = \Clower$; if $\fblue{\bar{x}_2} = \Vupper \bar{y}_2 + (\Cupper - \Vupper - V) (\bar{y}_3 - \bar{u}_3 - \bar{u}_2)$, we let $\hat{x}_3 = \Vupper$.
\fblue{For both cases, we have $\bar{x}_2 = \lambda_2 \hat{x}_3 + \lambda_3 \hat{x}_5 + \lambda_4 \hat{x}_7$.}

\item Satisfying $\bar{x}_2 - \bar{x}_1 \leq \Vupper \bar{y}_2 - \Clower \bar{y}_1 + (\Clower + V - \Vupper) (\bar{y}_3 - \bar{u}_3 - \bar{u}_2)$ at equality. We obtain $\Vupper \bar{y}_2 + (\Clower + V - \Vupper) (\bar{y}_3 - \bar{u}_3 - \bar{u}_2) \leq \bar{x}_2 \leq \Vupper \bar{y}_2 + (\Cupper - \Vupper - V) (\bar{y}_3 - \bar{u}_3 - \bar{u}_2)$. By letting $\hat{x}_3 = \Vupper$, we have $\bar{x}_2 - \bar{x}_1 = (\lambda_2 \hat{x}_3 + \lambda_3 \hat{x}_5 + \lambda_4 \hat{x}_7) - (\lambda_1 \hat{x}_1 + \lambda_2 \hat{x}_2 + \lambda_3 \hat{x}_4) $. As result, we have $\Clower+V \leq \hat{x}_5 \leq \Cupper- V$. If $\fblue{\bar{x}_2} = \Vupper \bar{y}_2 + (\Clower + V - \Vupper) (\bar{y}_3 - \bar{u}_3 - \bar{u}_2)$, we let $\hat{x}_5 = \Clower + V$; if $\bar{x}_2 = \Vupper \bar{y}_2 + (\Cupper - \Vupper - V) (\bar{y}_3 - \bar{u}_3 - \bar{u}_2)$, we let $\hat{x}_5 = \Cupper - V$.
\fblue{For both cases, we have $\bar{x}_2 = \lambda_2 \hat{x}_3 + \lambda_3 \hat{x}_5 + \lambda_4 \hat{x}_7$ and thus $\bar{x}_1 = \lambda_1 \hat{x}_1 + \lambda_2 \hat{x}_2 + \lambda_3 \hat{x}_4$.}

\item Satisfying $\bar{x}_1 - \bar{x}_2 \leq \Clower \bar{y}_1 - \Clower \bar{y}_2 - V \bar{y}_3 + V \bar{u}_3 + (\Clower + V - \Vupper) \bar{u}_2$ at equality. We obtain $\Clower \bar{y}_2 + V(\bar{y}_3 - \bar{u}_3) - (\Clower + V - \Vupper) \bar{u}_2 \leq \bar{x}_2 \leq \Clower \bar{y}_2 + (\Cupper - \Clower - V) (\bar{y}_3 - \bar{u}_3 - \bar{u}_2) + (\Vupper - \Clower) \bar{u}_2$. By letting $\hat{x}_3 = \Clower$, we have $\bar{x}_1 - \bar{x}_2 = (\lambda_1 \hat{x}_1 + \lambda_2 \hat{x}_2 + \lambda_3 \hat{x}_4) - (\lambda_2 \hat{x}_3 + \lambda_3 \hat{x}_5 + \lambda_4 \hat{x}_7)$. As result, we have $\Clower+V \leq \hat{x}_5 \leq \Cupper- V$. If $\bar{x}_2 = \Clower \bar{y}_2 + V(\bar{y}_3 - \bar{u}_3) - (\Clower + V - \Vupper) \bar{u}_2$, we let $\hat{x}_5 = \Clower + V$; if $\bar{x}_2 = \Clower \bar{y}_2 + (\Cupper - \Clower - V) (\bar{y}_3 - \bar{u}_3 - \bar{u}_2) + (\Vupper - \Clower) \bar{u}_2$, we let $\hat{x}_5 = \Cupper - V$.
\fblue{For both cases, we have $\bar{x}_2 = \lambda_2 \hat{x}_3 + \lambda_3 \hat{x}_5 + \lambda_4 \hat{x}_7$ and thus $\bar{x}_1 = \lambda_1 \hat{x}_1 + \lambda_2 \hat{x}_2 + \lambda_3 \hat{x}_4$.}
\end{enumerate}

Similar analyses hold for \eqref{eqn-C:p-x1-x3-ub-t3l2} due to the similar structure between \eqref{eqn-C:p-x3-x1-ub-t3l2} and \eqref{eqn-C:p-x1-x3-ub-t3l2} and thus are omitted here.

\textbf{Satisfying \eqref{eqn-C:p-x1-x2+x3-ub-t3l2} at equality}. For this case, substituting $\bar{x}_3 = \bar{x}_2 - \bar{x}_1 + \Vupper \bar{y}_1 - (\Vupper - V) \bar{y}_2 + \Vupper \bar{y}_3 + (\Cupper - \Vupper) (\bar{y}_3 - \bar{u}_3 - \bar{u}_2)$ into \eqref{eqn-C:p-lower-bound-t3l2} - \eqref{eqn-C:p-x1-x2+x3-ub-t3l2}, we obtain the feasible region of $(\bar{x}_1, \bar{x}_2)$ as $C' = \{(\bar{x}_1, \bar{x}_2) \in \R^2: \Vupper \bar{y}_1 + (\Clower + V - \Vupper) (\bar{y}_2 - \bar{y}_3 + \bar{u}_3) + (\Cupper - \Vupper) (\bar{y}_3 - \bar{u}_3 - \bar{u}_2) \leq \bar{x}_1 \leq \Vupper \bar{y}_1 + V (\bar{y}_2 - \bar{u}_2) + (\Cupper - \Vupper - V) (\bar{y}_3 - \bar{u}_3 - \bar{u}_2), \ \bar{x}_2 - \bar{x}_1 \leq (\Vupper-V) \bar{y}_2 - \Vupper \bar{y}_1 + V \bar{u}_2, \ \bar{x}_1 - \bar{x}_2 \leq \Vupper \bar{y}_1 - (\Vupper - V) \bar{y}_2 - (\Clower + V - \Vupper) \bar{u}_2 \}$.

First, by letting $\hat{x}_1 = \Vupper$, $\hat{x}_2 - \hat{x}_3 = V$, $\hat{x}_4 = \hat{x}_6 = \Cupper$, $\hat{x}_5 = \Cupper - V$,  $\hat{x}_8 - \hat{x}_7 = V$, and $\hat{x}_9 = \Vupper$, we have $\bar{x}_1 - \bar{x}_2 + \bar{x}_3 = (\lambda_1 \hat{x}_1 + \lambda_2 \hat{x}_2 + \lambda_3 \hat{x}_4) - (\lambda_2 \hat{x}_3 + \lambda_3 \hat{x}_5 + \lambda_4 \hat{x}_7) + (\lambda_3 \hat{x}_6 + \lambda_4 \hat{x}_8 + \lambda_5 \hat{x}_9)$.
Then the corresponding feasible region for $(\hat{x}_3, \hat{x}_7)$ can be described as set $A' = \{ (\hat{x}_3, \hat{x}_7) \in \R^2: \Clower \leq \hat{x}_3 \leq \Vupper, \ \Clower \leq \hat{x}_7 \leq \Vupper \}$.
Next, we only need to show $\bar{x}_1 = \lambda_1 \hat{x}_1 + \lambda_2 \hat{x}_2 + \lambda_3 \hat{x}_4$ and $\bar{x}_2 = \lambda_2 \hat{x}_3 + \lambda_3 \hat{x}_5 + \lambda_4 \hat{x}_7$. {Then \eqref{eqn-C: x1x2x3bar} will accordingly hold}.
We consider that one of inequalities in $C'$ is satisfied at equality to obtain the values of $(\hat{x}_3, \hat{x}_7)$ from $A'$ as follows.
\begin{enumerate}[1)]
\item Satisfying $\bar{x}_1 \geq \Vupper \bar{y}_1 + (\Clower + V - \Vupper) (\bar{y}_2 - \bar{y}_3 + \bar{u}_3) + (\Cupper - \Vupper) (\bar{y}_3 - \bar{u}_3 - \bar{u}_2)$ at equality. We obtain $\Clower \bar{y}_2 + (\Cupper - \Clower - V) (\bar{y}_3 - \bar{u}_3 - \bar{u}_2) \leq \bar{x}_2 \leq \Clower \bar{y}_2 + (\Cupper - \Clower - V) (\bar{y}_3 - \bar{u}_3) - (\Cupper - \Vupper - V) \bar{u}_2$. By letting $\hat{x}_3 = \Clower$ and thus $\hat{x}_2 = \Clower + V$, we have $\bar{x}_1 = \lambda_1 \hat{x}_1 + \lambda_2 \hat{x}_2 + \lambda_3 \hat{x}_4$. If $\bar{x}_2 = \Clower \bar{y}_2 + (\Cupper - \Clower - V) (\bar{y}_3 - \bar{u}_3 - \bar{u}_2)$, we let $\hat{x}_7 = \Clower$ and thus $\hat{x}_8 = \Clower + V$; if $\bar{x}_2 = \Clower \bar{y}_2 + (\Cupper - \Clower - V) (\bar{y}_3 - \bar{u}_3) - (\Cupper - \Vupper - V) \bar{u}_2$, we let $\hat{x}_7 = \Vupper$ and thus $\hat{x}_8 = \Vupper + V$.
\fblue{For both cases, we have $\bar{x}_2 = \lambda_2 \hat{x}_3 + \lambda_3 \hat{x}_5 + \lambda_4 \hat{x}_7$.}

\item Satisfying $\bar{x}_1 \leq \Vupper \bar{y}_1 + V (\bar{y}_2 - \bar{u}_2) + (\Cupper - \Vupper - V) (\bar{y}_3 - \bar{u}_3 - \bar{u}_2)$ at equality. We obtain $\Vupper \bar{y}_2 + (\Cupper - \Vupper - V) (\bar{y}_3 - \bar{u}_3) - (\Cupper - \Clower - V) \bar{u}_2 \leq \bar{x}_2 \leq \Vupper \bar{y}_2 + (\Cupper - \Vupper - V) (\bar{y}_3 - \bar{u}_3 - \bar{u}_2)$. By letting $\hat{x}_3 = \Vupper$ and thus $\hat{x}_2 = \Vupper + V$, we have $\bar{x}_1 = \lambda_1 \hat{x}_1 + \lambda_2 \hat{x}_2 + \lambda_3 \hat{x}_4$. If $\bar{x}_2 = \Vupper \bar{y}_2 + (\Cupper - \Vupper - V) (\bar{y}_3 - \bar{u}_3) - (\Cupper - \Clower - V) \bar{u}_2$, we let $\hat{x}_7 = \Clower$ and thus $\hat{x}_8 = \Clower + V$; if $\bar{x}_2 = \Vupper \bar{y}_2 + (\Cupper - \Vupper - V) (\bar{y}_3 - \bar{u}_3 - \bar{u}_2)$, we let $\hat{x}_7 = \Vupper$ and thus $\hat{x}_8 = \Vupper + V$.
\fblue{For both cases, we have $\bar{x}_2 = \lambda_2 \hat{x}_3 + \lambda_3 \hat{x}_5 + \lambda_4 \hat{x}_7$.}

\item Satisfying $\bar{x}_2 - \bar{x}_1 \leq (\Vupper-V) \bar{y}_2 - \Vupper \bar{y}_1 + V \bar{u}_2$ at equality. We obtain $\Clower \bar{y}_2 + (\Cupper - \Clower - V) (\bar{y}_3 - \bar{u}_3) - (\Cupper - \Vupper - V) \bar{u}_2 \leq \bar{x}_2 \leq \Vupper \bar{y}_2 + (\Cupper - \Vupper - V) (\bar{y}_3 - \bar{u}_3 - \bar{u}_2)$. By letting $\hat{x}_7 = \Vupper$ and thus $\hat{x}_8 = \Vupper + V$, we have $\bar{x}_2 - \bar{x}_1 = (\lambda_2 \hat{x}_3 + \lambda_3 \hat{x}_5 + \lambda_4 \hat{x}_7) - (\lambda_1 \hat{x}_1 + \lambda_2 \hat{x}_2 + \lambda_3 \hat{x}_4)$. If $\bar{x}_2 = \Clower \bar{y}_2 + (\Cupper - \Clower - V) (\bar{y}_3 - \bar{u}_3) - (\Cupper - \Vupper - V) \bar{u}_2$, we let $\hat{x}_3 = \Clower$ and thus $\hat{x}_2 = \Clower + V$; if $\bar{x}_2 = \Vupper \bar{y}_2 + (\Cupper - \Vupper - V) (\bar{y}_3 - \bar{u}_3 - \bar{u}_2)$, we let $\hat{x}_3 = \Vupper$ and thus $\hat{x}_2 = \Vupper + V$.
\fblue{For both cases, we have $\bar{x}_2 = \lambda_2 \hat{x}_3 + \lambda_3 \hat{x}_5 + \lambda_4 \hat{x}_7$ and thus $\bar{x}_1 = \lambda_1 \hat{x}_1 + \lambda_2 \hat{x}_2 + \lambda_3 \hat{x}_4$.}

\item Satisfying $\bar{x}_1 - \bar{x}_2 \leq \Vupper \bar{y}_1 - (\Vupper - V) \bar{y}_2 - (\Clower + V - \Vupper) \bar{u}_2$ at equality. We obtain 
$\Clower \bar{y}_2 + (\Cupper - \Clower - V) (\bar{y}_3 - \bar{u}_3 - \bar{u}_2) \leq \bar{x}_2 \leq \Vupper \bar{y}_2 + (\Cupper - \Vupper - V) (\bar{y}_3 - \bar{u}_3) - (\Cupper - \Clower - V) \bar{u}_2$.  By letting $\hat{x}_7 = \Clower$ and thus $\hat{x}_8 = \Clower + V$, we have $\bar{x}_1 - \bar{x}_2 = (\lambda_1 \hat{x}_1 + \lambda_2 \hat{x}_2 + \lambda_3 \hat{x}_4) - (\lambda_2 \hat{x}_3 + \lambda_3 \hat{x}_5 + \lambda_4 \hat{x}_7)$. If $\bar{x}_2 = \Clower \bar{y}_2 + (\Cupper - \Clower - V) (\bar{y}_3 - \bar{u}_3 - \bar{u}_2)$, we let $\hat{x}_3 = \Clower$ and thus $\hat{x}_2 = \Clower + V$; if $\bar{x}_2 = \Vupper \bar{y}_2 + (\Cupper - \Vupper - V) (\bar{y}_3 - \bar{u}_3) - (\Cupper - \Clower - V) \bar{u}_2$, we let $\hat{x}_3 = \Vupper$ and thus $\hat{x}_2 = \Vupper + V$.
\fblue{For both cases, we have $\bar{x}_2 = \lambda_2 \hat{x}_3 + \lambda_3 \hat{x}_5 + \lambda_4 \hat{x}_7$ and thus $\bar{x}_1 = \lambda_1 \hat{x}_1 + \lambda_2 \hat{x}_2 + \lambda_3 \hat{x}_4$.}
\end{enumerate}
Thus, the whole claim holds, and we have proved the conclusion.
\end{proof}

\subsection{Proof for Theorem \ref{thm:general1}} \label{apx:thm:general1}
\begin{proof}
Here we only provide the proof for the case in which $\Cupper - \Vupper - V > 0$ and $\Cupper - \Clower - 2V > 0$ since the cases in which $\Cupper - \Vupper - V = 0$ or $\Cupper - \Clower - 2V = 0$ can be proved similarly. Similar to the proof for Proposition \ref{prop:t3l2-int-extpoint}, we prove that every point $z \in \bar{Q}_3^1$ can be written as $z = \sum_{s \in S} \lambda_s z^s$ for some $\lambda_s \geq 0$ and $\sum_{s \in S} \lambda_s = 1$, where $z^s \in {\bar{P}_3^1}, \ s \in S$ and $S$ is the index set for the candidate points.

For a given point $z = (\bar{x}_1, \bar{x}_2, \bar{x}_3, \bar{y}_1, \bar{y}_2, \bar{y}_3, \bar{u}_2, \bar{u}_3) \in \bar{Q}_3^1$, we let the candidate points $z^1, z^2, \allowbreak \cdots, z^8 \allowbreak \in {\bar{P}_3^1}$ in the forms such that $z^1 = (\hat{x}_1, \hat{x}_2, \allowbreak \hat{x}_3, \allowbreak 1, \allowbreak 1, \allowbreak 1, 0, 0)$, $z^2 = (\hat{x}_4, \hat{x}_5, 0, 1, 1, 0, 0, 0)$,  $z^3 = (\hat{x}_6, 0, \allowbreak \hat{x}_7, 1, 0, \allowbreak 1, 0, 1)$,  $z^4 = (\hat{x}_8, 0, 0, 1, 0, 0, 0, 0)$, $z^5 = (0, \hat{x}_9, \hat{x}_{10}, 0, \allowbreak 1, 1, 1, 0)$, $z^6 = (0, \hat{x}_{11}, 0, 0, \allowbreak 1, 0, 1, 0)$,  $z^7 = (0, 0, \hat{x}_{12}, 0, \allowbreak 0, 1, 0, 1)$, and $z^8 = (0, 0, 0, \allowbreak 0, 0, 0, 0, 0)$, where $\hat{x}_i, i = 1, \cdots, 12$ are to be decided later. Meanwhile, we let 
\begin{subeqnarray} \label{eqn:lamdba-extra}
&& \lambda_1 = \bar{y}_3 - \bar{u}_3 - \lambda_5, \ \lambda_2 = \bar{y}_2 - \bar{u}_2 - \bar{y}_3 + \bar{u}_3 + \lambda_5, \\
&& \max\{0, \bar{y}_1 + \bar{u}_2 + \bar{u}_3 - 1 \} \leq \lambda_3 \leq \min\{\bar{u}_3, \bar{y}_1 - \bar{y}_2 + \bar{u}_2 \}, \\
&& \lambda_4 = \bar{y}_1 - \bar{y}_2 + \bar{u}_2 - \lambda_3, \ \max\{0, \bar{y}_3 - \bar{u}_3 - \bar{y}_2 + \bar{u}_2 \} \leq \lambda_5 \leq \min\{\bar{u}_2, \bar{y}_3 - \bar{u}_3 \}, \\
&& \lambda_6 = \bar{u}_2 - \lambda_5, \ \lambda_7 = \bar{u}_3 - \lambda_3 \mbox{ and } \lambda_8 = 1 - \bar{y}_1 - \bar{u}_2 - \bar{u}_3 + \lambda_3. 
\end{subeqnarray} 

First of all, {based on this construction, we can check} that {both} $\lambda_3$ and $\lambda_5$ exist, $\sum_{s = 1}^{8} \lambda_s = 1$, and $\lambda_s \geq 0$ for $\forall s = 1, \cdots, 8$ due to \eqref{eqn:p-udef-t3l2}, \eqref{eqn:u2-3period-t3l2}, \eqref{eqn:Q-3-1-mu}, and \eqref{eqn:Q-3-1-md}.
{Meanwhile}, it {can be checked} that $\bar{y}_i = y_i(z) = \sum_{s=1}^{8} \lambda_s y_i(z^s)$ for $i = 1, 2, 3$ and $\bar{u}_i = u_i(z) = \sum_{s=1}^{8} \lambda_s u_i(z^s)$ for $i = 2, 3$, {where $y_i(z)$ represents the $\bar{y}_i$ component value in the given point $z$ and $u_i(z)$ represents the $\bar{u}_i$ component value in the given point $z$}.

{Thus, in the remaining part of this proof, we only need to} decide the values of $\hat{x}_i$ for $i = 1, \cdots, 12$, $\lambda_3$, and $\lambda_5$ such that $\bar{x}_i = x_i(z) = \sum_{s=1}^{8} \lambda_s x_i(z^s)$ for $i = 1, 2, 3$, i.e., 
\begin{subeqnarray} \label{eqn-C: x1x2x3bar-extra}
 && \bar{x}_1 = \lambda_1 \hat{x}_1 + \lambda_2 \hat{x}_4 + \lambda_3 \hat{x}_6 + \lambda_4 \hat{x}_8, \slabel{eqn-C: x1x2x3bar-extra-1} \\
 && \bar{x}_2 = \lambda_1 \hat{x}_2 + \lambda_2 \hat{x}_5 + \lambda_5 \hat{x}_9 + \lambda_6 \hat{x}_{11},\slabel{eqn-C: x1x2x3bar-extra-2} \\
 && \bar{x}_3 = \lambda_1 \hat{x}_3 + \lambda_3 \hat{x}_7 + \lambda_5 \hat{x}_{10} + \lambda_7 \hat{x}_{12}.  \slabel{eqn-C: x1x2x3bar-extra-3}
\end{subeqnarray}

{To show~\eqref{eqn-C: x1x2x3bar-extra}, in the following, we prove that for any $(\bar{x}_1, \bar{x}_2, \bar{x}_3)$ in its feasible region corresponding to a given $(\bar{y}_1, \bar{y}_2, \bar{y}_3, \bar{u}_2, \bar{u}_3)$, we can always find a $(\hat{x}_1, \hat{x}_2, \cdots, \hat{x}_{12})$ in its feasible region, corresponding to the same given $(\bar{y}_1, \bar{y}_2, \bar{y}_3, \bar{u}_2, \bar{u}_3)$. Now we describe the feasible regions for $(\hat{x}_1, \hat{x}_2, \cdots, \hat{x}_{12})$ and $(\bar{x}_1, \bar{x}_2, \bar{x}_3)$, respectively. First, since}
$y$ and $u$ in $z^1, \cdots, z^8$ are given, {by substituting $z^1, \cdots, z^8$ into ${\bar{P}_3^1}$}, the corresponding feasible region for $(\hat{x}_1, \hat{x}_2, \cdots, \hat{x}_{12})$ can be described as set $A = \{ (\hat{x}_1, \hat{x}_2, \allowbreak \cdots, \allowbreak \hat{x}_{12}) \in \R^{12}: \Clower \leq \hat{x}_i \leq \Cupper (i = 1, 2, 3), \ -V \leq \hat{x}_i - \hat{x}_{i+1} \leq V (i = 1, 2), \ \Clower \leq \hat{x}_4 \leq \Vupper + V, \  \Clower \leq \hat{x}_5 \leq \Vupper, \allowbreak \  -V \leq \hat{x}_4 - \hat{x}_5 \leq V, \ \Clower \leq \hat{x}_i \leq \Vupper (i = 6, 7, 8, 11, 12), \allowbreak \ \Clower \leq \hat{x}_9 \leq \Vupper, \allowbreak \ \Clower \leq \hat{x}_{10} \leq \Vupper + V, \    -V \leq \hat{x}_9 - \hat{x}_{10} \leq V \}$.
{Second, corresponding to a given $(\bar{y}_1, \bar{y}_2, \bar{y}_3, \bar{u}_2, \bar{u}_3)$, following the description of $\bar{Q}_3^1$, the feasible region for $(\bar{x}_1, \bar{x}_2, \bar{x}_3)$ can be described as follows:}
\begin{subeqnarray}
& \hspace{-0.5in} C =  \Bigl\{  (\bar{x}_1, \bar{x}_2, \bar{x}_3) \in \R^3: & \bar{x}_1 \geq \Clower \bar{y}_1, \ \bar{x}_2 \geq \Clower \bar{y}_2, \ \bar{x}_3 \geq \Clower \bar{y}_3, \slabel{eqn-C:lower-bound-t3l1} \\
&& \bar{x}_1  \leq \Vupper \bar{y}_1 + (\Cupper - \Vupper) (\bar{y}_2 - \bar{u}_2), \slabel{eqn-C:x1-ub-t3l1-1} \\
&& \bar{x}_1  \leq \Vupper \bar{y}_1 + V (\bar{y}_2 - \bar{u}_2) + (\Cupper - \Vupper - V) (\bar{y}_3 - \bar{u}_3), \slabel{eqn-C:x1-ub-t3l1-2} \\
&& \bar{x}_2 \leq \Cupper \bar{y}_2 - (\Cupper - \Vupper) \bar{u}_2, \slabel{eqn-C:x2-ub-t3l1-1} \\
&& \bar{x}_2  \leq \Vupper \bar{y}_2 + (\Cupper - \Vupper) (\bar{y}_3 - \bar{u}_3), \slabel{eqn-C:x2-ub-t3l1-2} \\
&& \bar{x}_3  \leq \Cupper \bar{y}_3 - (\Cupper - \Vupper) \bar{u}_3, \slabel{eqn-C:x3-ub-t3l1-1} \\
&& \bar{x}_3  \leq (\Vupper + V) \bar{y}_3 - V \bar{u}_3 + (\Cupper - \Vupper - V) (\bar{y}_2 - \bar{u}_2), \slabel{eqn-C:x3-ub-t3l1-2} \\
&& \bar{x}_2 - \bar{x}_1  \leq (\Clower + V) \bar{y}_2 - \Clower \bar{y}_1 - (\Clower + V - \Vupper) \bar{u}_2, \slabel{eqn-C:x2-x1-ub-t3l1} \\
&& \bar{x}_3 - \bar{x}_2  \leq (\Clower + V) \bar{y}_3 - \Clower \bar{y}_2 - (\Clower + V - \Vupper) \bar{u}_3, \slabel{eqn-C:x3-x2-ub-t3l1} \\
&& \bar{x}_1 - \bar{x}_2  \leq \Vupper \bar{y}_1 - (\Vupper - V) \bar{y}_2 - (\Clower + V - \Vupper) \bar{u}_2, \slabel{eqn-C:x1-x2-ub-t3l1} \\
&& \bar{x}_2 - \bar{x}_3  \leq \Vupper \bar{y}_2 - (\Vupper - V) \bar{y}_3 - (\Clower + V - \Vupper) \bar{u}_3, \slabel{eqn-C:x2-x3-ub-t3l1} \\
&& \bar{x}_3 - \bar{x}_1  \leq (\Clower + 2 V) \bar{y}_3 - \Clower \bar{y}_1 - (\Clower + 2 V - \Vupper) \bar{u}_3 - (\Clower + V - \Vupper) \bar{u}_2, \slabel{eqn-C:x3-x1-ub-t3l1} \\
&& \bar{x}_1 - \bar{x}_3  \leq \Vupper \bar{y}_1 - \Clower \bar{y}_3 + V (\bar{y}_2 - \bar{u}_2) + (\Clower + V - \Vupper) (\bar{y}_3 - \bar{u}_3 - \bar{u}_2), \slabel{eqn-C:x1-x3-ub-t3l1} \\
&& \bar{x}_3 - \alpha \bar{x}_1 \leq (\Vupper + V) \bar{y}_3 - V \bar{u}_3 - \alpha \Clower \bar{y}_1, \slabel{eqn-C:x3-ax1-ub-t3l1} \\
&& \bar{x}_1 - \alpha \bar{x}_3 \leq \Vupper \bar{y}_1 + V (\bar{y}_2 - \bar{u}_2) - \alpha \Clower \bar{y}_3, \slabel{eqn-C:x1-ax3-ub-t3l1}
\Bigr\},
\end{subeqnarray}
where $\alpha = (\Cupper - \Vupper - V)/(\Cupper - \Clower - 2V)$ ($0 < \alpha \leq 1$).

{Accordingly, we can set up the linear transformation $F$ from $(\hat{x}_1, \hat{x}_2, \cdots, \hat{x}_{12}) \in A$ to $(\bar{x}_1, \bar{x}_2, \bar{x}_3) \in C$ as follows:}
\begin{equation}       
F = \left(                 
  \begin{array}{ c   c  c   c c  c  c  c c c  c c}   
    \lambda_1 & 0 & 0 & \lambda_2 & 0 & \lambda_3 & 0 & \lambda_4 & 0 & 0  & 0 & 0   \\  
    0 & \lambda_1 & 0 & 0 & \lambda_2 & 0 & 0 & 0 & \lambda_5 & 0 & {\lambda_{6}} & 0 \\  
    0 & 0 & \lambda_1 & 0 & 0 & 0 & \lambda_3 & 0 & 0 & \lambda_5 & 0 & {\lambda_{7}} \\  
  \end{array}
\right), \nonumber
\end{equation}
{where $\lambda_1, \lambda_2, \cdots, \lambda_7$ follow the definitions described in \eqref{eqn:lamdba-extra}. Thus, in the following, we only need to prove that $F: A \rightarrow C$ is surjective.}

Since $C$ is a closed and bounded polytope, any point can be expressed as a convex combination of the extreme points in $C$. 
Accordingly, we only need to show that for any extreme point $w^i \in C$ ($i = 1, \cdots, M$), there exists a point $p^i \in A$ such that $F p^i = w^i$, where $M$ represents the number of extreme points in $C$ (because for an arbitrary point $w \in C$, which can be {represented} as $w = \sum_{i=1}^M \mu_i w^i$ and $\sum_{i=1}^M \mu_i = 1$, there exists $p = \sum_{i=1}^M \mu_i p_i \in A$ such that $F p = w$ due to the linearity of $F$ and the convexity of $A$). 
Since it is difficult to enumerate all the extreme points in $C$, in the following proof we show the conclusion holds for any point in the faces of $C$, i.e., satisfying one of \eqref{eqn-C:lower-bound-t3l1} - \eqref{eqn-C:x1-ax3-ub-t3l1}  at equality, which implies the conclusion holds for extreme points.

\textbf{Satisfying $\bar{x}_1 \geq \Clower \bar{y}_1$ at equality}. For this case, substituting $\bar{x}_1 = \Clower \bar{y}_1$ into  \eqref{eqn-C:x1-ub-t3l1-1} - \eqref{eqn-C:x1-ax3-ub-t3l1}, we obtain the feasible region of $(\bar{x}_2, \bar{x}_3)$ as $C' = \{(\bar{x}_2, \bar{x}_3) \in \R^2: \Clower \bar{y}_2 \leq \bar{x}_2 \leq (\Clower + V) \bar{y}_2 - (\Clower + V - \Vupper) \bar{u}_2, \ \Clower \bar{y}_3 \leq \bar{x}_3 \leq (\Clower + 2 V) \bar{y}_3 - (\Clower + 2 V - \Vupper) \bar{u}_3 + (\Vupper - \Clower - V) \min\{\bar{u}_2, \bar{y}_3 - \bar{u}_3 \}, \bar{x}_3 - \bar{x}_2 \leq (\Clower + V) \bar{y}_3 - \Clower \bar{y}_2 - (\Clower + V - \Vupper) \bar{u}_3, \bar{x}_2 - \bar{x}_3 \leq (\Clower + V) \bar{y}_2 - \Clower \bar{y}_3 - (\Clower + V - \Vupper) \bar{u}_2 + (\Clower + V - \Vupper) \max\{0, \bar{y}_3 - \bar{u}_3 - \bar{y}_2 + \bar{u}_2 \} \}$.

First, by letting $\hat{x}_1 = \hat{x}_4 = \hat{x}_6  = \hat{x}_8 = \Clower$, it is easy to check that $\bar{x}_1 = \lambda_1 \hat{x}_1 + \lambda_2 \hat{x}_4 + \lambda_3 \hat{x}_6 + \lambda_4 \hat{x}_8$. Then \eqref{eqn-C: x1x2x3bar-extra-1} holds. 
Note here that once $(\hat{x}_1, \hat{x}_4, \hat{x}_6, \hat{x}_8)$ {is} fixed, the corresponding feasible region for $(\hat{x}_2, \hat{x}_3, \hat{x}_5, \hat{x}_7, \hat{x}_9, \hat{x}_{10}, \allowbreak \hat{x}_{11}, \hat{x}_{12})$ can be described as set $A' = \{ (\hat{x}_2, \hat{x}_3, \hat{x}_5, \hat{x}_7, \hat{x}_9, \hat{x}_{10}, \allowbreak \hat{x}_{11}, \hat{x}_{12}) \in \R^8: \Clower \leq \hat{x}_2 \leq \Clower + V, \ \Clower \leq \hat{x}_3 \leq \Clower + 2V, \ -V \leq \hat{x}_2 - \hat{x}_3 \leq V, \ \Clower \leq \hat{x}_5 \leq \Clower + V, \ \allowbreak  \Clower \leq \hat{x}_i \leq \Vupper (i = 7, 11, 12), \allowbreak \ \Clower \leq \hat{x}_9 \leq \Vupper, \allowbreak \ \Clower \leq \hat{x}_{10} \leq \Vupper + V, \    -V \leq \hat{x}_9 - \hat{x}_{10} \leq V \}$. In the following, we repeat the argument above to consider that one of inequalities in $C'$ is satisfied at equality to obtain the values of $(\hat{x}_2, \hat{x}_3, \hat{x}_5, \hat{x}_7, \hat{x}_9, \hat{x}_{10}, \allowbreak \hat{x}_{11}, \hat{x}_{12})$ from $A'$. In addition, we decide the corresponding $\lambda_3$ and $\lambda_5$ when a particular value of $\lambda_3$ or $\lambda_5$ is required (it follows other $\lambda$'s are also decided), otherwise we let $\lambda_3$ and $\lambda_5$ be free in their ranges as described in \eqref{eqn:lamdba-extra} respectively. 
\begin{enumerate}[1)]
\item Satisfying $\bar{x}_2 \geq \Clower \bar{y}_2$ at equality. We obtain $ \bar{x}_3 \in C'' = \{ \bar{x}_3 \in \R: \Clower \bar{y}_3 \leq  \bar{x}_3 \leq (\Clower + V) \bar{y}_3 - (\Clower + V - \Vupper) \bar{u}_3 \}$ through substituting $\bar{x}_2 = \Clower \bar{y}_2$ into $C'$. By letting $\hat{x}_2 = \hat{x}_5 = \hat{x}_9 = \hat{x}_{11} = \Clower$, we have $\bar{x}_2 = \lambda_1 \hat{x}_2 + \lambda_2 \hat{x}_5 + \lambda_5 \hat{x}_9 + \lambda_6 \hat{x}_{11}$. Then \eqref{eqn-C: x1x2x3bar-extra-2} holds. 
Thus, the corresponding feasible region for $(\hat{x}_3, \hat{x}_7, \hat{x}_{10}, \hat{x}_{12})$ can be described as set $A'' = \{ (\hat{x}_3, \hat{x}_7, \hat{x}_{10}, \hat{x}_{12}) \in \R^4: \Clower \leq \hat{x}_i \leq \Clower + V (i = 3, 10), \ \Clower \leq \hat{x}_i \leq \Vupper (i = 7, 12) \}$. If $\bar{x}_3 \geq \Clower \bar{y}_3$ is satisfied at equality, we let $\hat{x}_3 = \hat{x}_7 = \hat{x}_{10} = \hat{x}_{12} = \Clower$; if $\bar{x}_3 \leq (\Clower + V) \bar{y}_3 - (\Clower + V - \Vupper) \bar{u}_3$ is satisfied at equality, we let $\hat{x}_3 = \hat{x}_{10} = \Clower + V$ and $\hat{x}_7 = \hat{x}_{12} = \Vupper$. It is easy to check that $\bar{x}_3 = \lambda_1 \hat{x}_3 + \lambda_3 \hat{x}_7 + \lambda_5 \hat{x}_{10} + \lambda_7 \hat{x}_{12}$. Then \eqref{eqn-C: x1x2x3bar-extra-3} holds.

\item Satisfying $\bar{x}_2 \leq (\Clower + V) \bar{y}_2 - (\Clower + V - \Vupper) \bar{u}_2$ at equality. We obtain $ \bar{x}_3 \in C'' = \{ \bar{x}_3 \in \R: \Clower \bar{y}_3 + (\Vupper - \Clower - V) \max\{0, \bar{y}_3 - \bar{u}_3 - \bar{y}_2 + \bar{u}_2 \} \leq  \bar{x}_3 \leq (\Clower + 2 V) \bar{y}_3 - (\Clower + 2 V - \Vupper) \bar{u}_3 + (\Vupper - \Clower - V) \min\{\bar{u}_2, \bar{y}_3 - \bar{u}_3 \} \}$. By letting $\hat{x}_2 = \hat{x}_5 = \Clower + V$ and $\hat{x}_9 = \hat{x}_{11} = \Vupper$, we have $\bar{x}_2 = \lambda_1 \hat{x}_2 + \lambda_2 \hat{x}_5 + \lambda_5 \hat{x}_9 + \lambda_6 \hat{x}_{11}$. Then \eqref{eqn-C: x1x2x3bar-extra-2} holds. 
Thus, the corresponding feasible region for $(\hat{x}_3, \hat{x}_7, \hat{x}_{10}, \hat{x}_{12})$ can be described as set $A'' = \{ (\hat{x}_3, \hat{x}_7, \hat{x}_{10}, \hat{x}_{12}) \in \R^4: \Clower \leq \hat{x}_3 \leq \Clower + 2V, \ \Clower \leq \hat{x}_i \leq \Vupper (i=7,12), \ \Vupper - V \leq \hat{x}_{10} \leq \Vupper+V \}$. If $\bar{x}_3 \geq \Clower \bar{y}_3 + (\Vupper - \Clower - V) \max\{0, \bar{y}_3 - \bar{u}_3 - \bar{y}_2 + \bar{u}_2 \}$ is satisfied at equality, we let $\hat{x}_3 = \hat{x}_7 = \hat{x}_{12} = \Clower$, $\hat{x}_{10} = \Vupper - V $ and $\lambda_5 = \max\{0, \bar{y}_3 - \bar{u}_3 - \bar{y}_2 + \bar{u}_2 \}$; if $\bar{x}_3 \leq (\Clower + 2 V) \bar{y}_3 - (\Clower + 2 V - \Vupper) \bar{u}_3 + (\Vupper - \Clower - V) \min\{\bar{u}_2, \bar{y}_3 - \bar{u}_3 \}$ is satisfied at equality, we let $\hat{x}_3 = \Clower + 2V$, $\hat{x}_7 = \hat{x}_{12} = \Vupper$, $\hat{x}_{10} = \Vupper + V$, and $\lambda_5 = \min\{\bar{u}_2, \bar{y}_3 - \bar{u}_3 \}$. 
{For both cases}, we have $\bar{x}_3 = \lambda_1 \hat{x}_3 + \lambda_3 \hat{x}_7 + \lambda_5 \hat{x}_{10} + \lambda_7 \hat{x}_{12}$. Then \eqref{eqn-C: x1x2x3bar-extra-3} holds.

\item Satisfying $\bar{x}_3 \geq \Clower \bar{y}_3$ at equality. We obtain $\bar{x}_2 \in C'' = \{ \bar{x}_2 \in \R: \Clower \bar{y}_2 \leq \bar{x}_2 \leq (\Clower + V) \bar{y}_2 - (\Clower + V - \Vupper) \bar{u}_2 - (\Vupper - \Clower - V) \max\{0, \bar{y}_3 - \bar{u}_3 - \bar{y}_2 + \bar{u}_2 \} \}$. By letting $\hat{x}_3 = \hat{x}_7 =  \hat{x}_{10} = \hat{x}_{12} = \Clower$, we have $\bar{x}_3 = \lambda_1 \hat{x}_3 + \lambda_3 \hat{x}_7 + \lambda_5 \hat{x}_{10} + \lambda_7 \hat{x}_{12}$. Then \eqref{eqn-C: x1x2x3bar-extra-3} holds. 
Thus, the corresponding feasible region for $(\hat{x}_2, \hat{x}_5, \hat{x}_9, \hat{x}_{11})$ can be described as set $A'' = \{ (\hat{x}_2, \hat{x}_5, \hat{x}_9, \hat{x}_{11}) \in \R^4: \Clower \leq \hat{x}_i \leq \Clower + V (i = 2, 5, 9), \ \Clower \leq \hat{x}_{11} \leq \Vupper \}$. 
If $\bar{x}_2 \geq \Clower \bar{y}_2$ is satisfied at equality, we let $\hat{x}_2 = \hat{x}_5 = \hat{x}_9 = \hat{x}_{11} = \Clower$; if $\bar{x}_2 \leq (\Clower + V) \bar{y}_2 - (\Clower + V - \Vupper) \bar{u}_2 - (\Vupper - \Clower - V) \max\{0, \bar{y}_3 - \bar{u}_3 - \bar{y}_2 + \bar{u}_2 \}$ is satisfied at equality, we let $\hat{x}_2 = \hat{x}_5 = \hat{x}_9 = \Clower + V$, $\hat{x}_{11} = \Vupper$, and $\lambda_5 = \max\{0, \bar{y}_3 - \bar{u}_3 - \bar{y}_2 + \bar{u}_2 \}$. {For both cases}, we have $\bar{x}_2 = \lambda_1 \hat{x}_2 + \lambda_2 \hat{x}_5 + \lambda_5 \hat{x}_9 + \lambda_6 \hat{x}_{11}$. Then \eqref{eqn-C: x1x2x3bar-extra-2} holds. 

\item Satisfying $\bar{x}_3 \leq (\Clower + 2 V) \bar{y}_3 - (\Clower + 2 V - \Vupper) \bar{u}_3 + (\Vupper - \Clower - V) \min\{\bar{u}_2, \bar{y}_3 - \bar{u}_3 \}$ at equality. We obtain $ \bar{x}_2 \in C'' = \{ \bar{x}_2 \in \R: \Clower \bar{y}_2 + V(\bar{y}_3 - \bar{u}_3) + (\Vupper - \Clower - V) \min\{\bar{u}_2, \bar{y}_3 - \bar{u}_3 \}  \leq \bar{x}_2 \leq (\Clower + V) \bar{y}_2 - (\Clower + V - \Vupper) \bar{u}_2 \}$. 
By letting $\hat{x}_3 = \Clower + 2V$, $\hat{x}_7 = \hat{x}_{12} = \Vupper$, $\hat{x}_{10} = \Vupper + V$, and $\lambda_5 = \min\{\bar{u}_2, \bar{y}_3 - \bar{u}_3 \}$, we have $\bar{x}_3 = \lambda_1 \hat{x}_3 + \lambda_3 \hat{x}_7 + \lambda_5 \hat{x}_{10} + \lambda_7 \hat{x}_{12}$. Then \eqref{eqn-C: x1x2x3bar-extra-3} holds. 
In addition, we further have $\hat{x}_2 = \Clower + V$ and $\hat{x}_9 = \Vupper$ based on the definition of $A'$.
Thus, the corresponding feasible region for $(\hat{x}_5, \hat{x}_{11})$ can be described as set $A'' = \{ (\hat{x}_5, \hat{x}_{11}) \in \R^2: \Clower \leq \hat{x}_5 \leq \Clower + V, \ \Clower \leq \hat{x}_{11} \leq \Vupper \}$. 
If $\bar{x}_2 \geq \Clower \bar{y}_2 + V(\bar{y}_3 - \bar{u}_3) + (\Vupper - \Clower - V) \min\{\bar{u}_2, \bar{y}_3 - \bar{u}_3 \}$ is satisfied at equality, we let $\bar{x}_5 = \bar{x}_{11} = \Clower$ {and $\lambda_5 = \min\{\bar{u}_2, \bar{y}_3 - \bar{u}_3 \}$};
if $\bar{x}_2 \leq (\Clower + V) \bar{y}_2 - (\Clower + V - \Vupper) \bar{u}_2$ is satisfied at equality, we let $\hat{x}_5 = \Clower + V$ and $\hat{x}_{11} = \Vupper$.
{For both cases}, we have $\bar{x}_2 = \lambda_1 \hat{x}_2 + \lambda_2 \hat{x}_5 + \lambda_5 \hat{x}_9 + \lambda_6 \hat{x}_{11}$. Then \eqref{eqn-C: x1x2x3bar-extra-2} holds. 

\item Satisfying $\bar{x}_3 - \bar{x}_2 \leq (\Clower + V) \bar{y}_3 - \Clower \bar{y}_2 - (\Clower + V - \Vupper) \bar{u}_3$ at equality. We obtain $\bar{x}_2 \in C'' = \{ \bar{x}_2 \in \R: \Clower \bar{y}_2 \leq \bar{x}_2 \leq \Clower \bar{y}_2 + V(\bar{y}_3 - \bar{u}_3) + (\Vupper - \Clower - V) \min\{\bar{u}_2, \bar{y}_3 - \bar{u}_3 \} \}$ through substituting $\bar{x}_3 = \bar{x}_2  + (\Clower + V) \bar{y}_3 - \Clower \bar{y}_2 - (\Clower + V - \Vupper) \bar{u}_3$ into set $C'$. 
By letting $\hat{x}_3 - \hat{x}_2 = \hat{x}_{10} - \hat{x}_9 = V$, $\hat{x}_7 = \hat{x}_{12} = \Vupper$, and $\hat{x}_5= \hat{x}_{11} = \Clower$, we have $\bar{x}_3 - \bar{x}_2 = (\lambda_1 \hat{x}_3 + \lambda_3 \hat{x}_7 + \lambda_5 \hat{x}_{10} + \lambda_7 \hat{x}_{12}) - (\lambda_1 \hat{x}_2 + \lambda_2 \hat{x}_5 + \lambda_5 \hat{x}_9 + \lambda_6 \hat{x}_{11})$. 
If $\bar{x}_2 \geq \Clower \bar{y}_2$ is satisfied at equality, we let $\hat{x}_2 = \hat{x}_9 = \Clower$ (and then $\hat{x}_3 = \hat{x}_{10} = \Clower+V$); 
if $\bar{x}_2 \leq \Clower \bar{y}_2 + V(\bar{y}_3 - \bar{u}_3) + (\Vupper - \Clower - V) \min\{\bar{u}_2, \bar{y}_3 - \bar{u}_3 \}$ is satisfied at equality, we let $\hat{x}_2 = \Clower + V$ (and then $\hat{x}_3 = \Clower + 2V$), $\hat{x}_9 = \Vupper$ (thus $\hat{x}_{10} = \Vupper + V$), and $\lambda_5 = \min\{\bar{u}_2, \bar{y}_3 - \bar{u}_3 \}$. For both cases, we have 
$\bar{x}_2 = \lambda_1 \hat{x}_2 + \lambda_2 \hat{x}_5 + \lambda_5 \hat{x}_9 + \lambda_6 \hat{x}_{11}$ and thus $\bar{x}_3 = \lambda_1 \hat{x}_3 + \lambda_3 \hat{x}_7 + \lambda_5 \hat{x}_{10} + \lambda_7 \hat{x}_{12}$. Then both \eqref{eqn-C: x1x2x3bar-extra-2} and \eqref{eqn-C: x1x2x3bar-extra-3} hold.

\item Satisfying $\bar{x}_2 - \bar{x}_3 \leq (\Clower + V) \bar{y}_2 - \Clower \bar{y}_3 - (\Clower + V - \Vupper) \bar{u}_2 + (\Clower + V - \Vupper) \max\{0, \bar{y}_3 - \bar{u}_3 - \bar{y}_2 + \bar{u}_2 \}$ at equality. We obtain $\bar{x}_3 \in C'' = \{ \bar{x}_3 \in \R: \Clower \bar{y}_3 \leq \bar{x}_3 \leq \Clower \bar{y}_3 + (\Vupper - \Clower - V) \max\{0, \bar{y}_3 - \bar{u}_3 - \bar{y}_2 + \bar{u}_2 \} \}$ through substituting $\bar{x}_2 = \bar{x}_3 + (\Clower + V) \bar{y}_2 - \Clower \bar{y}_3 - (\Clower + V - \Vupper) \bar{u}_2 + (\Clower + V - \Vupper) \max\{0, \bar{y}_3 - \bar{u}_3 - \bar{y}_2 + \bar{u}_2 \}$ into set $C'$. 
By letting $\hat{x}_2 - \hat{x}_3 = \hat{x}_9 - \hat{x}_{10} = V$, $\hat{x}_7 = \hat{x}_{12} = \Clower$, $\hat{x}_5= \Clower$, $\hat{x}_{11} = \Vupper$, and $\lambda_5 = \max\{0, \bar{y}_3 - \bar{u}_3 - \bar{y}_2 + \bar{u}_2 \}$, we have $\bar{x}_2 - \bar{x}_3 =  (\lambda_1 \hat{x}_2 + \lambda_2 \hat{x}_5 + \lambda_5 \hat{x}_9 + \lambda_6 \hat{x}_{11}) - (\lambda_1 \hat{x}_3 + \lambda_3 \hat{x}_7 + \lambda_5 \hat{x}_{10} + \lambda_7 \hat{x}_{12})$.
In addition, we further have $\hat{x}_2 = \Clower + V$ and $\hat{x}_3 = \Clower$ based on the definition of $A'$.
If $\bar{x}_3 \geq \Clower \bar{y}_3$ is satisfied at equality, we let $\hat{x}_{10} = \Clower$ (and then $\hat{x}_9 = \Clower+V$); 
if $\bar{x}_3 \leq \Clower \bar{y}_3 + (\Vupper - \Clower - V) \max\{0, \bar{y}_3 - \bar{u}_3 - \bar{y}_2 + \bar{u}_2 \}$ is satisfied at equality, we let $\hat{x}_{10} = \Vupper - V$ (and then $\hat{x}_9 = \Vupper$) {and $\lambda_5 = \max\{0, \bar{y}_3 - \bar{u}_3 - \bar{y}_2 + \bar{u}_2 \}$}. For both cases, we have $\bar{x}_3 = \lambda_1 \hat{x}_3 + \lambda_3 \hat{x}_7 + \lambda_5 \hat{x}_{10} + \lambda_7 \hat{x}_{12}$ and thus $\bar{x}_2 = \lambda_1 \hat{x}_2 + \lambda_2 \hat{x}_5 + \lambda_5 \hat{x}_9 + \lambda_6 \hat{x}_{11}$. Then both \eqref{eqn-C: x1x2x3bar-extra-2} and \eqref{eqn-C: x1x2x3bar-extra-3} hold.
\end{enumerate}
Similar analyses hold for $\bar{x}_2 \geq \Clower \bar{y}_2$ and $\bar{x}_3 \geq \Clower \bar{y}_3$ due to the similar structure and also hold for inequalities \eqref{eqn-C:x1-ub-t3l1-1} - \eqref{eqn-C:x1-x3-ub-t3l1} following the proofs for Proposition \ref{prop:t3l2-int-extpoint}, and thus are omitted here.

\textbf{Satisfying \eqref{eqn-C:x3-ax1-ub-t3l1} at equality}. For this case, substituting $\bar{x}_3 = \alpha \bar{x}_1  + (\Vupper + V) \bar{y}_3 - V \bar{u}_3 - \alpha \Clower \bar{y}_1$ into
\eqref{eqn-C:lower-bound-t3l1} - \eqref{eqn-C:x1-ax3-ub-t3l1}, we obtain the feasible region of $(\bar{x}_1, \bar{x}_2)$ as $C' = \{(\bar{x}_1, \bar{x}_2) \in \R^2: 
\Clower \bar{y}_1 + (\Cupper - \Clower - 2V)(\bar{y}_3 - \bar{u}_3 - \min\{\bar{u}_2, \bar{y}_3 - \bar{u}_3 \}) \leq \bar{x}_1 \leq \Clower \bar{y}_1 + (\Cupper - \Clower - 2V)(\bar{y}_3 - \bar{u}_3 -\max\{0, \bar{y}_3 - \bar{u}_3 - \bar{y}_2 + \bar{u}_2 \}), \ 
\Clower \bar{y}_2 + (\Cupper - \Clower - V)(\bar{y}_3 - \bar{u}_3) - (\Cupper - \Vupper - V) \min\{\bar{u}_2, \bar{y}_3 - \bar{u}_3 \} \leq \bar{x}_2 \leq (\Clower + V) \bar{y}_2 - (\Clower + V - \Vupper) \bar{u}_2 + (\Cupper - \Clower - 2V)(\bar{y}_3 - \bar{u}_3 - \max\{0, \bar{y}_3 - \bar{u}_3 - \bar{y}_2 + \bar{u}_2 \}), \
\bar{x}_2 - \bar{x}_1  \leq (\Clower + V) \bar{y}_2 - \Clower \bar{y}_1 - (\Clower + V - \Vupper) \bar{u}_2, \
\bar{x}_1 - \bar{x}_2 \leq \Clower (\bar{y}_1 - \bar{y}_2) - V(\bar{y}_3 - \bar{u}_3) - (\Vupper - \Clower - V) \max\{0, \bar{y}_3 - \bar{u}_3 - \bar{y}_2 + \bar{u}_2 \}, \
\bar{x}_2 - \alpha \bar{x}_1 \geq \Clower \bar{y}_2 + (\Vupper - \Clower) (\bar{y}_3 - \bar{u}_3) - \alpha \Clower \bar{y}_1 \}$.

First, by letting $\hat{x}_3 = \Cupper$, $\hat{x}_1 = \Cupper - 2V$, $\hat{x}_4 = \hat{x}_6 = \hat{x}_8 = \Clower$, $\hat{x}_7 = \hat{x}_{12} = \Vupper$,  and $\hat{x}_{10} = \Vupper + V$, we have $\bar{x}_3 - \alpha \bar{x}_1 = (\lambda_1 \hat{x}_3 + \lambda_3 \hat{x}_7 + \lambda_5 \hat{x}_{10} + \lambda_7 \hat{x}_{12}) - \alpha (\lambda_1 \hat{x}_1 + \lambda_2 \hat{x}_4 + \lambda_3 \hat{x}_6 + \lambda_4 \hat{x}_8)$.
In addition, we further have $\hat{x}_2 = \Cupper - V$ and $\hat{x}_9 = \Vupper$ based on the definition of $A$.
Then the corresponding feasible region for $(\hat{x}_5, \hat{x}_{11})$ can be described as set $A' = \{ (\hat{x}_5, \hat{x}_{11}) \in \R^2: \Clower \leq \hat{x}_5 \leq \Clower + V, \ \Clower \leq \hat{x}_{11} \leq \Vupper \}$.
Next, we only need to show $\bar{x}_1 = \lambda_1 \hat{x}_1 + \lambda_2 \hat{x}_4 + \lambda_3 \hat{x}_6 + \lambda_4 \hat{x}_8$ and $\bar{x}_2 = \lambda_1 \hat{x}_2 + \lambda_2 \hat{x}_5 + \lambda_5 \hat{x}_9 + \lambda_6 \hat{x}_{11}$. Then \eqref{eqn-C: x1x2x3bar-extra} will hold.
We consider that one of inequalities in $C'$ is satisfied at equality to obtain the values of $(\hat{x}_5, \hat{x}_{11})$ from $A'$. In addition, we decide the corresponding $\lambda_3$ and $\lambda_5$ when a particular value of $\lambda_3$ or $\lambda_5$ is required (it follows other $\lambda$'s are also decided), otherwise we let $\lambda_3$ and $\lambda_5$ be free in their ranges as described in \eqref{eqn:lamdba-extra} respectively. 
\begin{enumerate}[1)]
\item Satisfying $\bar{x}_1 \geq \Clower \bar{y}_1 + (\Cupper - \Clower - 2V)(\bar{y}_3 - \bar{u}_3 - \min\{\bar{u}_2, \bar{y}_3 - \bar{u}_3 \})$ at equality. 
We obtain $\bar{x}_2 \in C'' = \{ \bar{x}_2 \in \R: \Clower \bar{y}_2 + (\Cupper - \Clower - V)(\bar{y}_3 - \bar{u}_3) - (\Cupper - \Vupper - V) \min\{\bar{u}_2, \bar{y}_3 - \bar{u}_3 \} \leq \bar{x}_2 \leq (\Clower + V) \bar{y}_2 - (\Clower + V - \Vupper) \bar{u}_2 + (\Cupper - \Clower - 2V)(\bar{y}_3 - \bar{u}_3 - \min\{\bar{u}_2, \bar{y}_3 - \bar{u}_3 \} ) \} $. 
By letting $\lambda_5 = \min\{\bar{u}_2, \bar{y}_3 - \bar{u}_3\}$, we have $\bar{x}_1 = \lambda_1 \hat{x}_1 + \lambda_2 \hat{x}_4 + \lambda_3 \hat{x}_6 + \lambda_4 \hat{x}_8$.
If $\bar{x}_2 = \Clower \bar{y}_2 + (\Cupper - \Clower - V)(\bar{y}_3 - \bar{u}_3) - (\Cupper - \Vupper - V) \min\{\bar{u}_2, \bar{y}_3 - \bar{u}_3 \}$, we let $\hat{x}_5 = \hat{x}_{11} = \Clower$; 
if $\bar{x}_2 = (\Clower + V) \bar{y}_2 - (\Clower + V - \Vupper) \bar{u}_2 + (\Cupper - \Clower - 2V)(\bar{y}_3 - \bar{u}_3 - \min\{\bar{u}_2, \bar{y}_3 - \bar{u}_3 \} )$, we let $\hat{x}_5 = \Clower + V$ and $\hat{x}_{11} = \Vupper$.
For both cases, we have $\bar{x}_2 = \lambda_1 \hat{x}_2 + \lambda_2 \hat{x}_5 + \lambda_5 \hat{x}_9 + \lambda_6 \hat{x}_{11}$.

\item Satisfying $\bar{x}_1 \leq \Clower \bar{y}_1 + (\Cupper - \Clower - 2V)(\bar{y}_3 - \bar{u}_3 -\max\{0, \bar{y}_3 - \bar{u}_3 - \bar{y}_2 + \bar{u}_2 \})$ at equality. 
We obtain $\bar{x}_2 \in C'' = \{ \bar{x}_2 \in \R: \Clower \bar{y}_2 + (\Cupper - \Clower - V)(\bar{y}_3 - \bar{u}_3) - (\Cupper - \Vupper - V) \max\{0, \bar{y}_3 - \bar{u}_3 - \bar{y}_2 + \bar{u}_2 \} \leq \bar{x}_2 \leq (\Clower + V) \bar{y}_2 - (\Clower + V - \Vupper) \bar{u}_2 + (\Cupper - \Clower - 2V)(\bar{y}_3 - \bar{u}_3 - \max\{0, \bar{y}_3 - \bar{u}_3 - \bar{y}_2 + \bar{u}_2 \} ) \} $. 
By letting $\lambda_5 = \max\{0, \bar{y}_3 - \bar{u}_3 - \bar{y}_2 + \bar{u}_2 \}$, we have $\bar{x}_1 = \lambda_1 \hat{x}_1 + \lambda_2 \hat{x}_4 + \lambda_3 \hat{x}_6 + \lambda_4 \hat{x}_8$.
If $\bar{x}_2 = \Clower \bar{y}_2 + (\Cupper - \Clower - V)(\bar{y}_3 - \bar{u}_3) - (\Cupper - \Vupper - V) \max\{0, \bar{y}_3 - \bar{u}_3 - \bar{y}_2 + \bar{u}_2 \}$, we let $\hat{x}_5 = \hat{x}_{11} = \Clower$; 
if $\bar{x}_2 = (\Clower + V) \bar{y}_2 - (\Clower + V - \Vupper) \bar{u}_2 + (\Cupper - \Clower - 2V)(\bar{y}_3 - \bar{u}_3 - \max\{0, \bar{y}_3 - \bar{u}_3 - \bar{y}_2 + \bar{u}_2 \} )$, we let $\hat{x}_5 = \Clower + V$ and $\hat{x}_{11} = \Vupper$.
For both cases, we have $\bar{x}_2 = \lambda_1 \hat{x}_2 + \lambda_2 \hat{x}_5 + \lambda_5 \hat{x}_9 + \lambda_6 \hat{x}_{11}$.

\item Satisfying $\bar{x}_2 \geq \Clower \bar{y}_2 + (\Cupper - \Clower - V)(\bar{y}_3 - \bar{u}_3) - (\Cupper - \Vupper - V) \min\{\bar{u}_2, \bar{y}_3 - \bar{u}_3 \}$ at equality. 
We obtain $\bar{x}_1 =  \Clower \bar{y}_1 + (\Cupper - \Clower - 2V)(\bar{y}_3 - \bar{u}_3 - \min\{\bar{u}_2, \bar{y}_3 - \bar{u}_3 \})$. 
By letting $\hat{x}_5 = \hat{x}_{11} = \Clower $ and $\lambda_5 = \min\{\bar{u}_2, \bar{y}_3 - \bar{u}_3 \}$, we have $\bar{x}_2 = \lambda_1 \hat{x}_2 + \lambda_2 \hat{x}_5 + \lambda_5 \hat{x}_9 + \lambda_6 \hat{x}_{11}$.
Furthermore, we also have  $\bar{x}_1 = \lambda_1 \hat{x}_1 + \lambda_2 \hat{x}_4 + \lambda_3 \hat{x}_6 + \lambda_4 \hat{x}_8$.

\item Satisfying $\bar{x}_2 \geq (\Clower + V) \bar{y}_2 - (\Clower + V - \Vupper) \bar{u}_2 + (\Cupper - \Clower - 2V)(\bar{y}_3 - \bar{u}_3 - \max\{0, \bar{y}_3 - \bar{u}_3 - \bar{y}_2 + \bar{u}_2 \})$ at equality. 
We obtain $\bar{x}_1 =  \Clower \bar{y}_1 + (\Cupper - \Clower - 2V)(\bar{y}_3 - \bar{u}_3 - \max\{0, \bar{y}_3 - \bar{u}_3 - \bar{y}_2 + \bar{u}_2 \})$. 
By letting $\hat{x}_5 = \Clower + V$, $\hat{x}_{11} = \Vupper$, and $\lambda_5 = \max\{0, \bar{y}_3 - \bar{u}_3 - \bar{y}_2 + \bar{u}_2 \}$, we have $\bar{x}_2 = \lambda_1 \hat{x}_2 + \lambda_2 \hat{x}_5 + \lambda_5 \hat{x}_9 + \lambda_6 \hat{x}_{11}$.
Furthermore, we also have  $\bar{x}_1 = \lambda_1 \hat{x}_1 + \lambda_2 \hat{x}_4 + \lambda_3 \hat{x}_6 + \lambda_4 \hat{x}_8$.

\item Satisfying $\bar{x}_2 - \bar{x}_1  \leq (\Clower + V) \bar{y}_2 - \Clower \bar{y}_1 - (\Clower + V - \Vupper) \bar{u}_2$ at equality. 
We obtain $\bar{x}_1 \in C'' = \{ \bar{x}_1 \in \R: \Clower \bar{y}_1 + (\Cupper - \Clower - 2V)(\bar{y}_3 - \bar{u}_3 - \min\{\bar{u}_2, \bar{y}_3 - \bar{u}_3 \}) \leq \bar{x}_1 \leq \Clower \bar{y}_1 + (\Cupper - \Clower - 2V)(\bar{y}_3 - \bar{u}_3 -\max\{0, \bar{y}_3 - \bar{u}_3 - \bar{y}_2 + \bar{u}_2 \}) \} $. 
By letting $\hat{x}_5 = \Clower + V$ and $\hat{x}_{11} = \Vupper$, we have $\bar{x}_2 - \bar{x}_1  = (\lambda_1 \hat{x}_2 + \lambda_2 \hat{x}_5 + \lambda_5 \hat{x}_9 + \lambda_6 \hat{x}_{11}) - (\lambda_1 \hat{x}_1 + \lambda_2 \hat{x}_4 + \lambda_3 \hat{x}_6 + \lambda_4 \hat{x}_8)$.
If $\bar{x}_1 = \Clower \bar{y}_1 + (\Cupper - \Clower - 2V)(\bar{y}_3 - \bar{u}_3 - \min\{\bar{u}_2, \bar{y}_3 - \bar{u}_3 \})$, we let $\lambda_5 = \min\{\bar{u}_2, \bar{y}_3 - \bar{u}_3 \}$; 
if $\bar{x}_1 = \Clower \bar{y}_1 + (\Cupper - \Clower - 2V)(\bar{y}_3 - \bar{u}_3 - \max\{0, \bar{y}_3 - \bar{u}_3 - \bar{y}_2 + \bar{u}_2 \})$, we let $\lambda_5 = \max\{0, \bar{y}_3 - \bar{u}_3 - \bar{y}_2 + \bar{u}_2 \}$.
For both cases, we have $\bar{x}_1 = \lambda_1 \hat{x}_1 + \lambda_2 \hat{x}_4 + \lambda_3 \hat{x}_6 + \lambda_4 \hat{x}_8$ and thus $\bar{x}_2 = \lambda_1 \hat{x}_2 + \lambda_2 \hat{x}_5 + \lambda_5 \hat{x}_9 + \lambda_6 \hat{x}_{11}$.

\item Satisfying $\bar{x}_1 - \bar{x}_2 \leq \Clower (\bar{y}_1 - \bar{y}_2) - V(\bar{y}_3 - \bar{u}_3) - (\Vupper - \Clower - V) \max\{0, \bar{y}_3 - \bar{u}_3 - \bar{y}_2 + \bar{u}_2 \}$ at equality. 
We obtain $\bar{x}_2 = \Clower \bar{y}_2 + (\Cupper - \Clower - V)(\bar{y}_3 - \bar{u}_3) - (\Cupper - \Vupper - V) \max\{0, \bar{y}_3 - \bar{u}_3 - \bar{y}_2 + \bar{u}_2 \} $.
By letting $\hat{x}_5 = \hat{x}_{11} = \Clower $ and $\lambda_5 = \max\{0, \bar{y}_3 - \bar{u}_3 - \bar{y}_2 + \bar{u}_2 \}$, we have $\bar{x}_1 - \bar{x}_2 = (\lambda_1 \hat{x}_1 + \lambda_2 \hat{x}_4 + \lambda_3 \hat{x}_6 + \lambda_4 \hat{x}_8) - (\lambda_1 \hat{x}_2 + \lambda_2 \hat{x}_5 + \lambda_5 \hat{x}_9 + \lambda_6 \hat{x}_{11})$.
Furthermore, we also have  $\bar{x}_1 = \lambda_1 \hat{x}_1 + \lambda_2 \hat{x}_4 + \lambda_3 \hat{x}_6 + \lambda_4 \hat{x}_8$ and thus $\bar{x}_2 = \lambda_1 \hat{x}_2 + \lambda_2 \hat{x}_5 + \lambda_5 \hat{x}_9 + \lambda_6 \hat{x}_{11}$.

\item Satisfying $\bar{x}_2 - \alpha \bar{x}_1 \geq \Clower \bar{y}_2 + (\Vupper - \Clower) (\bar{y}_3 - \bar{u}_3) - \alpha \Clower \bar{y}_1$ at equality.
We obtain $\bar{x}_1 \in C'' = \{ \bar{x}_1 \in \R: \Clower \bar{y}_1 + (\Cupper - \Clower - 2V)(\bar{y}_3 - \bar{u}_3 - \min\{\bar{u}_2, \bar{y}_3 - \bar{u}_3 \}) \leq \bar{x}_1 \leq \Clower \bar{y}_1 + (\Cupper - \Clower - 2V)(\bar{y}_3 - \bar{u}_3 -\max\{0, \bar{y}_3 - \bar{u}_3 - \bar{y}_2 + \bar{u}_2 \}) \} $. 
By letting $\hat{x}_5 = \hat{x}_{11} = \Clower$, we have $\bar{x}_2 - \alpha \bar{x}_1  = (\lambda_1 \hat{x}_2 + \lambda_2 \hat{x}_5 + \lambda_5 \hat{x}_9 + \lambda_6 \hat{x}_{11}) - \alpha (\lambda_1 \hat{x}_1 + \lambda_2 \hat{x}_4 + \lambda_3 \hat{x}_6 + \lambda_4 \hat{x}_8)$.
If $\bar{x}_1 = \Clower \bar{y}_1 + (\Cupper - \Clower - 2V)(\bar{y}_3 - \bar{u}_3 - \min\{\bar{u}_2, \bar{y}_3 - \bar{u}_3 \})$, we let $\lambda_5 = \min\{\bar{u}_2, \bar{y}_3 - \bar{u}_3 \}$; 
if $\bar{x}_1 = \Clower \bar{y}_1 + (\Cupper - \Clower - 2V)(\bar{y}_3 - \bar{u}_3 - \max\{0, \bar{y}_3 - \bar{u}_3 - \bar{y}_2 + \bar{u}_2 \})$, we let $\lambda_5 = \max\{0, \bar{y}_3 - \bar{u}_3 - \bar{y}_2 + \bar{u}_2 \}$.
For both cases, we have $\bar{x}_1 = \lambda_1 \hat{x}_1 + \lambda_2 \hat{x}_4 + \lambda_3 \hat{x}_6 + \lambda_4 \hat{x}_8$ and thus $\bar{x}_2 = \lambda_1 \hat{x}_2 + \lambda_2 \hat{x}_5 + \lambda_5 \hat{x}_9 + \lambda_6 \hat{x}_{11}$.
\end{enumerate}
Similar {analyses hold} for \eqref{eqn-C:x1-ax3-ub-t3l1} due to the similar structure and thus are omitted here.
\end{proof}

\section{Proofs for Multi-period Formulations} \label{apx:sec:multi-period}

\subsection{Proof for Proposition \ref{prop:x_t-ub-2-multi-period}} \label{apx:subsec:x_t-ub-2-multi-period}
\begin{proof}
\textbf{(Validity)} We discuss the following four \fblue{possible} cases in terms of the values of $y_t$ and $y_{t+1}$:
\begin{enumerate}[1)]
\item If $y_t = y_{t+1} = 1$, \fblue{then} we have $u_{t+1} = 0$ due to constraints \eqref{eqn:p-mindn} and $\sum_{s=1}^{k-1} u_{t-s+1} \leq 1$ due to constraints \eqref{eqn:p-minup} since $k \leq L$. We further discuss the following two \fblue{possible} cases.
\begin{itemize}
\item If $u_{t-s+1} = 0$ for all $s \in [1, k-1]_{\Z}$, then \eqref{eqn:x_t-ub-2-multi-period} converts to $x_t \leq \Cupper$, which is valid due to \eqref{eqn:p-upper-bound}.
\item If $u_{t-s+1} = 1$ for some $s \in [1, k-1]_{\Z}$, then \eqref{eqn:x_t-ub-2-multi-period} converts to $x_t \leq \Vupper + (s-1)V$, which is valid due to \eqref{eqn:p-ramp-up}.
\end{itemize}
\item If $y_t = 1$ and $y_{t+1} = 0$, then $u_{t-s+1} = 0$ for all $s \in [0, k-1]_{\Z}$ due to constraints \eqref{eqn:p-minup} since $k \leq L$. It follows that \eqref{eqn:x_t-ub-2-multi-period} converts to $x_t \leq \Vupper$, which is valid because of ramp-down constraints \eqref{eqn:p-ramp-down}.
\item If $y_t = 0$ and $y_{t+1} = 1$, \fblue{then} we have $u_{t+1} = 1$ due to constraints \eqref{eqn:p-udef} and $u_{t-s+1} = 0$ for all $s \in [1, k-1]_{\Z}$ due to \eqref{eqn:p-minup} since $k \leq L$. It follows \eqref{eqn:x_t-ub-2-multi-period} is valid.
\item If $y_t = y_{t+1} = 0$, \eqref{eqn:x_t-ub-2-multi-period} is clearly valid.
\end{enumerate}

\textbf{(Facet-defining)} Here we only provide the facet-defining proof for condition (1), as the proof for condition (2) is similar and thus is omitted here. We generate $3T-1$ affinely independent points in conv($P$) that satisfy \eqref{eqn:x_t-ub-2-multi-period} at equality.
Since $0 \in $ conv($P$), we generate another $3T-2$ linearly independent points in conv($P$) in the following groups. In the following proofs of Appendix, we use the superscript of $(x,y,u)$, e.g., $r$ in $(x^r, y^r, u^r)$, to indicate the index of different points in conv($P$), and define $\epsilon$ as an arbitrarily small positive real number.

\begin{enumerate}[1)]
\item For each $r \in [1,t-1]_{\Z}$ (totally $t-1$ points), we create $(\acute{x}^r, \acute{y}^r, \acute{u}^r) \in$ conv($P$) such that
\begin{equation*}
\acute{x}_s^r = \left\{\begin{array}{l}
  \Clower, \ s \in [1,r]_{\Z} \\
  0, \ s \in [r+1,T]_{\Z}
\end{array} \right., \
\acute{y}_s^r = \left\{\begin{array}{l}
  1, \ s \in [1,r]_{\Z} \\
  0, \ s \in [r+1,T]_{\Z}
\end{array} \right., \ \mbox{and} \
\begin{array}{l}
\acute{u}_s^r = 0, \\
\forall s
\end{array}.
\end{equation*}

\item For each $r \in [1,t-1]_{\Z}$ (totally $t-1$ points), we create $(\bar{x}^r, \bar{y}^r, \bar{u}^r) \in$ conv($P$) such that
\begin{equation*}
\bar{x}_s^r = \left\{\begin{array}{l}
  \Clower + \epsilon, \ s \in [1,r]_{\Z} \\
  0, \ s \in [r+1,T]_{\Z}
\end{array} \right., \
\bar{y}_s^r = \left\{\begin{array}{l}
  1, \ s \in [1,r]_{\Z} \\
  0, \ s \in [r+1,T]_{\Z}
\end{array} \right., \ \mbox{and} \
\begin{array}{l}
\bar{u}_s^r = 0, \\
\forall s
\end{array}.
\end{equation*}

\item For $r = t$ (totally one point), we create $(\bar{x}^r, \bar{y}^r, \bar{u}^r) \in$ conv($P$) such that
\begin{equation*}
\bar{x}_s^r = \left\{\begin{array}{l}
  \Vupper, \ s \in [1,r]_{\Z} \\
  0, \ s \in [r+1,T]_{\Z}
\end{array} \right., \
\bar{y}_s^r = \left\{\begin{array}{l}
  1, \ s \in [1,r]_{\Z} \\
  0, \ s \in [r+1,T]_{\Z}
\end{array} \right., \ \mbox{and} \
\begin{array}{l}
\bar{u}_s^r = 0, \\
\forall s
\end{array}.
\end{equation*}

\item For each $r \in [t+1,T-1]_{\Z}$ (totally $T-t-1$ points), we create $(\bar{x}^r, \bar{y}^r, \bar{u}^r) \in$ conv($P$) such that $\bar{y}_s^r = 1$ for each $s \in [t-k+2,r]_{\Z}$ and $\bar{y}_s^r = 0$ otherwise. Thus $\bar{u}_s^r = 1$ for $s = t-k+2$ and $\bar{u}_s^r = 0$ otherwise {due to constraints \eqref{eqn:p-minup} - \eqref{eqn:p-udef}}. Moreover, we let $\bar{x}_s^r = \max\{\Clower, \Vupper + (s-(t-k+2))V\}$ for each $s \in [t-k+2,t]_{\Z}$, $\bar{x}_s^r = \max\{\Clower, \Vupper + (k-3)V\}$ for each $s \in [t+1,r]_{\Z}$, and $\bar{x}_s^r = 0$ otherwise.

\item For $r = T$ (totally one point), we create $(\bar{x}^r, \bar{y}^r, \bar{u}^r) \in$ conv($P$) such that $\bar{y}_s^r = 1$ for each $s \in [1,T]_{\Z}$, $\bar{u}_s^r = 0$ for each $s \in [2,T]_{\Z}$, and $\bar{x}_s^r = \Cupper$ for each $s \in [1,T]_{\Z}$.

\item For each $r \in [2,t-k+1]_{\Z}$ (totally $t-k$ points), we create $(\hat{x}^r, \hat{y}^r, \hat{u}^r) \in$ conv($P$) such that
\begin{equation*}
\hat{x}_s^r = \left\{\begin{array}{l}
  \Vupper, \ s \in [r,t]_{\Z} \\
  0, \ \mbox{o.w.}
\end{array} \right., \
\hat{y}_s^r = \left\{\begin{array}{l}
  1, \ s \in [r,t]_{\Z} \\
  0, \ \mbox{o.w.}
\end{array} \right., \ \mbox{and} \
\hat{u}_s^r = \left\{\begin{array}{l}
  1, \ s = r \\
  0, \ \mbox{o.w.}
\end{array} \right..
\end{equation*}

\item For each $r \in [t-k+2,t]_{\Z}$ (totally $k-1$ points), we create $(\hat{x}^r, \hat{y}^r, \hat{u}^r) \in$ conv($P$) such that
\begin{equation*}
\hat{x}_s^r = \left\{\begin{array}{l}
  \Vupper + (s-r)V, \ s \in [r,t]_{\Z} \\
  \max \{\Clower, \Vupper + (t-r-1)V \}, \\
  \hspace{0.8in} s \in [t+1,T]_{\Z} \\
  0, \ \mbox{o.w.}
\end{array} \right., \
\hat{y}_s^r = \left\{\begin{array}{l}
  1, \ s \in [r,T]_{\Z} \\
  0, \ \mbox{o.w.}
\end{array} \right., \ \mbox{and} \
\hat{u}_s^r = \left\{\begin{array}{l}
  1, \ s = r \\
  0, \ \mbox{o.w.}
\end{array} \right..
\end{equation*}

\item For each $r \in [t+1,T]_{\Z}$ (totally $T-t$ points), we create $(\hat{x}^r, \hat{y}^r, \hat{u}^r) \in$ conv($P$) such that
\begin{equation*}
\hat{x}_s^r = \left\{\begin{array}{l}
  \Clower, \ s \in [r,T]_{\Z} \\
  0, \ \mbox{o.w.}
\end{array} \right., \
\hat{y}_s^r = \left\{\begin{array}{l}
  1, \ s \in [r,T]_{\Z} \\
  0, \ \mbox{o.w.}
\end{array} \right., \ \mbox{and} \
\hat{u}_s^r = \left\{\begin{array}{l}
  1, \ s = r \\
  0, \ \mbox{o.w.}
\end{array} \right..
\end{equation*}

\item For each $r \in [t+1,T]_{\Z}$ (totally $T-t$ points), we create $(\grave{x}^r, \grave{y}^r, \grave{u}^r) \in$ conv($P$) such that
\begin{equation*}
\grave{x}_s^r = \left\{\begin{array}{l}
  \Clower + \epsilon, \ s \in [r,T]_{\Z} \\
  0, \ \mbox{o.w.}
\end{array} \right., \
\grave{y}_s^r = \left\{\begin{array}{l}
  1, \ s \in [r,T]_{\Z} \\
  0, \ \mbox{o.w.}
\end{array} \right., \ \mbox{and} \
\grave{u}_s^r = \left\{\begin{array}{l}
  1, \ s = r \\
  0, \ \mbox{o.w.}
\end{array} \right..
\end{equation*}
\end{enumerate}

Finally, it is clear that $(\bar{x}^r, \bar{y}^r, \bar{u}^r)_{r=1}^{T}$ and $(\hat{x}^r, \hat{y}^r, \hat{u}^r)_{r=2}^{T}$ are linearly independent because they can construct a lower-diagonal matrix. In addition, $(\acute{x}^r, \acute{y}^r, \acute{u}^r)_{r=1}^{t-1}$ and $(\grave{x}^r, \grave{y}^r, \grave{u}^r)_{r=t+1}^{T}$ are also linearly independent with them after Gaussian eliminations between $(\acute{x}^r, \acute{y}^r, \acute{u}^r)_{r=1}^{t-1}$ and $(\bar{x}^r, \bar{y}^r, \bar{u}^r)_{r=1}^{t-1}$, and between $(\hat{x}^r, \hat{y}^r, \hat{u}^r)_{r=t+1}^{T}$ and $(\grave{x}^r, \grave{y}^r, \grave{u}^r)_{r=t+1}^{T}$.
\end{proof}

\subsection{Proof for Proposition \ref{prop:x_t-up-exp}} \label{apx:subsec:x_t-up-exp}
\begin{proof}
\textbf{(Validity)} It is clear that inequality \eqref{eqn:x_t-up-exp} is valid when $y_t = 0$ due to constraints \eqref{eqn:p-minup}. We then discuss the following two possible cases in which $y_t = 1$ by considering the last start-up time {(denoted as $t-s$, $s \geq 0$)} before $t$. Meanwhile, we let $f = \sum_{i \in S} (i - d_i)$.
\begin{enumerate}[1)]
\item $s \geq m + L$, i.e., $t-s+L-1 \leq t-m-1$. We further discuss the following two possible cases in terms of the first shut-down time {(denoted as $t+\hat{s}$, $\hat{s} \geq 1$)} after $t$. 
\begin{enumerate}[(1)]
\item $\hat{s} \geq n+1$, i.e., $t+\hat{s} \geq t+n+1$. Inequality \eqref{eqn:x_t-up-exp} converts to $x_t \leq \Cupper y_t$, which is clearly valid.
\item $\hat{s} \in [1,n]_{\Z}$, i.e., $t+1 \leq t+\hat{s} \leq t+n$. In this case, we have $x_t \leq \Vupper + \min \{s, \hat{s}-1\}V \leq \Vupper + (\hat{s}-1)V \leq \Vupper y_t + V \sum_{k=1}^{{[n-1]^+}} (y_{t+k} - \sum_{j=0}^{L-1} u_{t+k-j}) \leq \mbox{the right-hand side (RHS) of}$ \eqref{eqn:x_t-up-exp}.
\end{enumerate}
\item $s \in [0, m+L-1]_{\Z}$, i.e., $t-s+L-1 \in [t-m, t+L-1]_{\Z}$. \sred{In this case, we} have $y_i - \sum_{j=0}^{L-1}u_{i-j} = 0$ for all $i \in [t-m, t-s+L-1]_{\Z}$. We further discuss the following two possible cases in terms of the first shut-down time {(denoted as $t+\hat{s}$, $\hat{s} \geq 1$)} after $t$.
\begin{enumerate}[(1)]
\item $\hat{s} \geq n+1$, i.e., $t+\hat{s} \geq t+n+1$. We further discuss the following three possible cases in terms of the value of $t-s+L-1$.
\begin{enumerate}[(a)]
\item $t-s+L-1 \in [t-m, t]_{\Z}$, i.e., $L-1 \leq s \leq L+m-1$, which indicates $\phi = 0$ in the RHS of \eqref{eqn:x_t-up-exp}. In this case, we have $x_t \leq \Vupper + \min \{s, \hat{s}-1\}V$. From inequality \eqref{eqn:x_t-up-exp}, we have $x_t \leq \Vupper + \tilde{f}V + (m+L-1-f)V$, where $\tilde{f} = \sum_{i \in S \setminus [t-m,t-s+L-1]_{\Z}} (i-d_i)$. Now we {prove} the validity by showing that 
\begin{equation}
\Vupper + \min \{s, \hat{s}-1\}V \leq \Vupper + \tilde{f}V + (m+L-1-f)V. \label{eq:inter-validlity-1}
\end{equation}
If $S = \emptyset$, then \eqref{eq:inter-validlity-1} holds since $f = \tilde{f} = 0$ and $s \leq m+L-1$. We then discuss the following three possible cases in which $S \neq \emptyset$. Before that, we let $t-q = \max \{a \in S\}$ and $t-p = \max \{a \in S, a \leq t-s+L-1\}$. Note that $t-q$ exists since $S \neq \emptyset$ and therefore $f = m-q$.
\begin{itemize}
\item If $t-q \leq t-s+L-1$, i.e., $s \leq q+L-1$, then we have $\tilde{f} = 0$. It follows that  \eqref{eq:inter-validlity-1} holds since $\min \{s, \hat{s}-1\} \leq s \leq q+L-1 = \tilde{f} + (m+L-1-f)$.
\item If $t-q \geq t-s+L$ and $t-p$ does not exist, then we have $\tilde{f} = f = m-q$. It follows that  \eqref{eq:inter-validlity-1} holds since $\min \{s, \hat{s}-1\} \leq s \leq m+L-1 = \tilde{f} + (m+L-1-f)$.
\item If $t-q \geq t-s+L$ and $t-p$ exists, then we have $\tilde{f} = p-q$ and meanwhile $s \leq p+L-1$ since $t-p \leq t-s+L-1$. It follows that  \eqref{eq:inter-validlity-1} holds since $\min \{s, \hat{s}-1\} \leq s \leq p+L-1 = \tilde{f} + (m+L-1-f)$.
\end{itemize}
\item $t-s+L-1 \in [t+1, t+n-1]_{\Z}$, i.e., $s \in [L-1-n, L-2]_{\Z}$. In this case, we have $x_t \leq \Vupper + \min \{s, \hat{s}-1\}V$. From inequality \eqref{eqn:x_t-up-exp}, we have $x_t \leq \Vupper + (t+{[n-1]^+} - (t-s+L-1))V + (m+L-1-f-{[n-1]^+})V + \phi = \Vupper + sV + (m-f)V+\phi$. Clearly we have $\Vupper + \min \{s, \hat{s}-1\}V \leq \Vupper + sV + (m-f)V+\phi$ since $f \leq m$ and $\phi \geq 0$ and therefore it {proves} the validity of \eqref{eqn:x_t-up-exp}.
\item $t-s+L-1 \in [t+n, t+\hat{s}-1]_{\Z}$, i.e., $s \in [L-\hat{s}, L-1-n]_{\Z}$. Since $n \leq T-t-1$ in this case, we should have $n \geq (L-1)/2$ by the proposition condition and therefore it follows that $s \leq L-1-n \leq (L-1)/2$ and furthermore $\min \{L-1-s, s\} = s$. In this case, we have $x_t \leq \Vupper + \min \{s, \hat{s}-1\}V \leq \Vupper + sV = \Vupper + \min \{L-1-s, s\}V = \Vupper + \phi \leq \mbox{{the} RHS of}$ \eqref{eqn:x_t-up-exp}.
\end{enumerate}
\item $\hat{s} \in [1,n]_{\Z}$. It follows that $t-s+L-1 \leq t+\hat{s}-1$, i.e., $s \geq L-\hat{s}$. We further discuss the following two possible cases in terms of the value of $t-s+L-1$.
\begin{enumerate}[(a)]
\item $t-s+L-1 \in [t-m, t]_{\Z}$, i.e., $s \in [L-1, L+m-1]_{\Z}$. In this case, we have $x_t \leq \Vupper + \min \{s, \hat{s}-1\}V$. From inequality \eqref{eqn:x_t-up-exp}, we have $x_t \leq \Vupper + \tilde{f} V + (\hat{s}-1) V + \phi$, where $\tilde{f} = \sum_{i \in S \setminus [t-m,t-s+L-1]_{\Z}} (i-d_i)$. Clearly we have $\Vupper + \min \{s, \hat{s}-1\}V \leq \Vupper + \tilde{f} V + (\hat{s}-1)  V + \phi$ since $\tilde{f} \geq 0$ and $\phi \geq 0$ and therefore it {proves} the validity of \eqref{eqn:x_t-up-exp}.
\item $t-s+L-1 \in [t+1, t+\hat{s}-1]_{\Z}$, i.e., $s \in [L-\hat{s}, L-2]_{\Z}$. In this case, we have $x_t \leq \Vupper + \min \{s, \hat{s}-1\}V$. From inequality \eqref{eqn:x_t-up-exp}, we have $x_t \leq \Vupper + (t+\hat{s}-1 - (t-s+L-1)) V + \phi = \Vupper + (s+\hat{s}-L)V + \phi$. Therefore we only need to show 
\begin{equation}
\Vupper + \min \{s, \hat{s}-1\}V \leq \Vupper + (s+\hat{s}-L)V + \phi. \label{eq:inter-validlity-2}
\end{equation}
We discuss two possible cases in the following in terms of the value $s$.
\begin{itemize}
\item If $s \leq {t+L-T-1}$, then $\phi = sV$ and \eqref{eq:inter-validlity-2} holds since $\min \{s, \hat{s}-1\} \leq s$ and $s+\hat{s}-L \geq 0$.
\item If $s \geq {[t+L-T]^+}$, then $\phi = \min \{L-1-s, s\}V$ and {the} RHS of \eqref{eq:inter-validlity-2} becomes $\Vupper + \min \{\hat{s}-1, s + s + \hat{s} - L\}$, which shows that \eqref{eq:inter-validlity-2} holds since $s+\hat{s}-L \geq 0$.
\end{itemize}
\end{enumerate}
\end{enumerate}
\end{enumerate}

\textbf{(Facet-defining)} We generate $3T-2$ linearly independent points (i.e., $(\tilde{x}^{\tilde{r}}, \tilde{y}^{\tilde{r}}, \tilde{u}^{\tilde{r}})_{\tilde{r}=2}^{T}$,  $(\hat{x}^{\hat{r}}, \allowbreak \hat{y}^{\hat{r}}, \allowbreak \hat{u}^{\hat{r}})_{\hat{r}=1, \hat{r} \neq t}^{T}$, and $(\bar{x}^{\bar{r}}, \bar{y}^{\bar{r}}, \bar{u}^{\bar{r}})_{\bar{r}=1}^{T}$) in conv($P$) that satisfy \eqref{eqn:x_t-up-exp} at equality.

\begin{enumerate}[1)]
\item For each $\tilde{r} \in [2, t-L]_{\Z}$ (totally $t-L-1$ points), we create $(\tilde{x}^{\tilde{r}}, \tilde{y}^{\tilde{r}}, \tilde{u}^{\tilde{r}}) \in$ conv($P$) such that $\tilde{y}_s^{\tilde{r}} = 1$ for $s \in [\tilde{r}, \tilde{r}+L-1 ]_{\Z}$ and $\tilde{y}_s^{\tilde{r}} = 0$ otherwise, $\tilde{u}_s^{\tilde{r}} = 1$ for $s = \tilde{r}$ and $\tilde{u}_s^{\tilde{r}} = 0$ otherwise, and $\tilde{x}_s^{\tilde{r}} = \Clower$ for $s \in [\tilde{r},\tilde{r}+L-1]_{\Z}$ and $\tilde{x}_s^{\tilde{r}} = 0$ otherwise.
\item For each $\tilde{r} \in [t-L+1, t+n-L+1 ]_{\Z}$ (totally $n+1$ points), we create $(\tilde{x}^{\tilde{r}}, \tilde{y}^{\tilde{r}}, \tilde{u}^{\tilde{r}}) \in$ conv($P$) such that $\tilde{y}_s^{\tilde{r}} = 1$ for $s \in [\tilde{r}, \tilde{r}+L-1 ]_{\Z}$ and $\tilde{y}_s^{\tilde{r}} = 0$ otherwise and $\tilde{u}_s^{\tilde{r}} = 1$ for $s = \tilde{r}$ and $\tilde{u}_s^{\tilde{r}} = 0$ otherwise. For the value of $\tilde{x}^{\tilde{r}}$, \sred{we have the following two cases:}
	\begin{enumerate}[(1)]
	\item if $\tilde{r} = t+n-L+1$ and $n=T-t$, then we let $\tilde{x}_s^{\tilde{r}} = \Vupper + (s-\tilde{r}) V$ for $s \in [\tilde{r}, \tilde{r}+L-1 ]_{\Z}$ and $\tilde{x}_s^{\tilde{r}} = 0$ otherwise;
	\item otherwise, we let $\tilde{x}_s^{\tilde{r}} = \Vupper + \min \{ s-\tilde{r}, L-1-s+\tilde{r} \}V$ for $s \in [\tilde{r}, \tilde{r}+L-1 ]_{\Z}$ and $\tilde{x}_s^{\tilde{r}} = 0$ otherwise.
	\end{enumerate}
\item For each $\tilde{r} \in [t+n-L+2,t]_{\Z}$ (totally $[L-n-1]^+$ points), we create $(\tilde{x}^{\tilde{r}}, \tilde{y}^{\tilde{r}}, \tilde{u}^{\tilde{r}}) \in$ conv($P$) such that $\tilde{y}_s^{\tilde{r}} = 1$ for $s \in [\tilde{r}, \min \{\tilde{r}+L-1, T\}]_{\Z}$ and $\tilde{y}_s^{\tilde{r}} = 0$ otherwise and  $\tilde{u}_s^{\tilde{r}} = 1$ for $s = \tilde{r}$ and $\tilde{u}_s^{\tilde{r}} = 0$ otherwise.
For the value of $\tilde{x}^{\tilde{r}}$, \sred{we have the following two cases:}
	\begin{enumerate}[(1)]
	\item if $\tilde{r} = T-L+1$, then we let $\tilde{x}_s^{\tilde{r}} = \Vupper + (s-\tilde{r}) V$ for $s \in [\tilde{r}, \tilde{r}+L-1 ]_{\Z}$ and $\tilde{x}_s^{\tilde{r}} = 0$ otherwise;
	\item otherwise, we let $\tilde{x}_s^{\tilde{r}} = \Vupper + \min \{ s-\tilde{r}, L-1-s+\tilde{r} \}V$ for $s \in [\tilde{r}, \tilde{r}+L-1 ]_{\Z}$ and $\tilde{x}_s^{\tilde{r}} = 0$ otherwise.
	\end{enumerate}
\item For each $\tilde{r} \in [t+1, T]_{\Z}$ (totally $T-t$ points), 
we create $(\tilde{x}^{\tilde{r}}, \tilde{y}^{\tilde{r}}, \tilde{u}^{\tilde{r}}) \in$ conv($P$) such that $\tilde{y}_s^{\tilde{r}} = 1$ for $s \in [\tilde{r}, \min \{\tilde{r}+L-1, T\}]_{\Z}$ and $\tilde{y}_s^{\tilde{r}} = 0$ otherwise, $\tilde{u}_s^{\tilde{r}} = 1$ for $s = \tilde{r}$ and $\tilde{u}_s^{\tilde{r}} = 0$ otherwise, and $\tilde{x}_s^{\tilde{r}} = \Clower$ for $s \in [\tilde{r},\min \{\tilde{r}+L-1, T\}]_{\Z}$ and $\tilde{x}_s^{\tilde{r}} = 0$ otherwise.
\item For each $\hat{r} \in [1,t-m-1]_{\Z}$ (totally $t-m-1$ points), 
we create $(\hat{x}^{\hat{r}}, \hat{y}^{\hat{r}}, \hat{u}^{\hat{r}}) \in$ conv($P$) such that
$\hat{y}_s^{\hat{r}} = 1$ for $s \in [1,\hat{r}]_{\Z}$ and $\hat{y}_s^{\hat{r}} = 0$ otherwise, $\hat{u}_s^{\hat{r}} = 0$ for all $s \in [2,T]_{\Z}$, and $\hat{x}_s^{\hat{r}} = \Clower+\epsilon$ for $s \in [1,\hat{r}]_{\Z}$ and $\hat{x}_s^{\hat{r}} = 0$ otherwise.
\item For each $\hat{r} \in [t-m, t-1]_{\Z}$ (totally $m$ points), 
we create $(\hat{x}^{\hat{r}}, \hat{y}^{\hat{r}}, \hat{u}^{\hat{r}}) \in$ conv($P$) such that
$\hat{y}_s^{\hat{r}} = 1$ for $s \in [\hat{r}-L+1,\hat{r}]_{\Z}$ and $\hat{y}_s^{\hat{r}} = 0$ otherwise,  $\hat{u}_s^{\hat{r}} = 1$ for $s = \hat{r}-L+1$ and $\hat{u}_s^{\hat{r}} = 0$ otherwise, and
$\hat{x}_s^{\hat{r}} = \Clower+\epsilon$ for $s \in [\hat{r}-L+1,\hat{r}]_{\Z}$ and $\hat{x}_s^{\hat{r}} = 0$ otherwise.
\item For each $\hat{r} \in [t+1, T]_{\Z}$ (totally $T-t$ points), 
we create $(\hat{x}^{\hat{r}}, \hat{y}^{\hat{r}}, \hat{u}^{\hat{r}}) \in$ conv($P$) such that
$\hat{y}_s^{\hat{r}}=1$ for $s \in [\hat{r}, \min \{\hat{r}+L-1, T\}]_{\Z}$ and $\hat{y}_s^{\hat{r}} = 0$ otherwise, $\hat{u}_s^{\hat{r}}=1$ for $s = \hat{r}$ and $\hat{u}_s^{\hat{r}} = 0$ otherwise,
 and $\hat{x}_s^{\hat{r}} = \Clower + \epsilon$ for $s \in [\hat{r}, \min \{\hat{r}+L-1, T\}]_{\Z}$ and $\hat{x}_s^{\hat{r}} = 0$ otherwise.
\item For each $\bar{r} \in [1,t-m-1]_{\Z}$ (totally $t-m-1$ points), 
we create $(\bar{x}^{\bar{r}}, \bar{y}^{\bar{r}}, \bar{u}^{\bar{r}}) \in$ conv($P$) such that
$\bar{y}_s^{\bar{r}} = 1$ for $s \in [1,\bar{r}]_{\Z}$ and $\bar{y}_s^{\bar{r}} = 0$ otherwise, $\bar{u}_s^{\bar{r}} = 0$ for all $s \in [2,T]_{\Z}$, and $\bar{x}_s^{\bar{r}} = \Clower$ for $s \in [1,\bar{r}]_{\Z}$ and $\bar{x}_s^{\bar{r}} = 0$ otherwise.
\item For $\bar{r} = t-m$ (totally one point),
we create $(\bar{x}^{\bar{r}}, \bar{y}^{\bar{r}}, \bar{u}^{\bar{r}}) \in$ conv($P$) such that
$\bar{y}_s^{\bar{r}} = 1$ for all $s \in [1,T]_{\Z}$, $\bar{u}_s^{\bar{r}} = 0$ for all $s \in [2,T]_{\Z}$, and $\bar{x}_s^{\bar{r}} = \Cupper$ all $s \in [1,T]_{\Z}$. \\
Finally, we create the remaining $T-t+m$ points $(\bar{x}^{\bar{r}}, \bar{y}^{\bar{r}}, \bar{u}^{\bar{r}})$ ($\bar{r} \in [t-m+1, T]_{\Z}$).
Without loss of generality, we let $S = \{n_1, \cdots, n_p, \cdots, n_q, \cdots, n_{|S|}\} \subseteq [t-m+1,t]_{\Z}$ and $S' = S \cup \{t-m\}$. For notation convenience, we define $n_0 = t-m$ and $n_{|S|+1} = t$ if $n_{|S|} \neq t$ and $n_{|S|} = t$ otherwise.
For simplicity we assume $n_{|S|} \neq t$ and $n \geq 1$, as the case $ n_{|S|} = t$ or $n=0$ can be analyzed similarly.
\item For each $n_p \in S'$, $p \in [0,|S|]_{\Z}$, we create $n_{p+1} - n_p - 1$ points $(\bar{x}^{\bar{r}}, \bar{y}^{\bar{r}}, \bar{u}^{\bar{r}}) \in$ conv($P$) with $\bar{r} \in [n_p+1, n_{p+1}-1]_{\Z}$ (totally there are $\sum_{p=0}^{|S|} (n_{p+1} - n_p - 1) = m-|S|-1$ points) such that 
$\bar{y}_s^{\bar{r}} = 1$ for $s \in [n_p-L+1, \bar{r}]_{\Z}$ and $\bar{y}_s^{\bar{r}} = 0$ otherwise, 
$\bar{u}_s^{\bar{r}} = 1$ for $s = n_p-L+1$ and $\bar{u}_s^{\bar{r}} = 0$ otherwise, and
$\bar{x}_s^{\bar{r}} = \Clower$ for $s \in [n_p-L+1, \bar{r}]_{\Z}$ and $\bar{x}_s^{\bar{r}} = 0$ otherwise.
\item For each $n_p \in S$ ($p \in [1, |S|]_{\Z}$), we create one point $(\bar{x}^{\bar{r}}, \bar{y}^{\bar{r}}, \bar{u}^{\bar{r}})$  with $\bar{r} = n_p$ (totally there are $|S|$ points) such that
$\bar{y}_s^{\bar{r}} = 1$ for $s \in [n_{p-1}-L+1, T]_{\Z}$ and $\bar{y}_s^{\bar{r}} = 0$ otherwise, $\bar{u}_s^{\bar{r}} = 1$ for $s = n_{p-1}-L+1$ and $\bar{u}_s^{\bar{r}} = 0$ otherwise,
$\bar{x}_s^{\bar{r}} = \Vupper + (\min\{s,t\}-n_{p-1}+L-1)V$ for $s \in [n_{p-1}-L+1, T]_{\Z}$ and $\bar{x}_s^{\bar{r}} = 0$ otherwise.
\item For $\bar{r} = n_{|S|+1} = t$ (totally one point),
we create $(\bar{x}^{\bar{r}}, \bar{y}^{\bar{r}}, \bar{u}^{\bar{r}}) \in$ conv($P$) such that
$\bar{y}_s^{\bar{r}} = 1$ for $s \in [n_{|S|}-L+1, \bar{r}]_{\Z}$ and $\bar{y}_s^{\bar{r}} = 0$ otherwise, 
$\bar{u}_s^{\bar{r}} = 1$ for $s = n_{|S|}-L+1$ and $\bar{u}_s^{\bar{r}} = 0$ otherwise, and
$\bar{x}_s^{\bar{r}} = \Clower$ for $s \in [n_{|S|}-L+1, \bar{r}]_{\Z} \setminus \{t\}$ and  $\bar{x}_s^{\bar{r}} = \Vupper$ for $s = t$ and $\bar{x}_s^{\bar{r}} = 0$ otherwise.
\item For each $\bar{r} \in [t+1, t+n-1]_{\Z}$ (totally $n-1$ points),
we create $(\bar{x}^{\bar{r}}, \bar{y}^{\bar{r}}, \bar{u}^{\bar{r}}) \in$ conv($P$) such that
\sred{(i)} $\bar{y}_s^{\bar{r}} = 1$ for $s \in [t-L+1, \bar{r}]_{\Z}$ and $\bar{y}_s^{\bar{r}} = 0$ otherwise, \sred{(ii)} $\bar{u}_s^{\bar{r}} = 1$ for $s = t-L+1$ and $\bar{u}_s^{\bar{r}} = 0$ otherwise,
and \sred{(iii)} $\bar{x}_s^{\bar{r}} = \Vupper + \min \{s-t+L-1, n-1\} V$ for $s \in [t-L+1, t]_{\Z}$, $\bar{x}_s^{\bar{r}} = \Vupper + (n-1-s) V$ for $s \in [t+1, \bar{r}]_{\Z}$, and $\bar{x}_s^{\bar{r}} = 0$ otherwise.
\item For $\bar{r} = t+n$ (totally one point),
we create $(\bar{x}^{\bar{r}}, \bar{y}^{\bar{r}}, \bar{u}^{\bar{r}}) \in$ conv($P$) such that
$\bar{y}_s^{\bar{r}} = 1$ for $s \in [n_{|S|}-L+1, \min \{2t-n_{|S|}+L-1, T\}]_{\Z}$ and $\bar{y}_s^{\bar{r}} = 0$ otherwise and $\bar{u}_s^{\bar{r}} = 1$ for $s = n_{|S|}-L+1$ and $\bar{u}_s^{\bar{r}} = 0$ otherwise.
For the value of $\bar{x}^{\bar{r}}$, \sred{we have the following two cases:}
	\begin{enumerate}[(1)]
	\item if $2t-n_{|S|}+L-1 \leq T-1$, then we let $\bar{x}_s^{\bar{r}} = \Vupper + (\min \{s, 2t-s\}-n_{|S|}+L-1)V$ for $s \in [n_{|S|}-L+1, \min \{2t-n_{|S|}+L-1, T\}]_{\Z}$ and $\bar{x}_s^{\bar{r}} = 0$ otherwise;
	\item otherwise, we let $\bar{x}_s^{\bar{r}} = \Vupper + (\min \{s, t\}-n_{|S|}+L-1)V$ for $s \in [n_{|S|}-L+1, \min \{2t-n_{|S|}+L-1, T\}]_{\Z}$ and $\bar{x}_s^{\bar{r}} = 0$ otherwise.
	\end{enumerate}
\item For each $\bar{r} \in [t+n+1, T]_{\Z}$ (totally $T-t-n$ points),
we create $(\bar{x}^{\bar{r}}, \bar{y}^{\bar{r}}, \bar{u}^{\bar{r}}) \in$ conv($P$) such that
\sred{(i)} $\bar{y}_s^{\bar{r}} = 1$ for $s \in [t+n-L+1, \bar{r}]_{\Z}$ and $\bar{y}_s^{\bar{r}} = 0$ otherwise, \sred{(ii)} $\bar{u}_s^{\bar{r}} = 1$ for $s = t+n-L+1$ and $\bar{u}_s^{\bar{r}} = 0$ otherwise, and \sred{(iii)} $\bar{x}_s^{\bar{r}} = \Vupper + (s-t-n+L-1)V$ for $s \in [t+n-L+1, t]_{\Z}$, $\bar{x}_s^{\bar{r}} = \Vupper + \min \{L-1-n, \bar{r}-s \}V$ for $s \in [t+1, \bar{r}]_{\Z}$, and $\bar{x}_s^{\bar{r}} = 0$ otherwise.
\end{enumerate}

In summary, we create $3T-2$ points $(\bar{x}^{\bar{r}}, \bar{y}^{\bar{r}}, \bar{u}^{\bar{r}})_{\bar{r}=1}^{T}$, $(\hat{x}^{\hat{r}}, \hat{y}^{\hat{r}}, \hat{u}^{\hat{r}})_{\hat{r}=1, \hat{r} \neq t}^{T}$, and $(\tilde{x}^{\tilde{r}}, \tilde{y}^{\tilde{r}}, \tilde{u}^{\tilde{r}})_{\tilde{r}=2}^{T}$. It is easy to see that they are valid and satisfy \eqref{eqn:x_t-up-exp} at equality. Meanwhile, based on the construction, they are clearly linearly independent since $(\bar{x}, \bar{y}, \bar{u})$ and $(\hat{x}, \hat{y}, \hat{u})$ can construct a lower-triangular matrix in terms of the values of $x$ and $y$ and $(\tilde{x}, \tilde{y}, \tilde{u})$ is further linearly independent with them since it constructs an upper-triangular matrix in terms of the value of $u$.
\end{proof}

\subsection{Proof for Proposition \ref{prop:x_t-down-exp}} \label{apx:subsec:x_t-down-exp}
\begin{proof}
\textbf{(Validity)} It is clear that inequality \eqref{eqn:x_t-down-exp} is valid when $y_t = 0$ due to constraints \eqref{eqn:p-minup}. Also, for the case in which $t = 1$ when $y_t = 1$, it is easy to show \eqref{eqn:x_t-down-exp} is valid due to constraints \eqref{eqn:p-minup} and \eqref{eqn:p-ramp-down}. In the following, we continue to prove the validity by discussing the cases in which $y_t = 1$ with $t \in [2,T-1]_{\Z}$. We consider the last start-up time
{(denoted as $t-s$, $s \geq 0$)}
before $t$ in the following two possible cases.
\begin{enumerate}[1)]
\item $s \geq L-1$. It follows that $\phi = 0$ and $t-s+L-1 \leq t$. We further discuss the following two possible cases in terms of the first shut-down time 
{(denoted as $t+\hat{s}$, $\hat{s} \geq 1$)}
after $t$. We observe that $y_i - \sum_{j=0}^{\min \{L-1, i-2\}} u_{i-j} = 1$ for each $i \in [t,t+\min\{m+1,\hat{s}-1\}]_{\Z}$.
\begin{enumerate}[(1)]
\item $\hat{s} \in [1, \hat{t}]_{\Z}$. In this case, we have $x_t \leq \Vupper + \min \{s, \hat{s}-1\} V \leq \Vupper + (\hat{s}-1)V = \Vupper y_t + V \sum_{i =1 }^{\hat{s}-1} (y_i - \sum_{j=0}^{\min \{L-1,i-2\}} u_{i-j})$, which is clearly less than the RHS of \eqref{eqn:x_t-down-exp}.
\item $\hat{s} \in [\hat{t}+1, m+1]_{\Z}$. In this case, we have  $x_t \leq \Vupper + \min \{s, \hat{s}-1\} V \leq \Vupper + (\hat{s}-1)V = \Vupper + (\hat{t}-1)V + (\hat{s}-\hat{t})V \leq \Vupper y_t + V \sum_{i \in S_0} (y_i - \sum_{j=0}^{\min \{L-1,i-2\}} u_{i-j}) + V \sum_{i \in (S \cap [\hat{t}+1, t+\hat{s}-1]_{\Z}) \cup \{\hat{t}\}} (d_i - i) (y_i - \sum_{j=0}^{L-1} u_{i-j})$, which is clearly less than the RHS of \eqref{eqn:x_t-down-exp}.
\item $\hat{s} \geq m+2$. Inequality \eqref{eqn:x_t-down-exp} converts to $x_t \leq \Cupper$, which is clearly valid due to constraints \eqref{eqn:p-upper-bound}.
\end{enumerate}
\item $s \in [0,L-2]_{\Z}$. It follows that $t-s+L-1 \geq t+1$ and $y_i - \sum_{j=0}^{\min \{L-1, i-2\}}u_{i-j} = 0$ for all $i \in [t+1, t-s+L-1]_{\Z}$. We let $t+\hat{s}$ ($\hat{s} \geq 1$) be the first shut-down time after $t$ and further discuss the following three possible cases in terms of the value of $t-s+L-1$.
\begin{enumerate}[(1)]
\item $t-s+L-1 \in [t+1, \hat{t}-1]_{\Z}$. We further discuss the following three possible cases in terms of the value $t+\hat{s}$.
\begin{enumerate}[(a)]
\item $t+\hat{s} \geq t+m+2$. In this case, we have $x_t \leq \min \{\Cupper, \Vupper + \min \{s, \hat{s}-1\} V\}$. From inequality \eqref{eqn:x_t-down-exp}, we have $x_t \leq \Vupper + (\hat{t}-1-(t-s+L-1))V + V \sum_{i \in S \cup \{\hat{t}\}} (d_i - i) + (\Cupper - \Vupper - mV) + \phi = \Vupper + (\hat{t}-t+s-L)V+(t+m+1-\hat{t})V + (\Cupper - \Vupper - mV) + \phi$ since $V \sum_{i \in S \cup \{\hat{t}\}} (d_i - i) = t+m+1-\hat{t}$. We only need to show 
\begin{equation}
\min \{\Cupper, \Vupper + \min \{s, \hat{s}-1\} V\} \leq \Vupper + (s-L+m+1)V + (\Cupper - \Vupper - mV) + \phi. \label{eq:inter-validlity-1down-2}
\end{equation}
If $s \leq t+L-T-1$, then \eqref{eq:inter-validlity-1down-2} holds since $\phi = sV$. If $s \geq t+L-T$, we have $\phi = \min \{L-1-s, s\}V$. If $s \leq L-1-s$, we further have $\phi = sV$ and then \eqref{eq:inter-validlity-1down-2} holds; otherwise, $s \geq L-s$, we have $\phi = L-1-s$ and the RHS of \eqref{eq:inter-validlity-1down-2} becomes $\Cupper$. It follows that \eqref{eq:inter-validlity-1down-2} holds.
\item $t+\hat{s} \in [\hat{t}+1, t+m+1]_{\Z}$. In this case, inequality \eqref{eqn:x_t-down-exp} converts to $x_t \leq \Vupper + (\hat{t}-1 - (t-s+L-1))V + \tilde{f} + \phi = \Vupper + (\hat{t}-t+s-L)V+\tilde{f}+\phi$, where $\tilde{f} = V \sum_{i \in (S \cup \{\hat{t}) \cap [\hat{t}, t+\hat{s}-1]_{\Z} \}} (d_i - i)$. It is easy to observe that $\tilde{f} \geq t+\hat{s}-\hat{t}$. Now we only need to show 
\begin{equation}
\Vupper + \min \{s, \hat{s}-1\} V \leq \Vupper + (\hat{t}-t+s-L)V+\tilde{f}+\phi. \label{eq:inter-validlity-1down-3}
\end{equation}
If $s \leq t+L-T-1$ or $s \in [t+L-T,L-1-s]_{\Z}$, we have $\phi = sV$, indicating \eqref{eq:inter-validlity-1down-3} holds; otherwise, $s \geq L-s$ and then we have $\phi = L-1-s$. It follows that the RHS of \eqref{eq:inter-validlity-1down-3} becomes $\Vupper + (\hat{t}-t-1)V+\tilde{f} \geq \Vupper+(\hat{s}-1)V$, indicating \eqref{eq:inter-validlity-1down-3} holds.
\item $t+\hat{s} \in [t+1, \hat{t}]_{\Z}$. It follows that $\hat{s}-1 \leq \hat{t}-t-1$. In this case, we have $x_t \leq \Vupper + \min \{s, \hat{s}-1\} V \leq \Vupper + sV = \Vupper + \phi \leq \mbox{{the} RHS of}$ \eqref{eqn:x_t-down-exp}  if $s \leq t+L-T-1$ or $s \in [t+L-T,L-1-s]_{\Z}$ following the argument above. Hence, we only need to consider the case in which $s \geq L-s$ and therefore $\phi = L-1-s$. In this case, inequality \eqref{eqn:x_t-down-exp} converts to $x_t \leq \Vupper + (\hat{t}-1-(t-s+L-1))V + (L-1-s)V = \Vupper + (\hat{t}-t-1)V \geq \Vupper + (\hat{s}-1)V \geq \Vupper + \min \{s, \hat{s}-1\} V$, indicating that \eqref{eqn:x_t-down-exp} is valid.
\end{enumerate}
\item $t-s+L-1 \in [\hat{t}, t+m]_{\Z}$. Similar to the argument above, we only need to show
\begin{equation}
\Vupper + \min \{s, \hat{s}-1\} V \leq \Vupper + \tilde{f}+\phi + (\Cupper-\Vupper-mV)y_{t+m+1}, \label{eq:inter-validlity-1down-4}
\end{equation}
where $\tilde{f} = V \sum_{i \in (S \cup \{\hat{t}\}) \cap [t-s+L, t+\hat{s}-1]_{\Z} } (d_i - i)$. If $s \leq t+L-T-1$ or $s \in [t+L-T,L-1-s]_{\Z}$, we have $\phi = sV$ and therefore \eqref{eq:inter-validlity-1down-4} holds since $\tilde{f} \geq 0$ and $(\Cupper-\Vupper-mV)y_{t+m+1} \geq 0$. If $s \geq L-s$, i.e., $L-2 \geq s \geq L/2$ then $\phi = (L-1-s)V$. Next, in the following we assume $s \geq L-s$ and try to obtain the contradiction, which indicates that $s \geq L-s$ is not possible to happen.
\begin{enumerate}[(a)]
\item If $t \geq L$, then $\min \{t-2, L-2\} = L-2 \geq L/2$ and therefore $\hat{t} = t + (L-2)$ by the definition of $\hat{t}$. It follows that $t-s+L-1 \leq t+L/2-1 \leq t+(L-2)-1 = \hat{t}-1$ (the first and second inequalities follow since $L-2 \leq s \leq L/2$), which contradicts to the condition  $t-s+L-1 \geq \hat{t}$.
\item If $t \leq L-1$, then $t -s +L-1 \leq 2(L-1)-s \leq -2$ since $s \geq L/2$, which contradicts to the condition $t-s+L-1 \geq \hat{t}$.
\end{enumerate} 
\item $t-s+L-1 \geq t+m+1$. Hence $t+\hat{s} \geq t-s+L$. It follows that for all $i \in S_0 \cup S \cup \{\hat{t},t+m+1\}$, $y_i - \sum_{j=0}^{\min \{L-1, i-2\}}u_{i-j} = 0$ and thus inequality \eqref{eqn:x_t-down-exp} converts to $x_t \leq \Vupper + \phi$. In this case, $x_t \leq \Vupper + \min \{s, \hat{s}-1\}V$ and therefore we only need to show 
\begin{equation}
\min \{s, \hat{s}-1\}V \leq \phi. \label{eq:inter-validlity-1down-1}
\end{equation}
If $s \leq t+L-T-1$, then \eqref{eq:inter-validlity-1down-1} holds since $\phi = sV$. If $s \geq t+L-T$, we have $\phi = \min \{L-1-s, s\}V$. If $s \leq L-1-s$, we further have $\phi = sV$ and then \eqref{eq:inter-validlity-1down-1} holds; otherwise, $s \geq L-s$, i.e., $L-2 \geq s \geq L/2$. Next, in the following we assume $s \geq L-s$ and try to obtain the contradiction, which indicates that $s \geq L-s$ is not possible to happen.
\begin{enumerate}[(a)]
\item If $t \geq L$, then $\min \{t-2, L-2\} = L-2 \geq L/2$ and therefore $\hat{t} = t + (L-2)$ by the definition of $\hat{t}$. It follows that $t-s+L-1 \leq t+L/2-1 \leq t+(L-2)-1 = \hat{t}-1 \leq t+m$, which contradicts to the condition  $t-s+L-1 \geq t+m+1$.
\item If $t \leq L-1$, then $t -s +L-1 \leq 2(L-1)-s \leq -2$ since $s \geq L/2$, which contradicts to the condition  $t-s+L-1 \geq t+m+1$.
\end{enumerate}
\end{enumerate}
\end{enumerate}

\textbf{(Facet-defining)} 
We generate $3T-2$ linearly independent points (i.e., $(\tilde{x}^{\tilde{r}}, \tilde{y}^{\tilde{r}}, \tilde{u}^{\tilde{r}})_{\tilde{r}=2}^{T}$,  $(\hat{x}^{\hat{r}}, \allowbreak \hat{y}^{\hat{r}}, \allowbreak \hat{u}^{\hat{r}})_{\hat{r}=1, \hat{r} \neq t}^{T}$, and $(\bar{x}^{\bar{r}}, \bar{y}^{\bar{r}}, \bar{u}^{\bar{r}})_{\bar{r}=1}^{T}$) in conv($P$) that satisfy \eqref{eqn:x_t-down-exp} at equality.

\begin{enumerate}[1)]
\item For each $\tilde{r} \in [2, T]_{\Z}$ (totally $T-1$ points), we create $(\tilde{x}^{\tilde{r}}, \tilde{y}^{\tilde{r}}, \tilde{u}^{\tilde{r}}) \in$ conv($P$) such that
$\tilde{y}_s^{\tilde{r}} = 1$ for $s \in [\tilde{r}, \min \{\tilde{r}+L-1, T\}]_{\Z}$ and $\tilde{y}_s^{\tilde{r}} = 0$ otherwise and $\tilde{u}_s^{\tilde{r}} = 1$ for $s = \tilde{r}$ and $\tilde{u}_s^{\tilde{r}} = 0$ otherwise.
For the value of $\tilde{x}^{\tilde{r}}$, \sred{we have the following three cases:}
	\begin{enumerate}[(1)]
	\item if $t \notin [\tilde{r}, \min \{\tilde{r}+L-1, T\}]_{\Z}$, then we let $\tilde{x}_s^{\tilde{r}} = \Clower$ for $s \in  [\tilde{r}, \min \{\tilde{r}+L-1, T\}]_{\Z}$ and $\tilde{x}_s^{\tilde{r}} = 0$ otherwise;
	\item if $t \in [\tilde{r}, \min \{\tilde{r}+L-1, T\}]_{\Z}$ and $\tilde{r}+L-1 \leq T-1$, then we let $\tilde{x}_s^{\tilde{r}} = \Vupper + \min \{s-\tilde{r}, L-1-s+\tilde{r}\}V$ for $s \in  [\tilde{r}, \min \{\tilde{r}+L-1, T\}]_{\Z}$ and $\tilde{x}_s^{\tilde{r}} = 0$ otherwise;
	\item otherwise, we let $\tilde{x}_s^{\tilde{r}} = \Vupper + (\min \{s,t\}-\tilde{r})V$ for $s \in  [\tilde{r}, \min \{\tilde{r}+L-1, T\}]_{\Z}$ and $\tilde{x}_s^{\tilde{r}} = 0$ otherwise.
	\end{enumerate}
\item For each $\hat{r} \in [1,t-1]_{\Z}$ (totally $t-1$ points), 
we create $(\hat{x}^{\hat{r}}, \hat{y}^{\hat{r}}, \hat{u}^{\hat{r}}) \in$ conv($P$) such that
$\hat{y}_s^{\hat{r}} = 1$ for $s \in [1,\hat{r}]_{\Z}$ and $\hat{y}_s^{\hat{r}} = 0$ otherwise, $\hat{u}_s^{\hat{r}} = 0$ for all $s \in [2,T]_{\Z}$, and $\hat{x}_s^{\hat{r}} = \Clower+\epsilon$ for $s \in [1,\hat{r}]_{\Z}$ and $\hat{x}_s^{\hat{r}} = 0$ otherwise.
\item For each $\hat{r} \in [t+1, T]_{\Z}$ (totally $T-t$ points), 
we create $(\hat{x}^{\hat{r}}, \hat{y}^{\hat{r}}, \hat{u}^{\hat{r}}) \in$ conv($P$) such that
$\hat{y}_s^{\hat{r}} = 1$ for $s \in [\hat{r}, \min \{\hat{r}+L-1, T\}]_{\Z}$ and $\hat{y}_s^{\hat{r}} = 0$ otherwise, $\hat{u}_s^{\hat{r}} = 1$ for $s = \hat{r}$ and $\hat{u}_s^{\hat{r}} = 0$ otherwise, and $\hat{x}_s^{\hat{r}} = \Clower+\epsilon$ for $s \in [\hat{r}, \min \{\hat{r}+L-1, T\}]_{\Z}$ and $\hat{x}_s^{\hat{r}} = 0$ otherwise.
\item For each $\bar{r} \in [1,t-1]_{\Z}$ (totally $t-1$ points), 
we create $(\bar{x}^{\bar{r}}, \bar{y}^{\bar{r}}, \bar{u}^{\bar{r}}) \in$ conv($P$) such that
$\bar{y}_s^{\bar{r}} = 1$ for $s \in [1,\bar{r}]_{\Z}$ and $\bar{y}_s^{\bar{r}} = 0$ otherwise, $\bar{u}_s^{\bar{r}} = 0$ for all $s \in [2,T]_{\Z}$, and $\bar{x}_s^{\bar{r}} = \Clower$ for $s \in [1,\bar{r}]_{\Z}$ and $\bar{x}_s^{\bar{r}} = 0$ otherwise.
\item For each $\bar{r} \in [t,\hat{t}-1]_{\Z}$ (totally $\hat{t}-t$ points), 
we create $(\bar{x}^{\bar{r}}, \bar{y}^{\bar{r}}, \bar{u}^{\bar{r}}) \in$ conv($P$) such that
$\bar{y}_s^{\bar{r}} = 1$ for $s \in [1, \bar{r}]_{\Z}$ and $\bar{y}_s^{\bar{r}} = 0$ otherwise, $\bar{u}_s^{\bar{r}} = 0$ for all $s \in [2,T]_{\Z}$, and $\bar{x}_s^{\bar{r}} = \Vupper+(\bar{r}-\max\{t,s\})V$ for $s \in [1, \bar{r}]_{\Z}$ and $\bar{x}_s^{\bar{r}} = 0$ otherwise. \\
In the following, we assume $S = \{n_1, \cdots, n_p, \cdots, n_q, \cdots, n_{|S|}\}$ without loss of generality and let $S' = S \cup \{\hat{t}\}$. For simplicity, we define $n_0 = \hat{t}$ and $n_{|S|+1}=t+m+1$.
\item For each $\bar{r} = n_p \in S'$ with $p \in [0, |S|]_{\Z}$ (totally $|S|+1$ points), 
we create $(\bar{x}^{\bar{r}}, \bar{y}^{\bar{r}}, \bar{u}^{\bar{r}}) \in$ conv($P$) such that
$\bar{y}_s^{n_p} = 1$ for $s \in [1, n_{p+1}-1]_{\Z}$ and $\bar{y}_s^{n_p} = 0$ otherwise, $\bar{u}_s^{n_p} = 0$ for all $s \in [2,T]_{\Z}$, and $\bar{x}_s^{\bar{r}} = \Vupper+(n_{p+1}-1-\max\{s,t\})V$ for $s \in [1, n_{p+1}-1]_{\Z}$ and $\bar{x}_s^{n_p} = 0$ otherwise.
\item For $\bar{r} = t+m+1$ (totally one point), 
we create $(\bar{x}^{\bar{r}}, \bar{y}^{\bar{r}}, \bar{u}^{\bar{r}}) \in$ conv($P$) such that
$\bar{y}_s^{\bar{r}}=1$ for all $s \in [1,T]_{\Z}$, $\bar{u}_s^{\bar{r}}=0$ for all $s \in [2,T]_{\Z}$, and $\bar{x}_s^{\bar{r}}=\Cupper$ for all $s \in [1,T]_{\Z}$.
\item For each $n_p \in S'$, $p \in [0,|S|]_{\Z}$, we create $n_{p+1}-n_p - 1$ points $(\bar{x}^{\bar{r}}, \bar{y}^{\bar{r}}, \bar{u}^{\bar{r}}) \in$ conv($P$) with $\bar{r} \in [n_p+1, n_{p+1}-1]_{\Z}$ (totally there are $\sum_{p=0}^{|S|} (n_{p+1} - n_p - 1) = t+m-\hat{t}-|S|$ points) such that 
$\bar{y}_s^{\bar{r}} = 1$ for $s \in [n_p-L+1, \bar{r}]_{\Z}$ and $\bar{y}_s^{\bar{r}} = 0$ otherwise and $\bar{u}_s^{\bar{r}} = 1$ for $s = n_p-L+1$ and $\bar{u}_s^{\bar{r}} = 0$ otherwise.
For the value of $\bar{x}^{\bar{r}}$, \sred{we have the following two cases:}
	\begin{enumerate}[(1)]
	\item if $t \in [n_p-L+1, \bar{r}]_{\Z}$, then we let $\bar{x}_s^{\bar{r}} = \Vupper + \min \{\bar{r}-s, s-n_p+L-1\}V$ for $s \in [n_p-L+1, \bar{r}]_{\Z}$ and $\bar{y}_s^{\bar{r}} = 0$ otherwise;
	\item otherwise, we let $\bar{x}_s^{\bar{r}} = \Clower$ for all $s \in [n_p-L+1, \bar{r}]_{\Z}$.
	\end{enumerate}
\item For each $\bar{r} \in [t+m+2, T]_{\Z}$ (totally $T-t-m-1$ points), 
we create $(\bar{x}^{\bar{r}}, \bar{y}^{\bar{r}}, \bar{u}^{\bar{r}}) \in$ conv($P$) such that
$\bar{y}_s^{\bar{r}} = 1$ for $s \in [t+m-L+2, \bar{r}]_{\Z}$ and $\bar{y}_s^{\bar{r}} = 0$ otherwise and $\bar{u}_s^{\bar{r}} = 1$ for $s = t+m-L+2$ accordingly and $\bar{u}_s^{\bar{r}} = 0$ otherwise.
For the value of $\bar{x}^{\bar{r}}$, \sred{we have the following three cases:}
	\begin{enumerate}[(1)]
	\item if $t \notin [t+m-L+2, \bar{r}]_{\Z}$, we let 
$\bar{x}_s^{\bar{r}} = \Clower$ for $s \in [t+m-L+2, \bar{r}]_{\Z}$ and $\bar{x}_s^{\bar{r}} = 0$ otherwise;
	\item if $t \in [t+m-L+2, \bar{r}]_{\Z}$ and $\bar{r} \leq T-1$, we let 
$\bar{x}_s^{\bar{r}} = \Vupper + \min \{m+1, s-t-m+L-2, \bar{r}-s\}V$ for $s \in [t+m-L+2, \bar{r}]_{\Z}$ and $\bar{x}_s^{\bar{r}} = 0$ otherwise;
	\item otherwise, we let $\bar{x}_s^{\bar{r}} = \Vupper + (\min \{s,t\}-t-m+L-2)V$ for $s \in [t+m-L+2, \bar{r}]_{\Z}$ and $\bar{x}_s^{\bar{r}} = 0$ otherwise.
	\end{enumerate}
\end{enumerate}

In summary, we create $3T-2$ points $(\bar{x}^{\bar{r}}, \bar{y}^{\bar{r}}, \bar{u}^{\bar{r}})_{\bar{r}=1}^{T}$, $(\hat{x}^{\hat{r}}, \hat{y}^{\hat{r}}, \hat{u}^{\hat{r}})_{\hat{r}=1, \hat{r} \neq t}^{T}$, and $(\tilde{x}^{\tilde{r}}, \tilde{y}^{\tilde{r}}, \tilde{u}^{\tilde{r}})_{\tilde{r}=2}^{T}$. It is easy to see that they are valid and satisfy \eqref{eqn:x_t-down-exp} at equality.
Meanwhile, based on the construction, they are clearly linearly independent since $(\bar{x}, \bar{y}, \bar{u})$ and $(\hat{x}, \hat{y}, \hat{u})$ can construct a lower-triangular matrix in terms of the values of $x$ and $y$ and $(\tilde{x}, \tilde{y}, \tilde{u})$ is further linearly independent with them since it constructs an upper-triangular matrix in terms of the value of $u$.
\end{proof}

\subsection{Proof for Proposition \ref{prop:ru-2-exp}} \label{apx:subsec:ru-2-exp}
\begin{proof}
\textbf{(Validity)} We first prove that inequality \eqref{eqn:ru-2-exp-1} is valid. It is clear that inequality \eqref{eqn:ru-2-exp-1} is valid when $y_{t-m} = y_t = 0$ and $y_{t-m} = 1, \ y_t = 0$ due to constraints  \eqref{eqn:p-minup} and \eqref{eqn:p-lower-bound}. We continue to discuss the remaining two cases in terms of the values of $y_{t-m}$ and $y_t$. Meanwhile, we let $f = \sum_{i \in S \setminus \{t-m+L\} } (i - d_i)$.
\begin{enumerate}[1)]
\item $y_{t-m} = 0, \ y_t = 1$. There is at least a start-up between $t-m$ and $t$ and without loss of generality we consider there is only one start-up between $t-m$ and $t$. We let this start-up time be $t-s$ ($s \in [0,m-1]_{\Z}$). Meanwhile, we only consider the case in which $n \geq 1$ since the case in which $n=0$ can be proved similarly, and let the first shut-down time after $t$ be $t+\hat{s}$ ($\hat{s} \geq 1$). We further discuss the following three possible cases in terms of the value of $t-s+L-1$.
\begin{enumerate}[(1)]
\item $t-s+L-1 \leq t$, i.e., $s \geq L-1$. It follows that $\phi = 0$ and $y_i - \sum_{j=0}^{L-1}u_{i-j} = 1$ for all $i \in [t+1, t+\hat{s}-1]_{\Z}$. Therefore, in this case, $x_{t} \leq \Vupper + \min \{s, \hat{s}-1\}V \leq \min \{\Vupper + (\hat{s}-1)V, \Clower + mV\} = \mbox{{the} RHS of}$ \eqref{eqn:ru-2-exp-1}, indicating \eqref{eqn:ru-2-exp-1} is valid. Note that the second inequality holds because $\Vupper + sV \leq \Vupper + (m-1)V \leq \Clower + mV$.
\item $t-s+L-1 \in [t+1, t+n-1]_{\Z}$. From inequality \eqref{eqn:ru-2-exp-1}, we have $x_t \leq \Vupper + (t+{[n-1]^+} - (t-s+L-1))V + (\Clower + (m-{[n-1]^+})V - \Vupper) + \phi$. We only need to show
\begin{equation}
\Vupper + \min \{s, \hat{s}-1\}V \leq \Vupper + (t+{[n-1]^+} - (t-s+L-1))V + (\Clower + (m-{[n-1]^+})V - \Vupper) + \phi. \label{eq:inter-validlity-2up-1}
\end{equation}
If $s \leq t+L-T-1$ or $s \in [t+L-T,L-1-s]_{\Z}$, we have $\phi = sV$ and therefore \eqref{eq:inter-validlity-2up-1} holds since $t+{[n-1]^+} - (t-s+L-1) \geq 0$ and $\Clower + (m-{[n-1]^+})V - \Vupper \geq 0$; otherwise, $s \geq L-s$, we have $\phi = (L-1-s)V$ and thus {the} RHS of \eqref{eq:inter-validlity-2up-1} becomes $\Clower+mV$. It follows that \eqref{eq:inter-validlity-2up-1} holds since $\Vupper + sV \leq \Vupper + (m-1)V \leq \Clower + mV$.
\item $t-s+L-1 \geq t+n$. It follows that $t-s \geq t+n-L+1 \geq t-L+2$ and $y_i - \sum_{j=0}^{L-1}u_{i-j} = 0$ for all $i \in [t+1, t+n]_{\Z}$ and inequality \eqref{eqn:ru-2-exp-1} converts to $x_t \leq \Vupper + \phi$. We only need to show
\begin{equation}
\Vupper + \min \{s, \hat{s}-1\}V \leq \Vupper + \phi. \label{eq:inter-validlity-2up-2}
\end{equation}
If $s \leq t+L-T-1$ or $s \in [t+L-T,L-1-s]_{\Z}$, we have $\phi = sV$ and therefore \eqref{eq:inter-validlity-2up-2} holds; otherwise, $s \geq L-s$, i.e., $s \geq L/2$, we have $\phi = (L-1-s)V$ and $\max \{t-L+2, t-m+1\} \leq t-s \leq t-L/2$, i.e., $\min \{m-1, L-2\} \geq L/2$. On the other side, by condition (i) in Proposition \ref{prop:ru-2-exp}, we have $t-s+L-1 \leq t+L/2-1 \leq t+\min \{m-1, L-2\}-1 \leq t+n-1$, which contradicts to the condition $t-s+L-1 \geq t+n$. It indicates that $s \geq L-s$ is not possible to happen and \eqref{eq:inter-validlity-2up-2} holds.
\end{enumerate}
\item $y_{t-m} = y_t = 1$. If there is a shut-down and thereafter a start-up between $t-m+1$ and $t-1$, the discussion is the same as the case discussed above since $x_{t-m} \geq \Clower y_{t-m}$; otherwise, we consider the case in which the {machine} keeps online throughout $t-m$ to $t$. It follows that $\phi = 0$. We continue to discuss the following two possible cases in terms of the first shut-down time {(denoted as $t+\hat{s}$, $\hat{s} \geq 1$)} after time $t$.
\begin{enumerate}[(1)]
\item $t+\hat{s} \geq t+n+1$. In this case, we have $x_t - x_{t-m} \leq \min \{mV, \Vupper + (\hat{s}-1)V -\Clower\} \leq mV = \mbox{{the} RHS of}$ \eqref{eqn:ru-2-exp-1}.
\item $t+\hat{s} \in [t+1, t+n]_{\Z}$. In this case, we have $x_t - x_{t-m} \leq \min \{mV, \Vupper + (\hat{s}-1)V -\Clower\} \leq \Vupper + (\hat{s}-1)V -\Clower = \mbox{{the} RHS of}$ \eqref{eqn:ru-2-exp-1}.
\end{enumerate}
\end{enumerate}

Next, we prove that inequality \eqref{eqn:ru-2-exp-2} is valid. It is clear that inequality \eqref{eqn:ru-2-exp-2} is valid when $y_{t-m} = y_t = 0$ and $y_{t-m} = 1, \ y_t = 0$ due to constraints  \eqref{eqn:p-minup} and \eqref{eqn:p-lower-bound}. We continue to discuss the remaining two cases in terms of the values of $y_{t-m}$ and $y_t$.
\begin{enumerate}[1)]
\item $y_{t-m} = 0, \ y_t = 1$.  There is at least a start-up between $t-m$ and $t$. We let the last start-up time before $t$ be $t-s$ ($s \in [0,m-1]_{\Z}$). Meanwhile, we consider the case in which $n \geq 1$ since the case in which $n=0$ indicates $t = T$ and clearly \eqref{eqn:ru-2-exp-2} is valid, and let the first shut-down time after $t$ be $t+\hat{s}$ ($\hat{s} \geq 1$). We further discuss the following three possible cases in terms of the value of $t-s+L-1$.
\begin{enumerate}[(1)]
\item $t-s+L-1 \leq t$, i.e., $s \geq L-1$. It follows that $\phi = 0$, $y_i - \sum_{j=0}^{L-1}u_{i-j} = 0$ for all $i \in [t-s, t-s+L-1]_{\Z}$,  and $y_i - \sum_{j=0}^{L-1}u_{i-j} = 1$ for all $i \in [t-s+L, t+\hat{s}-1]_{\Z}$. We further discuss the following two possible cases in terms of the value of $t+\hat{s}$.
\begin{enumerate}[(a)]
\item $t+\hat{s} \geq t+n+1$, i.e., $\hat{s} \geq n+1$. In this case, we have $x_t \leq \Vupper + \min \{s, \hat{s}-1\}V$. From inequality \eqref{eqn:ru-2-exp-2}, we have $x_t \leq \Vupper + \tilde{f}V + (m-1-f)V + \psi $, where 
$\tilde{f} = \sum_{i \in (S \setminus \{t-m+L\}) \cap [t-s+L,t]_{\Z} } (i - d_i)$ and $\psi = (\Clower + V-\Vupper) (y_q - \sum_{j=0}^{\min \{L-1, q-t+m-1\}} u_{q-j}) \geq 0$. We only need to show that 
\begin{equation}
\Vupper + \min \{s, \hat{s}-1\}V \leq \Vupper + \tilde{f}V + (m-1-f)V + \psi. \label{eq:inter-validlity-2up-3}
\end{equation}
We let $t-q = \max \{a \in S\}$ and $t-p = \max \{a \in S, a \leq t-s+L-1\}$. Note that $t-q$ exists since $S \neq \emptyset$ and therefore $f = (t-q) - (t-m+L) = m-L-q$.
\begin{itemize}
\item If $t-q \leq t-s+L-1$, i.e., $s \leq q+L-1$, then we have $\tilde{f} = 0$. It follows that  \eqref{eq:inter-validlity-2up-3} holds since $\min \{s, \hat{s}-1\} \leq s \leq q+L-1 = \tilde{f} + (m-1-f)$.
\item If $t-q \geq t-s+L$ and $t-p$ does not exist, then we have $\tilde{f} = f = m-q$. It follows that  \eqref{eq:inter-validlity-2up-3} holds since $\min \{s, \hat{s}-1\} \leq s \leq m-1 = \tilde{f} + (m-1-f)$.
\item If $t-q \geq t-s+L$ and $t-p$ exists, then we have $\tilde{f} = p-q$ and meanwhile $s \leq p+L-1$ since $t-p \leq t-s+L-1$. It follows that  \eqref{eq:inter-validlity-2up-3} holds since $\min \{s, \hat{s}-1\} \leq s \leq p+L-1 = \tilde{f} + (m-1-f)$.
\end{itemize}
\item $t+\hat{s} \in [t+1, t+n]_{\Z}$. We only need to show $\Vupper + \min \{s, \hat{s}-1\}V \leq \Vupper + \tilde{f}V + (\hat{s}-1)V + \psi$, where $\tilde{f}$ and $\psi$ are defined in \eqref{eq:inter-validlity-2up-3}. Clearly it is true since $\min \{s, \hat{s}-1\} \leq (\hat{s}-1)$ and $\tilde{f}, \psi \geq 0$.
\end{enumerate}
\item $t-s+L-1 \in [t+1, t+n-1]_{\Z}$. From inequality \eqref{eqn:ru-2-exp-2}, we have $x_t \leq \Vupper + (t+{[n-1]^+} - (t-s+L-1)) V + (m-1-f-{[n-1]^+})V + \phi = \Vupper+sV+(m-L-f)V+\phi$. Since $\Vupper + \min \{s, \hat{s}-1\}V \leq \Vupper + sV \leq \Vupper+sV+(m-L-f)V+\phi$, it follows that inequality \eqref{eqn:ru-2-exp-2} is valid. Note that $m-L-f \geq 0$.
\item $t-s+L-1 \geq t+n$. It follows that $t-s \geq t+n-L+1 \geq t-L+2$ and $y_i - \sum_{j=0}^{L-1}u_{i-j} = 0$ for all $i \in [t+1, t+n]_{\Z}$ and inequality \eqref{eqn:ru-2-exp-2} converts to $x_t \leq \Vupper + \phi$. We only need to show
\begin{equation}
\Vupper + \min \{s, \hat{s}-1\}V \leq \Vupper + \phi. \label{eq:inter-validlity-2up-4}
\end{equation}
If $s \leq t+L-T-1$ or $s \in [t+L-T,L-1-s]_{\Z}$, we have $\phi = sV$ and therefore \eqref{eq:inter-validlity-2up-4} holds; otherwise, $s \geq L-s$, i.e., $s \geq L/2$, we have $\phi = (L-1-s)V$ and $\max \{t-L+2, t-m+1\} \leq t-s \leq t-L/2$, i.e., $\min \{m-1, L-2\} \geq L/2$. On the other side, by condition (i) in Proposition \ref{prop:ru-2-exp}, we have $t-s+L-1 \leq t+L/2-1 \leq t+\min \{m-1, L-2\}-1 \leq t+n-1$, which contradicts to the condition $t-s+L-1 \geq t+n$. It indicates that $s \geq L-s$ is not possible to happen and \eqref{eq:inter-validlity-2up-4} holds.
\end{enumerate}
\item $y_{t-m} = y_t = 1$. If there is a shut-down and thereafter a start-up between $t-m+1$ and $t-1$, the discussion is same as the case discussed above since $x_{t-m} \geq \Clower y_{t-m}$; otherwise, we consider the case in which the {machine} keeps online throughout $t-m$ to $t$. It follows that $\phi = 0$. We continue to discuss the following two possible cases in terms of the first shut-down time 
{(denoted as $t+\hat{s}$, $\hat{s} \geq 1$)}
after time $t$.
\begin{enumerate}[(1)]
\item $t+\hat{s} \geq t+n+1$. In this case, we have $x_t - x_{t-m} \leq \min \{mV, \Vupper + (\hat{s}-1)V -\Clower\} \leq mV = \mbox{{the} RHS of}$ \eqref{eqn:ru-2-exp-2}.
\item $t+\hat{s} \in [t+1, t+n]_{\Z}$. In this case, we have $x_t - x_{t-m} \leq \min \{mV, \Vupper + (\hat{s}-1)V -\Clower\} \leq \Vupper + (\hat{s}-1)V -\Clower \leq \mbox{{the} RHS of}$ \eqref{eqn:ru-2-exp-2}.
\end{enumerate}
\end{enumerate}

\textbf{(Facet-defining)} Here we only provide the facet-defining proof for inequality \eqref{eqn:ru-2-exp-2}, as inequality \eqref{eqn:ru-2-exp-1} can be proved to be facet-defining similarly. We generate $3T-2$ linearly independent points (i.e., $(\tilde{x}^{\tilde{r}}, \tilde{y}^{\tilde{r}}, \tilde{u}^{\tilde{r}})_{\tilde{r}=2}^{T}$, $(\hat{x}^{\hat{r}}, \hat{y}^{\hat{r}}, \hat{u}^{\hat{r}})_{\hat{r}=1, \hat{r} \neq t-m}^{T}$, and $(\bar{x}^{\bar{r}}, \bar{y}^{\bar{r}}, \bar{u}^{\bar{r}})_{\bar{r}=1}^{T}$) in conv($P$) that satisfy \eqref{eqn:ru-2-exp-2} at equality.

\begin{enumerate}[1)]
\item For each $\tilde{r} \in [2, T]_{\Z}$ (totally $T-1$ points), we create $(\tilde{x}^{\tilde{r}}, \tilde{y}^{\tilde{r}}, \tilde{u}^{\tilde{r}}) \in$ conv($P$) such that
$\tilde{y}_s^{\tilde{r}} = 1$ for $s \in [\tilde{r}, \min \{\tilde{r}+L-1, T\}]_{\Z}$ and $\tilde{y}_s^{\tilde{r}} = 0$ otherwise and $\tilde{u}_s^{\tilde{r}} = 1$ for $s = \tilde{r}$ and $\tilde{u}_s^{\tilde{r}} = 0$ otherwise.
For the value of $\tilde{x}^{\tilde{r}}$, \sred{we have the following three cases:}
	\begin{enumerate}[(1)]
	\item if $t \in [\tilde{r}, \min \{\tilde{r}+L-1, T\}]_{\Z}$ and $\tilde{r}+L-1 \leq T-1$, then we let $\tilde{x}_s^{\tilde{r}} = \Vupper + \min \{s-\tilde{r}, L-1-s+\tilde{r}\}V$ for $s \in  [\tilde{r}, \min \{\tilde{r}+L-1, T\}]_{\Z}$ and $\tilde{x}_s^{\tilde{r}} = 0$ otherwise;
	\item if $t \in [\tilde{r}, \min \{\tilde{r}+L-1, T\}]_{\Z}$ and $\tilde{r}+L-1 = T$, then we let $\tilde{x}_s^{\tilde{r}} = \Vupper + (\min \{s, t\}-\tilde{r})V$ for $s \in  [\tilde{r}, \min \{\tilde{r}+L-1, T\}]_{\Z}$ and $\tilde{x}_s^{\tilde{r}} = 0$ otherwise;
	\item otherwise, we let $\tilde{x}_s^{\tilde{r}} = \Clower$ for $s \in  [\tilde{r}, \min \{\tilde{r}+L-1, T\}]_{\Z}$ and $\tilde{x}_s^{\tilde{r}} = 0$ otherwise.
	\end{enumerate}
\item For each $\hat{r} \in [1,t-1]_{\Z} \setminus \{t-m\}$ (totally $t-2$ points), 
we create $(\hat{x}^{\hat{r}}, \hat{y}^{\hat{r}}, \hat{u}^{\hat{r}}) \in$ conv($P$) such that
\sred{(i)} $\hat{y}_s^{\hat{r}} = 1$ for $s \in [1,\hat{r}]_{\Z}$ and $\hat{y}_s^{\hat{r}} = 0$ otherwise, \sred{(ii)} $\hat{u}_s^{\hat{r}} = 0$ for all $s \in [2,T]_{\Z}$, and \sred{(iii)} $\hat{x}_s^{\hat{r}} = \Clower+\epsilon$ for $s \in [1, \hat{r}]_{\Z} \setminus \{t-m\}$, $\hat{x}_s^{\hat{r}} = \Clower$ for $s \in [1, \hat{r}]_{\Z} \cap \{t-m\}$, and $\hat{x}_s^{\hat{r}} = 0$ otherwise.
\item For $\hat{r} = t$ (totally one point), 
we create $(\hat{x}^{\hat{r}}, \hat{y}^{\hat{r}}, \hat{u}^{\hat{r}}) \in$ conv($P$) such that
\sred{(i)} $\hat{y}_s^{\hat{r}}= 1$ for $s \in [1, t+m]_{\Z}$ and $\hat{y}_s^{\hat{r}} = 0$ otherwise, \sred{(ii)} $\hat{u}_s^{\hat{r}} = 0$ for all $s \in [2,T]_{\Z}$, and \sred{(iii)}
$\hat{x}_s^{\hat{r}} = \Clower+ [s-t+m]^+ V + \epsilon$ for $s \in [1, t]_{\Z}$, $\hat{x}_s^{\hat{r}} = \Clower+ (t+m-s) V + \epsilon$ for $s \in [t+1, t+m]_{\Z}$, and $\hat{x}_s^{\hat{r}} = 0$ otherwise.
\item For each $\hat{r} \in [t+1,t+m]_{\Z}$ (totally $m$ points), 
we create $(\hat{x}^{\hat{r}}, \hat{y}^{\hat{r}}, \hat{u}^{\hat{r}}) \in$ conv($P$) such that
$\hat{y}_s^{\hat{r}} = 1$ for $s \in [\hat{r}, \min \{\hat{r}+L-1, T\}]_{\Z}$ and $\hat{y}_s^{\hat{r}} = 0$ otherwise, $\hat{u}_s^{\hat{r}} = 1$ for $s = \hat{r}$ and $\hat{u}_s^{\hat{r}} = 0$ otherwise, and $\hat{x}_s^{\hat{r}} = \Clower+\epsilon$ for $s \in [\hat{r}, \min \{\hat{r}+L-1, T\}]_{\Z}$ and $\hat{x}_s^{\hat{r}} = 0$ otherwise.
\item For each $\hat{r} \in [t+m+1,T]_{\Z}$ (totally $T-t-m$ points), 
we create $(\hat{x}^{\hat{r}}, \hat{y}^{\hat{r}}, \hat{u}^{\hat{r}}) \in$ conv($P$) such that
\sred{(i)} $\hat{y}_s^{\hat{r}} = 1$ for $s \in [1,\hat{r}]_{\Z}$ and $\hat{y}_s^{\hat{r}} = 0$ otherwise, \sred{(ii)} $\hat{u}_s^{\hat{r}} = 0$ for all $s \in [2,T]_{\Z}$, 
and \sred{(iii)} $\hat{x}_s^{\hat{r}} = \Clower+ [s-t+m]^+ V + \epsilon$ for $s \in [1, t]_{\Z}$, $\hat{x}_s^{\hat{r}} = \Clower+ (\max\{t+m,s\}-s) V + \epsilon$ for $s \in [t+1, \hat{r}]_{\Z}$, and $\hat{x}_s^{\hat{r}} = 0$ otherwise.
\item For each $\bar{r} \in [1,t-1]_{\Z}$ (totally $t-1$ points), 
we create $(\bar{x}^{\bar{r}}, \bar{y}^{\bar{r}}, \bar{u}^{\bar{r}}) \in$ conv($P$) such that
$\bar{y}_s^{\bar{r}} = 1$ for $s \in [1,\bar{r}]_{\Z}$ and $\bar{y}_s^{\bar{r}} = 0$ otherwise, $\bar{u}_s^{\bar{r}} = 0$ for all $s \in [2,T]_{\Z}$, and $\bar{x}_s^{\bar{r}} = \Clower$ for $s \in [1,\bar{r}]_{\Z}$ and $\bar{x}_s^{\bar{r}} = 0$ otherwise.
\item For $\bar{r} = t$ (totally one point), 
we create $(\bar{x}^{\bar{r}}, \bar{y}^{\bar{r}}, \bar{u}^{\bar{r}}) \in$ conv($P$) such that
\sred{(i)} $\bar{y}_s^{\bar{r}}= 1$ for $s \in [1, t+m]_{\Z}$ and $\bar{y}_s^{\bar{r}} = 0$ otherwise, \sred{(ii)} $\bar{u}_s^{\bar{r}} = 0$ for all $s \in [2,T]_{\Z}$,
and \sred{(iii)} $\bar{x}_s^{\bar{r}} = \Clower+ [s-t+m]^+ V$ for $s \in [1, t]_{\Z}$, $\bar{x}_s^{\bar{r}} = \Clower+ (t+m-s) V$ for $s \in [t+1, t+m]_{\Z}$, and $\bar{x}_s^{\bar{r}} = 0$ otherwise.
\item For each $\bar{r} \in [t+1, t+n-1]_{\Z}$ (totally $n-1$ points), 
we create $(\bar{x}^{\bar{r}}, \bar{y}^{\bar{r}}, \bar{u}^{\bar{r}}) \in$ conv($P$) such that
\sred{(i)} $\bar{y}_s^{\bar{r}} = 1$ for $s \in [t-L+1, \bar{r}]_{\Z}$ and $\bar{y}_s^{\bar{r}} = 0$ otherwise, 
\sred{(ii)} $\bar{u}_s^{\bar{r}} = 1$ for $s=t-L+1$ and $\bar{u}_s^{\bar{r}} = 0$ otherwise,
and \sred{(iii)} $\bar{x}_s^{\bar{r}} = \Vupper+\min\{s-t+L-1, \bar{r}-t\}V$ for $s \in [t-L+1, t-1]_{\Z}$,
$\bar{x}_s^{\bar{r}} = \Vupper+(\bar{r}-s)V$ for $s \in [t, \bar{r}]_{\Z}$, and $\bar{x}_s^{\bar{r}} = 0$ otherwise.
\item For $\bar{r} = t+n$ (totally one point), 
we create $(\bar{x}^{\bar{r}}, \bar{y}^{\bar{r}}, \bar{u}^{\bar{r}}) \in$ conv($P$) such that
$\bar{y}_s^{\bar{r}} = 1$ for $s \in [t-p-L+1, \min \{t+p+L-1, T\}]_{\Z}$ and $\bar{y}_s^{\bar{r}} = 0$ otherwise and 
$\bar{u}_s^{\bar{r}} = 1$ for $s = t-p-L+1$ and $\bar{u}_s^{\bar{r}} = 0$ otherwise, where we define $t-p = \max \{a \in S\}$, which is greater than $\geq t-m+L$.
For the value of $\bar{x}^{\bar{r}}$, \sred{we have the following two cases:}
	\begin{enumerate}[(1)]
	\item if $t+p+L-1 \leq T-1$, then we let $\bar{x}_{\sred{s}}^{\bar{r}} = \Vupper + (\min \{s,2t-s\} - t+p+L-1)V$ for $s \in [t-p-L+1, \min \{t+p+L-1, T\}]_{\Z}$ and $\bar{x}_s^{\bar{r}} = 0$ otherwise;
	\item otherwise, we let $\bar{x}_{\sred{s}}^{\bar{r}} = \Vupper + (\min \{s,t\} - t+p+L-1)V$ for $s \in [t-p-L+1, \min \{t+p+L-1, T\}]_{\Z}$ and $\bar{x}_s^{\bar{r}} = 0$ otherwise.
	\end{enumerate}
\item For each $\bar{r} \in [t+n+1, t+m]_{\Z}$ (totally $m-n$ points), 
we create $(\bar{x}^{\bar{r}}, \bar{y}^{\bar{r}}, \bar{u}^{\bar{r}}) \in$ conv($P$) such that
$\bar{y}_s^{\bar{r}} = 1$ for $s \in [t+n-L+1, \bar{r}]_{\Z}$ and $\bar{y}_s^{\bar{r}} = 0$ otherwise and
$\bar{u}_s^{\bar{r}} = 1$ for $s = t+n-L+1$ and $\bar{u}_s^{\bar{r}} = 0$ otherwise.
For the value of $\bar{x}^{\bar{r}}$, \sred{we have the following two cases:}
	\begin{enumerate}[(1)]
	\item if $t \leq t+n-L$, then we let $\bar{x}_{\sred{s}}^{\bar{r}} = \Clower$ for $s \in [t+n-L+1, \bar{r}]_{\Z}$ and $\bar{x}_{\sred{s}}^{\bar{r}} = 0$ otherwise;
	\item otherwise, we let $\bar{x}_s^{\bar{r}} = \Vupper + \min \{s-t-n+L-1, n-s\}V$ for $s \in [t+n-L+1, t]_{\Z}$, $\bar{x}_s^{\bar{r}} = \Vupper + \min \{n+L-1, n, \bar{r}-s\}V$ for $s \in [t+1, \bar{r}]_{\Z}$, and $\bar{x}_s^{\bar{r}} = 0$ otherwise.
	\end{enumerate}
\item For each $\bar{r} \in [t+m+1,T]_{\Z}$ (totally $T-t-m$ points), 
we create $(\bar{x}^{\bar{r}}, \bar{y}^{\bar{r}}, \bar{u}^{\bar{r}}) \in$ conv($P$) such that
\sred{(i)} $\bar{y}_s^{\bar{r}} = 1$ for $s \in [1,\bar{r}]_{\Z}$ and $\bar{y}_s^{\bar{r}} = 0$ otherwise, \sred{(ii)} $\bar{u}_s^{\bar{r}} = 0$ for all $s \in [2,T]_{\Z}$, and \sred{(iii)}
$\bar{x}_s^{\bar{r}} = \Clower+ [s-t+m]^+ V$ for $s \in [1, t]_{\Z}$, $\bar{x}_s^{\bar{r}} = \Clower+ (\max\{t+m,s\}-s) V$ for $s \in [t+1, \bar{r}]_{\Z}$, and $\bar{x}_s^{\bar{r}} = 0$ otherwise.
\end{enumerate}

In summary, we create $3T-2$ points $(\bar{x}^{\bar{r}}, \bar{y}^{\bar{r}}, \bar{u}^{\bar{r}})_{\bar{r}=1}^{T}$, $(\hat{x}^{\hat{r}}, \hat{y}^{\hat{r}}, \hat{u}^{\hat{r}})_{\hat{r}=1, \hat{r} \neq t-m}^{T}$, and $(\tilde{x}^{\tilde{r}}, \tilde{y}^{\tilde{r}}, \tilde{u}^{\tilde{r}})_{\tilde{r}=2}^{T}$. 
It is easy to see that they are valid and satisfy \eqref{eqn:ru-2-exp-2} at equality. Meanwhile, they are clearly linearly independent since $(\bar{x}, \bar{y}, \bar{u})$ and $(\hat{x}, \hat{y}, \hat{u})$ can construct a lower-triangular matrix in terms of the values of $x$ and $y$ and $(\tilde{x}, \tilde{y}, \tilde{u})$ is further linearly independent with them since it constructs an upper-triangular matrix in terms of the value of $u$.
\end{proof}

\subsection{Proof for Proposition \ref{prop:rd-2-exp}} \label{apx:subsec:rd-2-exp}
\begin{proof}
\textbf{(Validity)} It is clear that inequality \eqref{eqn:rd-2-exp} is valid when $y_t = y_{t+m} = 0$ and $y_t = 0, \ y_{t+m} = 1$ due to constraints  \eqref{eqn:p-minup} and \eqref{eqn:p-lower-bound}. We continue to discuss the remaining two cases in terms of the values of $y_{t}$ and $y_{t+m}$.
\begin{enumerate}[1)]
\item $y_{t} = 1, \ y_{t+m} = 0$. There is at least a shut-down between $t$ and $t+m$ and without loss of generality we consider there is only one shut-down between them and let it be $t+\hat{s}$ ($\hat{s} \in [1, m]_{\Z}$). Meanwhile, we let the last start-up time before $t$ be $t-s$ ($s \geq 0$) and continue to discuss the following three possible cases in terms of the value of $t-s+L-1$ ($t-s+L-1 \leq t + \hat{s} - 1$).
\begin{enumerate}[(1)]
\item $t-s+L-1 \leq t$, i.e., $s \geq L-1$. It follows that $\phi = 0$. We continue to discuss the following three possible cases in terms of the value of $t+\hat{s}$.
\begin{enumerate}[(a)]
\item $t+\hat{s} \in [t+1, \tilde{t}-1]_{\Z}$. In this case, we have $x_t \leq \Vupper + \min \{s, \hat{s}-1\}V \leq \Vupper + (\hat{s}-1)V = \mbox{{the} RHS of}$ \eqref{eqn:rd-2-exp}.
\item $t+\hat{s} \in [\tilde{t}, q]_{\Z}$. In this case, we have $x_t \leq \Vupper + \min \{s, \hat{s}-1\}V \leq \Vupper + (\tilde{t}-1-t)V + (t+\hat{s}-t)V \leq \Vupper + (\tilde{t}-1-t)V + V \sum_{i \in (S \setminus \{t+m\}) \cap [\tilde{t}, t+\hat{s}-1]_{\Z}} (d_i - i) = \mbox{{the} RHS of}$ \eqref{eqn:rd-2-exp}.
\item $t+\hat{s} \geq q+1$. In this case, we have $x_t \leq \Vupper + \min \{s, \hat{s}-1\}V \leq \Vupper + (m-1)V \leq \Clower + mV = \Vupper + (\tilde{t}-1-t)V + (t+m-\tilde{t})V + (\Clower + V- \Vupper) = \mbox{{the} RHS of}$ \eqref{eqn:rd-2-exp}.
\end{enumerate}
\item $t-s+L-1\in [t+1, \tilde{t}-1]_{\Z}$. It follows that $y_i - \sum_{j=0}^{L-1} u_{i-j} = 0$ for all $i \in [t, t-s+L-1]_{\Z}$. We continue to discuss the following two possible cases in terms of the value of $t+\hat{s}$.
\begin{enumerate}[(a)]
\item $t+\hat{s} \in [\tilde{t}, q]_{\Z}$. In this case, we have $x_t \leq \Vupper + \min \{s, \hat{s}-1\}V$. From inequality \eqref{eqn:rd-2-exp}, we have $x_t \leq \Vupper + [(\tilde{t}-1)-(t-s+L-1)]V + \tilde{f} V+ \phi$, where $\tilde{f} = \sum_{i \in (S \setminus \{t+m\}) \cap [\tilde{t}, t+\hat{s}-1]_{\Z}} (d_i - i) \geq t+\hat{s}-\tilde{t}$. We only need to show 
\begin{equation}
\Vupper + \min \{s, \hat{s}-1\}V \leq \Vupper + [(\tilde{t}-1)-(t-s+L-1)]V + \tilde{f} V+ \phi. \label{eq:inter-validlity-2down-1}
\end{equation}
If $s \leq t+L-T-1$ or $s \in [t+L-T,L-1-s]_{\Z}$, we have $\phi = sV$ and therefore \eqref{eq:inter-validlity-2down-1} holds since $\tilde{f} \geq t+\hat{s}-\tilde{t} \geq 0$ and $(\tilde{t}-1)-(t-s+L-1) \geq 0$; otherwise, $s \geq L-s$, i.e., $s \geq L/2$, we have $\phi = (L-1-s)V$ and the RHS of \eqref{eq:inter-validlity-2down-1} converts to $\Vupper + (\tilde{t}-1-t)V + \tilde{f} V \geq \Vupper + (\tilde{t}-1-t)V + (t+\hat{s}-\tilde{t}) V =  \Vupper + (\hat{s}-1)V$, which indicates that \eqref{eq:inter-validlity-2down-1} holds.
\item $t+\hat{s} \geq q+1$. In this case, we have $x_t \leq \Vupper + \min \{s, \hat{s}-1\}V$. From inequality \eqref{eqn:rd-2-exp}, we have $x_t \leq \Vupper + [(\tilde{t}-1)-(t-s+L-1)]V + (t+m-\tilde{t}) V+ (\Clower + V-\Vupper) + \phi$. If $\phi = sV$, we have $\Vupper + \min \{s, \hat{s}-1\}V \leq \Vupper + \phi$, indicating \eqref{eqn:rd-2-exp} is valid; otherwise $\phi = (L-1-s)V$, we have $x_t \leq \Clower + mV$ from inequality \eqref{eqn:rd-2-exp}, which indicates that \eqref{eqn:rd-2-exp} is valid since $\Vupper + \min \{s, \hat{s}-1\}V \leq \Vupper + sV \leq \Vupper + (m-1)V \leq \Clower + mV$.
\end{enumerate}
\item $t-s+L-1\in [\tilde{t}, t+\hat{s}-1]_{\Z}$. Since $[\tilde{t}, t+\hat{s}-1]_{\Z} \neq \emptyset$ in this case, we have $\tilde{t} = \hat{t}$, otherwise $\tilde{t} \geq t+m \geq t+\hat{s}$. From inequality \eqref{eqn:rd-2-exp}, we have $x_t \leq \Vupper + \tilde{f} V+ \phi + \psi$, where $\tilde{f} = \sum_{i \in (S \setminus \{t+m\}) \cap [t-s+L, t+\hat{s}-1]_{\Z}} (d_i - i) \geq 0$ and $\psi = (\Clower + V - \Vupper) (y_{q} - \sum_{j=0}^{ \min \{L-1, q-2\} } u_{q-j}) \geq 0$. 
Now, we show $\phi = sV$ and therefore $\Vupper + \min \{s, \hat{s}-1\}V \leq \Vupper + sV = \Vupper + \phi \leq \Vupper + \tilde{f} V+ \phi + \psi$, which indicates that \eqref{eqn:rd-2-exp} is valid. 
By contradiction, if $\phi \neq sV$, i.e., $s \geq L-s$ and $\phi = (L-1-s)V$, then we have $s \geq L/2$ and $s \leq \min \{t-2, L-2\}$ since $t-s \geq 2$ and $t-s+L-1 \geq t+1$. It follows that $\min \{t-2, L-2\} \geq L/2$ and $\hat{t} = t+\min \{t-2, L-2\}$ by the definition of $\hat{t}$. It follows that $t-s+L-1 \leq t+L/2-1 \leq t+\min \{t-2, L-2\}-1 = \hat{t}-1 = \tilde{t}-1$, which contradicts to this case that $t-s+L-1 \geq \tilde{t}$.
\end{enumerate}
\item $y_{t} = y_{t+m} = 1$. If there is a shut-down and thereafter a start-up between $t+1$ and $t+m-1$, the discussion is same as the case discussed above since $x_{t+m} \geq \Clower y_{t+m}$; otherwise, we consider the case in which the {machine} keeps online throughout $t$ to $t+m$. We continue to discuss the following three possible cases in terms of the value of $t-s+L-1$, where $t-s$ is {defined as} the last start-up time  before time $t$.
\begin{enumerate}[(1)]
\item $t-s+L-1 \leq t$, i.e., $s \geq L-1$. From inequality \eqref{eqn:rd-2-exp}, we have $x_t - x_{t+m} \leq mV$, which is valid due to ramp-down constraints \eqref{eqn:p-ramp-down}.
\item $t-s+L-1\in [t+1, \tilde{t}-1]_{\Z}$. In this case, we have $x_t - x_{t+m} \leq \min \{mV, \Vupper + sV -\Clower \}$. From inequality \eqref{eqn:rd-2-exp}, we have $x_t - x_{t+m} \leq \Vupper - \Clower + [(\tilde{t}-1)-(t-s+L-1)]V + (t+m-\tilde{t}) V+ (\Clower + V-\Vupper) + \phi$. 
If $\phi = sV$, we have $x_t - x_{t+m} \leq \Vupper + sV - \Clower = \Vupper + \phi - \Clower$, indicating \eqref{eqn:rd-2-exp} is valid; otherwise $\phi = (L-1-s)V$, we have $x_t-x_{t+m} \leq mV$ from inequality \eqref{eqn:rd-2-exp}, which indicates that \eqref{eqn:rd-2-exp} is valid.
\item $t-s+L-1 \geq \tilde{t}$. We further discuss the following two possible cases.
\begin{enumerate}[(a)]
\item If $\hat{t} \leq t+m$, then we have $\tilde{t} = \hat{t}$. From inequality \eqref{eqn:rd-2-exp}, we have $x_t - x_{t+m} \leq \Vupper -\Clower + \tilde{f} V+ \phi + \psi$, where $\tilde{f} = \sum_{i \in (S \setminus \{t+m\}) \cap [t-s+L, t+m]_{\Z}} (d_i - i) \geq 0$ and $\psi = (\Clower + V - \Vupper) (y_{q} - \sum_{j=0}^{ \min \{L-1, q-2\} } u_{q-j}) \geq 0$. 
Now, we show $\phi = sV$ and therefore $x_t - x_{t+m} \leq \Vupper + sV - \Clower = \Vupper + \phi - \Clower \leq \Vupper - \Clower + \tilde{f} V+ \phi + \psi$, which indicates that \eqref{eqn:rd-2-exp} is valid. 
By contradiction, if $\phi \neq sV$, i.e., $s \geq L-s$ and $\phi = (L-1-s)V$, then we have $s \geq L/2$ and $s \leq \min \{t-2, L-2\}$ since $t-s \geq 2$ and $t-s+L-1 \geq t+1$. It follows that $\min \{t-2, L-2\} \geq L/2$ and $\hat{t} = t+\min \{t-2, L-2\}$ by the definition of $\hat{t}$. It follows that $t-s+L-1 \leq t+L/2-1 \leq t+\min \{t-2, L-2\}-1 = \hat{t}-1 = \tilde{t}-1$, which contradicts to this case that $t-s+L-1 \geq \tilde{t}$.
\item If $\hat{t} \geq t+m+1$, then we have $\tilde{t} = t+m$ and $t-s+L-1 \geq \tilde{t} = t+m$, i.e., $m \leq L-1-s$. It follows that $y_i - \sum_{j=0}^{L-1}u_{i-j} = 0$ for all $i \in [t+1, t+m]$. From inequality \eqref{eqn:rd-2-exp}, we have $x_t - x_{t+m} \leq \Vupper - \Clower + \phi$, which is valid no matter $\phi = sV$ or $(L-1-s)V$ since $x_t - x_{t+m} \leq \Vupper + sV -\Clower$ and $x_t - x_{t+m} \leq mV \leq (L-1-s)V$.
\end{enumerate}
\end{enumerate}
\end{enumerate}

\textbf{(Facet-defining)} The facet-defining proof is similar with that in Appendix \ref{apx:subsec:ru-2-exp} for Proposition \ref{prop:ru-2-exp} and thus is omitted here.
\end{proof}

\subsection{Proof for Proposition \ref{prop:3degree-1-multi-period}} \label{apx:subsec:3degree-1-multi-period}
\begin{proof}
\textbf{(Validity)} We discuss the following two \fblue{possible} cases in terms of the value of $y_{t}$:
\begin{enumerate}[1)]
\item If $y_{t} = 0$, then $u_{t-s} = 0$ for all $s \in [0, L-1]_{\Z}$ due to constraints \eqref{eqn:p-minup} and $y_{t+1} = u_{t+1}$ due to constraints \eqref{eqn:p-minup} and \eqref{eqn:p-udef}. It follows that inequality \eqref{eqn:3degree-1-multi-period} converts to $x_{t-2} - x_{t-1} \leq \Vupper y_{t-2} - (\Vupper - V) y_{t-1}$, which can be easily verified to be valid through considering all the three possible cases, i.e., (1) $y_{t-2} = y_{t-1} = 1$, (2) $y_{t-2} = 1$ and $y_{t-1} = 0$, and (3) $y_{t-2} = y_{t-1} = 0$.
\item If $y_{t} = 1$, then $\sum_{s=0}^{L-1} u_{t-s} \leq 1$ due to constraints \eqref{eqn:p-minup}. We further discuss the following four possible cases.
\begin{enumerate}[(1)]
\item If $u_{t-s} = 0$ for all $s \in [0, L-1]_{\Z}$, then $y_{t-1} = 1$ due to constraints \eqref{eqn:p-udef} and $L \geq 2$. It follows that inequality \eqref{eqn:3degree-1-multi-period} converts to $x_{t-2} - x_{t-1} + x_{t} \leq \Vupper y_{t-2} - (\Vupper - V) + \Cupper + (\Clower + V - \Vupper) (y_{t+1} - 1)$, which can be easily verified to be valid through considering all the four possible cases, i.e., (1) $y_{t-2} = y_{t+1} = 1$, (2) $y_{t-2} = 1$ and $y_{t+1} = 0$, (3) $y_{t-2} = 0$ and $y_{t+1} = 1$, and (4) $y_{t-2} = y_{t+1} = 0$.
\item If $u_{t} = 1$, then $u_{t-s} = 0$ for all $s \in [1, L-1]_{\Z}$. Meanwhile, we have $y_{t+1} = 1$ and $u_{t+1} = 0$ due to $L \geq 2$ and $ y_{t-1} = 0$ due to constraints \eqref{eqn:p-ramp-down}. It follows that inequality \eqref{eqn:3degree-1-multi-period} converts to $x_{t-2} + x_{t} \leq \Vupper y_{t-2} + \Vupper$, which is valid since $x_{t-2} \leq \Vupper y_{t-2}$ and $x_{t} \leq \Vupper$ due to constraints \eqref{eqn:p-ramp-up} and \eqref{eqn:p-ramp-down}.
\item If $u_{t-1} = 1$, then $u_{t-s} = 0$ for all $s \in [2, L-1]_{\Z}$ and $u_{t} = 0$. Meanwhile, we have $y_{t-2} = 0$ due to \eqref{eqn:p-ramp-down} and $y_{t} = 1$ and $u_{t+1} = 0$ due to $L \geq 2$. It follows that inequality \eqref{eqn:3degree-1-multi-period} converts to $ x_{t} - x_{t-1} \leq V + (\Clower + V - \Vupper) (y_{t+1} - 1)$, which can be easily verified to be valid for either $y_{t+1} = 1$ or $y_{t+1} = 0$.
\item If $u_{t-s-2} = 1$ for some $s \in [0, L-3]_{\Z}$ when $L \geq 3$, then $y_{t-2} = y_{t-1} = y_{t} = 1$ due to minimum-up time constraints \eqref{eqn:p-minup}. It follows that inequality \eqref{eqn:3degree-1-multi-period} converts to $x_{t-2} - x_{t-1} + x_{t} \leq \Vupper + sV + V + (\Clower + V - \Vupper) (y_{t+1} - 1)$, which can be easily verified to be valid either $y_{t+1} = 1$ or $y_{t+1} = 0$ since $x_{t-2} \leq \Vupper + sV$ and $ x_{t} - x_{t-1} \leq V + (\Clower + V - \Vupper) (y_{t+1} - 1)$.
\end{enumerate}
\end{enumerate}

\textbf{(Facet-defining)} We only provide the facet-defining proof for the case when $L=3$ since the case when $L=2$ can be proved similarly and thus is omitted here.
We generate generate $3T-2$ linearly independent points in conv($P$) that satisfy \eqref{eqn:3degree-1-multi-period} at equality in the following groups.
\begin{enumerate}[1)]
\item For each $r \in [1,t-3]_{\Z}$ (totally $t-3$ points), we create $(\acute{x}^r, \acute{y}^r, \acute{u}^r) \in$ conv($P$) such that
\begin{equation*}
\acute{x}_s^r = \left\{\begin{array}{l}
  \Clower, \ s \in [1,r]_{\Z} \\
  0, \ s \in [r+1,T]_{\Z}
\end{array} \right., \
\acute{y}_s^r = \left\{\begin{array}{l}
  1, \ s \in [1,r]_{\Z} \\
  0, \ s \in [r+1,T]_{\Z}
\end{array} \right., \ \mbox{and} \
\begin{array}{l}
\acute{u}_s^r = 0, \\
\forall s
\end{array}.
\end{equation*}

\item For $r = t-1$ (totally one point), we create $(\acute{x}^r, \acute{y}^r, \acute{u}^r) \in$ conv($P$) such that
\begin{equation*}
\acute{x}_s^r = \left\{\begin{array}{l}
  \Clower + V, \ s \in [1,r-1]_{\Z} \\
  \Clower, \ s = r \\
  0, \ s \in [r+1,T]_{\Z}
\end{array} \right., \
\acute{y}_s^r = \left\{\begin{array}{l}
  1, \ s \in [1,r]_{\Z} \\
  0, \ s \in [r+1,T]_{\Z}
\end{array} \right., \ \mbox{and} \
\begin{array}{l}
\acute{u}_s^r = 0, \\
\forall s
\end{array}.
\end{equation*}

\item For each $r \in [1,t-1]_{\Z}$ (totally $t-1$ points), we create $(\bar{x}^r, \bar{y}^r, \bar{u}^r) \in$ conv($P$) such that $\bar{y}_s^r = 1$ for each $s \in [1,r]_{\Z}$ and $\bar{y}_s^r = 0$ otherwise. Thus $\bar{u}_s^r = 0$ for each $s \in [2,T]_{\Z}$ {due to constraints \eqref{eqn:p-minup} - \eqref{eqn:p-udef}}. For the value of $\bar{x}^r$: (1) for each $r \in [1,t-3]_{\Z}$, we let $\bar{x}_s^r = \Clower + \epsilon$ for each $s \in [1,r]_{\Z}$; (2) for $r= t-2$, we let $\bar{x}_s^r = \Vupper$ for each $s \in [1,r]_{\Z}$; (3) for $r= t-1$, we let $\bar{x}_s^r = \Clower + V + \epsilon$ for each $s \in [1,r-1]_{\Z}$ and $\bar{x}_s^r = \Clower + \epsilon$ for $s = r$.

\item For $r = t$ (totally one point), we create $(\bar{x}^r, \bar{y}^r, \bar{u}^r) \in$ conv($P$) such that $\bar{y}_s^r = 1$ for each $s \in [1,T]_{\Z}$ and thus $\bar{u}_s^r = 0$ for each $s \in [2,T]_{\Z}$ {due to constraints \eqref{eqn:p-minup} - \eqref{eqn:p-udef}}. For the value of $\bar{x}^r$, we let $\bar{x}_s^r = \Cupper - V$ for $s = t-1$ and $\bar{x}_s^r = \Cupper$ otherwise.

\item For each $r \in [t+1,T]_{\Z}$ (totally $T-t$ points), we create $(\bar{x}^r, \bar{y}^r, \bar{u}^r) \in$ conv($P$) such that
\begin{equation*}
\bar{x}_s^r = \left\{\begin{array}{l}
  \Vupper, \ s = t-2 \\
  \Clower+V, \ s = t \\
  \Clower, \ s \in [r,T]_{\Z} \cup \{t-1\} \\
  0, \ \mbox{o.w.}
\end{array} \right., \
\bar{y}_s^r = \left\{\begin{array}{l}
  1, \ s \in [t-2,r]_{\Z} \\
  0, \ \mbox{o.w.}
\end{array} \right., \ \mbox{and} \
\bar{u}_s^r = \left\{\begin{array}{l}
  1, \ s = t-2 \\
  0, \ \mbox{o.w.}
\end{array} \right..
\end{equation*}

\item For each $r \in [2,t]_{\Z}$ (totally $t-1$ points), we create $(\hat{x}^r, \hat{y}^r, \hat{u}^r) \in$ conv($P$) such that $\hat{y}_s^r = 1$ for each $s \in [r,r+L-1]_{\Z}$ (i.e., $s \in [r,r+2]_{\Z}$) and $\hat{y}_s^r = 0$ otherwise. Thus $\hat{u}_s^r = 1$ for $s = r$ {due to constraints \eqref{eqn:p-minup} - \eqref{eqn:p-udef}}. For the value of $\hat{x}^r$: 
(1) for each $r \in [2,t-3]_{\Z} \cup \{t-1\}$, we let $\hat{x}_s^r = \Clower $ for each $s \in [r,r+2]_{\Z} \setminus \{t-2\}$ and $\hat{x}_s^r = \Clower+V $ for each $s \in [r,r+2]_{\Z} \cap \{t-2\}$;
(2) for $r= t-2$, we let $\hat{x}_s^r = \Vupper$ for each $s \in \{t-2, t\}$ and $\hat{x}_s^r = \Clower$ for each $s = t-1$; 
(3) for $r= t$, we let $\hat{x}_s^r = \Vupper$ for each $s \in [r,r+L-1]_{\Z}$.

\item For each $r \in [t+1,T]_{\Z}$ (totally $T-t$ points), we create $(\hat{x}^r, \hat{y}^r, \hat{u}^r) \in$ conv($P$) such that
\begin{equation*}
\hat{x}_s^r = \left\{\begin{array}{l}
  \Clower, \ s \in [r,T]_{\Z} \\
  0, \ \mbox{o.w.}
\end{array} \right., \
\hat{y}_s^r = \left\{\begin{array}{l}
  1, \ s \in [r,T]_{\Z} \\
  0, \ \mbox{o.w.}
\end{array} \right., \ \mbox{and} \
\hat{u}_s^r = \left\{\begin{array}{l}
  1, \ s = r \\
  0, \ \mbox{o.w.}
\end{array} \right..
\end{equation*}

\item For each $r \in [t+1,T]_{\Z}$ (totally $T-t$ points), we create $(\grave{x}^r, \grave{y}^r, \grave{u}^r) \in$ conv($P$) such that
\begin{equation*}
\grave{x}_s^r = \left\{\begin{array}{l}
  \Clower + \epsilon, \ s \in [r,T]_{\Z} \\
  0, \ \mbox{o.w.}
\end{array} \right., \
\grave{y}_s^r = \left\{\begin{array}{l}
  1, \ s \in [r,T]_{\Z} \\
  0, \ \mbox{o.w.}
\end{array} \right., \ \mbox{and} \
\grave{u}_s^r = \left\{\begin{array}{l}
  1, \ s = r \\
  0, \ \mbox{o.w.}
\end{array} \right..
\end{equation*}

\item We create $(\dot{x}, \dot{y}, \dot{u}) \in$ conv($P$) such that $\dot{y}_s = 1$ for each $s \in \{t-1, t, t+1\}$ and $\dot{y}_s = 0$ otherwise. Thus we have $\dot{u}_s = 1$ for $s = t-1$. Meanwhile, we let $\dot{x}_{t-2}=\dot{x}_{t}=\Clower+\epsilon$ and $\dot{x}_{t-1}=\Clower+V+\epsilon$.
\end{enumerate}

Finally, it is clear that $(\bar{x}^r, \bar{y}^r, \bar{u}^r)_{r=1}^{T}$ and $(\hat{x}^r, \hat{y}^r, \hat{u}^r)_{r=2}^{T}$ are linearly independent because they can construct a lower-diagonal matrix. In addition, $(\acute{x}^r, \acute{y}^r, \acute{u}^r)_{r=1, r \neq t-2}^{t-1}$, $(\grave{x}^r, \grave{y}^r, \grave{u}^r)_{r=t+1}^{T}$, and $(\dot{x}, \dot{y}, \dot{u})$ are also linearly independent with them after Gaussian eliminations between $(\acute{x}^r, \acute{y}^r, \acute{u}^r)_{r=1, r \neq t-2}^{t-1}$, $(\dot{x}, \dot{y}, \dot{u})$, and $(\bar{x}^r, \bar{y}^r, \bar{u}^r)_{r=1}^{t-1}$, and between $(\hat{x}^r, \hat{y}^r, \hat{u}^r)_{r=t+1}^{T}$ and $(\grave{x}^r, \grave{y}^r, \grave{u}^r)_{r=t+1}^{T}$.
\end{proof}

\subsection{Proof for Proposition \ref{prop:ru-3-exp}} \label{apx:subsec:ru-3-exp}
\begin{proof}
\textbf{(Validity)} Here we only provide the validity proof for the case in which $L \geq 3$, as the case in which $L=2$ can be proved similarly. We discuss the following two possible cases in terms of the start-up before $t$. Meanwhile, we let $f = \sum_{i \in S} (i - d_i)$.
\begin{enumerate}[1)]
\item There is no start-up before $t$. If the {machine} is offline throughout the first time period to $t$, clearly \eqref{eqn:ru-3-exp} is valid; otherwise, we consider the {machine} is online since the first time period and shuts down at $t+\hat{s}$.
\begin{enumerate}[(1)]
\item $t+\hat{s} \leq t-2$. It follows that $y_{t-2} = y_{t-1} = y_t = 0$ and clearly \eqref{eqn:ru-3-exp} is valid due to constraints \eqref{eqn:p-minup}.
\item $t+\hat{s} = t-1$. Inequality \eqref{eqn:ru-3-exp} converts to $x_{t-2} \leq \Vupper + \psi$, which is valid due to constraints \eqref{eqn:p-ramp-down} and $\psi = V \sum_{i \in S \cap [t-m+1,t-2]_{\Z}} (i - d_i) + (\Cupper - \Vupper - {[m+L-3]^+} V) \geq 0$.
\item $t+\hat{s} = t$. Inequality \eqref{eqn:ru-3-exp} converts to $x_{t-2} - x_{t-1} \leq V + \psi$, which is valid due to constraints \eqref{eqn:p-ramp-down} and $\psi = fV + (\Cupper - \Vupper - {[m+L-3]^+} V) \geq 0$.
\item $t+\hat{s} \geq t+1$. Inequality \eqref{eqn:ru-3-exp} converts to $x_{t-2} - x_{t-1} + x_t \leq V + \Cupper$, which is valid due to constraints \eqref{eqn:p-upper-bound} and \eqref{eqn:p-ramp-down}.
\end{enumerate}
\item There is at least one start-up before $t$. Without loss of generality, we consider there is only one start-up before $t$ and let its time be $t-s$ ($s \geq 0$). Meanwhile, we let the first shut-down time after this start-up be $t+\hat{s}$. We continue to discuss the following four possible cases in terms of the value of $t-s+L-1$.
\begin{enumerate}[(1)]
\item $t-s+L-1 \leq t-m-1$. In this case, we can follow the same discussions in case 1) to show that \eqref{eqn:ru-3-exp} is valid and thus omit them here.
\item $t-s+L-1 \in [t-m, t-1]_{\Z}$. We continue to discuss the following four possible cases in terms of the value of $t+\hat{s}$.
\begin{enumerate}[(a)]
\item $t+\hat{s} \leq t-2$. It follows that $y_{t-2} = y_{t-1} = y_t = 0$ and clearly \eqref{eqn:ru-3-exp} is valid due to constraints \eqref{eqn:p-minup}.
\item $t+\hat{s} = t-1$. Inequality \eqref{eqn:ru-3-exp} converts to $x_{t-2} \leq \Vupper + \psi$, which is valid due to constraints \eqref{eqn:p-ramp-down} and $\psi = V \sum_{i \in S \setminus [t-s, t-s+L+1]_{\Z}} (i - d_i) \geq 0$.
\item $t+\hat{s} = t$. Inequality \eqref{eqn:ru-3-exp} converts to $x_{t-2} - x_{t-1} \leq V + \psi$, which is valid due to constraints \eqref{eqn:p-ramp-down} and $\psi = V \sum_{i \in S \setminus [t-s, t-s+L+1]_{\Z}} (i - d_i) \geq 0$.
\item $t+\hat{s} \geq t+1$. Inequality \eqref{eqn:ru-3-exp} converts to $x_{t-2} - x_{t-1} + x_t \leq \Vupper + V + \tilde{f}V + ({[m+L-3]^+}-f)V$, where $\tilde{f} = \sum_{i\in S \setminus [t-m+1,t-s+L-1]_{\Z}}(i-d_i)$. If $t-s+L-1 = t-m$, i.e., $s = m+L-1$, then $f=\tilde{f}$ and \eqref{eqn:ru-3-exp} further converts to $x_{t-2} - x_{t-1} + x_t \leq \Vupper + V + {[m+L-3]^+}V$, which is valid since $x_{t-2} \leq \Vupper + {[s-2]^+} V \leq \Vupper + {[m+L-3]^+}V$ and $x_t - x_{t-1} \leq V$. Otherwise, we consider $t-s+L-1 \geq t-m+1$, i.e., $s \leq m+L-2$, and let $t-q = \max \{a \in S\}$ and $t-p = \max \{a \in S, a \leq t-s+L-1\}$. If $t-q$ does not exist, then $f=\tilde{f}=0$ and we also have 
\eqref{eqn:ru-3-exp} is valid. Otherwise, we consider $S \neq \emptyset$ and therefore $t-q$ exists and $f=(t-q)-(t-m) = m-q$.
We continue to discuss the following three possible cases.
\begin{itemize}
\item If $t-q \leq t-s+L-1$, i.e., $s \leq q+L-1$, then we have $\tilde{f} = 0$. \eqref{eqn:ru-3-exp} further converts to $x_{t-2} - x_{t-1} + x_t \leq \Vupper + V + {[q+L-3]^+}V$, which is valid since $x_{t-2} \leq \Vupper + {[s-2]^+}V \leq \Vupper + {[q+L-3]^+}V$ and $x_t - x_{t-1} \leq V$.
\item If $t-q \geq t-s+L$ and $t-p$ does not exist, then we have $\tilde{f} = f = m-q$. \eqref{eqn:ru-3-exp} further converts to $x_{t-2} - x_{t-1} + x_t \leq \Vupper + V + {[m+L-3]^+}V$, which is valid since $x_{t-2} \leq \Vupper + {[s-2]^+}V \leq \Vupper + {[m+L-3]^+}V$ and $x_t - x_{t-1} \leq V$.
\item If $t-q \geq t-s+L$ and $t-p$ exists, then we have $\tilde{f} = p-q$ and meanwhile $s \leq p+L-1$ since $t-p \leq t-s+L-1$. \eqref{eqn:ru-3-exp} further converts to $x_{t-2} - x_{t-1} + x_t \leq \Vupper + V + {[p+L-3]^+} V$, which is valid since $x_{t-2} \leq \Vupper + {[s-2]^+}V \leq \Vupper + {[p+L-3]^+}V$ and $x_t - x_{t-1} \leq V$.
\end{itemize}
\end{enumerate}
\item $t-s+L-1 = t$. It follows that $y_i - \sum_{j=0}^{L-1}u_{i-j}=0$ for all $i \in [t-s,t]_{\Z}$. Inequality \eqref{eqn:ru-3-exp} converts to $x_{t-2} - x_{t-1} + x_t \leq V+\Vupper+ {[L-3]^+}V$, which is valid since $x_{t-2} \leq \Vupper + {[L-3]^+}V$ and $x_t - x_{t-1} \leq V$.
\item $t-s+L-1 \geq t+1$, i.e., $s \leq L-2$. It is clearly that \eqref{eqn:ru-3-exp} is valid when $s \in [0,2]_{\Z}$. Now we consider $s \geq 3$, i.e., $s \in [3, L-2]_{\Z}$. Inequality \eqref{eqn:ru-3-exp} converts to $x_{t-2} - x_{t-1} + x_t \leq V+\Vupper+{[L-3]^+}V$, which is valid since $x_{t-2} \leq \Vupper + {[L-3]^+}V$ and $x_t - x_{t-1} \leq V$.
\end{enumerate}
\end{enumerate}

\textbf{(Facet-defining)}
Here we only provide the facet-defining proof for the case in which $L \geq 3$ and  $n_{|S|} \leq t-2$, as the case in which $L=2$ or $n_{|S|} = t-1$ can be proved similarly. We generate $3T-2$ linearly independent points (i.e., $(\tilde{x}^{\tilde{r}}, \tilde{y}^{\tilde{r}}, \tilde{u}^{\tilde{r}})_{\tilde{r}=2}^{T}$, $(\hat{x}^{\hat{r}}, \hat{y}^{\hat{r}}, \hat{u}^{\hat{r}})_{\hat{r}=1, \hat{r} \neq t-2}^{T}$, and $(\bar{x}^{\bar{r}}, \bar{y}^{\bar{r}}, \bar{u}^{\bar{r}})_{\bar{r}=1}^{T}$) in conv($P$) that satisfy \eqref{eqn:ru-3-exp} at equality. As the condition in Proposition \ref{prop:ru-3-exp} described, we have $t=T$ in the following proof.

\begin{enumerate}[1)]
\item For each $\tilde{r} \in [2, t-L-2]_{\Z}$ (totally $[t-L-3]^+$ points), we create $(\tilde{x}^{\tilde{r}}, \tilde{y}^{\tilde{r}}, \tilde{u}^{\tilde{r}}) \in$ conv($P$) such that
$\tilde{y}_s^{\tilde{r}} = 1$ for $s \in [\tilde{r}, \min \{\tilde{r}+L-1, T\}]_{\Z}$ and $\tilde{y}_s^{\tilde{r}} = 0$ otherwise, $\tilde{u}_s^{\tilde{r}} = 1$ for $s = \tilde{r}$ and $\tilde{u}_s^{\tilde{r}} = 0$ otherwise, and $\tilde{x}_s^{\tilde{r}} = \Clower$ for $s \in [\tilde{r}, \min \{\tilde{r}+L-1, T\}]_{\Z}$ and $\tilde{x}_s^{\tilde{r}} = 0$ otherwise.
\item For $\tilde{r}=t-L-1$ (totally one point), 
we create $(\tilde{x}^{\tilde{r}}, \tilde{y}^{\tilde{r}}, \tilde{u}^{\tilde{r}}) \in$ conv($P$) such that
$\tilde{y}_s^{\tilde{r}} = 1$ for $s \in [\tilde{r}, \min \{\tilde{r}+L-1, T\}]_{\Z}$ and $\tilde{y}_s^{\tilde{r}} = 0$ otherwise, $\tilde{u}_s^{\tilde{r}} = 1$ for $s = \tilde{r}$ and $\tilde{u}_s^{\tilde{r}} = 0$ otherwise, and $\tilde{x}_s^{\tilde{r}} = \Vupper$ for $s \in [\tilde{r},\tilde{r}+L-1]_{\Z}$ and $\tilde{x}_s^{\tilde{r}} = 0$ otherwise.
\item For $\tilde{r}=t-L$ (totally one point), 
we create $(\tilde{x}^{\tilde{r}}, \tilde{y}^{\tilde{r}}, \tilde{u}^{\tilde{r}}) \in$ conv($P$) such that
\sred{(i)} $\tilde{y}_s^{\tilde{r}} = 1$ for $s \in [\tilde{r}, \min \{\tilde{r}+L-1, T\}]_{\Z}$ and $\tilde{y}_s^{\tilde{r}} = 0$ otherwise, \sred{(ii)} $\tilde{u}_s^{\tilde{r}} = 1$ for $s = \tilde{r}$ and $\tilde{u}_s^{\tilde{r}} = 0$ otherwise, and \sred{(iii)} $\tilde{x}_s^{\tilde{r}} = \Clower$ for $s \in [\tilde{r},\tilde{r}+L-1]_{\Z} \setminus \{t-2\}$, $\tilde{x}_s^{\tilde{r}} = \Clower+V$ for $s = t-2$, and $\tilde{x}_s^{\tilde{r}} = 0$ otherwise.
\item For each $\tilde{r} \in [t-L+1,t-2]_{\Z}$ (totally $L-3$ points), we create $(\tilde{x}^{\tilde{r}}, \tilde{y}^{\tilde{r}}, \tilde{u}^{\tilde{r}}) \in$ conv($P$) such that
$\tilde{y}_s^{\tilde{r}} = 1$ for $s \in [\tilde{r}, \min \{\tilde{r}+L-1, T\}]_{\Z}$ and $\tilde{y}_s^{\tilde{r}} = 0$ otherwise, $\tilde{u}_s^{\tilde{r}} = 1$ for $s = \tilde{r}$ and $\tilde{u}_s^{\tilde{r}} = 0$ otherwise, and $\tilde{x}_s^{\tilde{r}} = \Vupper+(s-\tilde{r})V$ for $s \in [\tilde{r},t]_{\Z}$ for $s \in [\tilde{r},t]_{\Z}$ and $\hat{x}_s^{\hat{r}} = 0$ otherwise. Note here that we have $\tilde{r}+L-1 \in [t, \min \{t+L-3,T\}]_{\Z}=\{t\}$.
\item For $\tilde{r}=t-1$ (totally one point), 
we create $(\tilde{x}^{\tilde{r}}, \tilde{y}^{\tilde{r}}, \tilde{u}^{\tilde{r}}) \in$ conv($P$) such that
\sred{(i)} $\tilde{y}_s^{\tilde{r}} = 1$ for $s\in \{t-1,t\}$ and $\tilde{y}_s^{\tilde{r}} = 0$ otherwise, \sred{(ii)} $\tilde{u}_s^{\tilde{r}} = 1$ for $s = \tilde{r}$ and $\tilde{u}_s^{\tilde{r}} = 0$ otherwise,
and \sred{(iii)} $\tilde{x}_s^{\tilde{r}} = \Clower$ for $s =t-1$, $\tilde{x}_s^{\tilde{r}} = \Clower+V$ for $s = t$, and $\tilde{x}_s^{\tilde{r}} = 0$ otherwise.
\item For $\tilde{r}=t$ (totally one point), 
we create $(\tilde{x}^{\tilde{r}}, \tilde{y}^{\tilde{r}}, \tilde{u}^{\tilde{r}}) \in$ conv($P$) such that
 $\tilde{y}_s^{\tilde{r}} = 1$ for $s = t$ and $\tilde{y}_s^{\tilde{r}} = 0$ otherwise, $\tilde{u}_s^{\tilde{r}} = 1$ for $s = \tilde{r}$ and $\tilde{u}_s^{\tilde{r}} = 0$ otherwise,
and $\tilde{x}_s^{\tilde{r}} = \Vupper$ for $s = t$ and $\tilde{x}_s^{\tilde{r}} = 0$ otherwise.
\item For each $\hat{r} \in [1,t-m-1]_{\Z}$ (totally $[t-m-1]^+$ points), 
we create $(\hat{x}^{\hat{r}}, \hat{y}^{\hat{r}}, \hat{u}^{\hat{r}}) \in$ conv($P$) such that
$\hat{y}_s^{\hat{r}} = 1$ for $s \in [1,\hat{r}]_{\Z}$ and $\hat{y}_s^{\hat{r}} = 0$ otherwise, $\hat{u}_s^{\hat{r}} = 0$ for all $s \in [2,T]_{\Z}$, and $\hat{x}_s^{\hat{r}} = \Clower+\epsilon$ for $s \in [1,\hat{r}]_{\Z}$ and $\hat{x}_s^{\hat{r}} = 0$ otherwise.
\item For $\hat{r} = t-m$ (totally one point), 
we create $(\hat{x}^{\hat{r}}, \hat{y}^{\hat{r}}, \hat{u}^{\hat{r}}) \in$ conv($P$) such that
$\hat{y}_s^{\hat{r}}= 1$ for all $s \in [1, T]_{\Z}$, $\hat{u}_s^{\hat{r}} = 0$ for all $s \in [2,T]_{\Z}$, and $\hat{x}_s^{\hat{r}} = \Cupper-\epsilon$ for $s \in [1,T]_{\Z} \setminus \{t-1\}$ and $\hat{x}_s^{\hat{r}} = \Cupper-V-\epsilon$ otherwise.
\item For each $\hat{r} \in [t-m+1, t-3]_{\Z}$ (totally $[m-3]^+$ points), 
we create $(\hat{x}^{\hat{r}}, \hat{y}^{\hat{r}}, \hat{u}^{\hat{r}}) \in$ conv($P$) such that
$\hat{y}_s^{\hat{r}} = 1$ for $s \in [\hat{r}-L+1,\hat{r}]_{\Z}$ and $\hat{y}_s^{\hat{r}} = 0$ otherwise,
$\hat{u}_s^{\hat{r}} = 1$ for $s = \hat{r}-L+1$ and $\hat{u}_s^{\hat{r}} = 0$ otherwise, and 
$\hat{x}_s^{\hat{r}} = \Clower+\epsilon$ for $s \in [\hat{r}-L+1,\hat{r}]_{\Z}$ and $\hat{x}_s^{\hat{r}} = 0$ otherwise.
\item For $\hat{r} = t-1$ (totally one point), 
we create $(\hat{x}^{\hat{r}}, \hat{y}^{\hat{r}}, \hat{u}^{\hat{r}}) \in$ conv($P$) such that
\sred{(i)} $\hat{y}_s^{\hat{r}} = 1$ for $s \in [\hat{r}-L+1,\hat{r}]_{\Z}$ and $\hat{y}_s^{\hat{r}} = 0$ otherwise,
\sred{(ii)} $\hat{u}_s^{\hat{r}} = 1$ for $s = \hat{r}-L+1$ and $\hat{u}_s^{\hat{r}} = 0$ otherwise, 
and \sred{(iii)} $\hat{x}_s^{\hat{r}} = \Clower+\epsilon$ for $s \in [\hat{r},\hat{r}+L-1]_{\Z} \setminus \{t-2\}$, $\hat{x}_s^{\hat{r}} = \Clower+V+\epsilon$ for $s = t-2$, and $\hat{x}_s^{\hat{r}} = 0$ otherwise.
\item For $\hat{r} = t$ (totally one point), 
we create $(\hat{x}^{\hat{r}}, \hat{y}^{\hat{r}}, \hat{u}^{\hat{r}}) \in$ conv($P$) such that
\sred{(i)} $\hat{y}_s^{\hat{r}} = 1$ for $s \in \{t-1,t\}$ and $\hat{y}_s^{\hat{r}} = 0$ otherwise,
\sred{(ii)} $\hat{u}_s^{\hat{r}} = 1$ for $s = t-1$ and $\hat{u}_s^{\hat{r}} = 0$ otherwise,
and \sred{(iii)} $\hat{x}_s^{\hat{r}} = \Clower+\epsilon$ for $s=t-1$, $\hat{x}_s^{\hat{r}} = \Clower+V+\epsilon$ for $s = t$, and $\hat{x}_s^{\hat{r}} = 0$ otherwise.
\item For each $\bar{r} \in [1,t-m-1]_{\Z}$ (totally $[t-m-1]^+$ points), 
we create $(\bar{x}^{\bar{r}}, \bar{y}^{\bar{r}}, \bar{u}^{\bar{r}}) \in$ conv($P$) such that
$\bar{y}_s^{\bar{r}} = 1$ for $s \in [1,\bar{r}]_{\Z}$ and $\bar{y}_s^{\bar{r}} = 0$ otherwise, $\bar{u}_s^{\bar{r}} = 0$ for all $s \in [2,T]_{\Z}$, and $\bar{x}_s^{\bar{r}} = \Clower$ for $s \in [1,\bar{r}]_{\Z}$ and $\bar{x}_s^{\bar{r}} = 0$ otherwise.
\item For $\bar{r} = t-m$ (totally one point), 
we create $(\bar{x}^{\bar{r}}, \bar{y}^{\bar{r}}, \bar{u}^{\bar{r}}) \in$ conv($P$) such that
$\bar{y}_s^{\bar{r}}= 1$ for all $s \in [1, T]_{\Z}$, $\bar{u}_s^{\bar{r}} = 0$ for all $s \in [2,T]_{\Z}$, and $\bar{x}_s^{\bar{r}} = \Cupper$ for $s \in [1,T]_{\Z} \setminus \{t-1\}$ and $\bar{x}_s^{\bar{r}} = \Cupper-V$ otherwise. \\
Finally, we create the remaining $T-t+m = m$ points $(\bar{x}^{\bar{r}}, \bar{y}^{\bar{r}}, \bar{u}^{\bar{r}})$ ($\bar{r} \in [t-m+1, T]_{\Z}$). Without loss of generality, we let $S = \{n_1, \cdots, n_p, \cdots, n_q, \cdots, n_{|S|}\} \subseteq [t-m+1,t-1]_{\Z}$ \sred{and} $S' = S \cup \{t-m\}$. 
For notation convenience, we define $n_0 = t-m$ and $n_{|S|+1} = t$.
\item For each $n_p \in S'$, $p \in [0,|S|]_{\Z}$, we create $n_{p+1}-n_p - 1$ points $(\bar{x}^{\bar{r}}, \bar{y}^{\bar{r}}, \bar{u}^{\bar{r}})  \in$ conv($P$) with $\bar{r} \in [n_p+1, n_{p+1}-1]_{\Z}$ (totally there are $\sum_{p=0}^{|S|} n_{p+1}-n_p - 1 = m-|S|-1$ points) such that  
$\bar{y}_s^{\bar{r}} = 1$ for $s \in [n_p-L+1, \bar{r}]_{\Z}$ and $\bar{y}_s^{\bar{r}} = 0$ otherwise, $\bar{u}_s^{\bar{r}} = 1$ for $s = n_p-L+1$ and $\bar{u}_s^{\bar{r}} = 0$ otherwise.
For the value of $\bar{x}^{\bar{r}}$, \sred{we have the following three cases:}
	\begin{enumerate}[(1)]
	\item if $\bar{r} \leq t-3$, then we let $\bar{x}_s^{\bar{r}} = \Clower$ for $s \in [n_p-L+1, \bar{r}]_{\Z}$ and $\bar{x}_s^{\bar{r}} = 0$ otherwise;
	\item if $\bar{r} = t-2$, then we let $\bar{x}_s^{\bar{r}} = \Vupper$ for $s \in [n_p-L+1, \bar{r}]_{\Z}$ and $\bar{x}_s^{\bar{r}} = 0$ otherwise;
	\item if $\bar{r} = t-1$, then we let  $\bar{x}_s^{\bar{r}} = \Clower$ for $s \in [n_p-L+1, \bar{r}]_{\Z} \setminus \{t-1\}$, $\bar{x}_s^{\bar{r}} = \Clower+V$ for $s = t-1$, and $\bar{x}_s^{\bar{r}} = 0$ otherwise.	
	\end{enumerate}
\item For each $n_p \in S \cup \{t\}$, $p \in [1, |S|+1]_{\Z}$, we create one point $(\bar{x}^{\bar{r}}, \bar{y}^{\bar{r}}, \bar{u}^{\bar{r}}) \in$ conv($P$) \sred{with} $\bar{r} = n_p$ (totally there are $|S|+1$ points) such that
$\bar{y}_s^{\bar{r}} = 1$ for $s \in [n_{p-1}-L+1, T]_{\Z}$ and $\bar{y}_s^{\bar{r}} = 0$,
$\bar{u}_s^{\bar{r}} = 1$ for $s = n_{p-1}-L+1$ and $\bar{u}_s^{\bar{r}} = 0$ otherwise, and 
$\bar{x}_s^{\bar{r}} = \Vupper + (s-n_{p-1}+L-1)V$ for $s \in [n_{p-1}-L+1, t]_{\Z}$ and $\bar{x}_s^{\bar{r}} = 0$ otherwise.
\end{enumerate}

In summary, we create $3T-2$ points $(\bar{x}^{\bar{r}}, \bar{y}^{\bar{r}}, \bar{u}^{\bar{r}})_{\bar{r}=1}^{T}$, $(\hat{x}^{\hat{r}}, \hat{y}^{\hat{r}}, \hat{u}^{\hat{r}})_{\hat{r}=1, \hat{r} \neq t-2}^{T}$, and $(\tilde{x}^{\tilde{r}}, \tilde{y}^{\tilde{r}}, \tilde{u}^{\tilde{r}})_{\tilde{r}=2}^{T}$.
It is easy to see that they are valid and satisfy \eqref{eqn:ru-3-exp} at equality.
They are clearly linearly independent since $(\bar{x}, \bar{y}, \bar{u})$ and $(\hat{x}, \hat{y}, \hat{u})$ can construct a lower-triangular matrix in terms of the values of $x$ and $y$ and $(\tilde{x}, \tilde{y}, \tilde{u})$ is further linearly independent with them since it constructs an upper-triangular matrix in terms of the value of $u$.
\end{proof}

\subsection{Proof for Proposition \ref{prop:rd-3-exp}} \label{apx:subsec:rd-3-exp}
\begin{proof}
\textbf{(Validity)} It is easy to show that \eqref{eqn:rd-3-exp} is valid when there exists $i \in [t-2, t]_{\Z}$ such that $y_i = 0$. Now we consider the case in which $y_{t-2} = y_{t-1} = y_t = 1$ and let the start-up time before $t-2$ be $t-s$ ($s\geq 2$) and the shut-down time after $t$ be $t+\hat{s}$ ($\hat{s} \geq 1$). We discuss the following three possible cases in terms of the value of $t-s+L-1$.
\begin{enumerate}[1)]
\item $t-s+L-1 \leq t$. We discuss the following three possible cases in terms of the value of $t+\hat{s}$.
\begin{enumerate}[(1)]
\item $t+\hat{s} \in [t+1, \hat{t}-1]_{\Z}$. Inequality \eqref{eqn:rd-3-exp} converts to $x_{t-2} - x_{t-1} + x_t \leq V + \Vupper + (\hat{t}-1)V$, which is valid since $x_{t-2} - x_{t-1} \leq V$ and $x_t \leq \Vupper + (\hat{t}-1)V$.
\item $t+\hat{s} \in [\hat{t}, t+m+1]_{\Z}$. Inequality \eqref{eqn:rd-3-exp} converts to $x_{t-2} - x_{t-1} + x_t \leq V + \Vupper + (\hat{t}-1)V + \tilde{f}V$, which is valid since $x_{t-2} - x_{t-1} \leq V$ and $x_t \leq \Vupper + (\hat{s}-1)V = \Vupper + (\hat{t}-1)V + (\hat{s}-\hat{t})V \leq \Vupper + \sum_{i \in S_0}V + V \sum_{(S\cup\{\hat{t}\}) \cap [\hat{t}+1,t+\hat{s}-1]_{\Z}} = \Vupper + (\hat{t}-1)V + \tilde{f}V$.
\item $t+\hat{s} \geq t+m+2$. Inequality \eqref{eqn:rd-3-exp} converts to $x_{t-2} - x_{t-1} + x_t \leq V + \Cupper$, which is valid.
\end{enumerate}
\item $t-s+L-1 \geq t+1$, i.e., $s \in [2, L-2]_{\Z}$. It follows that $\phi = (s-2)V$ and inequality \eqref{eqn:rd-3-exp} converts to $x_{t-2} - x_{t-1} + x_t \leq \Vupper + (s-2)V + V + \psi$, which is valid since $x_{t-2} \leq \Vupper + (s-2)V$, $x_t - x_{t-1} \leq V$, and $\psi = V \sum_{i \in S_0} (y_i - \sum_{m=0}^{\min \{L-1,i-2\}} u_{i-m}) + V \sum_{i \in S \cup \{\hat{t}\}} (d_i - i) (y_i - \sum_{m=0}^{L-1} u_{i-m}) + (\Cupper - \Vupper - m V) (y_{t+m+1} - \sum_{j=0}^{L-1} u_{t+m+1-j}) \geq 0$.
\end{enumerate}

\textbf{(Facet-defining)} The facet-defining proof is similar with that in Appendix \ref{apx:subsec:ru-3-exp} for Proposition \ref{prop:ru-3-exp} and thus is omitted here.
\end{proof}

\end{appendices}

\end{document}